\documentclass{amsart}

\usepackage{amsmath,amssymb,amsthm, mathrsfs, mathtools, bm, amsfonts}
\usepackage{mathabx}
\usepackage{amscd,mathtools}
\usepackage[utf8]{inputenc}
\usepackage[english]{babel}
\usepackage{cite}
\usepackage{color}
\usepackage[pagebackref,colorlinks,citecolor=blue,linkcolor=red]{hyperref}

\renewcommand*{\backrefalt}[4]{\ifcase #1 (Not cited).\or (Cited p.~#2).\else (Cited pp.~#2).\fi} 
\usepackage{float}
\usepackage{tikz}
\usepackage{lmodern}
\usetikzlibrary{decorations.pathmorphing}

\usepackage{multicol}
\tikzset{snake it/.style={decorate, decoration=snake}}
\usetikzlibrary{automata}
\usetikzlibrary{cd}
\usepackage[margin=3cm]{geometry}

\usepackage{comment}
\usepackage{mathtools}

\usepackage{csquotes}

\usepackage{amsthm}

\usepackage{tikz}
\usetikzlibrary{calc}
\usetikzlibrary{arrows}
\usetikzlibrary{decorations.pathreplacing}
\usetikzlibrary{intersections}

\usepackage{hyperref}
\usepackage{tikz}
\usepackage{caption, subcaption}
\usepackage{tikz-cd}
\usetikzlibrary{calc,decorations.pathreplacing}
\usepackage{tikz}
\usetikzlibrary{calc}
\usetikzlibrary{arrows}
\usetikzlibrary{decorations.pathreplacing}
\usetikzlibrary{intersections}

 \newtheorem{theorem}{Theorem}[section]
  \newtheorem{proposition}[theorem]{Proposition}

  \newtheorem{corollary}[theorem]{Corollary}
  \newtheorem{lemma}[theorem]{Lemma}

  \theoremstyle{definition}
  \newtheorem{definition}[theorem]{Definition}
  \newtheorem{claim}[theorem]{Claim}
  
  \newtheorem*{claim*}{Claim}

    \newtheorem{notation}[theorem]{Notation}

  \newtheorem*{question*}{Question}
  \newtheorem*{answer*}{Answer}
  \newtheorem*{application*}{Application}

  \theoremstyle{remark}
  \newtheorem{remark}[theorem]{Remark}
  \newtheorem*{remark*}{Remark}

\usepackage{amssymb}

\newcommand{\id}{\text{id}}

\DeclarePairedDelimiterX{\Norm}[1]{\lVert}{\rVert}{#1}
  
\theoremstyle{definition}

  \newcommand{\hh}{{\sf h}}

    \newcommand{\RR}{{\mathbb{R}}}

  \newcommand{\gothic}{\mathfrak}
  \newcommand{\go}{{\gothic o}}
 \newcommand{\calA}{\mathcal{A}}
  \newcommand{\calB}{\mathcal{B}}
  \newcommand{\calC}{\mathcal{C}}
  \newcommand{\calD}{\mathcal{D}}
  
  \newcommand{\calF}{\mathcal{F}}
  
  \newcommand{\calH}{\mathcal{H}}

  \newcommand{\calN}{\mathcal{N}}

  \newcommand{\calQ}{\mathcal{Q}}
  
  \newcommand{\calS}{\mathcal{S}}
  
  \newcommand{\calU}{\mathcal{U}}
    \newcommand{\calV}{\mathcal{V}}
  \newcommand{\calW}{\mathcal{W}}
  \newcommand{\calX}{\mathcal{X}}
    \newcommand{\calY}{\mathcal{Y}}
  \newcommand{\calZ}{\mathcal{Z}}
      \newcommand{\oV}{\overline{\mathcal{V}}}
            
  \newcommand{\ZZ}{\mathbb Z}
  \newcommand{\Dual}{\calD}

    \newcommand{\MCG}{\mathcal{MCG}}

    \newcommand{\hT}{\widehat{T}}
    \newcommand{\hx}{\hat{x}}
     \newcommand{\ha}{\hat{a}}
     
         \newcommand{\hmed}{\hat{\mathfrak{m}}}
     
      \newcommand{\hb}{\hat{b}}
          \newcommand{\hc}{\hat{c}}
       \newcommand{\hz}{\hat{z}}
    \newcommand{\hy}{\hat{y}}
    \newcommand{\hd}{\hat{\delta}}

    \newcommand{\hw}{\hat{w}}
    \newcommand{\hlam}{\hat{\lambda}}
    \newcommand{\om}{\omega}
    \newcommand{\hO}{\widehat{\Omega}}
    
    \newcommand{\hpsi}{\hat{\psi}}
    \newcommand{\hPsi}{\widehat{\Psi}}
        \newcommand{\hPhi}{\widehat{\Phi}}
    \newcommand{\hf}{\hat{f}}
    \newcommand{\hY}{\widehat{\mathcal Y}}
    \newcommand{\BM}{\calB_{\min}}
    
    \renewcommand{\hh}{\mathbf{h}}
    \newcommand{\PP}{\mathbf{P}}
    \newcommand{\hpi}{\widehat{\pi}}
    \newcommand{\Dust}{\mathrm{Dust}}
    \newcommand{\bp}{\mathfrak{p}}
    
    \newcommand{\bF}{\mathbb{F}}
    \newcommand{\bE}{\mathbb{E}}

       \newcommand{\oZ}{\overline{\mathcal Z}}
    \newcommand{\oQ}{\overline{\mathcal Q}}
    \newcommand{\oY}{\overline{\mathcal Y}}

    \newcommand{\Rel}{\mathrm{Rel}}
  
  \newcommand{\dist}{\mathbf{d}}

\newcommand{\diam}{\mathrm{diam}}
\newcommand{\hull}{\mathrm{hull}}
\newcommand{\med}{\mathfrak{m}}
\newcommand{\gate}{\mathfrak{g}}
\newcommand{\Neb}{\mathcal{N}}
\newcommand{\ret}{\mathrm{ret}}
\newcommand{\Real}{\mathrm{Real}}
\newcommand{\supp}{\mathrm{supp}}

\newcommand{\nest}{\sqsubset}

\newtheorem{thmi}{Theorem}
\newtheorem{cori}[thmi]{Corollary}

\newtheorem{propi}[thmi]{Proposition}

\begin{document}

\title[Cubulating infinity in HHSes]{Cubulating infinity in hierarchically hyperbolic spaces}


 \author   {Matthew Gentry Durham}
 \address{Department of Mathematics, University of California, Riverside, CA }
 \email{mdurham@ucr.edu}

\begin{abstract}
We prove that the hierarchical hull of any finite set of interior points, hierarchy rays, and boundary points in a hierarchically hyperbolic space (HHS) is quasi-median quasi-isometric to a CAT(0) cube complex of bounded dimension.  Our construction extends and refines a theorem of Behrstock-Hagen-Sisto about modeling hulls of interior points and our previous work with Zalloum on modeling finite sets of rays via limits of these finite models.

We further prove that the quasi-median quasi-isometry between the hull of a finite set of rays or boundary points and its cubical model extends to an isomorphism between their respective hierarchical and simplicial boundaries.  In this sense, we prove that the hierarchical boundary of any proper HHS is locally modeled by the simplicial boundaries of CAT(0) cube complexes.  This is a purely geometric statement, allowing one to important various topologies from the cubical setting.

Our proof of the cubical model theorem is new, even for the interior points case.  In particular, we provide a concrete description of the cubical model as a cubical subcomplex of a product of simplicial trees into which the hierarchical data is directly encoded.  Moreover, the above boundary isomorphism is new for all non-cubical HHSes, including mapping class groups and Teichm\"uller spaces of finite-type surfaces.

As an application of our techniques, we show that in most HHSes, including all hierarchically hyperbolic groups, the distance between any pair of points in the top-level hyperbolic space is coarsely the length of a maximal $0$-separated chain of hyperplanes separating them in an appropriate cubical model.  For mapping class groups, this says that these cubical models cubically encode distance in the curve graph of the surface.
\end{abstract}

\maketitle

\section{Introduction} \label{sec:intro}

The hierarchy machinery of Masur-Minsky \cite{MM99, MM00}, powered by techniques from hyperbolic geometry, revolutionized our understanding of the geometry of mapping class groups and Teichm\"uller spaces.  In recent years, an analogous influx of ideas from CAT(0) cubical complexes has transformed the area, leading to a similar seismic shift.

At the center of this recent wave is the \emph{cubical model machinery} of Behrstock-Hagen-Sisto \cite{HHS_quasi}.  Roughly speaking, they proved that the coarse convex hulls of any finite set of interior points in any \emph{hierarchically hyperbolic space} \cite{HHS_1}, such as the mapping class group, are quasi-median quasi-isometric to CAT(0) cube complexes (see also Bowditch \cite{Bowditch_hulls}).  This is a higher-rank generalization of the fact that coarse hulls of finite sets of points in hyperbolic spaces are quasi-trees.  These cubical models have proven to be quite useful.  They were essential in the resolution of Farb's quasi-flats conjecture \cite{HHS_quasi}, proofs of semihyperbolicity \cite{DMS20, HHP} and coarse injectivity \cite{HHP} of the mapping class group, and inspired Petyt's proof that mapping class groups are quasi-isometric to cube complexes \cite{Petyt_quasicube} (though necessarily non-equivariantly \cite{Bridson_notCCC}), among other results \cite{ABD, DZ22}.

The main purpose of this article is to extend this cubical machinery to model the hulls of finitely-many hierarchy rays and boundary points in a natural compactification of any proper HHS called the \emph{hierarchical boundary} \cite{DHS_boundary}, which we introduced with Hagen and Sisto.  Moreover, we provide a local description of this boundary via these cubical models.  We summarize our main results as follows:

\vspace{.1in}
\framebox{
\begin{minipage}{6in}
\begin{thmi}\label{thmi:main informal}
In any (proper) HHS, the hierarchical hull of any finite set of interior points, hierarchy rays, and boundary points admits a quasi-median quasi-isometry to a bounded dimensional CAT(0) cube complex.  Moreover, this map extends to a simplicial isomorphism between the simplicial boundary of the cubical model and the hierarchical boundary of the hull.
\end{thmi}
\end{minipage}
}
\vspace{.1in}

In fact, our proof of the cubical model theorem is entirely new, even in the interior point case.  Our construction is quite different from those contained in \cite{HHS_quasi, Bowditch_hulls}, and provides a wealth of new hierarchical information, especially in the infinite case as compared with the limiting cubical models for hierarchy rays that we constructed with Zalloum \cite{DZ22} (see Remark \ref{rem:main rem}).  Moreover, the construction is quite explicit and hands-on, and this alternative perspective has some added benefits, which we will describe below.

This article lays a broad technical foundation for future work using cubical techniques to analyze the geometry of hierarchically hyperbolic spaces and their hierarchical boundaries.   This includes further stabilizations of these cubical models along the lines of our work with Minsky and Sisto \cite{DMS20}, as well as metric and topological properties of hierarchical boundaries, neither of which are discussed at any length in this article.  Infinite diameter objects in HHSes can be extremely complicated, and this complexity accounts for the vast majority of the work in this paper.  In a future expository paper, we will explain how to prove the stable cubical model theorem from \cite{DMS20} in the relatively simple case of modeling the hull of a pair of internal points, which is enough to recover a proof of semihyperbolicity of the mapping class group along the lines of \cite{DMS20}.

The rest of the introduction provides the context of our results, a detailed discussion of the construction itself, and also gives some applications of our refined techniques.  We begin with a brief discussion of the relevant notions from the world of HHSes, then a comparative discussion between the cubical model construction from \cite{HHS_quasi} (Subsection \ref{subsec:BHS cube}) and our own (Subsection \ref{subsec:Dur cube}).  With this foundation, we then discuss some applications, including our simplicial boundary results (Subsection \ref{subsec:infinity intro}) and our ability to extract curve graph data from the cubical models (Subsection \ref{subsec:extract}).  The rest of the introduction then discusses the new technical aspects of our results.  We begin in Subsection \ref{subsec:ray example} by working a detailed example motivating our new hierarchical accounting techniques (Subsection \ref{subsec:PU intro}), and a refined version of the distance-formula (Subsection \ref{subsec:TT intro}).  We end with a detailed outline of how it all fits together (Subsection \ref{subsec:sketch}).

\subsection{A brief HHS primer}\label{subsec:HHS primer}

Building off of work of Masur-Minsky \cite{MM99,MM00} for mapping class groups, Brock \cite{Brock_WP}, Rafi \cite{Raf07}, and the author \cite{Dur16} for Teichm\"uller spaces, and Hagen \cite{Hagen_contact} and Kim-Koberda \cite{KK_HHS} for right-angled Artin groups, Behrstock-Hagen-Sisto \cite{HHS_1} introduced the axiomatic notion of a \emph{hierarchically hyperbolic space} (Definition \ref{defn:HHS}) to capture some of the coarse geometric features common to all of these objects.  This abstract perspective has proven to be quite useful, and, at the time of writing, the class of HHSes has expanded considerably to include:

\begin{itemize}
\item Hyperbolic spaces.
\item Mapping class groups \cite{MM99, MM00, Aougab_hyp, HHS_1} and Teichm\"uller spaces, with both the Teichm\"uller \cite{Raf07, Dur16} and Weil-Petersson metrics \cite{Brock_WP}.
\item Many \cite{HagenSusse}, but not all \cite{Shep_noHHS}, CAT(0) cubical groups,  including all right-angled Artin and Coxeter groups \cite{HHS_1}.
\item Fundamental groups of closed 3-manifolds without Nil or Sol summands \cite{HHS_2, hagen2022equivariant}.
\item Quotients of mapping class groups by large powers of Dehn twists \cite{BHMS_combo} and pseudo-Anosovs \cite{HHS_asdim}.
\item Surface group extensions of lattice Veech groups \cite{DDLS1, DDLS2} and multicurve stabilizers \cite{russell2021extensions} of mapping class groups.
\item  The genus-2 handlebody group \cite{chesser2022stable}.
\item Artin groups of extra-large type \cite{HMS21}.
\item The class of HHSes is closed under products and appropriate combinations \cite{HHS_2}, including graph products \cite{BR_combo} and graph braid groups  \cite{berlyne2021hierarchical}.
\end{itemize}

The general philosophy is that every HHS $\calX$ is built from an indexed family of uniform $\delta$-hyperbolic spaces $\{\calC(U)\}_{U \in \mathfrak S}$, in the follow sense.  The ambient space $\calX$ has a uniform coarse Lipschitz projection $\pi_U: \calX \to \calC(U)$ to each $\calC(U)$, which function like a closest point projections to quasiconvex subsets in a number of ways.  These combine into a global map to the full product of the associated hyperbolic spaces
$$\Pi: \calX \to \prod_{U \in \mathfrak S} \calC(U).$$

The elements of the index set $\mathfrak S$, called \emph{domains}, satisfy certain mutually exclusive \emph{relations}---\emph{orthogonality} $\perp$, \emph{nesting} $\nest$, and \emph{transversality} $\pitchfork$.  These relations generalize the various ways (isotopy classes of) subsurfaces can interact---namely when $U,V$ are essential subsurfaces, then $U \perp V$ stands for disjointness, $U \nest V$ for $U$ being a subsurface of $V$, and $U \pitchfork V$ when neither of those holds.  The nesting and transverse relations determine certain \emph{consistency inequalities} (Definition \ref{defn:consistency}) which control the tuples appearing in the image of $\Pi$, in the sense that every image tuple satisfies these inequalities, while every tuple satisfying the inequalities can be coarsely \emph{realized} as the image of a point (Theorem \ref{thm:realization}) \cite{BKMM, HHS_1}.  On the other hand, orthogonal domains have mutually unconstrained coordinates and determine \emph{product regions} (Subsection \ref{subsec:product region}), which are infinite-diameter quasi-isometrically embedded products coarsely generalizing the product structure of the thin parts of Teichm\"uller space discovered by Minsky \cite{Minsky_prod}.

Importantly, there is a Masur-Minsky-style \cite{MM00, HHS_2} \emph{distance formula} (Theorem \ref{thm:DF}), 
\begin{equation}\label{eq:DF intro}
d_{\calX}(a,b) \asymp \sum_{U \in \mathfrak S} [[d_{\calC(U)}(\pi_U(a), \pi_U(b))]]_K
\end{equation}
which says that the distance between any pair of points $a,b \in \calX$ is coarsely equal to the $\ell^1$-distance in the above product, as long as we ignore terms smaller than some uniform threshold $K = K(\calX)>0$.  Any pair of points is connected by a \emph{hierarchy path} \cite{MM00, Dur16, HHS_1}, a nice uniform quasi-geodesic which projects to an unparameterized quasi-geodesic in each hyperbolic space (Definition \ref{defn:hp}).  Notably, we will recover new proofs of the distance formula and the existence of hierarchy paths in this paper via our refined cubical model techniques (Corollary \ref{cori:DF and HP}), as Bowditch \cite{Bowditch_hulls} does in his alternate construction of cubical models; see Subsection \ref{subsec:related results} for a comparison of his related but very different median-based approach.

The \emph{hierarchical} or \emph{HHS boundary} $\partial \calX$ of $\calX$ \cite{DHS_boundary, DHS_corr} is built out of the Gromov boundaries of the various hyperbolic spaces in the hierarchy.  The boundary has an underlying simplicial structure, where each $0$-simplex is a point in $\partial \calC(U)$ for some $U \in \mathfrak S$, with higher dimensional simplices corresponding to collections of $0$-simplices supported on pairwise orthogonal domains.  Besides Gromov boundaries themselves, one motivation for this construction was Thurston's compactification of the Teichm\"uller space of a finite-type surface by projective classes of measured laminations on the surface.  The HHS boundary of Teichm\"uller space with the Teichm\"uller metric is morally a combinatorial version of this, where one forgets the measure classes on the minimal components and only remembers the topological (or combinatorial) information of the laminations and the pairwise disjoint subsurfaces on which these minimal components are supported.  Like Thurston's boundary for Teichm\"uller space, the HHS boundary compactifies any proper HHS \cite[Theorem 3.4]{DHS_boundary}, however we focus on the simplicial structure in this article and will not deal with a choice of topology.  Moreover, automorphisms of $\calX$ act by homeomorphisms on the HHS boundary of $\calX$ (see Definition \ref{defn:HHS auto}), a fact which proved quite useful in \cite{DHS_boundary} for proving algebraic results like the Tits alternative for hierarchically hyperbolic groups.

Notably, when an HHS $\calX$ is a CAT(0) cube complex, then there is a simplicial isomorphism from the simplicial boundary \cite{Hagen_simplicial} of the cube complex to its (untopologized) hierarchical boundary \cite[Theorem 10.1]{DHS_boundary}.  Theorem \ref{thmi:main boundary} below gives a local version of this theorem for any proper HHS.

Finally, one can take a \emph{hierarchical hull} of any subset $A \subset \calX$ in an HHS $\calX$ (Definition \ref{defn:hier hull}).  Roughly, the hull of a set $A$ is a coarse object consisting of points in $\calX$ which project close to the hyperbolic hull, $\hull_{\calC(U)}(\pi_U(A))$, for all $U \in \mathfrak S$.  Hierarchical hulls satisfy a notion of quasi-convexity called \emph{hierarchical quasi-convexity} \cite{HHS_2} (Definition \ref{defn:hqc}), which directly generalizes the notion of quasi-convexity from hyperbolic spaces and coincides with median convexity in the HHS setting \cite{RST18}.  The hulls of finite sets of interior points \cite{BKMM, EMR_rank, HHS_2, HHS_quasi}, hierarchy rays \cite{DZ22}, and boundary points (Definition \ref{defn:hier hull}) are particularly nice, and in many ways encode all of the hierarchical geometry between the objects in question.  These hulls are the main objects we study in this paper.

\subsection{The original cubical model construction}\label{subsec:BHS cube}

The main construction in this paper is a new proof of Behrstock-Hagen-Sisto's cubical model theorem \cite[Theorem 2.1]{HHS_quasi} which also works for points at infinity.  See also Bowditch \cite{Bowditch_hulls} for an alternative proof of the interior point case in the slightly more general setting of coarse median spaces satisfying certain hierarchical axioms (see also Subsection \ref{subsec:related results}).  We will now briefly discuss their construction, before moving on to explain our version.

For the purpose of this discussion, let $F \subset \calX$ be a finite subset of interior points in an HHS.  One of the basic properties of an HHS is that there exists a constant $K = K(\calX)>0$ so that the set $\calU$ of \emph{relevant domains} $U \in \mathfrak S$ where $F$ has a large diameter projection $\calU = \{U \in \mathfrak S| \diam_{\calC(U)}(F)>K\}$ is finite (Corollary \ref{cor:rel sets are finite}).  This allows us to work only with the domains in $\calU$.  Projecting $F$ to $\calC(U)$ for each $U \in \calU$, we can take the hyperbolic hull $H_U = \hull_{\calC(U)}(\pi_U(F)) \subset \calC(U)$.  The \emph{hierarchical hull} $H = \hull_{\calX}(F)$ of $F$ can defined to be the set of points in $\calX$ which project sufficiently close to $H_U$ in each $U \in \calU$ (see Definition \ref{defn:hier hull}).

Behrstock-Hagen-Sisto's approach to cubulating $H$ proceeds as follows: To each hyperbolic hull $H_U$, we can associate its Gromov modeling tree $\phi_U:T_U \to \calC(U)$.  If $\#\calU = 1$ (which can be arranged when $\calX$ is hyperbolic), then this tree would suffice as a cubical model, but in general one needs a higher-dimensional model because of the product regions in $\calX$.  At this point, one might hope to use the fact that $H$ admits a geometrically meaningful map $H \to \prod_{U \in \calU} T_U$ to a finite product of trees, which is morally a CAT(0) cube complex, though not exactly as the $T_U$ need not be simplicial trees.  However, the consistency inequalities (Definition \ref{defn:consistency}) constrain the image---this map is usually abritrarily far from being coarsely surjective.  The idea is to induce a wall-space structure on $H$ from a carefully defined collection of ambient ``walls" in the product which are defined in terms of the trees.

The challenge is that domains in $\calU$ are usually related in a complicated fashion, and any given $U \in \calU$ usually contains many domains which \emph{nest} ($\nest$) into $U$ (recall that nesting corresponds to subsurface inclusion in the surface case).  Each domain $V \nest U$ with $V \in \calU$ not only has its own hyperbolic hull $H_U$ and modeling tree $T_U$, but also is encoded in $\calC(U)$ by its \emph{$\rho$-set} $\rho^V_U$  (e.g., its boundary curves in the surface case) which are located near $H_U$ (and hence $\phi_U(T_U)$) in $\calC(U)$.  From the beginning \cite{MM00}, these $\rho$-sets have been understood to encode important organizational information about how the various pieces of the hierarchy fit together.

\begin{figure}
    \centering
    \includegraphics[width=.6\textwidth]{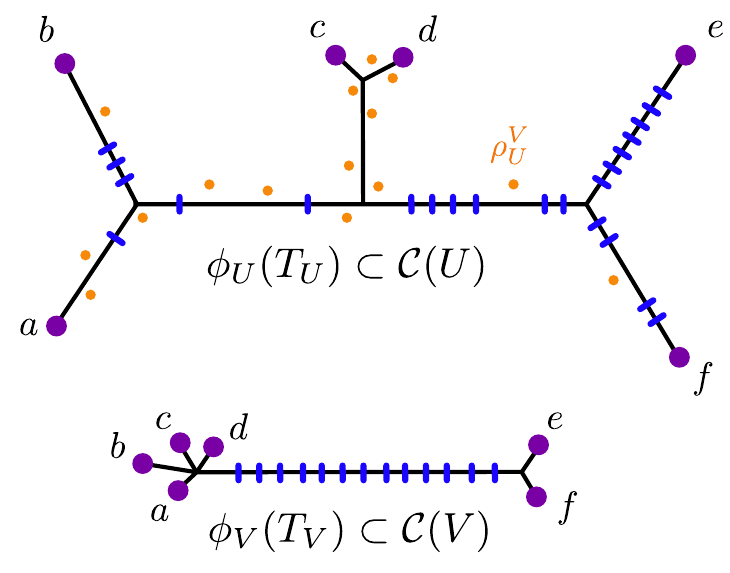}
    \caption{A cartoon of the tree-walls in the Behrstock-Hagen-Sisto construction of cubical model for the hierarchical hull of $H = \hull_{\calX}(\{a,b,c,d,e,f\})$: The tree $\phi_U(T_U) \subset \calC(U)$ models the hyperbolic hull $H_U = (\pi_U(\{a,b,c,d,e,f\}))$.  One chooses the tree-walls (in blue) to coarsely encode lengths along the subtrees of $\phi_U(T_U)$ which are far from the $\rho$-points (gold) and leaves (purple).  One then pulls back these tree-walls to obtain a wall-space on $H$, with the cubical model for $H$ being the Sageev dual cube complex to this wallspace.  Below is the setup for a domain $V \in \calU$ with $V \nest_{\calU} U$, i.e. no smaller domain in $\calU$ nests into $V$.  In this case, since $\rho^V_U$ is far from any branched point or point of $F$, the Bounded Geodesic Image Property \ref{ax:BGIA} forces this tree $\phi_V(T_V)$ to be coarsely an interval, and the lack of any nesting domains means that the number of tree-walls is coarsely the length of the interval.}
    \label{fig:BHS_cube}
\end{figure}

One of the key observations in \cite{HHS_quasi} is that, from the perspective of $U$, the information coming from $\rho$-sets in $U$ is redundant in a strong way.  In particular, the $\rho$-sets coming from nested domains $V \nest U$ can coarsely cover arbitrarily large portions of the tree $\phi_U(T_U)$ (including possibly all of it).  Each $\rho^V_U$ for $V \nest U$ also represents a large-diameter projection of $F$ to $\calC(V)$, making the small-diameter $\rho$-sets in $\calC(U)$ redundant.  The construction in \cite{HHS_quasi} involves carefully picking a family of tree-walls in $\phi_U(T_U)$ in the complement of these $\rho$-sets, and then pulling them back to walls on the hull $H \subset \calX$ using projections coming from the hierarchical setup.  See Figure \ref{fig:BHS_cube} for a heuristic picture.

These tree-walls allow them to build a wall-space structure on the hull $H$ as follows: Each tree-wall in $\phi_U(T_U)$ determines two half-trees of $\phi_U(T_U)$, which pull back to a wall and two half-spaces in $H$ under (the inverses of) the domain projection $\pi_U:H \to \calC(U)$ and closest-point projection $p_U:\calC(U) \to \phi_U(T_U)$.  They then directly prove that the dual cube complex $\calW$ (provided by Sageev \cite{Sageev_machine}) to this wall-space admits a quasi-median quasi-isometry $\mathfrak{p}: \calW \to H$.  The map $\mathfrak{p}$ is constructed coordinate-wise for each $U \in \calU$ by taking a consistent tuple of half-spaces (in the wall-space sense), projecting them to $\calC(U)$, and then intersecting over all of half-space projections.  They then show that this tuple is consistent (in the hierarchical sense), and that it defines a quasi-isometry, both of which are nontrivial tasks. 

One significantly complicating feature of this construction is the necessity of translating cubical techniques in the model to hierarchical facts in the hull via the map described above, which, for instance, notably elongated arguments in our paper with Minsky and Sisto (see especially \cite[Section 5]{DMS20}).  One of the main technical advantages of our alternate approach is an intermediate object which encodes a wealth of hierarchical information in an explicitly cubical form, thus acting as an interface between these two perspectives.  We explain how to build this bridge next.

\subsection{Cubulating hulls via products of trees} \label{subsec:Dur cube}

\begin{figure}
    \centering
    \includegraphics[width=1\textwidth]{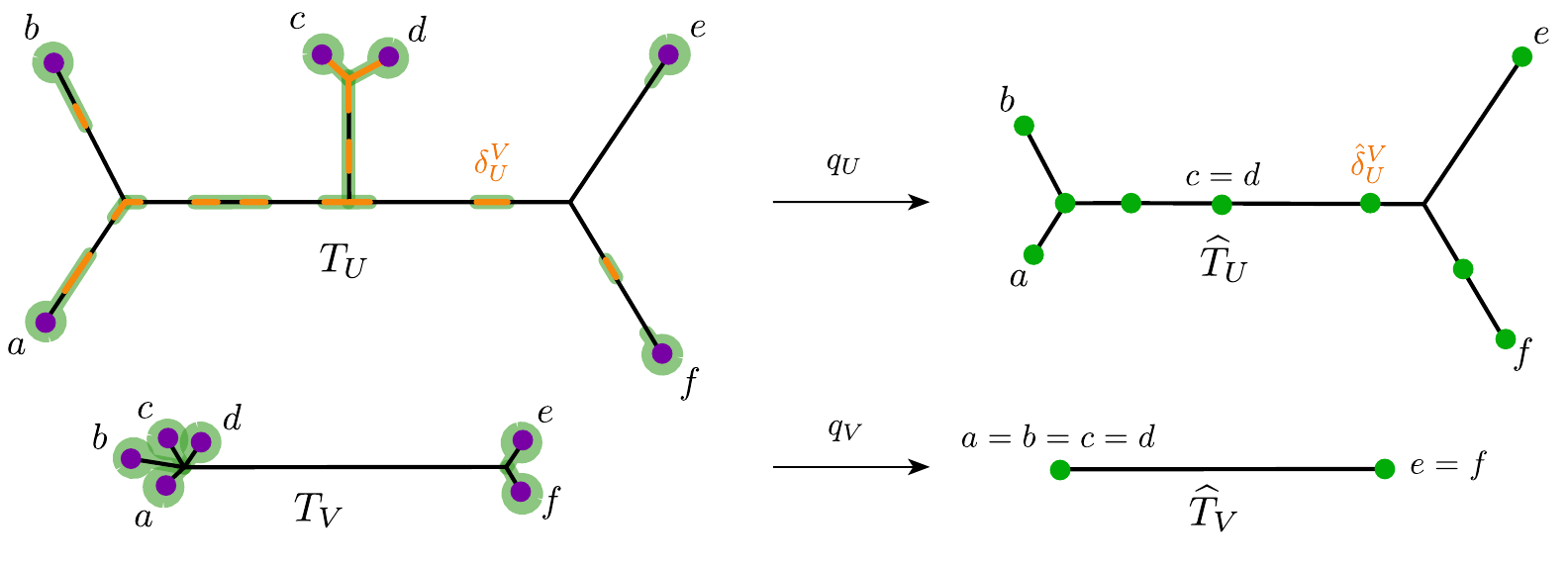}
    \caption{A cartoon of cluster formation and the collapsed trees in a hierarchical family of trees coming the setup in the example in Figure \ref{fig:BHS_cube}: Instead of working in the hyperbolic space $\calC(U)$, we work directly in the (abstract) modeling trees, $T_U$.  Here, we identify \emph{shadow} subtrees (gold) of $T_U$ using the $\rho$-set data in $\calC(U)$, group nearby shadows and leaves into \emph{cluster} subtrees (green), and then collapse down these cluster subtrees to obtain the collapsed tree $q_U: T_U \to \hT_U$.  This process results in special \emph{marked points} (green) on the collapsed trees $\hT_U$, which can encode both leaf data and $\rho$-set data (denoted by $\hd$).  These cluster subtrees can have arbitrarily large diameter and collapsing them can lead to the identification of multiple leaves.  For the domain $V \nest U$ which is $\nest$-minimal in $\calU$, the corresponding collapsed tree $\hT_V$ is an interval of coarsely the same length as the diameter of $T_V$.  In fact, almost every domain in $\calU$ is one of these two-sided \emph{bipartite} domains (Definition \ref{defn:bipartite}), a fact which plays a key role in our arguments.}
    \label{fig:Dur_cube}
\end{figure}

The starting point of our construction of the cubical model for the hull $H$ is a different conclusion from their observation that $\rho$-sets encode redundant information.  Instead of ignoring how the $\rho$-set information is encoded in the images of the trees $\phi_U(T_U) \subset \calC(U)$, we pull them back to the (abstract) tree $T_U$, group them into \emph{cluster} subtrees, and then collapse these subtrees $q_U: T_U \to \hT_U$ to obtain a family of trees $\hT_U$ which only encode the desired geometric information from each domain $U \in \calU$ while still encoding the organizational information provided by the $\rho$-sets.  See Figure \ref{fig:Dur_cube} for a heuristic picture.

We show that this family of trees $\{\hT_U\}_{U \in \calU}$ with its accompanying collapsed $\rho$-set data admits the structure of what we call a \emph{hierarachical family of trees} (Definition \ref{defn:HFT}).  We think of such a structure as an ``exact'' HHS structure, in the sense that it possesses many of the features of an HHS, while all of the coarseness has been eliminated.  In particular, every hierarchical family of trees admits a family of \emph{consistency equations}, for which the $0$-consistent set $\calQ \subset \prod_{U \in \calU} \hT_U$ (Definition \ref{defn:Q consistent}) behaves like an HHS in many important ways.  Notably, the standard HHS structure on any right-angled Artin group \cite{HHS_1} admits such an ``exact'' structure, where all of the associated coarseness constants are $0$.  In this way, a hierarchical family of trees is as hierarchically nice as possible.  See Subsection \ref{subsec:ray example} for an explicit example in a RAAG, and item \eqref{item:R-cubings} of Remark \ref{rem:main rem} below for a discussion of how this setup relates to Casals-Ruiz--Hagen--Kazachkov's notion of an $\mathbb{R}$-cubing \cite{CRHK}.

One could stop here and observe that $\calQ$ is a median subalgebra of a product of finite-valent trees, and hence quasi-median quasi-isometric to a CAT(0) cube complex (see \cite[Section 2]{Bowditch_hulls}, \cite[Proposition 2.12]{HP_proj}, or \cite[Theorem 4.16]{CRHK}).  However, we give an explicit direct proof.  We prove that, up to quasi-median quasi-isometry of the space $\calQ$, each collapsed tree $\hT_U$ can be replaced by a simplicial tree in an appropriate sense via our Tree Trimming \ref{thm:tree trimming} techniques, which function like a fine-tuned version of the distance formula; see Subsection \ref{subsec:TT intro} below for a discussion.  As a result, the $0$-consistent set $\calQ$ is contained in product of simplicial trees, i.e., a genuine CAT(0) cube complex $\prod_{U \in \calU} \hT_U$.  This allows us to work with the induced wall structure on $\calQ$, and we prove that $\calQ$ with its induced $\ell^1$-metric is \emph{isometric} to the cubical dual to this wall-space.  This space $\calQ$ is the aforementioned bridge, simultaneously hierarchical and cubical.

While the above discussion ignored the possibility for allowing points at infinity and hierarchy rays, our construction works in that context, and we will discuss this aspect in the next subsection.  However, for the purposes of summarizing the above discussion, we state the full version of our modeling theorem (Theorem \ref{thm:dual}):

\begin{center}
\framebox{
\begin{minipage}{6.25in}
\begin{thmi} \label{thmi:main model}
Let $\calX$ be a proper HHS, $F \subset \calX$ a nonempty finite set of interior points, $\Lambda$ a (possibly empty) set of hierarchy rays and points in $\partial \calX$, and $H$ the hierarchical hull of $F \cup \Lambda$.  Let $\calU$ be the set of relevant domains for $F \cup \Lambda$, and $\{\hT_U\}_{U \in \calU}$ the corresponding hierarchical family of simplicial trees.  Then the following hold:

\begin{enumerate}
\item The $0$-consistent canonical subset $\calQ \subset \prod_{U \in \calU} \left(\hT_U \cup \partial \hT_U\right)$ is weakly convex and isometric to its cubical dual.  In particular, $\calQ$ is a finite dimensional CAT(0) cube complex, whose dimension is bounded in terms $\calX$.
\item There exists an $L = L(|F \cup \Lambda|, \calX)>0$ and an $L$-quasi-inverse pair of maps
\begin{itemize}
\item a $L$-quasi-median $(L,L)$-quasi-isometry $\hPsi: H \to \calQ$ and 
\item a $L$-quasi-median $(L,L)$-quasi-isometric embedding $\hO: \calQ \to \calX$ with $d^{Haus}_{\calX}(\hO(\calQ), H) < L$.
\end{itemize}
\item There exists a collection of $0$-cells $\widehat{F} \cup \widehat{\Lambda}$ (possibly at infinity) of $\calQ$ which are in bijective correspondence under $\hO$ and $\hPsi$ with $F \cup \Lambda$.  Moreover, $\calQ$ is the cubical convex hull of $\widehat{F} \cup \widehat{\Lambda}$. 
\item (Equivariance) If $g \in \mathrm{Aut}(\calX)$ is an HHS automorphism, then the hull of $g \cdot (F \cup \Lambda)$ admits a cubical model $\calQ_g$ with $g$ inducing a cubical isomorphism $\calQ \to \calQ_g$.
\end{enumerate}
\end{thmi}
\end{minipage}
}
\end{center}

\begin{remark} \label{rem:main rem}
Some comments are in order:
\begin{enumerate}
\item Properness in the statement above is only required when the set of points at infinity $\Lambda$ is nonempty.  The assumption that $F$ is nonempty provides a basepoint in the interior of $\calX$, which is necessary for many arguments.
\item In conclusion (1), $\calQ$ is \emph{weakly convex} in the sense that any pair of points in $\calQ$ can be connected by a geodesic in $\calQ$, in fact a combinatorial geodesic in its cubical structure which $\hO:\calQ \to H$ sends to a hierarchy path; see Proposition \ref{prop:cube paths hp}.

\item In the case where $\Lambda$ is nonempty, the set of relevant domains $\calU$ is frequently (and indeed generically \cite{ST_random, DZ22}) infinite.  This fact alone necessitates new techniques, even if one wanted to utilize existing cubical model constructions from \cite{HHS_quasi, Bowditch_hulls}.  See Subsection \ref{subsec:PU intro} for a discussion of these new ``passing-up''-type techniques. 
\item The isometry between $\calQ$ and its cubical dual involves showing that hierarchical consistency in $\calQ$ corresponds \emph
{exactly} to cubical consistency in its dual, in the sense of Sageev \cite{Sageev_machine}.  A coarse version of this correspondence is at the heart of the original construction in \cite{HHS_quasi}, but our version makes it completely explicit.  See Subsection \ref{subsec:dual to Q} for a discussion.
\item The notion of ``canonical'' for a tuple in $\prod_{U \in \calU} \left(\hT_U \cup \partial \hT_U\right)$ (Definition \ref{defn:Q consistent}) is necessary to pick out an ``internal point'' of the model space $\calQ$.  It is a hierarchical analogue of the cubical canonicality condition which determines interior $0$-cubes (or equivalently the descending chain condition on ultrafilters, Definition \ref{defn:ultrafilter}), and we make this connection explicit in the proof of the isometry between $\calQ$ and its cubical dual; see Proposition \ref{prop:DCC} in particular.
\item The map $\hPsi: H \to \calQ \subset \prod_{U \in \calU} \hT_U$ is defined as simply as possible in a coordinate-wise fashion as a combination of the various projection and collapsing maps described above; see Subsection \ref{subsec:hPsi defined} for a precise definition.
\item On the other hand, the coordinate-wise definition of the map $\hO: \calQ \to \calX$ is more involved, as possibly arbitrarily large pieces of nesting information have been collapsed in the construction of $\calQ$.  The process of uncollapsing exploits the fundamentally cubical nature of $\calQ$, namely the two-sided $\nest_{\calU}$-minimal \emph{bipartite} domains as discussed in Figure \ref{fig:Dur_cube}, though this part of the argument does not depend on any cubical techniques, including the fact that $\calQ$ is isometric to its cubical dual, which is a separate argument.  In this way, we see our construction as highlighting the fundamentally cubical nature of HHSes, at least at the local level.  See Section \ref{sec:Q to Y} for more details, especially Subsection \ref{subsec:honing clusters} for a detailed outline of this point of the argument.

\item \underline{Equivariance}: The construction in Theorem \ref{thmi:main model} is invariant under any \emph{HHS automorphisms} $g \in \mathrm{Aut}(\calX)$.  HHS automorphisms are isometries of $\calX$ which preserves the projections and other relevant hierarchical data of $\calX$.  In particular, the mapping class group is a group of HHS automorphisms of itself; see Subsection \ref{subsec:HHS auto} for more details, \cite[Subsection 1G]{HHS_2} (or equivalently \cite[Section 2]{PS20}) for a precise definition, and \cite{DHS_boundary, DHS_corr} for various results about HHS automorphisms.  Equivariance of the cubical models is an immediate consequence of the fact that the atoms of the entire construction are fundamental parts of an HHS, which are automorphism-equivariant by definition. 

\item \underline{The quasi-median property}: \emph{Coarse median} spaces were introduced by Bowditch \cite{Bowditch_coarsemedian} to generalize feature common to both hyperbolic spaces and cube complexes, with the main motivating example being the mapping class group \cite{BM_centroid}.  CAT(0) cube complexes are in fact \emph{median} spaces (see e.g. \cite{CDH10, Chepoi_median,Roller} and \cite{Bowditch_medianbook} for an overview), with their median structure directly encoding their $\ell^1$ (or combinatorial) geometry, while the coarse median structure on any HHS encodes a large amount of its hierarchical geometry (see e.g. \cite{RST18, HHS_quasi}).  The fact that the the model map $\hO:\calQ \to H$ is \emph{quasi-median} (Definition \ref{defn:quasi-median}) says, at least in principle, that the cubical geometry of $\calQ$ directly encodes the hierarchical geometry of $H$.  We note that the quasi-median property of these model maps was already obtained in the interior point case in \cite{HHS_quasi} and for limiting versions modeling hierarchy rays in \cite{DZ22}, and this property has proved useful in essentially every application of this technology \cite{HHS_quasi, HHP, DMS20, DZ22}.

\item \underline{Refined infinite models}: \label{item:infinite model rem} In our work with Zalloum \cite{DZ22}, we produce a cubical model for any finite set of hierarchy rays by taking an unrescaled ultralimit of finite cubical models.  While we are able to extract a good amount of information out of the fact that these limiting model maps were quasi-median quasi-isometries, the cubical models in Theorem \ref{thmi:main model} contain a vast amount of additional hierarchical information that is lost in the limiting process.  In particular, as with the original construction in \cite{HHS_quasi}, every hyperplane in $\calQ$ is labeled by some relevant domain in $\calU$ and two hyperplanes cross if and only if their labels are orthogonal.  Moreover, we recover the fact that the flat subspaces of the cubical models are mapped to (quasi-)flat subspaces of standard product regions in the HHS $\calX$ (Lemma \ref{lem:hyp PU}).  These facts alone are quite useful for applications; see Theorem \ref{thmi:curve graph} below and \cite[Section 16]{CRHK} for an in-depth discussion.

\item \underline{Extended tuples, $0$-cubes at infinity, and Roller boundaries}: Another motivation for the definition of points in the HHS boundary \cite{DHS_boundary} was the notion of a \emph{generalized marking} from Minsky's work on the Ending Lamination Theorem \cite[Section 5]{Min_ELC}, which were required to build model manifolds for hyperbolic $3$-manifolds with geometrically infinite ends.  The information of a generalized marking on a surface $S$ can be coarsely encoded in a pair consisting of a marking on $S$ (in the sense of \cite{MM00}) and a point in the HHS boundary, or an appropriately constructed hierarchy ray (Proposition \ref{prop:ray replace boundary}).  Any such generalized marking determines a tuple in $\prod_{U \in \mathfrak S} (\calC(U) \cup \partial \calC(U))$ satisfying a notion of \emph{extended consistency} (see Definition \ref{defn:extended consistency} and Proposition \ref{prop:extended consistency}).  These extended tuples play an analogous role to the $0$-cubes at infinity of a cube complex.  In this article, we lay the foundation for locally modeling the space of extended tuples by the $0$-cubes at infinity of appropriately chosen cubical models, and thus for the possibility of defining a Roller-type boundary for any proper HHS.\label{item:roller remark}

One upshot of this discussion is that the infinite cubical models constructed in Theorem \ref{thmi:main model} also can also be used to organize the infinite hierarchies that arise in the context of the Ending Lamination Theorem, though with much less specificity than required for that theorem.  See Bowditch \cite[Section 11]{Bowditch_hulls} for an in-depth discussion of this connection in the context of the interval between a pair of points in the pants graph of a compact surface, where the pants graph is quasi-isometric to the Teichm\"uller space of the surface with the Weil-Petersson metric \cite{Brock_WP}.

\item \underline{Relation to $\mathbb{R}$-cubings}: \label{item:R-cubings}
The $0$-consistent subset $\calQ$ of the (possibly infinite) product of trees $\prod_{U \in \calU} \hT_U$ satisfies most the axioms for being a \emph{discrete $\mathbb{R}$-cubing} in the sense of Casals-Ruiz--Hagen--Kazachkov \cite[Definition 4.2]{CRHK}; see especially Definition \ref{defn:HFT} of our notion of a hierarchical family of trees which collects a similar set of properties as axioms.  The key difference is that \cite[Definition 4.2]{CRHK} requires the ambient space to be path-connected, which, in our setting, is tantamount to assuming that $\calQ$ is a CAT(0) cube complex.  Thus one consequence of Theorem \ref{thmi:main model} is that the finite dimensional CAT(0) cube complex $\calQ$ admits a special discrete $\mathbb{R}$-cubing structure that is more hierarchical in flavor than the structure one might produce directly from its cubical structure; see \cite[Example 4.25]{CRHK} for a discussion of the latter.  As a consequence of their impressive machinery, we automatically obtain a host of other structural consequences for $\calQ$, which would be interesting to explore.
\end{enumerate}

\end{remark}

As we will see below, our cubical model construction is quite robust, and we develop a number of tools for producing many variations of cubical models which can be tailored to a number of potential applications.  For the rest of the introduction, we refer to $\calQ$, or any of its variants, as a \emph{cubical model} for the hull $H$.

Before we discuss boundaries, we pause to observe the following consequence of Theorem \ref{thmi:main model}, namely a new proof of (the upper bound of) the Distance Formula \ref{thm:DF} and the existence of hierarchy paths:

\begin{center}
\framebox{
\begin{minipage}{6.25in}
\begin{cori} \label{cori:DF and HP}
In any HHS in $\calX$, the upper-bound of the Distance Formula \ref{thm:DF} holds.  Moreover, any pair of interior points in $\calX$ are connected by a hierarchy path.
\end{cori}
\end{minipage}
}
\end{center}

\begin{proof}
Let $a,b \in \calX$ be points in an HHS $\calX$, and fix $K = K(\calX)>0$ sufficiently large.  Let $\calU = \{U|d_U(a,b) > K\}$ be the set of $K$-relevant domains.  Let $H = \hull_{\calX}(a,b)$ and $\calQ$ be a cubical model with cubical model map $\hO:\calQ \to H$ identifying $\ha,\hb \in \calQ$ with $a,b\in H$.

Theorem \ref{thm:Q distance estimate} provides a distance formula-free proof that
$$\sum_{U \in \calU} [[d_U(a,b)]]_K \asymp \sum_{U \in \calU}d_{\hT_U}(\ha,\hb)$$
where the left-hand sum is the distance-formula sum (see equation \eqref{eq:DF intro} or Theorem \ref{thm:DF}) while the right-hand sum is exactly $d_{\calQ}(\ha,\hb)$.  On the other hand, Proposition \ref{prop:cube paths hp} provides a direct proof that $\hO$ sends any combinatorial geodesic in $\calQ$ between $\ha,\hb$ to a uniform quasi-geodesic---in fact, a hierarchy path---between $a,b$ in $\calX$, and hence the right-hand side gives a lower-bound for $d_{\calX}(a,b)$.  Both conclusions follow.
\end{proof}

To our knowledge, the difficult part (the upper-bound) of every proof of the distance formula involves constructing hierarchy paths, and our proof is no exception.  Note that the proof that combinatorial geodesics in cubical models are sent to hierarchy paths (Proposition \ref{prop:cube paths hp}) does not involve appealing to the median structure, which is necessary since currently the only proof that median paths are hierarchy paths uses the distance formula \cite[Theorem 5.2]{RST18}.

\subsection{Cubulating infinity in proper HHSes} \label{subsec:infinity intro}

In addition to a more robust cubical model construction, the main contribution of this article---and the bulk of the work therein---involves extending the cubical model construction to allow for points in the HHS boundary and hierarchy rays.  This leads to a local description of the HHS boundary of any proper HHS via the simplicial boundaries of appropriately chosen cubical models.

The \emph{simplicial boundary} of a CAT(0) cube complex was introduced by Hagen \cite{Hagen_simplicial}.  It is a fundamentally combinatorial object designed to capture not only the behavior of combinatorial geodesic rays, but also other more pathological ways of running off to infinity, like an infinite staircase (see \cite[Figure 1]{Hagen_simplicial} or \cite[Figure 1]{HFF23}).  It is a simplicial complex whose simplices are \emph{unidirectional boundary sets} or \emph{UBSes} (Definition \ref{defn:UBS}), which are infinite equivalence classes of hyperplanes which encode a combinatorial form of divergence.  The simplicial boundary of a cube complex is called \emph{fully visible} if every UBS admits a combinatorial ray representative.  Hagen proved that every UBS in a sufficiently nice cube complex (like those considered in this paper) admits a decomposition into \emph{minimal} UBSes (see Theorem \ref{thm:minimal decomp}).  Hence every UBS in a fully visible simplicial boundary admits a decomposition with minimal combinatorial ray representatives each of which ``points'' in a well-defined direction.

The following summarizes our boundary isomorphism result:

\vspace{.1in}
\begin{center}
\framebox{
\begin{minipage}{.96\textwidth}
\begin{thmi}\label{thmi:main boundary}
Let $\calX$ be a proper HHS, $F \subset \calX$ a nonempty finite set of interior points, $\Lambda$ a (possibly empty) set of hierarchy rays and points in $\partial \calX$, and $H$ the hierarchical hull of $F \cup \Lambda$.  Let $\calQ$ be any cubical model for $H$.  Then the following hold:

\begin{enumerate}
\item The map $\hO:\calQ \to H$ extends to a simplicial isomorphism $\partial \hO: \partial_{\Delta} \calQ \to \partial H$ from the simplicial boundary of $\calQ$ to the HHS boundary of $H$.
\item The simplicial boundary $\partial_{\Delta} \calQ$ of $\calQ$ is fully visible.
\item (Equivariance) If $g \in \mathrm{Aut}(\calX)$ is an HHS automorphism, then the cubical isomorphism $\calQ \to \calQ_g$ induced by $g$ induces a simplicial isomorphism $\partial_{\Delta} \calQ \to \partial_{\Delta} \calQ_g$. 
\end{enumerate}

\end{thmi}
\end{minipage}
}
\end{center}
\vspace{.1in}

The map $\partial \hO:\partial_{\Delta} \calQ \to \partial_{\Delta} H$ extends $\hO$ in the following sense: Any combinatorial ray $\gamma$ in $\calQ$ is sent to a hierarchy ray $\hO(\gamma) \subset H$ whose limiting simplex in $\partial H$ (Definition \ref{defn:ray support}) is the image under $\partial \hO$ of the simplex in $\partial_{\Delta} \calQ$ corresponding to $\gamma$.  Since $\partial_{\Delta} \calQ$ is fully visible, this correspondence is as complete as possible; see Theorem \ref{thm:boundary iso} for more details.

One consequence of the above theorem is a local description of the hierarchical boundary of any proper HHS, as we will now explain.  Note that the simplicial isomorphism $\partial_{\Delta} \calQ \to \partial H$ holds regardless of the number of boundary points and hierarchy rays in $\Lambda$.  Given a finite collection of simplices in $\partial \calX$, one can construct a finite collection of hierarchy rays whose limiting data encode exactly these simplices; see Proposition \ref{prop:ray replace boundary}.  Theorem \ref{thmi:main boundary} then says that the simplicial boundary of any cubical model of the hierarchical hull of this finite set of rays is isomorphic to the hull's hierarchical boundary, which is (as a simplicial complex) precisely the desired set of simplices.  Hence we obtain the following as a corollary:

\begin{center}
\framebox{
\begin{minipage}{.96\textwidth}
\begin{cori}\label{cori:local boundary}
Let $\calX$ be a proper HHS and $\Sigma \subset \partial \calX$ be any finite union of simplices.  Then $\Sigma$ can be realized as the hierarchical boundary of the hierarchical hull $H$ of a finite set of hierarchy rays $\Lambda$ in $\calX$, and if $\calQ_{\Sigma}$ is any cubical model for $H$, then the model map $\hO: \calQ \to \calX$ extends to a simplicial embedding $\partial \hO: \calQ_{\Sigma} \to \partial \calX$ whose image is exactly $\Sigma$.
\end{cori}
\end{minipage}
}
\end{center}
\vspace{.1in}

We note that both Theorem \ref{thmi:main boundary} and Corollary \ref{cori:local boundary} are entirely geometric statements, and thus do not reference any sort of topology on any of the relevant objects.  As such, they provide a foundation for future investigations of the various possible cubical-like topologies (see e.g., \cite{HFF23}) one can put on the HHS boundary of any proper HHS.

Moreover, Corollary \ref{cori:local boundary} provides a potential tool for studying homological and homotopical properties of the HHS boundary $\partial \calX$ of a proper HHS $\calX$.  In particular, suppose $\partial \calX$ is endowed with a topology $\tau$ for which the identity map $id:\partial \calX \to \partial \calX$ induces a continuous map $(\partial \calX, \tau_{\Delta}) \to (\partial \calX, \tau)$, where $\tau_{\Delta}$ is the topology on $\partial \calX$ realized as a CW complex.  Then any continuous map $S^n \to \partial \calX$ from the $n$-sphere intersects only finitely-many simplices, and hence Corollary \ref{cori:local boundary} would allow one to employ cubical techniques to study the map.  We thank Mark Hagen for pointing this connection out to us.

\subsection{Cubically capturing the curve graph}\label{subsec:extract}

In our view, the greatest utility of these cubical models \cite{HHS_quasi, DMS20, HHP, CRHK, DZ22} is that they encode a wealth of hierarchical information into their cubical structure, thereby allowing one to employ a mixture of cubical and hierarchical techniques.  A significant part of this article is spent developing a number of tools for extracting more information, not only in the case of infinite models but even for finite models.  While we will discuss some of these tools in detail below, we first highlight one feature of this more robust model machinery---namely that one can construct cubical models which coarsely encode curve graph distance.

In our work with Zalloum \cite{DZ22}, we were interested in studying, among other things, the cubical geometry of the hierarchical hulls associated to hierarchy rays which display ``hyperbolic'' behavior.  For instance, this includes hierarchy rays in the mapping class group or the Teichm\"uller space of a finite type surface $S$ which project to (unparameterized) quasi-geodesic rays in the curve graph $\calC(S)$ of $S$, examples of which are Morse \cite{Beh06, ABD} and sublinearly-Morse rays \cite{DZ22} (the latter were first defined in \cite{QRT20}).  As part of our analysis with Zalloum, we established one half of a correspondence between this sort of global hyperbolic behavior of rays in an HHS and certain hyperbolic features of their cubical representatives.   In this paper, we complete this correspondence in a strong way, which we now explain.

Following Genevois \cite{Gen_hyp}, we say that two hyperplanes in a CAT(0) cube complex $X$ are \emph{L-separated} if the number of hyperplanes crossing both is less than $L$ (see Definition \ref{defn:sep hyp}).  Genevois proved that one can define a new distance $d^X_L$ on the $0$-skeleton $X^{(0)}$ of $X$ by using the length of any maximal $L$-separated chain between any $0$-cells $x,y \in X^{(0)}$.  Importantly, he proved that $(X, d^X_L)$ is $\delta_L$-hyperbolic, with $\delta_L$ depending on the chosen $L$.  The philosophy here is that separated hyperplanes encode hyperbolic behavior in CAT(0) cube complexes.

To setup our next theorem, let $\calX$ be an HHS with \emph{unbounded products} (Definition \ref{defn:BDD}), a natural weak assumption which includes all of the main examples, including every hierarchically hyperbolic group.  Let $S$ be the label for the $\nest$-maximal hyperbolic space (e.g., the curve graph $\calC(S)$ when $S$ is a finite type surface and $\calX$ is the mapping class group), as constructed in our paper with Abbott and Behrstock \cite{ABD}; see Subsection \ref{subsec:ABD} for a discussion.

\begin{figure}
    \centering
    \includegraphics[width=.7\textwidth]{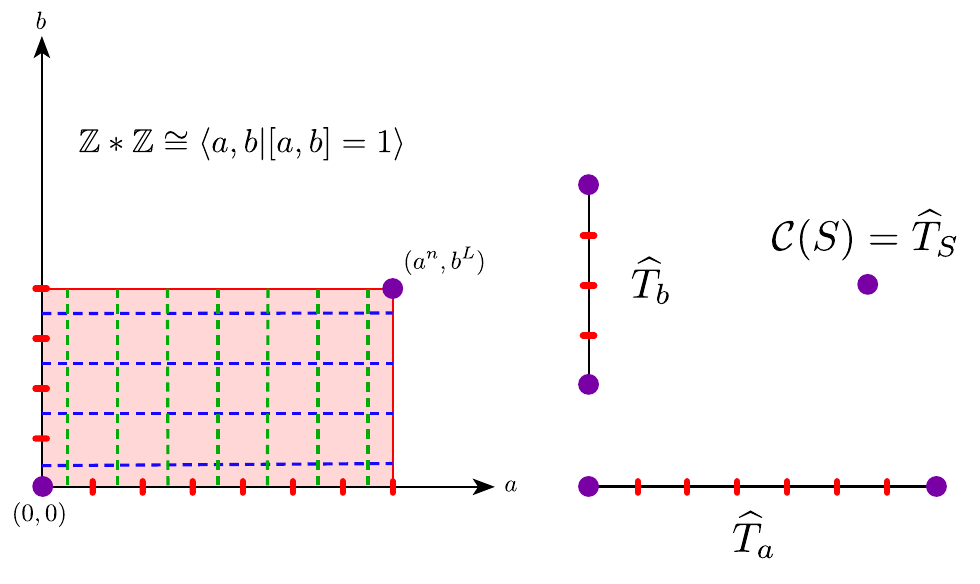}
    \caption{Fake $L$-separation: A simple example which shows that the upper bound $d^{\calQ}_L(x,y) \prec d_{\calC(S)}(\pi_S(x), \pi_S(y))$ does not hold for general cubical models $\calQ$ and separation constants $L>0$.  The standard HHG structure on $G = \ZZ \times \ZZ \cong \langle a,b|[a,b]=1\rangle$ consists of two lines labeled by $a,b$ and a point for the top-level hyperbolic space $\calC(S)$ (which can be arranged for any HHS which is coarsely a product).  Since this is a cube complex, cubical convex hulls coincide with hierarchical hulls, and are thus their own cubical models.  Given any $L>0$ and $n>0$, the cubical hull for $(id_G, a^nb^L)$ will be a rectangle (in faded red), where the collapsed trees are just intervals.  The set of (blue) vertical hyperplanes is $L$-separated because each is crossed by each of the $L$-many (green) horizontal hyperplanes, while the projection of the hull to the top level hyperbolic space $\calC(S)$ is a point (because $\calC(S)$ is a point).  Fixing $L$ and letting $n$ grow shows that the upper bound does not hold for general cubical models.   
Fake separation is easy to arrange elsewhere.  For instance, the cubical model of the hull of the identity and any sufficiently high power of a Dehn twist in the mapping class group of a finite type surface consists entirely of fake $0$-separated hyperplanes.}
    \label{fig:fake_sep}
\end{figure}

\vspace{.1in}
\begin{center}
\framebox{
\begin{minipage}{.96\textwidth}
\begin{thmi}\label{thmi:curve graph}
Let $\calX$ be an HHS with unbounded products.  For any $x,y \in \calX$, let $H$ be the hierarchical hull of $x,y$.  Then there exists a cubical model $\hO:\calQ_{x,y} \to H$ identifying $\hx,\hy$ with $x,y$ so that
$$d^{\calQ_{x,y}}_0(\hx,\hy) \asymp d_{\calC(S)}(\pi_S(x), \pi_S(y)),$$  
where the constants in $\asymp$ and the construction of $\calQ_{x,y}$ depend only on the hierarchical structure of $\calX$.
\begin{itemize}
\item Moreover, if $\gamma$ is a hierarchy ray, then
$$\diam_{\calC(S)}(\pi_S(\gamma)) = \infty \ \iff \ \diam_0^{\calQ_{\gamma(0), \gamma}}(\gamma) = \infty.$$
\end{itemize}
\end{thmi}
\end{minipage}
}
\end{center}
\vspace{.1in}

In other words, the distance in the top-level hyperbolic space is coarsely encoded in the cubical geometry of a family of cubical models.  One consequence of Theorem \ref{thmi:curve graph} is that a hierarchy ray $\gamma$ determines a point in the boundary of $\calC(S)$ if and only if $\gamma$ determines a point in the boundary of the contact graph of $\calQ_{\gamma(0), \gamma}$.  The \emph{contact graph}, introduced by Hagen \cite{Hagen_contact}, is a hyperbolic graph which plays an analogous role to the curve graph for cube complexes.  Hence hyperbolic directions in the HHS $\calX$ coincide the hyperbolic directions in a family of uniform cubical models.

We note that \cite[Theorem C]{DZ22} provides the bound $d^{\calQ}_L(\hx,\hy) \succ d_S(\pi_S(x), \pi_S(y))$ for any cubical model $\calQ$, where $L$ is some constant depending on $\calX$.  In Lemma \ref{lem:sep LB}, we strengthen this to $0$-separation (called \emph{strong separation} in \cite{BC_div}) with a simpler proof which uses the fact that crossing hyperplanes are labeled by orthogonal domains, as discussed above in item \eqref{item:infinite model rem} of Remark \ref{rem:main rem}.  Thus Theorem \ref{thmi:curve graph} completes the relationship in the strongest way possible.  See Figure \ref{fig:fake_sep} for a simple example of how ``fake'' separation can prevent the upper bound in Theorem \ref{thmi:curve graph} from holding in general cubical models, and Subsection \ref{subsec:sketch} for a discussion of how to alter the cubical models to prevent fake separation from occurring.

In future work, we will develop a variation on this cubical machinery (along the lines of \cite{DMS20}) to, following Genevois, define a new metric on the mapping class group which will make it quasi-isometric to the curve graph.  We note that Incerti-Medici--Zalloum \cite{IMZ21} have already done this for cube complexes which admit an HHS structure; see also Petyt-Spriano-Zalloum \cite{PSZ_curtain} for a generalization of these ideas to CAT(0) spaces, as well as Zalloum's expository article \cite{Zalloum_curtain}.

\subsection{Related results}\label{subsec:related results} 

The use of local approximations, for instance in manifold theory, is at least as old as modern mathematics.  In geometric group theory, paramount examples include Gromov hyperbolic spaces and CAT(0) spaces, which are defined by $3$-point comparisons to tripods and Euclidean space, respectively.  As discussed above, Behrstock-Hagen-Sisto's cubical model theorem \cite[Theorem F]{HHS_quasi} directly generalizes Gromov's tree modeling theorem \cite{Gromov1987} for hyperbolic spaces.

In \cite{Bowditch_coarsemedian}, Bowditch introduced the notion of a coarse median space, a class of objects including cube complexes and HHSes which has proven very useful.  Their defining property is a coarse version of the median property which hinges on local comparisons with cube complexes (see Definition \ref{defn:coarse median}), which is in some sense a weak version of the cubical model theorem (see Remark \ref{rem:Bowditch median remark}).  In \cite{Bowditch_hulls}, Bowditch recovers the full cubical model theorem for interior points in the more general setting of coarse median spaces which satisfy certain HHS-like axioms.  Notably, he recovers (as we do) a proof of the existence of median quasi-geodesic paths connecting any pair of points and gives a new proof of the distance formula (Theorem \ref{thm:DF}).  His approach is similar in spirit to ours, since his cubical model also arises as a subset of a product of trees.  However, his approach is fundamentally different from ours (and the approach in \cite{HHS_quasi}) because it utilizes the median property as part of its axiomatic foundation, whereas medians are essentially a by-product of the hierarchical framework.  Moreover, his construction does not allow for modeling hulls of median rays (which coincide with hierarchy rays in the HHS setting), whereas the infinite models we construct and their applications constitute the majority of work in this paper.  We believe that any such extension of Bowditch's construction would require an analogous set of tools to the ones we develop in this article; see Subsection \ref{subsec:PU intro} for a discussion of some of these techniques.

In \cite{DZ22}, we proved with Zalloum that the hull of any finite collection of hierarchy rays in an HHS can be approximated by a CAT(0) cube complex which is obtained by taking an unrescaled ultralimit of the finite cubical approximations from \cite{HHS_quasi}.  Our purpose in that paper was to study certain hyperbolic-like directions in HHSes, such as sublinearly-Morse geodesic rays, and we discussed in Remark \ref{rem:main rem} the ways in which the current article refines and strengthens this result.  Much of that discussion did not actually involve any hierarchical geometry, and so we carried some of that work out in the framework of \emph{locally quasi-cubical spaces}, namely coarse median spaces where median hulls of finite sets of points are quasi-median quasi-isometric to CAT(0) cube complexes.  This more general framework is currently being further developed by Petyt-Spriano-Zalloum \cite{PSZ_gencurtain}.  See Zalloum's expository article \cite{Zalloum_expo} for a broad discussion of cubical approximations in geometric group theory.

Theorem \ref{thmi:main boundary} and Corollary \ref{cori:local boundary} are related to recent work of Hamenst\"adt \cite{Ham_boundary} and Petyt \cite{Petyt_quasicube}, both of which give a global cubical model for boundaries of the mapping class group $\MCG(S)$ of a finite-type surface $S$.

In \cite{Petyt_quasicube}, Petyt proves that $\MCG(S)$ (in fact, any \emph{colorable} hierarchically hyperbolic group, more generally) admits a quasi-median quasi-isometry to a finite dimensional CAT(0) cube complex.  He accomplishes this in two ways, with the more relevant one being through work of Bestvina-Bromberg-Fujiwara \cite{BBF_quasitree}.  They proved that $\MCG(S)$ admits a quasi-isometric embedding into a finite product of quasi-trees, and Petyt proved that this map is quasi-median.  This global approximation is quite useful for studying the interior of $\MCG(S)$, e.g. it provides a new proof that $\MCG(S)$ is bicombable though it does not recover semihyperbolicity \cite{Ham_biauto, DMS20, HHP}.  Moreover, it allows one to cubically approximate the hierarchical hull of any finite set of interior points or hierarchy rays with uniform constants (independent of the size of the finite set).  In addition, this global model allows one to identify any number of boundaries, e.g. the Roller boundary, of this cube complex as a boundary of $\MCG(S)$.  The main drawback of this construction is a lack of equivariance.  In particular, the quasi-isometry from the global cubical model to $\MCG(S)$ is not $\MCG(S)$-equivariant (and could not possibly be \cite{Bridson_notCCC}), and thus any notion of a boundary for $\MCG(S)$ constructed in this way would necessarily provide limited dynamical information.  In addition, it is yet unclear how hierarchical information is encoded in the cubical geometry, since, for instance, hyperplanes in Petyt's global model do not come with domain labels, as in Theorem \ref{thmi:main model}.

In \cite{Ham_boundary}, Hamenst\"adt constructs a uniformly finite CAT(0) cube complex $C$ which admits a proper coarsely surjective Lipschitz map $F:C \to \MCG(S)$.  Remarkably, Hamenst\"adt constructs a homeomorphic identification of the Roller boundary of $C$ with the space of complete geodesic laminations on $S$.  As with Petyt's construction, this map $F$ is necessarily not $\MCG(S)$-equivariant, and in fact is also not coarsely bi-Lipschitz.  Notably, these results are established via train-track techniques, which are quite different from our hierarchical approach.

Finally, Casals-Ruiz--Hagen--Kazachkov \cite{CRHK} have an ongoing program to study the asymptotic cones of HHSes.  In doing so, they introduce the notion of a \emph{real cubing}, which is a higher-rank generalization of a real tree, and prove that every asymptotic cone of an HHS admits the structure of a real cubing.  Remarkably, they use this robust theory to prove that the mapping class group has a unique asymptotic cone (up to bi-Lipschitz equivalence).  Their program is similar to ours in spirit, in that they need to wrestle with infinite sets of hierarchical data and they use cubical techniques to do so.  See item \eqref{item:R-cubings} of Remark \ref{rem:main rem} for a discussion about how our setup is related to theirs.

\subsection{A basic motivating example} \label{subsec:ray example}

We will spend the rest of the introduction highlighting some of our techniques and laying out some questions.  To motivate this discussion, we begin with a basic example of a combinatorial (and thus hierarchy) ray in a familiar cubical HHS, the Cayley 2-complex $\calX$ of $\ZZ^2 * \ZZ \cong \langle a,b| [a,b]=1\rangle * \langle c \rangle$.  This space is known as the ``tree of flats"  because it is a family of 2-dimensional flats (copies of $\ZZ^2$) connected in a tree-like fashion.  See Figure \ref{fig:ray_example} for the setup.

The standard HHS structure on $\calX \cong \ZZ^2 *\ZZ$ consists of the infinite-valent tree $T$ obtained by coning off all copies of $\ZZ^2$, which is the top-level space in the hierarchy, along with lower level spaces which are copies of $\ZZ$ corresponding to the coordinate axes of each flat.  In particular, every hyperbolic space in the structure is a tree.  This HHS structure is ``exact'' in the sense discussed above: all of the constants associated to the structure can be taken to be $0$, and $\ZZ^2 * \ZZ$ is isometric to the $0$-consistent subset of the product of the hyperbolic spaces in the hierarchy (which are all lines except for $T$).  Moreover, all projections in the hierarchical setup can be taken to be either the collapsing map $\calX \to T$ or closest-point projections (in $\calX$) to the flats and their axes (see \cite{HHS_1} for the general case of cubical HHSes).  As mentioned above, this is the sort of exactness that we capture in the hierarchical families of trees used to construct our version of the cubical models (Theorem \ref{thmi:main model}) in the interior point case.  We will now explain in some detail how our construction plays out in this example.

The set of combinatorial (and thus hierarchy) rays in $\calX$ admits a fairly simple description.  Since we are in the cubical setting, the HHS and simplicial boundaries coincide (by, e.g., \cite[Theorem 10.1]{DHS_boundary}).  They consist of two main types of simplices: a Cantor set of isolated $0$-simplices coming from the Gromov boundary of  $T$ and a family of isolated $1$-dimensional diamonds corresponding to each copy $\ZZ^2$, where the vertices of the diamond correspond to the cardinal directions in the flat they represent.  In particular, the HHS (or simplicial) boundary of $\ZZ^2 * \ZZ$ is a disconnected union of these diamonds and the Cantor set.  Importantly, every combinatorial geodesic ray in $\ZZ^2 * \ZZ$ either eventually lies in a single flat or projects to a geodesic ray in $T$ which runs off to infinity.

It is instructive to consider the hierarchical hull of a single combinatorial ray of the latter type:
$$\gamma = abca^2b^2ca^3b^3c \cdots.$$

The most obvious hierarchical family of trees associated to this setup consists of trees which are just the intervals obtained by projecting $\gamma$ to all of the various hyperbolic spaces in the structure, excluding projections of diameter $0$.  In particular, these intervals have the following form:
\begin{enumerate}
\item The infinite geodesic ray $T_{\gamma} = \pi_T(\gamma)$ in $T$, and
\item For each $n$, an interval $T_{n, a}$ of length $n$ in the $a$-direction of the flat labeled by $d_n = abca^2b^2c\dots ca^{n-1}b^{n-1}c$.
\item For each $n$, an interval $T_{n, b}$ of length $n$ in the $b$-direction of the flat labeled by $d_n$.
\end{enumerate}

Thus we have the following product of intervals:
$$\calY:= T_{\gamma} \times \prod_n \left(T_{n,a} \times T_{n,b}\right),$$
which we note is the product of the projections of $\gamma$ to factors of the full product of $T$ and the various copies of $\RR$ which give the HHS structure on $\calX$, with the factors that are a point being ignored.  In particular, the $0$-consistent set 
$\calQ$ of $\calY$ is exactly the restriction of the $0$-consistent set of the full product of hyperbolic spaces, which is itself precisely $\calX$.  Note that this makes $\calQ$ precisely the $\ell^1$-convex hull of $\gamma$ in $\calX$, when the latter is identified as a subset of the product of hyperbolic spaces.

Moreover, the above two sets of information, namely (1) the geodesic ray $T_{\gamma}$ and (2) the set of labels $d_n$ which encodes the relevant flats (and thus their axes), both completely encode the limit set of $\gamma$ in the simplicial (hierarchical) boundary of $\calX$---namely the $0$-simplex corresponding to $\overline{\gamma} = \pi_T(\gamma) \in \partial T$---since the $\rho$-sets for the domains labeled by $d_n$ diverge to infinity along $\pi_T(\gamma)$.  See Figure \ref{fig:ray_example}.

As for the cubical model $\calQ$ itself, the set of domain relations
\begin{itemize}
\item $T_{n,a} \perp T_{n,b}$ are orthogonal for all $n$,
\item $T_{n,a}, T_{n,b} \nest T_{\gamma}$ nest for all $n$ with $\rho^{d_n}_{T_{\gamma}}$ being the collapsed point for the corresponding flat;
\item All other domains are transverse $\pitchfork$;
\end{itemize}
tells us how to piece together $\calQ$ as follows.  As discussed in Subsection \ref{subsec:HHS primer}, orthogonality does not constrain coordinates in $\calY$.  Thus for each $n$, $\calQ$ will contain an isometrically embedded copy of the square $T_{n,a} \times T_{n,b}$.  The rest of the relation data then tells us to glue these squares into their corresponding $\rho$-points along $T_{\gamma}$, namely that $T_{n,a} \times T_{n,b}$ replaces $\rho^{d_n}_{T_{\gamma}}$ for each $n$.  Again, see Figure \ref{fig:ray_example}.

Returning to our boundary discussion, we claim that only one of the above sets of data encoded by the ray is robust under (necessary) variations of the model construction.  While the choice of $T_{\gamma}$ as a modeling tree for the projection of $\gamma$ to $T$ is natural, using it also as the \emph{collapsed} tree in the hierarchical product of trees in our construction is very special.  As heuristically indicated in Figures \ref{fig:BHS_cube} and \ref{fig:Dur_cube}, the collapsing procedure in each relevant tree $T_U$ involves ``clustering'' the nearby $\rho$-sets for domains which nest into $V$.  In the entirely coarse general HHS setting, one needs some flexibility in choosing this \emph{cluster separation constant}.  See Subsection \ref{subsec:shadows} for more details.

\begin{figure}
    \centering
    \includegraphics[width=.95\textwidth]{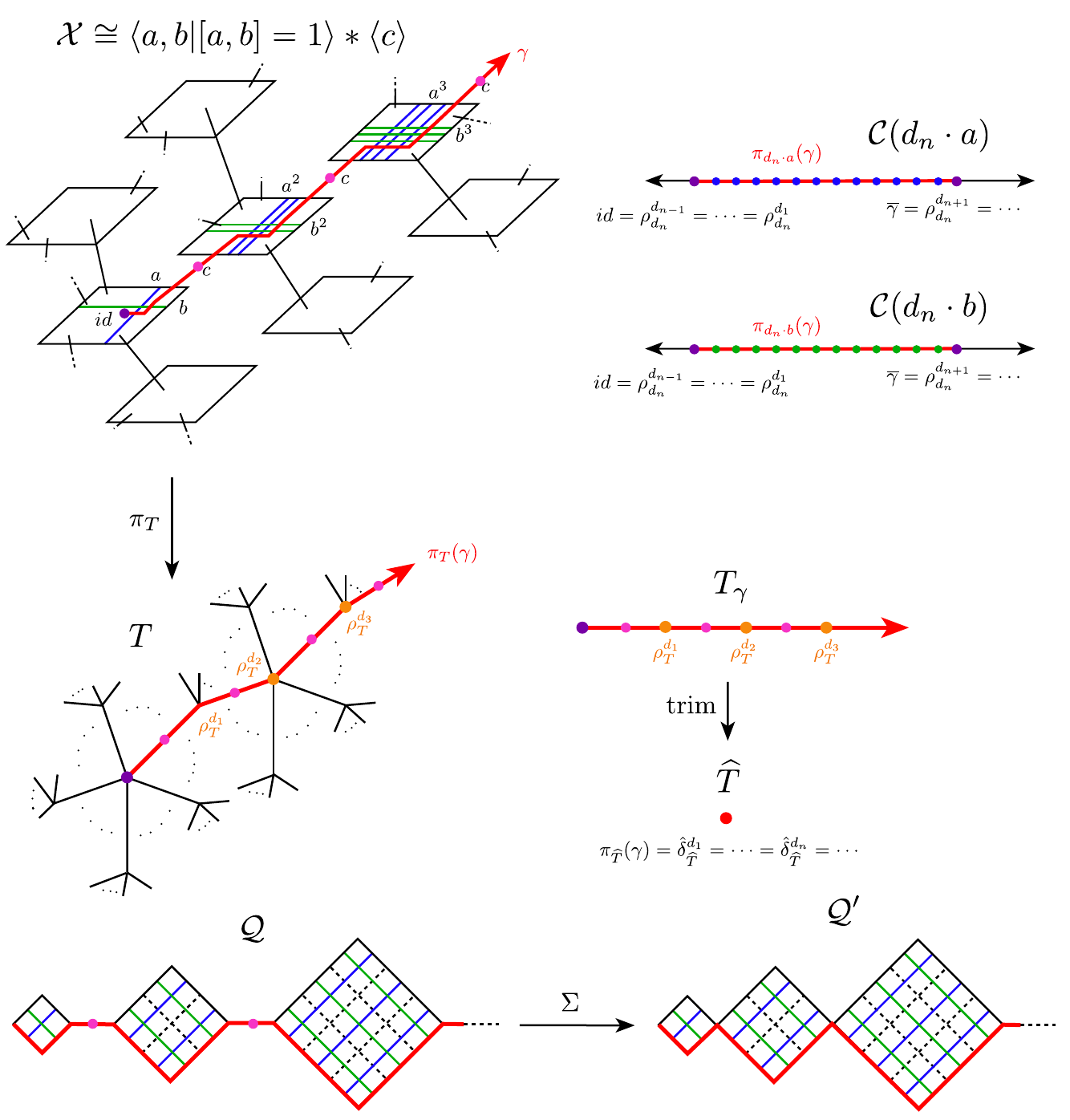}
    \caption{A ray in $\calX \cong \langle a,b \rangle * \langle c \rangle$ whose limiting behavior is only robustly encoded in domain data: The ray $\gamma= abca^2b^2ca^3b^3c \cdots$ passes through an infinite sequence of flats, moving progressively further in each successive flat.  The limiting behavior of $\gamma$ is encoded not only in the geodesic ray $\pi_{T}(\gamma)$ to which it projects in the tree $T$ obtained by collapsing down all of the flats to points, but also in the sequence of flats, which is encoded as a sequence of vertices $\rho^{d_n}_T$ of $T$.  When building a cubical model $\calQ$ for $\hull_{\calX}(\gamma)$, one could take $\pi_T(\gamma) = T_{\gamma}$, $\pi_{d_n\cdot a}(\gamma) = T_{n,a}$, $\pi_{d_n\cdot b}(\gamma) = T_{n,b}$ for the family of collapsed trees.  However, trimming $T_{\gamma} \to \hT$ the $c$-intervals from $T_{\gamma}$ collapses it to a point $\hT$ and gives another hierarchical family of trees, with a cubical model $\calQ'$ for $\hull_{\calX}(\gamma)$ which is $0$-quasi-median $(2,0)$-quasi-isometric under the induced collapsing map $\Sigma: \calQ \to \calQ'$.  At the level of simplicial boundaries, the hyperplanes in $\calQ$ induced by the green, blue, and pink hyperplanes in $\calX$ represent the class corresponding to $\gamma$ in $\partial_{\Delta} \calQ$, while the green and blue hyperplanes represent $\Sigma(\gamma)$ in $\partial_{\Delta}\calQ'$, with $\Sigma$ inducing the isomorphism $\partial \Sigma: \partial_{\Delta} \calQ \to \partial_{\Delta} \calQ'$ (between singletons).  In this sense, $\gamma$ is only robustly encoded in the domain labels for the green and blue hyperplanes.}
    \label{fig:ray_example}
\end{figure}

In this special ``exact'' setting, we were able to collapse nothing, implicitly setting the cluster separation constant to be $0$, but any other choice of cluster separation constant larger than $0$ would result in a cubical model that is quasi-median quasi-isometric to this special ``exact'' model, with constants depending on that choice of larger threshold.  In general, the cluster threshold is always bigger than $1$.

In this case, choosing it to be exactly $1$ results in collapsing the infinite ray $T_{\gamma}$ to a point.  While this may appear violent, one can show (using our Tree Trimming \ref{thm:tree trimming} techniques) that the corresponding map between hierarchical families of trees
$$T_{\gamma} \times \prod_n \left(T_{n,a} \times T_{n,b}\right) \to \prod_n \left(T_{n,a} \times T_{n,b}\right),$$
induces a $0$-quasi-median $(2,0)$-quasi-isometry $\Sigma:\calQ \to \calQ'$ between their $0$-consistent subsets.

Moreover, this map $\Sigma:\calQ \to \calQ'$ preserves the infinite family of (finite diameter) squares $T_{n,a} \times T_{n,b}$ and hence the labeling data of the hyperplanes they determine in the model.  These hyperplanes, on the other hand, encode $\gamma$ as the unique $0$-simplex in the simplicial boundary of the model $\calQ'$.  In fact, any choice of larger threshold would collapsed finitely-many of these squares, but any infinite subfamily is enough to determine $\pi_T(\gamma) \in \partial T$.  Thus the hyperplane domain label data robustly encodes the limiting behavior, since every hyperplane labeled by $T$ (and indeed even finitely-many hyperplanes with lower-level labels) may be lost in the process of forming a cubical model.

It is important to note that the purely coarse version of the cubical model first constructed by Behrstock-Hagen-Sisto \cite{HHS_quasi} has this exact same issue: there are no tree-hyperplanes (which they call \emph{separators}) along $T_{\gamma}$ in their setting.  We expect that the same issue would arise in any potential extension of Bowditch's construction \cite{Bowditch_hulls} which allows for modeling median rays.  Thus, these sorts of issues are unavoidable.

By this point, we hope that this example has motivated the need for (and utility of) the ability to modify the trees used in constructing cubical models via small surgeries.  The formal framework for this is what we have referred to as \emph{tree trimming}, and we will discuss this below in Subsection \ref{subsec:TT intro}.

However, this example also highlights another necessary feature of our analysis: the ability to use domain label data to interpret the various ways in which a family of hyperplanes can diverge to infinity in a cubical model.  This latter point is more subtle, and falls under the umbrella of what are often referred to as \emph{passing-up} style arguments, which play a key role at several points in this paper, including for tree-trimming.  We discuss these next.

\subsection{Hierarchical accounting techniques}\label{subsec:PU intro}

Passing-up style arguments have been foundational to the HHS setting since the work of Masur-Minsky \cite{MM00} on the hierarchical geometry of the mapping class group $\MCG(S)$ of a finite-type surface $S$.  Notably, this hierarchy machinery also played a central role in the construction of the model manifolds used in Brock-Canary-Minsky's \cite{Min_ELC, BCM_ELC} proof of the Ending Lamination Theorem, where the notion of ``generalized marking'' we discussed in Remark \ref{rem:main rem} arises.  To contextualize our results, which fully realize some of the related structural results from \cite{MM00} to the HHS setting, we will discuss that setting first.

Following \cite{MM00}, identify $\MCG(S)$ with its \emph{marking graph}, a combinatorial model for $\MCG(S)$ whose vertices are special filling collections of curves called \emph{markings}, with two markings being connected by an edge when they differ by an \emph{elementary move} on a minimal-complexity subsurface.

A \emph{hierarchy} $\mathcal H$ between a pair of markings $\mu,\nu$ is, roughly, a collection of geodesics between $\pi_U(\mu), \pi_U(\nu) \in \calC(U)$ in the curve graphs $\calC(U)$ of a collection $\calU$ of subsurfaces on which the markings $\mu, \nu$ interact in a sufficiently complicated fashion.  These hierarchies can be \emph{resolved} to build hierarchy paths between any pair of markings \cite[Theorem 6.10]{MM00}, and the structure of $\calH$ is essential for proving the original distance formula \cite[Theorem 6.5]{MM00}.  Our choice of notation $\calU$ here is intentional: all but boundedly-many of these subsurfaces $U \in \calU$ have the property that $\mu,\nu$ project far apart in $\calC(U)$, namely $d_{\calC(U)}(\mu,\nu) \gg K$ for some fixed $K = K(S)>0$ depending on $S$.  Hence most of the subsurfaces in $\calU$ are ``relevant domains'', in the sense of Subsection \ref{subsec:BHS cube}.

Importantly, a key property of any hierarchy $\calH$ between $\mu,\nu$ is the following: Not only will every relevant subsurface $U \in \calU$ support a geodesic in $\calH$, but the boundary curves $\partial U$ of $U$ will appear in the $1$-neighborhood of some vertex along a geodesic in $\calH$ which is supported on a relevant subsurface $V \in \calU$ of which $U \subset V$ is a subsurface (see \cite[Theorem 3.1 and Theorem 4.7]{MM00}).  This fact explains Masur-Minsky's terminology for such domains, namely \emph{large links} (see \cite[Lemma 6.2]{MM00}), and motivated the Large Links Axiom \ref{ax:LL} for HHSes \cite{HHS_2}.

The \emph{passing-up} principle, as it is now referred to, is a straight-forward consequence of the above: Given a large collection of relevant subsurfaces $\calV \subset \calU$ for some pair of markings $\mu,\nu$, there exists some relevant subsurface $W\in \calU$ so that the boundary curves of some further subcollection $\calV' \subset \calV$ spreads out along any geodesic between $\pi_W(\mu),\pi_W(\nu)$ in $\calC(W)$.  Moreover, by making $\#\calV$ larger, one can force the boundary curves from subsurfaces in $\calV'$ to spread further out in $\calC(W)$.  Put another way, the boundary curves of large links cannot cluster nearby in every subsurface curve graph.

A weak version of this passing-up principle has appeared in a number of guises in the HHS setting: see Lemma \ref{lem:passing-up} for the original HHS version (from \cite[Lemma 2.5]{HHS_1} and \cite[Lemma 3.2]{DHS_boundary}), as well as a slight generalization due to Behrstock-Hagen-Sisto \cite[Lemma 1.6]{HHS_quasi}, and Lemma \ref{lem:PS} for a useful consequence due to Petyt-Spriano \cite[Lemma 5.4]{PS20}.

In this article, we prove a strong version of the passing-up principle, as well as a number of related statements, which realize more fully the organizational power of hierarchies from the surface setting.  Moreover, we fully extend these techniques to allow for hierarchy rays and boundary points, thereby generalizing the organizational principles of Minsky's infinite hierarchies.  In particular, some version of the following theorem would be necessary to prove any version of the boundary isomorphism Theorem \ref{thmi:main boundary} with respect to any cubical model construction.  Since the full statement is somewhat technical, we will only give the following informal version:

\begin{center}
\framebox{
\begin{minipage}{.96\textwidth}
\begin{propi}[Strong passing-up]\label{propi:SPU}

Let $a,b$ be interior points, boundary points, or hierarchy rays in an HHS $\calX$.  Then for any sufficiently large collection $\calV$ of relevant domains for $a,b$, there exists a relevant domain $W$ and large subcollection $\calV' \subset \calV$, so that $V \nest W$ for all $V$ and
$$\diam_{\calC(W)}\left(\bigcup_{V \in \calV'} \rho^V_W\right)$$
can be made as large as necessary by increasing the size of $\#\calV$.  Moreover, the $\rho$-sets $\rho^V_W$ distribute proportionally evenly along any geodesic between $\pi_W(a)$ and $\pi_W(b)$.
\end{propi}
\end{minipage}
}
\end{center}
\vspace{.1in}

In the above, ``proportionally evenly'' means, roughly, that the $\rho$-sets $\rho^V_W$ for $V \in \calV'$ must not only globally spread out in $\calC(W)$, but must also locally spread out.  In particular, keeping the diameter of $d_{\calC(W)}(\pi_W(a),\pi_W(b))$ fixed and increasing $\#\calV$ forces a progressively larger proportion of any geodesic $[\pi_W(a), \pi_W(b)]$ to be coarsely covered by these $\rho$-sets.  See Proposition \ref{prop:SPU} for the full statement and discussion.  We note that Zalloum \cite{Zalloum_effective} has already used Proposition \ref{propi:SPU} in his proof of effective rank rigidity for hierarchically hyperbolic groups.

Just as the passing-up principle and related ideas are essential for proving the distance formula in all of the relevant contexts \cite{MM00, Dur16, HHS_2}, this stronger version plays an essential role in several distinct ways in this article.  The following are the highlights:

\begin{itemize}
\item As discussed above in Subsection \ref{subsec:ray example}, even though the Behrstock-Hagen-Sisto's original cubical model construction did not require Strong Passing-Up \ref{prop:SPU}, some form of it is unavoidable to perform an analysis of the simplicial boundaries of these models.  In particular, in proving any variant of the simplicial isomorphism in Theorem \ref{thmi:main boundary}, one is handed an infinite family of hyperplanes in a cubical model, whose only related hierarchical data is the family of hyperplane labels which come from the relevant domains for some finite set of interior points and hierarchy rays.  In lieu of one domain labeling infinitely-many of these hyperplanes, extracting a domain of support for an infinite subfamily family requires knowing that the $\rho$-sets for those labels spread out in some domain.  This argument uses Proposition \ref{propi:SPU} in a strong way.  See Lemma \ref{lem:dust spread} in particular.
\item Strong Passing-Up \ref{prop:SPU} is essential for proving that all but boundedly-many domains appearing in the set of relevant domains for a finite collection of points and rays in an HHS are \emph{bipartite} domains, as discussed in Figure \ref{fig:Dur_cube}.  As discussed in Remark \ref{rem:main rem}, these bipartite domains are the most cubical part of the quasi-isometry between the hierarchical hull and its model, an argument which is otherwise fairly cubical-free.  See Subsection \ref{subsec:bipartite} for details and Proposition \ref{prop:bipartite} specifically.
\item The Bounding Large Containers Proposition \ref{prop:bounding containers} is a distance formula-related statement coming out this analysis.  In Section \ref{sec:distance estimates}, it is essential for proving the distance comparison between hierarchical hulls and their cubical models, as well as proving the Tree Trimming Theorem \ref{thm:tree trimming} that we will discuss next.
\item Another application of Proposition \ref{prop:bounding containers} is an elementary direct proof that changing thresholds in the distance formula works, which, to our knowledge, does not exist in the literature.  See Corollary \ref{cor:DF accounting} for details.
\item Notably, Strong Passing-Up \ref{prop:SPU} for a pair of interior points is necessary to prove standard Passing-Up \ref{lem:PU rays} for hierarchy rays, as well as the basic but important fact that the set of relevant domains for a ray (or pair of rays) is countable (Lemma \ref{lem:rel sets countable}).
\item It is standard that the Large Links Axiom \ref{ax:LL} implies the Passing-Up Lemma \ref{lem:passing-up}.  We give a proof of the converse in Proposition \ref{prop:strong ll}, and hence that the Large Links Axiom is equivalent to Passing-Up, so that the latter could replace the former in the axiomatic setup of HHSes.  We note that this fact was independently discovered by Hagen and Petyt \cite{HP_pc}.
\end{itemize}

\subsection{Tree-trimming: Refining the distance formula} \label{subsec:TT intro}

In our discussion of the cubical model Theorem \ref{thmi:main model} and the variation which captures curve graph distance (Theorem \ref{thmi:curve graph}), we saw how modifying the trees involved in the construction of the cubical models can be quite useful, as it allows one to directly modify the cubical geometry of the models in a fine-tuned fashion.  We will now discuss the machinery that makes this work, namely the Tree Trimming Theorem \ref{thmi:tree trimming} below.  Since it can be reasonably understood a refinement of the distance formula, we will begin our discussion there.

Given a pair of points $a,b \in \calX$ in an HHS, the distance formula (Theorem \ref{thm:DF}) says that the distance between them in $\calX$ is coarsely the $\ell^1$-sum of distances in the hyperbolic spaces to which they project far apart:
\begin{equation}\label{eq:DF intro}
d_{\calX}(a,b) \asymp_K \sum_{U}[[d_{U}(a,b)]]_K.
\end{equation}
In the above coarse inequality, we need to take $K$ sufficiently large (depending only on $\calX$), and then the additive and multiplicative constants of $\asymp_K$ depend on $K$ and $\calX$.  In particular, taking $K'>K$ provides a new distance estimate, in which any term $d_V(a,b) < K'$ disappears.  While it is hard to overstate how useful this distance formula has proven in the literature, we have developed a refined version which has some added benefits.

One of the upshots of Theorem \ref{thmi:main model} is the introduction of a cubical model $\calQ$ where the distance is actually an $\ell^1$-sum:
\begin{equation}\label{eq:Q DF intro}
d_{\calX}(a,b) \asymp_K d_{\calQ}(\ha,\hb) = \sum_{U \in \calU} d_{\hT_U}(\ha,\hb),
\end{equation}
where $\ha,\hb$ are identified with $a,b$ under the model map $\hO:\calQ \to \hull_{\calX}(a,b)$, and the $\hT_U$ are the collapsed trees (actually intervals, in this case) associated to $a,b$ and each $U \in \calU$.

In our construction, we also need to choose some sufficiently large $K$ (again, depending only on $\calX$) to collect the $K$-relevant domains $\calU$, which are precisely the domains contributing to the right side of equation \eqref{eq:DF intro}.  While $\calQ$-distance again only gives a coarse estimate, an obvious benefit of this setup is that equation \eqref{eq:Q DF intro} allows one to (coarsely) account for how much each domain in $\calU$ contributes in this estimate.  In other words, distance in the model $\calQ$ is a concrete realization of the distance formula.

However, this exact model has another benefit: we can perform small surgeries on each of the intervals $\hT_U$ which only changes the coarse estimate to the distance in $\calX$ in a controlled fashion.  Put another way, by extracting the precise contributions of the relevant domains---which is achieved by excising the redundant information coming from nesting, as discussed in Subsection \ref{subsec:Dur cube}---we can further make \emph{partial} alterations to the contributions of each relevant domain, while preserving the combinatorial information of the $\rho$-sets encoded in the trees.  Compare this with the distance formula \eqref{eq:DF intro}, in which the only dimension of flexibility results in the complete deletion of the data associated to a relevant domain.

The general version of this idea is the statement of our Tree Trimming Theorem \ref{thm:tree trimming}, which plays a crucial role in a number of places in this article.  To state an informal version, observe that in the setup of our cubical model Theorem \ref{thmi:main model}, each of the collapsed trees $\hT_U$ for $U \in \calU$ admits a decomposition $\hT_U = \hT^e_U \sqcup \hT^c_U$ along the marked points $\hT^c_U$ which encode the leaves and cluster (nesting) data (see Figure \ref{fig:Dur_cube}).  The subtrees $\hT^e_U$ in the complements of these marked points encode the information for which $U$ is responsible in equation \eqref{eq:Q DF intro}.  The idea is that we can delete (or insert) small subtrees into these complementary subtrees, and produce a new hierarchical family of trees whose cubical realization is $0$-median quasi-isometric to the original.  See Figure \ref{fig:tree_trim} below for a heuristic picture.

The following is an informal version of Theorem \ref{thm:tree trimming}:

\vspace{.1in}

\begin{center}
\framebox{
\begin{minipage}{.96\textwidth}

\begin{thmi}[Tree trimming]\label{thmi:tree trimming}
Let $\calQ$ be a cubical model arising from a hierarchical family of trees $\{\hT_U\}_{U \in \calU}$ associated to a finite set $F \cup \Lambda$ of interior points, boundary points, and hierarchy rays in an HHS $\calX$.  Then there exists a $B_{tt}$ depending only on the constants in the construction of $\calQ$, so that if $B< B_{tt}$ and $\Delta_U:\hT_U \to \hT'_U$ collapses at most $B$-many subtrees of diameter at most $B$ in each complementary component of $\hT^e_U$ for each $U \in \calU$, then 
\begin{enumerate}
\item Then the collapsing maps $\Delta_U:\hT_U \to \hT'_U$ combine to produce a new hierarchical family of tree $\{\hT'_U\}_{U \in \calU}$ over the same index set, and
\item The induced map $\Delta: \calQ \to \calQ'$ on $0$-consistent sets is a $0$-median $(L, L)$-quasi-isometry, where $L = L(\calX, |F\cup \Lambda|, B)>0$.
\end{enumerate}
In particular, $\calQ'$ is a new cubical model for $\hull_{\calX}(F \cup \Lambda)$.

\end{thmi}

\end{minipage}
}
\end{center}
\vspace{.1in}

Notice that Theorem \ref{thmi:tree trimming} allows one to delete arbitrarily many---and perhaps infinitely-many---small pieces from the various trees $\hT_U$, as discussed in the example in Subsection \ref{subsec:ray example}, while only distorting the metric on the resulting model a bounded amount.  Since the hyperplanes in our cubical models come directly from the collapsed trees, this ability to modify the trees at a small scale provides us great flexibility in controlling the cubical geometry of the models.  For instance, Theorem \ref{thmi:curve graph} is a direct application of this philosophy.  We will discuss some of the details of these applications in the next subsection.

\subsection{Outline of the paper and some proofs} \label{subsec:sketch}

Having given the broad strokes of the paper with some emphasis on various parts of the argument, we will now go into a bit more detail of how all of the pieces fit together.

In Section \ref{sec:background}, we give some background on HHSes, cube complexes, and (coarse) median spaces.  In Section \ref{sec:ray projections}, we state the main definitions of the objects we are modeling in this paper: hierarchy rays, boundary points, and their hierarchical hulls.  Some of this analysis is new, like the construction of hierarchy rays whose projections encode those of a basepoint and boundary point pair (Proposition \ref{prop:ray replace boundary}), and Definition \ref{defn:extended consistency} and Proposition \ref{prop:extended consistency}, which extends the hierarchical notion of consistency to rays and boundary points.

In Section \ref{sec:qpu}, we develop the hierarchical accounting techniques we discussed in Subsection \ref{subsec:PU intro} above.  As discussed there, this is where we prove the Strong Passing-up Proposition \ref{prop:SPU}, which is useful for understanding how domain data spreads out in the various hyperbolic spaces, as well as the Bounding Containers Proposition \ref{prop:bounding containers}, which is important for distance formula-type arguments.  In other words, the material here is foundational for the rest of the paper.

At this point, our hierarchical foundation is set, and in Section \ref{sec:constants}, we set some notation, constants, and basic assumptions that we will carry throughout the rest of the paper.

In Section \ref{sec:reduce to tree}, we begin the construction of the cubical models in earnest.  Given a finite set of interior points, boundary points, and hierarchy rays $F \cup \Lambda$, the starting point is encoding the information of the projection of $F \cup \Lambda$ to each relevant domain $\calC(U)$ for $U \in \calU$ and its accompanying $\rho$-data into a Gromov modeling tree $T_U$ adorned with analogous $\rho$-data; see Definition \ref{defn:tree projections} and Lemma \ref{lem:tree control}, as well as Definition \ref{defn:tree consistency} for the related notion of consistency.  This reduction is quite useful, as the product of trees $\prod_{U \in \calU} T_U$ is a natural landing spot both for the coarse information coming from the hierarchical hull $H = \hull_{\calX}(F \cup \Lambda)$, as well as the exact information coming from the eventual cubical model $\calQ$.

In Section \ref{sec:Q}, we explain the procedure for building the collapsed trees $\hT_U$ from the modeling trees $T_U$, by collapsing down nesting data.  Part of this analysis involves translating the coarse hierarchical information encoded into the $T_U$ into exact hierarchical information in the $\hT_U$, see Definition \ref{defn:Q project} and Lemma \ref{lem:collapsed tree control}.  Definition \ref{defn:Q consistent} gives the precise definition of $0$-consistency, thus of the model $\calQ$, and the map $\hPhi:H \to \calQ$.  This also involves developing a notion of canonicality, which is necessary for differentiating between interior points and points at infinity.

In Section \ref{sec:distance estimates}, we prove our distance comparison estimates between $H$ and $\calQ$ (Theorem \ref{thm:Q distance estimate}).  In particular, we prove that the distance formula sum in $H$ is coarsely the $\ell^1$-distance in $\calQ$.  This statement is useful for proving the upper-bound of the Distance Formula \ref{thm:DF}, as discussed in Subsection \ref{subsec:PU intro} above.  More importantly, assuming the distance formula, this proves that $\hPhi:H \to \calQ$ is a uniform quasi-isometric embedding.

In Section \ref{sec:Q to Y}, we build our map $\hO: \calQ \to H$, the main part of which is defining a coordinate-wise map $\Omega: \calQ \to \prod_{U \in \calU} T_U$.  This involves a detailed analysis because of all of the information that has been collapsed $T_U \to \hT_U$.  More specifically, given a canonical $0$-consistent tuple $\hx \in \calQ$, almost every coordinate $\hx_U$ coincides with a \emph{marked point}, which corresponds to a possibly arbitrarily-large (if not infinite-diameter) cluster $C \subset T_U$.  Deciding where to send $\hx_U$ in $C$ involves considering essentially all of the domains nesting into $U$ whose nesting data is contained in $C$.   See Subsection \ref{subsec:honing clusters} for a detailed description of the \emph{cluster honing} process.  As mentioned above, this part of the argument has a distinctly cubical flavor, as it depends on the ubiquity of the two-sided \emph{bipartite domains}, and proving their ubiquity involves Strong Passing-up \ref{prop:SPU}.

In Section \ref{sec:tree trimming}, we develop our tree-trimming techniques, proving the Tree Trimming Theorem \ref{thm:tree trimming} and the fact that each of the collapsed trees $\hT_U$ can be replaced by a simplicial tree, where the marked points are vertices (Corollary \ref{cor:simplicial structure}).  As discussed in Subsection \ref{subsec:PU intro}, Theorem \ref{thm:tree trimming} is a distance-formula type result, and the quasi-isometry statement depends directly on the Bounding Containers Proposition \ref{prop:bounding containers}, as well as an ``exact'' version of the cluster honing process from the above definition of $\hO:\calQ \to H$.

In Section \ref{sec:coarsely surjective}, we prove that $\hPhi:H \to \calQ$ is coarsely surjective, the argument for which uses Theorem \ref{thm:tree trimming} in a crucial way.  This involves proving Theorem \ref{thm:coarse inverse}, which  states that the composition $\hPsi \circ \hO:\calQ \to \calQ$ is coarsely constant, where $\hPsi:H \to \calQ$ and $\hO:\calQ \to H$ are the model maps.  This is the point in the article where the coarse nature of the hull meet most plainly with the exact nature of the model.  In the course of the proof, one needs to compare the coordinate tuples for a point $\hx \in \calQ$ and its image $\hy = \hPsi \circ \hO (\hx) \in \calQ$.  Using coarse techniques, one can show that all but boundedly-many of these coordinates $\hx_U, \hy_U$ are within bounded distance of each other (see Proposition \ref{prop:coarsely surjective, trees}).  However, one must show that all but boundedly-many coordinates are exactly the same, since $d_{\calQ}(\hx,\hy) = \sum_U d_{\hT_U}(\hx_U,\hy_U)$.  Tree Trimming \ref{thm:tree trimming} solves this problem by allowing us to trim away the (a priori infinitely-many) intervals connecting the distinct $\hx_U, \hy_U$ which are at bounded distance.  Since the resulting map on cubical models is a uniform quasi-isometry with the images of $\hx,\hy$ at bounded distance, it follows that $d_{\calQ}(\hx,\hy)$ was bounded after all; see Subsection \ref{subsec:proof of coarsely surjective}.

In Section \ref{sec:walls in Q}, we analyze the induced wall-space structure on $\calQ$.  In the interior point case, this involves moving to the exact setting of cube complexes built from hierarchical family of trees (Definition \ref{defn:HFT}), an axiomatization of the setup we extract from hierarchical hulls.  The main purpose of this section is to develop a description of the walls in $\calQ$, as well as how their half-spaces nest and intersect (Subsection \ref{subsec:walls and halfspaces}), as well as characterizing infinite descending chains (Proposition \ref{prop:DCC}).  This description is crucial for our proof in Section \ref{sec:Q is cubical}, in which we prove that $\calQ$ is isometric to its cubical dual, while also showing that non-canonical tuples in $\calQ$ (namely, those encoding points at infinity), are in bijection with the $0$-cubes at infinite in the cubical dual of $\calQ$.  It is also important for our analysis of the simplicial boundary of $\calQ$ in Section \ref{sec:boundary compare}.

In Section \ref{sec:medians}, we complete the proof of Theorem \ref{thmi:main model} by analyzing and synchronizing the (coarse) median structures of the various objects involved, in particular by showing that the maps $\hO:H \to \calQ$ and between $\calQ$ and its cubical dual are median-preserving.  We have isolated the content of this section mainly because the straightforward nature of this discussion helps make transparent how the various parts of the construction play together. 

In Section \ref{sec:HP and DF}, we pause to observe how to use the cubical geometry of the model of a pair of points to prove the existence of hierarchy paths (Proposition \ref{prop:cube paths hp}) and the upper-bound in the distance formula (Corollary \ref{cor:DF lower bound}).  For this, we give a direct proof that any combinatorial geodesic (ray) in any cubical model $\calQ$ is mapped to a hierarchy path (or ray) back in $\calX$.  We then observe that the length of such image of such a \emph{cubical path} (Definition \ref{defn:cubical path}) between a pair of point $x,y \in \calX$ is coarsely the distance formula sum by Theorem \ref{thm:Q distance estimate}.

In Section \ref{sec:boundary compare}, we prove our boundary isomorphism statement, Theorem \ref{thm:boundary talk}, which we discussed in Subsection \ref{subsec:infinity intro}.  This involves laying out some background material, including Hagen's simplicial boundary for cube complexes \cite{Hagen_simplicial}  and the simplicial structure on the HHS boundary.  The key step for the isomorphism is Lemma \ref{lem:minimal spread}, which uniquely identifies a $0$-simplex in the HHS boundary of the hull $H$ to any $0$-simplex (i.e., a minimal UBS) in the simplicial boundary of a model $\calQ$, which combines with Hagen's decomposition Theorem \ref{thm:minimal decomp} to provide a hierarchical representative in the general case, and thus a simplicial embedding $\partial_{\Delta} \calQ \to \partial H$.  Surjectivity requires extra work, and proceeds by showing that $\partial_{\Delta} \calQ$ is visible, i.e. every point admits a combinatorial ray representative, which, moreover, can be chosen to have whatever limiting hierarchical data is required (Lemma \ref{lem:boundary surj}).

Finally, in Section \ref{sec:sep hyp}, we explain how to modify the cubical model construction to encode curve graph distance in the hyperplanes of the model, as in Theorem \ref{thmi:curve graph}.  To modify the cubical construction, we take our inspiration from the surface world.  Two simple closed curves $\alpha, \beta$ \emph{fill} a compact surface $S$ if the complement of their union is a union of disks, and it is not hard to show that if $\alpha_1, \dots, \alpha_n$ are a pairwise filling collection of curves on $S$, then $\diam_{\calC(S)}(\alpha_1 \cup \cdots \cup \alpha_n) \asymp n$.  We introduce an analogous notion of filling for any pair of domains in an HHS (Definition \ref{defn:filling}), and observe that the $\rho$-sets of a filling collection of domains must spread out in the top-level hyperbolic space.  As an application using our Tree Trimming \ref{thm:tree trimming} techniques, we explain how to construct a cubical model for any pair of points so that two hyperplanes are $0$-separated if and only if their domain labels are filling.  Hence long $0$-separated chains produce large distances in the top-level hyperbolic space; see Subsection \ref{subsec:well-sep cube}.

\subsection{Acknowledgements}

This work was supported in part by NSF grant DMS-1906487.  I would also like to thank Cornell University for their hospitality during the academic year of 2021-22, during which some of this work was completed.  I am also indebted to the following people:

\begin{itemize}
\item I would like to thank Mark Hagen for fielding many of my questions about HHSes and cube complexes, and always being interested in hearing my ideas.  Mark is generous with his time and an invaluable resource to the geometric group theory community.  His encouragement was necessary for me to finish this paper.  
\item I would like to thank Abdul Zalloum for numerous conversations about HHSes, cube complexes, and especially infinite cubical models.  This paper was born in part out of the incompleteness of the limiting cubical models we developed in our joint work \cite{DZ22}, and I would have neither started nor finished this project without Abdul's interest and encouragement.
\item Thanks to Mark Hagen, Alessandro Sisto, and Abdul Zalloum for helpful comments on an earlier draft of this paper.  Thanks in particular for Sisto's question about the relationship between \cite{CRHK} and this paper, and Hagen for an illuminating subsequent discussion. 
\item Thanks to Daniel Groves, with whom we realized that a different construction of the cubical models would be useful while working on another project, and for preliminary discussions about my approach.
\item Thanks to Yair Minsky and Alessandro Sisto for our collaboration on \cite{DMS20}, where we developed the idea of clustering relative projections into subtrees.  In particular, the clarity that the decomposition of our stable trees provided in that project also helped me see that collapsing the cluster subtrees to points was the way to go.  I also thank both of them for interesting conversations about and their interest in this project.
\item Thanks to Johanna Mangahas for inviting me to give a seminar talk at Buffalo, and for her interest in and insightful questions about my alternative construction of the cubical models.
\item Thanks to Montse Casals-Ruiz, Ilya Kazachkov, Harry Petyt, and Davide Spriano, for making time to understand some of this project in more depth, and to the organizers of semester long program on Geometric Group Theory during the spring of 2023 at CRM in Montreal for providing a place for these conversations to happen.
\item Thanks to the UCR graduate students in my Fall 2022 topics course during which I presented part of this work, especially Jacob Garcia, William Sablan, Elliott Vest, and Sean Wakasa for their useful questions and interest in the work.
\item Thanks to many friends, colleagues, and collaborators with whom I have spoken about this project and related ideas over the past few years, including Matt Cordes, Ruth Charney, R\'{e}mi Coulon, Spencer Dowdall, Ben Dozier, Talia Fern\'{o}s, Jonah Gaster, Anna Gilbert, Thomas Koberda, Chris Leininger, Marissa Loving, Katie Mann, Jason Manning, Chandrika Sadanand, Jenya Sapir, Rishi Sonthalia, and Sam Taylor.
\item Thanks to all of the participants at the Riverside Workshop in Geometric Group Theory 2023, for their questions and interest in my minicourse which involved part of this work.
\end{itemize}

\section{Background on hierarchically hyperbolic spaces} \label{sec:background}

In this section, we give the basic definitions and properties of hierarchically hyperbolic spaces.  Most of the material is standard here, with \cite{HHS_2} the main reference.  See also Sisto's expository article \cite{Sisto_HHS}.

\subsection{Coarse (in)equality notation}

Many of the calculations in this paper are estimates which only hold up to some multiplicative and additive constants, which can depend on various background objects or quantities.  In at attempt to limit notation burden, we will write
\begin{itemize}
\item $A \prec B$ when there exist constants $K\geq 1, C \geq 0$ so that $A \leq K \cdot B + C$.
\item We define $A \succ B$ similarly, and set $A \asymp B$ to mean $A \prec B$ and $A \succ B$.
\item When we say ``the constants in $A \prec B$ depend on'' certain quantities or objects, we mean that the constants $K,C$ can be written as functions of those quantities or quantities related to those objects.
\item In particular, when we say that a coarse (in)equality depends only on an HHS $\calX$, this means that they can be writing as functions of the fundamental constants associated to the HHS, which we lay out in the next subsection.  See also Section \ref{sec:constants} for a simplification of the constants associated to a fixed HHS.
\end{itemize}

\subsection{HHS axioms and key lemmas}

Hierarchically hyperbolic spaces were introduced in \cite{HHS_1, HHS_2}.  Since we need several of them, we will include the list of axioms, and some of their important consequences.  See \cite{CRHK} for a thorough introduction and \cite{Sisto_HHS} for a more philosophical survey.

We begin with the basic definition from \cite{HHS_2}:

\begin{definition}\label{defn:HHS}
The $q$--quasigeodesic space  $(\calX,\dist_{\calX})$ is a \textbf{\em hierarchically hyperbolic space} if there exists $\delta\geq0$, an index set $\mathfrak S$, and a set $\{\calC(W):W\in\mathfrak S\}$ of $\delta$--hyperbolic spaces $(\calC(U),\dist_U)$,  such that the following conditions are satisfied:\begin{enumerate}
\item\textbf{(Projections.)}\label{item:dfs_curve_complexes} There is
a set $\{\pi_W: \calX\rightarrow2^{\calC(W)}\mid W\in\mathfrak S\}$
of \textbf{\em projections} sending points in $\calX$ to sets of diameter
bounded by some $\xi\geq0$ in the various $\calC(W)\in\mathfrak S$.
Moreover, there exists $K$ so that for all $W\in\mathfrak S$, the coarse map $\pi_W$ is $(K,K)$--coarsely
Lipschitz and $\pi_W(\calX)$ is $K$--quasiconvex in $\calC(W)$.

 \item \textbf{(Nesting.)} \label{item:dfs_nesting} $\mathfrak S$ is
 equipped with a partial order $\nest$, and either $\mathfrak
 S=\emptyset$ or $\mathfrak S$ contains a unique $\nest$--maximal
 element; when $V\nest W$, we say $V$ is \textbf{\em nested} in $W$.  (We
 emphasize that $W\nest W$ for all $W\in\mathfrak S$.)  For each
 $W\in\mathfrak S$, we denote by $\mathfrak S_W$ the set of
 $V\in\mathfrak S$ such that $V\nest W$.  Moreover, for all $V,W\in\mathfrak S$
 with $V\nest W$ there is a specified subset
 $\rho^V_W\subset\calC(W)$ with $\diam_{\calC(W)}(\rho^V_W)\leq\xi$.
 There is also a \textbf{\em projection} $\rho^W_V: \calC(W)\rightarrow 2^{\calC(V)}$.  (The similarity in 
notation is justified by viewing $\rho^V_W$ as a coarsely constant map $\calC(V)\rightarrow 2^{\calC(W)}$.)
 
 \item \textbf{(Orthogonality.)} 
 \label{item:dfs_orthogonal} $\mathfrak S$ has a symmetric and
 anti-reflexive relation called \textbf{\em orthogonality}: we write $V\perp W$ when $V,W$ are orthogonal.  Also, whenever $V\nest W$ and $W\perp
 U$, we require that $V\perp U$.  We require that for each
 $T\in\mathfrak S$ and each $U\in\mathfrak S_T$ for which
 $\{V\in\mathfrak S_T\mid V\perp U\}\neq\emptyset$, there exists $W\in
 \mathfrak S_T-\{T\}$, so that whenever $V\perp U$ and $V\nest T$, we
 have $V\nest W$.  Finally, if $V\perp W$, then $V,W$ are not
 $\nest$--comparable.
 
 \item \textbf{(Transversality and consistency.)}
 \label{item:dfs_transversal} If $V,W\in\mathfrak S$ are not
 orthogonal and neither is nested in the other, then we say $V,W$ are
 \textbf{\em transverse}, denoted $V\pitchfork W$.  There exists
 $\kappa_0\geq 0$ such that if $V\pitchfork W$, then there are
  sets $\rho^V_W\subseteq\calC(W)$ and
 $\rho^W_V\subseteq\calC(V)$ each of diameter at most $\xi$ and 
 satisfying: 
 \begin{equation}\label{eq:transverse consistency}
   \min\left\{\dist_{
 W}(\pi_W(x),\rho^V_W),\dist_{
 V}(\pi_V(x),\rho^W_V)\right\}\leq\kappa_0
 \end{equation}
 for all $x\in\calX$.

 For $V,W\in\mathfrak S$ satisfying $V\nest W$ and for all
 $x\in\calX$, we have: 
\begin{equation}\label{eq:nested consistency}
\min\left\{\dist_{ W}(\pi_W(x),\rho^V_W),\diam_{\calC(V)}(\pi_V(x)\cup\rho^W_V(\pi_W(x)))\right\}\leq\kappa_0. 
\end{equation}

 The preceding two inequalities are the \textbf{\em consistency inequalities} for points in $\calX$.
 
 Finally, if $U\nest V$, then $\dist_W(\rho^U_W,\rho^V_W)\leq\kappa_0$ whenever $W\in\mathfrak S$ satisfies either $V\nest W$ or $V\pitchfork W$ and $W,U$ not $\perp$.
 
 \item \textbf{(Finite complexity.)} \label{item:dfs_complexity} There exists $n\geq0$, the \textbf{\em complexity} of $\calX$ (with respect to $\mathfrak S$), so that any set of pairwise--$\nest$--comparable elements has cardinality at most $n$.
  
 \item \textbf{(Large links.)} \label{ax:LL} There
exist $\lambda\geq1$ and $E\geq\max\{\xi,\kappa_0\}$ such that the following holds.
Let $W\in\mathfrak S$ and let $x,x'\in\calX$.  Let
$N=\lambda\dist_{_W}(\pi_W(x),\pi_W(x'))+\lambda$.  Then there exists $\{T_i\}_{i=1,\dots,\lfloor
N\rfloor}\subseteq\mathfrak S_W-\{W\}$ such that for all $T\in\mathfrak
S_W-\{W\}$, either $T\in\mathfrak S_{T_i}$ for some $i$, or $\dist_{
T}(\pi_T(x),\pi_T(x'))<E$.  Also, $\dist_{
W}(\pi_W(x),\rho^{T_i}_W)\leq N$ for each $i$.

 \item \textbf{(Bounded geodesic image.)}
 \label{ax:BGIA} There exists $E>0$ such that 
 for all $W\in\mathfrak S$,
 all $V\in\mathfrak S_W-\{W\}$, and all geodesics $\gamma$ of
 $\calC(W)$, either $\diam_{\calC(V)}(\rho^W_V(\gamma))\leq E$ or
 $\gamma\cap \Neb_{E}(\rho^V_W)\neq\emptyset$.
 
 \item \textbf{(Partial Realization.)} \label{item:dfs_partial_realization} There exists a constant $\alpha$ with the following property. Let $\{V_j\}$ be a family of pairwise orthogonal elements of $\mathfrak S$, and let $p_j\in \pi_{V_j}(\calX)\subseteq \calC(V_j)$. Then there exists $x\in \calX$ so that:
 \begin{itemize}
 \item $\dist_{V_j}(x,p_j)\leq \alpha$ for all $j$,
 \item for each $j$ and 
 each $V\in\mathfrak S$ with $V_j\nest V$, we have 
 $\dist_{V}(x,\rho^{V_j}_V)\leq\alpha$, and
 \item if $W\pitchfork V_j$ for some $j$, then $\dist_W(x,\rho^{V_j}_W)\leq\alpha$.
 \end{itemize}

\item\textbf{(Uniqueness.)} For each $\kappa\geq 0$, there exists
$\theta_u=\theta_u(\kappa)$ such that if $x,y\in\calX$ and
$\dist_{\calX}(x,y)\geq\theta_u$, then there exists $V\in\mathfrak S$ such
that $\dist_V(x,y)\geq \kappa$.\label{item:dfs_uniqueness}
\end{enumerate}
 
We refer to $\mathfrak S$, together with the nesting
and orthogonality relations, and the projections as a \textbf{\em hierarchically hyperbolic structure} for the space $\calX$, and denote the structure by $(\calX, \mathfrak S)$.
\end{definition}

We usually suppress the projection
map $\pi_U$ for $U \in \mathfrak S$ when writing distances in $\calC(U)$, e.g., given $x,y\in\calX$ and
$p\in\calC(U)$  we write
$\dist_U(x,y)$ for $\diam_{\calC(U)}(\pi_U(x)\cup\pi_U(y))$ and $\dist_U(x,p)$ for
$\diam_{\calC(U)}(\pi_U(x)\cup\{p\})$.
Given $A\subset \calX$ and $U\in\mathfrak S$ 
we let $\pi_{U}(A)$ denote $\displaystyle \bigcup_{a\in A}\pi_{U}(a)$.

\subsection{Basic HHS facts}

We now collect some useful facts about HHSes from various places.

The following definition from \cite{Behr06, BKMM, HHS_2}, will be important for us:

\begin{definition}[Consistent tuple]\label{defn:consistency}
Let $\kappa\geq 0$ and let $(b_U)\in\prod_{U\in\mathfrak S}2^{\calC(U)}$ be a tuple such that for each $U\in\mathfrak S$, 
the $U$--coordinate  $b_U$ has diameter $\leq\kappa$.  Then $(b_U)$ is \textbf{\em $\kappa$--consistent} if for all 
$V,W\in \mathfrak S$, we have $$\min\{\dist_V(b_V,\rho^W_V),\dist_W(b_W,\rho^V_W)\}\leq\kappa$$ whenever $V\pitchfork W$ and 
$$\min\{\dist_W(x,\rho^V_W),\diam_V(b_V\cup\rho^W_V)\}\leq\kappa$$ whenever $V\nest W$.
\begin{itemize}
\item We refer to the above inequalities as the \emph{consistency inequalities}.
\end{itemize}
\end{definition}

\medskip

The following is essentially immediate from Axiom \ref{item:dfs_transversal}:

\begin{lemma}\label{lem:base consistency}
There exists $\theta = \theta(\mathfrak S)>0$ so that if $x \in \calX$, then the tuple $(\pi_U(x)) \in \prod_{U \in \mathfrak S} \calC(U)$ is $\theta$-consistent.
\end{lemma}

The following Realization Theorem says that any tuple satisfying these inequalities can be coarsely realized as a point in $\calX$.  It is \cite[Theorem 4.3]{BKMM} for mapping class groups, \cite[Theorem 3.2]{EMR_rank} for Teichm\"uller space with the Teichm\"uller metric, and \cite[Theorem 3.1]{HHS_2} for the general HHS case:

\begin{theorem}[Realization]\label{thm:realization}
For each $\kappa \geq 1$ there exists $\theta_e, \theta_u \geq 0$ such that the following holds.  Let $\mathbf{b} = (b_U)_{U \in \mathfrak S} \in \prod_{U \in \mathfrak S} 2^{\calC(U)}$ be $\kappa$-consistent.

Then there exists $x \in X$ so that $d_U(b_U, \pi_W(x)) \leq \theta_e$ for all $U \in \mathfrak S$.  Moreover, $x$ is \emph{coarsely unique} in the sense that 
$$\diam_{\calX} \{x \in \calX| d_W(b_W, \pi_W(x)) \leq \theta_e \textrm{ for all } U \in \mathfrak S\} < \theta_u.$$
\end{theorem}

\begin{remark}
Notably, the proof of Theorem \ref{thm:realization} in \cite{HHS_2} does not involve the hyperbolicity of the associated spaces $\{\calC(U)\}_{U \in \mathfrak S}$.  In this sense, realization is a purely hierarchical fact which plays a similar role to the existence of a coarse median in Bowditch's analysis \cite{Bowditch_hulls}. 
\end{remark}

Perhaps the main technique in HHSes is taking an object in the ambient space $\calX$ and projecting it to all of the hyperbolic spaces, doing hyperbolic geometry, and then reassembling the pieces via hierarchical tools like the Realization Theorem \ref{thm:realization}.  Toward that end, and recognizing that HHSes are fundamentally coarse spaces, we will usually want to restrict ourselves to the \emph{relevant} collection of hyperbolic spaces to which an object has a sufficiently large projection.

\begin{definition}[Relevant domains] \label{defn:relevant}
Given a subset $A \subset \calX$ and $K > 0$, we say that $U \in \mathfrak S$ is $K$-\emph{relevant} for $A$ if $\diam_U(A) \geq K$.  We denote the set of $K$-relevant domains for $A$ by $\Rel_K(A)$.
\end{definition}

The following Passing-Up lemma is a foundational fact about HHSes.  Indeed, we will develop multiple more sophisticated version of it in Section \ref{sec:qpu} below, and, in Proposition \ref{prop:strong ll}, prove that it is equivalent to the Large Links Axiom \ref{ax:LL}.

Roughly speaking, the passing-up lemma says that if $x,y \in \calX$ have a large number of lower complexity relevant domains, then there must be some higher complexity relevant domain containing at least one of them, thereby allowing us to ``pass up'' the $\nest$-lattice. 

\begin{lemma}[Passing-up]\label{lem:passing-up}
Let $\calX$ be an HHS with constant $E$.  For every $D>0$ there is an integer $P=P(D)>0$ such that if $V \in \mathfrak S$ and $x,y \in \calX$ satisfy $d_{U_i}(x,y)>E$ for a collection of domains $\{U_i\}_{i=1}^{P}$ with $U_i \in \mathfrak S_V$, then there exists $W \in \mathfrak S_V$ with $U_i \sqsubsetneq W$ for some $i$ such that $d_W(x,y)>D.$
\end{lemma}

We observe the following corollary, which is embedded in the proof of Theorem \ref{thm:realization} in \cite[Theorem 3.1]{HHS_2}, but not explicitly stated.  We give a proof for completeness:

\begin{corollary}\label{cor:rel sets are finite}
For any $x,y \in \calX$, the set $\Rel_E(x,y)$ is finite.
\end{corollary}

\begin{proof}
Let $C$ be the coarse Lipschitz constant for all of the projections $\pi_U$ for $U \in \mathfrak S$, and set $D > C d_{\calX}(x,y) + C$.  If $\# \Rel_E(x,y) > P(D)$, then the Passing-Up Lemma \ref{lem:passing-up} provides a domain $W \in \mathfrak S$ with $d_W(x,y)>D$, which is impossible.
\end{proof}

Another fundamental fact about HHSes is the following distance formula, which originated in work of Masur-Minsky \cite{MM00}, and found various iterations elsewhere \cite{Brock_WP, Rafi_hyp, Dur16, KK_HHS} before being generalized to all HHSes in \cite{HHS_2}:

\begin{theorem}[Distance formula]\label{thm:DF}
There exists $K_0 = K_0(\mathfrak S)>0$ so that if $K>K_0$ and $x,y \in \calX$, then
$$d_{\calX}(x,y) \asymp \sum_{U \in \mathfrak S} [d_U(x,y)]_K$$where the coarse equality $\asymp$ depends only on $\mathfrak S$ and $K$.
\end{theorem}

As mentioned in the introduction, we will give a new proof of the more difficult lower bound in the distance formula in Corollary \ref{cor:DF lower bound} below.

\subsection{Normalized HHSes}

Sometimes, one is handed an HHS whose hyperbolic spaces contain more information than necessary.  Following \cite[Proposition 1.16]{DHS_boundary}, we assume that every HHS $\calX$ that we worked satisfies the following assumption:

\begin{definition}[Normalized HHS]\label{defn:normalized}
An HHS $(\calX, \mathfrak S)$ is \emph{normalized} if there exists $C>0$ so that for all $U \in \mathfrak S$, we have $\calC(U) = \calN_C(\pi_U(\calX)))$.
\end{definition}

That is, we will always assume that our domain projections $\pi_U: \calX \to \calC(U)$ are uniformly coarsely surjective.

\subsection{Hierarchical quasi-convexity, gates, and induced HHS structures}

The following notion of quasi-convexity generalizes the hyperbolic setting \cite[Definition 5.1]{HHS_2}:

\begin{definition}\label{defn:hqc}
A subset $A \subset \calX$ is $k$-\emph{hierarchically quasiconvex} for some $k:[0,\infty) \to [0,\infty)$ if
\begin{enumerate}
\item $\pi_U(A)$ is $k(0)$-quasiconvex for each $U \in \mathfrak S$;
\item For all $\kappa\geq 0$ and $\kappa$-consistent tuples $(b_U) \subset \prod_{U \in \mathfrak S} \calC(U)$ with $b_U \subset \pi_U(A)$ for all $U \in \mathfrak S$, if $x \in \calX$ satisfies $d_U(x, b_U)<\theta_e(\kappa)$ (where $\theta_e(\kappa)$ is as in Theorem \ref{thm:realization}), then $d_{\calX}(x,A)<k(\kappa)$.
\end{enumerate}
\end{definition}

There is a natural notion of projection to a hierarchically quasi-convex subset, described in \cite{HHS_2} as follows:

\begin{lemma}\label{lem:gates exist}
Let $A \subset \calX$ be $k$-hierarchically quasi-convex.  For any $x \in \calX$, define a tuple $\gate(x) \in \prod_{U \in \mathfrak S} 2^{\calC(U)}$ by taking $\gate(x)_U$ to be closest point projection to $\pi_U(A)$.  Then $\gate(x)$ is $\alpha$-consistent, where $\alpha = \alpha(k(0), \calX)$.   
\end{lemma}

The map $\gate_A:\calX \to A$ is referred to as a \emph{gate map} in much of the HHS literature because these hierarchically quasi-convex sets are median convex \cite{HHS_quasi, RST18} and coarsely coincide with the median gate associated to the coarse median structure on $\calX$.

Another useful fact is that hierarchically quasi-convex subsets admit an induced HHS structure as follows (see \cite[Proposition 5.6]{HHS_2}):

\begin{lemma}\label{lem:hqc induce}
Let $A \subset \calX$ be a $k$-hierarchically quasi-convex subset.  Then considering $A$ with the metric restricted from $\calX$, the following give an HHS structure on $A$ called the \emph{induced HHS structure}:
\begin{itemize}
\item Fixing a $K = K(A, k,\calX)>0$ sufficiently large, the index set is $\Rel_K(A) = \{U \in \mathfrak S| \diam_U(A)>K\}$.
\item The sets of hyperbolic spaces is $\{\pi_U(A)| U \in \Rel_K(A)\}$.
\item All relations, projections, and relative projections are those coming from the ambient structure $(\calX, \mathfrak S)$, where the projections involve a further composition with the retraction $r_U:\calC(U) \to \pi_U(A)$ provided by quasi-convexity of $\pi_U(A)$ in $\calC(U)$. 
\end{itemize}
\end{lemma}

\subsection{Product regions}\label{subsec:product region}

Another important set of structures in an HHS are its product regions.  The construction below, from \cite{HHS_2}, generalizes Minsky's Product Regions Theorem \cite{Minsky_prod} for Teichm\"uller space with the Teichm\"uller metric.

Roughly, each domain $U \in \mathfrak S$ determines a product region in $\calX$, which at the largest scale, consists of two components corresponding to those domains nesting into $U$ and those orthogonal to $U$.  Each of these components is in fact an HHS itself, where the underlying hierarchy is induced from the respective subcollections of domains in $\mathfrak S$.  In more details:

Let $U \in \mathfrak S$.  Let $\mathfrak S_U = \{V \in \mathfrak S | V \sqsubseteq U\}$ and note that $U$ is the $\nest$-minimal domain in $\mathfrak S$ so that $W \nest U$ for all $W \in \mathfrak S_U$.  On the other hand, observe that the container part of the orthogonality Axiom \ref{item:dfs_orthogonal} provides a domain $A \in \mathfrak S$ so that $V \nest A$ for all $V \nest U$, though we need not have $A \perp U$.  Nonetheless, set $\mathfrak S^{\perp}_U = \{W \in \mathfrak S| W \perp U\} \cup \{A\}$.

Let $\bF_U$ denote the set of $\theta$-consistent partial tuples in $\prod_{V \in \mathfrak S_U} 2^{\calC(V)}$, and let $\bE_U$ denote the set of $\theta$-consistent partial tuples in $\prod_{W \in \mathfrak S^{\perp}_U} 2^{\calC(W)}$.  Here, consistent partial tuples are exactly those satisfying the Consistency Inequalities \ref{defn:consistency} over a given restricted domain set.  Moreover, both $\bF_U$ and $\bE_U$ are HHSes with their structures just the restrictions of the structure on $\calX$ to the relevant pieces (see \cite[Construction 5.10]{HHS_2}).

We can now define a coarsely well-defined map $P_U:\bF_U \times \bE_U \to \prod_{V \in \mathfrak S} 2^{\calC(V)}$ as follows.  To any pair of tuples $(x,y) \in \bF_U \times \bE_U$ and each $V \in \mathfrak S$, we define 
$$\pi_V(P_U(x,y)) = \begin{cases}
x_V & V \sqsubseteq U\\
y_V & V \perp U\\
\rho^U_V & V \pitchfork U\\
\rho^U_V & U \nest V.
\end{cases}$$

The image of $P_U$ is $\theta$-consistent \cite[Construction 5.10]{HHS_2}, and hence defines a coarsely well-defined map $\PP_U:\bF_U \times \bE_U \to \calX$.  We denote the image by $\PP_U$.  Note that $\PP_U$ contains many \emph{parallel copies} of both $\bF_U$ and $\bE_U$.  We also have the following (see \cite[Proposition 5.11 and Remark 5.14]{HHS_2}):

\begin{lemma}\label{lem:PR hqc}
For every $U \in \mathfrak S$, we have that $\PP_U$ and every parallel copy of $\bF_U$ and $\bE_U$ are uniformly hierarchically quasi-convex, with constant depending only on $\calX$.
\end{lemma}

Hence admits a gate map $\gate_{P_U}:\calX \to P_U$ via Lemma \ref{lem:gates exist}.

For our purposes, we need the following three lemmas.  The first is immediate from the above definitions, but will prove useful in the proof of Proposition \ref{prop:ray replace boundary} below:

\begin{lemma}\label{lem:prod coord}
For any $U \in \mathfrak S$ and $x \in \PP_U$, we have $d_V(\pi_V(x), \rho^U_V) < \theta$ for all $U \nest V$ or $U \perp V$.
\end{lemma}

The second is also immediate, and is useful in Proposition \ref{prop:ray replace boundary} below:

\begin{lemma}\label{lem:double prod}
If $U \perp V$, then $\PP_U \cap \PP_V \neq \emptyset$.
\end{lemma}

We finish with another useful lemma, which is obvious in the surface setting.  We thank Mark Hagen for providing a proof in the general case.  It allows us to gain extra control over the structure of the map $\rho^V_U:\calC(V) \to \calC(U)$ when $U \nest V$:

\begin{lemma}\label{lem:rho project}
Suppose that $V, W \nest U$, $V \pitchfork W$, and $d_U(\delta^V_U, \delta^W_U) > 10E$.  If $\diam (\calC(W)) > 10^{10}E$, then $d_V(\rho^U_V(\rho^W_U), \rho^W_V) < 3E$.
\end{lemma}

\begin{proof}
Using surjectivity of $\pi_W: \calX \to \calC(W)$, choose some $x' \in \calX$ so that $d_W(x',\rho^V_W)> 10^{10}E$.  Now take the gate image $x = \gate_{P_W}(x')$ in $P_W$, so that we now have (1) $d_U(x, \rho^W_U)<E$ while still having (2) $d_W(x, \rho^V_W) > 10^9 E$.

Item (1), the BGIA \ref{ax:BGIA}, and our assumption that $d_U(\rho^V_U, \rho^W_U) > 10E$ imply that $d_V(\rho^U_V(x),\rho^U_V(\rho^W_U))<E$.  On the other hand, consistency of $x$ and item (2) imply that $d_V(x, \rho^W_V)<E$.

Hence $d_V(\rho^U_V(\rho^W_U), x) < 2E$.  But $d_W(x,\rho^V_W)>10^9E$, and thus consistency of $x$ (Lemma \ref{lem:base consistency}) gives $d_V(x, \rho^W_V)<E$.  It follows from the triangle inequality that $d_V(\rho^W_V, \rho^U_V(\rho^W_U))<3E$, as required.

\end{proof}

We note that in the above lemma, we need to assume that $\calC(W)$ has relatively large diameter.  The only places we use this lemma involve where $W$ is a $K$-relevant domain for some pair of points, where we can choose $K$ as large as necessary, so this does not present an issue.  Of course, most HHSes have the \emph{bounded domain dichotomy} (Definition \ref{defn:BDD}), i.e. all domains are either coarse points or infinite diameter.

Finally, we observe that neither (1) the existence of a gate map to a hierarchically quasi-convex subset, nor (2) the fact that product regions are hierarchically quasi-convex depends on the Distance Formula \ref{thm:DF}, as both are fundamentally consequences of Realization \ref{thm:realization}; see \cite[Section 5]{HHS_2} for details.  Hence Lemma \ref{lem:rho project} does not depend on the distance formula either.  This is important, because this lemma proves useful in a number of places throughout this article.

\subsection{HHS automorphisms} \label{subsec:HHS auto}

There is a natural notion of an automorphism of an HHS, which generalizes how mapping class groups act on their associated curve graph data.  The definition originally appeared in \cite{HHS_2} and automorphisms were studied in depth in \cite{DHS_boundary}.  The following definition, which is now standard (see e.g., \cite[Section 2]{PS20} and \cite[Definition 10.9]{CRHK}), is equivalent to the original via \cite[Section 2]{DHS_cor}:

\begin{definition}\label{defn:HHS auto}
An \emph{automorphism} $g$ of an HHS $(\calX, \mathfrak S)$ is a quasi-isometry of $g:\calX \to \calX$ together with the following:
\begin{enumerate}
\item A bijection $\mathfrak S \to \mathfrak S$ denoted by $U \mapsto gU$, which preserves nesting ($\nest$) and orthogonality ($\perp$) (and hence also transversality).
\item For each $U \in \mathfrak S$, an isometry $g:\calC(U) \to \calC(g U)$.
\item For each $U \in \mathfrak S$ and all $x \in \calX$, we require $g \pi_U(x) = \pi_{g U}(gx)$.
\item For all $U,V \in \mathfrak S$, we require $g \rho^U_V = \rho^{gU}_{gV}$, whenever $\rho^U_V$ is defined.
\end{enumerate}
\begin{itemize}
\item We let $\mathrm{Aut}(\calX, \mathfrak S)$ denote the group of HHS automorphisms of $(\calX, \mathfrak S)$.
\end{itemize}
\end {definition}

\begin{remark}
In what follows, we will see that our constructions only depend on these elementary parts, namely the ambient space $\calX$, the hyperbolic spaces $\calC(U)$, the projections $\pi_U$, and the relative projections $\rho^U_V$.  As such, equivariance of our construction, as in Theorem \ref{thmi:main model}, will be automatic.
\end{remark}

\section{Hierarchy rays and boundary points: Projections and consistency} \label{sec:ray projections}

The original cubulation construction from \cite{HHS_quasi} involved hulls of finite sets of interior points, but we are interested in extending this to hulls of finite sets of interior points, boundary points, and hierarchy rays, with the latter two first being studied in \cite{DHS_boundary}.

In this section, we will discuss boundary points, hierarchy rays and their associated domain projections.  This leads us to an extended notion of consistency for these objects (Proposition \ref{prop:extended consistency}), which we use in various coarse arguments throughout the paper.  

For the rest of the section, we fix a proper HHS $(\calX, \mathfrak S$), where properness is useful so that we can apply Arzel\`a-Ascoli.

\subsection{Boundary points and their domain projections}

In \cite{DHS_boundary}, we introduced, with Hagen and Sisto, a boundary construction for HHSes and used them to study HHS automorphisms and subgroup embeddings.  For most of this paper, we will just need the notion of a boundary point and how to project it to the domains in $\mathfrak S$.  However, in Section \ref{sec:boundary compare}, we will need a description of the simplicial structure of the boundary, and we defer that (equivalent) discussion until later:

\begin{definition}[Boundary point, support set, orthogonal complement] \label{defn:boundary point}
The HHS boundary of $(\calX, \mathfrak S)$ as a set $\partial \calX$ consists of points $\lambda \in \partial \calX$ consisting of the following data:

\begin{enumerate}
\item A pairwise orthogonal set of domains $\supp(\lambda) \subset \mathfrak S$ called the \emph{support set} of $\lambda$;
\item For each $U \in \supp(\lambda)$, a boundary point $\lambda_U \in \partial \calC(U)$.
\item A collection of constants $\{a_U| U \in \supp(\lambda)\}$ with $\sum_{U \in \supp(\lambda)} a_U =1$.
\end{enumerate}

We denote the set of domains orthogonal to $\supp(\lambda)$ by $\supp^{\perp}(\lambda) = \{V \in \mathfrak S| V \perp U \text{ for all } U \in \supp(\lambda)\}$.
\end{definition}

\begin{remark}[Ignoring simplex constants]\label{rem:simplex constants}
The constants $a_U$ in item (3) of Definition \ref{defn:boundary point} are superfluous for our purposes.  They pick out a point in the boundary simplex generated by the rest of the data.  Since domain projections do not distinguish different points in the interior of the same simplex, we will ignore these constants moving forward.  See Subsection \ref{subsec:HHS simplex} below for a discussion of the simplicial structure on the HHS boundary. 
\end{remark}

We would like to associate to $\lambda \in \partial \calX$ a tuple of points $(\lambda_U) \in \prod_{U \in \mathfrak S} \calC(U) \cup \partial \calC(U)$.  For domains which are not in $\supp^{\perp}(\lambda)$, we will see how to use $\lambda_U$ for $U \in \supp(\lambda)$ to produce a projection.  On the other hand, for domains in $\supp^{\perp}(\lambda)$, we will need to fix a basepoint.  This will not cause problems in what follows, since we will always assume that our set of finite points (whose hierarchical hull we are modeling) contains an interior point.

Boundary projections were defined in \cite[Definition 2.6]{DHS_boundary}, but our approach actually coincides with the slightly more general and canonical approach from \cite[Definition 3.12]{ABR_boundary1}.  For this, we need a brief discussion and some notation.

First, we note that the $\delta$-hyperbolic spaces $\calC(U)$ for $U \in \mathfrak S$ are frequently not proper, and hence we cannot guarantee the existence of a geodesic ray representing a point in $\partial \calC(U)$.  Nonetheless, one can always find a $(1,20\delta)$-quasi-geodesic representative ray.

Let $\lambda \in \partial \calX$ be a boundary point.  Let $N$ denote the function provided by the Morse lemma for a $\delta$-hyperbolic space, where $\calC(U)$ is $\delta$-hyperbolic for each $U \in \mathfrak S$.

If $V \in \supp(\lambda)$ and $U \nest V$, then we let $\lambda^U_V$ denote all points at least $E + N(1,20\delta)$ away from $\rho^U_V$ on all $(1,20\delta)$-quasigeodesic rays which represent $\lambda_U\in \partial \calC(U)$ and are based at points in $\rho^U_V$.  We note that $\diam_U(\rho^V_U(\lambda^U_V))< E$ by the BGIA \ref{ax:BGIA}.

Given $\lambda \in \partial \calX$ and a basepoint $\go \in \calX$, we can now define a tuple $(\lambda_U) \in \prod_{U \in \mathfrak S} \calC(U) \cup \partial \calC(U)$:

\begin{definition}[Based boundary projection]\label{defn:boundary projection}
Let $\lambda \in \partial \calX$ and fix a basepoint $\go \in \calX$.  If $U \in \supp^{\perp}(\lambda)$, then we set $\lambda_U = \pi_U(\go)$.  If $U \notin \supp^{\perp}(\lambda)$, then we define $\lambda_U \in \calC(U) \cup \partial \calC(U)$ as follows:
\begin{enumerate}
    \item If $U \in \supp(\lambda)$, then $\pi_U(\lambda)$ is the equivalence class of $\lambda_U$ in $\partial \calC(U)$.
    \item If $U\notin \supp(\lambda) \cup \supp^{\perp}(\lambda)$ does not nest into any $W \in \supp(\lambda)$, then set $\calV = \{V \in \supp(\lambda)| V \pitchfork U \text{ or } V \sqsubsetneq U\}$ and define $\displaystyle \pi_U(\lambda) = \bigcup_{V \in \calV} \rho^V_U.$
    \item If $U \nest V$ for some $V \in \supp(\lambda)$, then $V$ is the unique such domain in $\supp(\lambda)$.  In this case, we let $\pi_U(\lambda)$ be $\rho^V_U(\lambda^U_V)$.
    \begin{itemize}
        \item Moreover, we set $\rho^V_U(\pi_V(\lambda)) = \pi_U(\lambda)$.
    \end{itemize}
\end{enumerate}
\end{definition}

\begin{remark}
Unlike in \cite{DHS_boundary}, the benefit of this definition is that there are no choices made, and so the projections are canonical relative to the given HHS structure on $\calX$.
\end{remark}

\subsection{Hierarchy rays and their projections}

One of the fundamental constructions in an HHS is that of a hierarchy path, the definition of which has a natural ray extension:

\begin{definition}[Hierarchy path and ray]\label{defn:hp}
Given $L>0$, a map $\hh:[a,b] \to \calX$ is called a $L$-\emph{hierarchy path} if $\pi_U(\hh)$ is an unparameterized $(L,L)$-quasigeodesic for each $U \in \mathfrak S$.  If instead the domain of $\hh$ is $[0, \infty)$, then $\hh$ is called a $L$-\emph{hierarchy ray}.
\begin{itemize}
\item We denote the set of $L$-hierarchy rays in $\calX$ by $\calX_L^{ray}$.
\end{itemize}
\end{definition}

The existence of $L$-hierarchy paths in an HHS $(\calX, \mathfrak S)$ for uniform $L = L(\mathfrak S)>0$ is known in several settings \cite{MM00, Brock_WP, Dur16, HHS_2, Bowditch_hulls}, and we will deduce it from our model construction in Section \ref{sec:HP and DF} below.  Moreover, $L$-hierarchy rays can be built from sequences of $L$-hierarchy paths when $\calX$ is proper via an Arzel\`a-Ascoli argument.  We record their existence in the following lemma:

\begin{lemma}\label{lem:hp exist}
There exists $L = L(\calX)>0$ so that any pair of points $x,y \in \calX$ are connected by an $L$-hierarchy path.
\end{lemma} 

\begin{notation}[Ray notation] \label{not:ray notation}
Throughout the paper, when working with a hierarchy ray $\hh$, we will almost exclusively use $\hh$ as a label for a ray, not the image of a ray itself.  So we adopt the notation $\overline{\hh} = \hh([0,\infty)) \subset \calX$ to refer to the image of a hierarchy ray $\hh$.
\end{notation}

Since a $L$-hierarchy ray is itself an infinite-diameter $(L,L)$-quasigeodesic in $\calX$, we will see that it must have some domains where it has an infinite projection.  Since these play a similar to role to support sets of boundary points, we use the same terminology:

\begin{definition}[Support set of a ray]\label{defn:ray support}
Given an $L$-hierarchy ray $\hh$ in $(\calX, \mathfrak S)$, the \emph{support set} of $\hh$ is 
$$\supp(\hh) = \{U \in \mathfrak S| \diam_U(\pi_U(\overline{\hh})) = \infty\}.$$
\begin{itemize}
    \item We denote the set of domains orthogonal to $\supp(\hh)$ by $\supp^{\perp}(\hh)$.
\end{itemize}
\end{definition}

\begin{lemma}\label{lem:big orth}
For any hierarchy ray $\hh$, the set $\supp(\hh)$ is nonempty and pairwise orthogonal.  Hence there exists $\xi= \xi(\mathfrak S)>0$ so that $\# \supp(\hh) < \xi$.
\end{lemma}

\begin{proof}
The fact that $\supp(\hh) \neq \emptyset$ is \cite[Lemma 3.3]{DHS_boundary}.  To see pairwise orthogonality, let $U, V \in \supp(\hh)$ and suppose that $V \pitchfork U$ or $V \nest U$.

Using the fact that $\pi_U(\overline{\hh}))$ is an $(L,L)$-quasi-geodesic ray in $\calC(U)$, then we may choose a time $t \in [0,\infty)$ sufficiently large so that $\pi_U(\hh([t, \infty)) \cap \calN_E(\rho^V_U) = \emptyset$.  Hence the transverse (when $U \pitchfork V$) or nested (when $V \nest U$) consistency inequalities \ref{eq:transverse consistency} plus the BGIA \ref{ax:BGIA} imply that, for any $t' > t$, we have $d_V(\pi_V(\hh(t')), \rho^U_V)<2E$ (when $U \pitchfork V$) or $d_V(\pi_V(\hh(t')), \rho^U_V(\hh(t')))<2E$ (when $V \nest U$).  Either way, we get a contradiction of the assumption that $V \in \supp(\hh)$.
\end{proof}

For our purposes, the main difference between a hierarchy ray and a point in the HHS boundary is that the latter is possibly missing some data that needs to be supplemented by a basepoint.  On the other hand, a hierarchy ray $\hh$ consists of points in $\calX$, and hence has a natural projection to every domain.  This, however, creates a new subtlety, because we want to record the whole projection of a $\hh$ to each domain in $\supp^{\perp}(\hh)$, which is possibly an arbitrarily long quasi-geodesic ray and not a coarse point.  In Proposition \ref{prop:ray replace boundary} below, we will show that, when considering hierarchical hulls, one can replace boundary points with carefully constructed hierarchy rays, but we need two different but closely related notions of projections in the meantime.

Suppose then that $\hh: [0,\infty) \to \calX$ is an $L$-hierarchy ray and $U \in \supp^{\perp}(\hh)$.  Let $t_U \in [0,\infty)$ be a time so that 
$$d_U(\hh(0), \hh(t_U)) + 1 \geq \diam_U(\pi_U(\overline{\hh})).$$
Note that since $U \in \supp^{\perp}(\hh)$, we have $\diam_U(\pi_U(\overline{\hh})) < \infty$, and hence such a time $t_U$ exists, though is likely not unique.

We can now associate a tuple $(\hh_U) \in \prod_{U \in \mathfrak S} \calC(U) \cup \partial \calC(U)$ to a hierarchy ray $\hh$ as follows:

\begin{definition}[Ray projections]\label{defn:ray projection}
Let $\hh:[0, \infty) \to \calX$ be a hierarchy ray.  For each $U \notin \supp(\hh)$, we define $\hh_U \in  \calC(U) \cup \partial \calC(U)$ as follows:
\begin{enumerate}
    \item If $U \in \supp(\hh)$, then $\pi_U(\hh) = \hh_U \in \partial \calC(U)$ is the equivalence class of $\pi_U(\overline{\hh})$ in $\partial \calC(U)$.
    \item If $U \notin \supp^{\perp}(\hh)$ does not nest into any $W \in \supp(\hh)$, then set $\calV = \{V \in \supp(\hh)| V \pitchfork U \text{ or } V \sqsubsetneq U\}$ and define $\displaystyle \pi_U(\hh) = \bigcup_{V \in \calV} \rho^V_U.$
    \item If $U\nest V$ for some $V \in \supp(\hh)$, then $V$ is the unique such domain in $\supp(\hh)$.  In this case, we let $\pi_U(\hh)$ be $\rho^V_U(\hh^U_V)$.
    \begin{itemize}
        \item Moreover, we set $\rho^V_U(\pi_V(\hh)) = \pi_U(\hh)$.
    \end{itemize}
\end{enumerate}

\begin{itemize}
\item On the other hand, if $U \in \supp^{\perp}(\hh)$, then we define $\pi_U(\hh) = \pi_U(\hh(t_U))$, where $t_U \in [0,\infty)$ is as above.
\end{itemize}
\end{definition}

\subsection{Some useful statements about projections of hierarchy rays}

We briefly pause to observe some useful facts about projecting hierarchy rays, which appear both in the proof that the hull of a finite collection of boundary points can replaced by the hull of a carefully constructed collection of hierarchy rays (Proposition \ref{prop:ray replace boundary}), and in our proof that tuples of ray (and hence boundary) projections satisfy an extended notion of consistency (Proposition \ref{prop:extended consistency}).

First, we need a basic fact from Gromov hyperbolic geometry, whose proof we include for completeness:

\begin{lemma}\label{lem:no backtracking hyper}
For every $L>0, \delta>0$ there exists $M_0 = M(L,\delta)>0$ so that the following holds:  Let $X$ be $\delta$-hyperbolic and $q:[0,a]\to X$ an $(L,L)$-quasi-geodesic.  Then for any time $t \in [0,a]$ such that
$$\displaystyle d_X(q(0), q(t))+1 \geq \diam_X(q([0,a]))$$
we have $|a - t| < M$.

\end{lemma}

\begin{proof}
    Let $t \in [0,a]$ be as in the statement.  If $[q(0),q(a)]$ is a geodesic between $q(0), q(a)$, then the Morse property says that $d_X(q(t), [q(0), q(a)])<N = N(\delta, L)$.  Hence there exists some $p \in [q(0), q(a)]$ so that $d_X(q(t),p)<N$.  Since $q(t)$ is maximally far from $q(0)$, it follows that $d_X(p, q(a)) < N + 1$.  But then $d_X(q(t), q(a)) < 2N + 1$, and thus $a-t < (2N+1)L + L$.  This completes the proof.
\end{proof}

Lemma \ref{lem:no backtracking hyper} allows us to prove the following useful non-back-tracking statement for hierarchy rays:

\begin{lemma}\label{lem:no backtracking hp}
For any $L>0$ there exists $M_1 = M_1(L, \mathfrak S)>0$ so that the following holds.  Suppose $\hh:[0,\infty) \to \calX$ is an $L$-hierarchy ray and $U \notin \supp(\hh)$.  If $T_U \in [0,\infty)$ is such that $d_U(\hh(0), \hh(T_U)) +1 \geq \diam_U(\pi_U(\overline{\hh}))$, then for all $t \geq T_U$, we have 
$$d_U(\hh(t), \hh(T_U)) < M_1.$$
\end{lemma}

\begin{proof}
Suppose first that $U \in \supp^{\perp}(\hh)$.  Then we have $\diam_U(\hh)< \infty$, and hence $\pi_U(\hh)$ is a segment.  Since $\hh$ projects to an unparameterized $(L,L)$-quasi-geodesic, it can be reparameterized into a quasi-geodesic satisfying the assumptions of Lemma \ref{lem:no backtracking hyper}.  The conclusion in this case is now immediate.

On the other hand, if $U \notin \supp(\hh) \cup \supp^{\perp}(\hh)$, then either (1) $U \pitchfork V$ or $U \nest V$ for some $V \in \supp(\hh)$, or (2) $V \nest U$ for some $V \in \supp(\hh)$.  In each of these cases, the definition of $\pi_U(\hh)$ requires taking points sufficiently far out the tail of $\hh$, and so the conclusion follows from Definition \ref{defn:ray projection}.  This completes the proof.
\end{proof}

\subsection{Replacing boundary points with rays}

In this subsection, we explain why we can reduce our considerations to hulls of rays instead of hulls of boundary points when our ambient HHS $\calX$ is proper.  We will make this properness assumption going forward when dealing with boundary points, since these ray approximations will be necessary for certain arguments.

\begin{lemma}\label{lem:ray for boundary}
Suppose $\calX$ is a proper HHS and $L >0$.  There exists $\theta'' = \theta''(\mathfrak S, L)>0$ so that if $\go \in \calX$ and $\lambda \in \partial \calX$, then there exists an $L$-hierarchy ray $\hh_{\lambda}$ with
\begin{enumerate}
\item $\supp(\hh_{\lambda}) = \supp(\lambda)$,
\item For all $U \in \supp(\lambda)$, we have $[\pi_U(\hh_{\lambda})] = [\lambda_U] \in \partial \calC(U)$,
\item For all $U \notin \supp(\lambda)$, we have $d_U(\pi_U(\hh_{\lambda}), \pi_U(\lambda))<\theta''$,
\item $\hh_{\lambda} \subset \hull_{\theta''}(\go, \lambda)$.
\end{enumerate}
\end{lemma}

\begin{proof}
For each $U \in \supp(\lambda)$, let $\gamma_U$ be a $(1,20)$-quasi-geodesic ray in $\calC(U)$ from $\go$ to $\lambda_U$.  For each $n$, let $a^n_U \in \gamma_U$ be a point so that the sequence $(a^n_U)$ converges to $\lambda_U$ in $\calC(U) \cup \partial \calC(U)$ as $n\to \infty$.  For each partial tuple $(a^n_U)_{U \in \supp(\lambda)}$, use the Partial Realization Axiom \ref{item:dfs_partial_realization} to produce a sequence of points $p'_n$ with $d_U(p'_n, a^n_U)<\alpha = \alpha(\mathfrak S)$ for each $n$.

Let $H= \hull_{\calX}(\go, \lambda)$ and $\gate_H:\calX \to H$ be the gate map from Lemma \ref{lem:gate retract} above.  If we set $p_n = \gate_H(p'_n)$, then $d_U(p_n, a^n_U)$ is uniformly bounded for each $U \in \supp(\lambda)$, while $d_V(p_n, H_V)$ is uniformly bounded  for each $V \in \mathfrak S - \supp(\lambda)$, where $H_U$ is the hyperbolic hull of $\pi_V(\go), \pi_V(\lambda)$ in $\calC(V)$.

Let $h_n$ be an $L$-hierarchy path between $\go$ and $p_n$.  Since $\calX$ is proper, we can apply Arzela-Ascoli to the sequence $(h_n)$ to obtain a subsequence which converges to a hierarchy ray $h$.  It follows from our choice of $a^n_U$ above that $h$ satisfies $\supp(h) = \supp(\lambda)$, and for each $U \in \supp^{\perp}(h)$, we have $[h_U] = [\lambda_U]$ by construction, proving (1) and (2).

For (3), observe that if $V \in \mathfrak S$ with $V \nest U$, $V \pitchfork U$, or $U \nest V$ for some $U \in \supp(\lambda) = \supp(\hh_{\lambda})$, then the definitions of $\pi_V(\lambda)$ and $\pi_V(\hh_{\lambda})$ (Definition \ref{defn:boundary projection} and \ref{defn:ray projection}) coarsely coincide.  For $V \in \supp^{\perp}(\lambda)$, the desired bound follows Lemma \ref{lem:no backtracking hp}. 

Finally for (4),  observe that by items (2) and (3), quasi-convexity of $H_U = \hull_U(\go, \lambda_U)$ for each $U \in \mathfrak S$, and the fact that $h_n$ projects to an $(L,L)$-quasi-geodesic in each $U \in \mathfrak S$, it follows that $\pi_U(\overline{h})$ is coarsely contained in $H_U$, with constants depending only on $\mathfrak S$ and $L$.  Hence there exists $\theta''= \theta''(\mathfrak S, L)>0$ so that $h \subset \hull_{\theta''}(\go, \lambda)$, by Definition \ref{defn:hier hull} of the hull.  This completes the proof.

 \end{proof}

With this lemma in hand, our replacement statement is essentially immediate.

\begin{proposition}\label{prop:ray replace boundary}
Let $F \subset \calX$ be finite and nonempty.  Let $\Lambda$ be a finite sets of boundary points and $L$-hierarchy rays, with $\lambda(0) \in F$ for any hierarchy ray $\lambda \in \Lambda$.  Let $\widehat{\Lambda}$ denote the set of hierarchy rays in $\Lambda$ along with a replacement hierarchy ray $\hh_{\lambda}$ for each boundary point $\lambda$ as constructed in Lemma \ref{lem:ray for boundary}.

There exists $\theta_0 = \theta_0(L, \mathfrak S, |F \cup \Lambda|)>0$ so that 
\begin{itemize}
\item $H_{\theta_0} = \hull_{\theta_0}(F \cup \Lambda)$ and $\widehat{H}_{\theta_0} = \hull_{\theta_0}(F \cup \widehat{\Lambda})$ are within bounded Hausdorff distance,
\item The gate maps $\gate_{H_{\theta_0}}:H_{\theta_0} \to \widehat{H}_{\theta_0}$ and $\gate_{\widehat{H}_{\theta_0}}: \widehat{H}_{\theta_0} \to H_{\theta_0}$ are both uniform quasi-isometries depending only on $\mathfrak S, L, |F \cup \Lambda|$,
\end{itemize}
with all constants in the above depending only on $\mathfrak S, L, |F\cup \Lambda|$.
\end{proposition}

\begin{proof}
The first conclusion is an immediate consequence of Lemma \ref{lem:ray for boundary} and Lemma \ref{lem:no backtracking hp}.  The second part of the statement is a straight-forward application of the definitions, the first conclusion of this proposition, Lemma \ref{lem:ray for boundary}, and the Distance Formula \ref{thm:DF}.
\end{proof}

Hence for the rest of the paper, we can restrict our consideration to only dealing with hierarchy rays.

\begin{remark}
To see that the above proposition has content, observe that if $\calX = \ZZ \times \ZZ$, then one can define a sequence of geodesic (hierarchy) rays $\gamma_n$ which limit onto $+\infty$ in the first component, while their projection of $\gamma_n$ to the second component is $[0,n]$.  The boundary point corresponding to $+\infty$ in the first component does not see the second component.  In the general setting, given $\lambda \in \partial \calX$ so that $\supp^{\perp}(\lambda)$ contains a domain $U$ with $\diam \calC(U) = \infty$ and $U$ not $\nest$-maximal, then the union of all hierarchy rays which limit onto points in the smallest simplex in $\partial \calX$ containing $\lambda$ will have infinite Hausdorff distance with $\hull_{\theta}(\lambda(0), \lambda)$ for any $\theta$, as hierarchy rays are free to move in domains $V \in \supp^{\perp}(\lambda)$ as they choose.  Thus these special hierarchy ray representatives for boundary points are strictly necessary.  See also work of Mousely \cite{Mousely19} for examples of more exotic Teichm\"uller rays and their limit sets in the HHS boundary of Teichm\"uller space, which takes advantage of the fact that Teichm\"uller geodesics are not hierarchy rays, but are very close to being so \cite{Rafi_hyp}.
\end{remark}

\subsection{Consistency of boundary and ray tuples}

Our next goal is to show that the tuples we associated to boundary points (Definition \ref{defn:boundary projection}) and hierarchy rays (Definition \ref{defn:ray projection}) satisfy uniform consistency conditions.  We are treating them together because most of the discussion is similar, and having a statement for both may be useful in the future, despite our replacement Proposition \ref{prop:ray replace boundary} above.

The following definition simply generalizes standard consistency (Definition \ref{defn:consistency}) to allow for some coordinates to be at infinite.

\begin{definition}[Extended consistency]\label{defn:extended consistency}
For $\theta>0$, we say that a tuple
$$(\lambda_U) \in \prod_{U \in \mathfrak S} \left(\calC(U) \cup \partial \calC(U)\right)$$
satisfies \emph{extended $\theta$- consistency} if given any $U, V \in \mathfrak S$, we have the following:
	\begin{itemize}
            \item If $U\pitchfork V$ or $V \nest U$, then $\min\{d_U(\lambda_U, \rho^V_U), d_V(\lambda_V, \rho^U_V)\} < \theta$.
            \item If $U \nest V$, then $\min\{d_V(\lambda_V, \rho^U_V), \diam_U(\lambda_U \cup \rho^V_U(\lambda_V)\}< \theta$.
        \end{itemize}
\end{definition}

\begin{remark}
When $U \in \supp(\lambda)$, then $\lambda_U \in \partial \calC(U)$, and so necessarily $d_U(\lambda_U, \rho^V_U) = \infty$ when $V \pitchfork U$ or $V \nest U$, and $\diam_U(\lambda_U \cup \rho^V_U(\lambda_V)) = \infty$ when $U \nest V$.  Hence when $(\lambda_U)$ is $\theta$-consistent, then in all of these cases we have $d_V(\lambda_V, \rho^U_V)<\theta$.
\end{remark}

Observe that an immediate consequence of this definition is that if $(\lambda_U)$ is $\theta$-consistent, then $\{U \in \mathfrak S| \lambda_U \in \partial \calC(U)\}$ is pairwise orthogonal.

We can now prove our consistency statement for tuples associated to boundary points and hierarchy rays.

\begin{proposition}\label{prop:extended consistency}
    For every $L>0$, there exists $\theta_L = \theta_L(L, \mathfrak S)>0$ so that if $\lambda$ is either an $L$-hierarchy ray or a point in $\partial \calX$ with a basepoint $\go \in \calX$, then the tuple of projections 
    $$(\lambda_U) \in \prod_{U \in \mathfrak S} \left(\calC(U) \cup \partial \calC(U)\right)$$
    is $\theta_L$-consistent.

\end{proposition}

\begin{proof}

Let $U, V \in \mathfrak S$.  The proof is a case-wise analysis depending on how $U,V$ relate to each other and to the domains in $\supp(\lambda)$.

First, note that $U \perp V$ when $U,V \in \supp(\lambda)$, so we may exclude this.  Similarly, when $U \in \supp(\lambda)$, then $\lambda_V$ is given by Definitions \ref{defn:boundary projection} and \ref{defn:ray projection}, and we may that take $\theta_L = E$ in these cases.

There are three main cases: (1) $U,V \notin \supp^{\perp}(\lambda)$; (2) $U \notin \supp^{\perp}(\lambda)$ and $V \in \supp^{\perp}(\lambda)$; (3) $U, V \in \supp^{\perp}(\lambda)$.

\textbf{\underline{(1) $U,V \notin \supp^{\perp}(\lambda)$}}: There are three main subcases: (1.a) $U \pitchfork V$; (1.b) $U \nest V$, with the case of $V \nest U$ dealt with by the same argument as (1.b).

\underline{(1.a) $U \pitchfork V$ and $U,V \notin \supp^{\perp}(\lambda)$}:  First suppose that both $\calW_U = \{W \in \supp(\lambda)| W \pitchfork U \text{ or } W \nest U\} \neq \emptyset$ and $\calW_V \neq \emptyset$.  Then $\lambda_U = \bigcup_{W \in \calW_U} \rho^W_U$ and $\lambda_V = \bigcup_{Z \in \calW_V} \rho^Z_V$.

We deal with further subcases:

\begin{itemize}
\item There is some domain $Z \in \calW_U \cap \calW_V$.
\begin{itemize}
\item If $Z \nest U$, then $d_V(\rho^Z_V, \rho^U_V) < E$, and hence $d_V(\lambda_V, \rho^U_V) < 2E$.  Similarly, $d_U(\lambda_U, \rho^V_U)<2E$ when $Z \nest V$.  So we are done in this case.
\item Suppose instead that $Z \pitchfork U$ and $Z \pitchfork V$.  Let $x \in \PP_Z$, the product region for $Z$ in $\calX$ (see Subsection \ref{subsec:product region}).  Then $d_U(x,\rho^Z_U)<E$ and $d_V(x,\rho^Z_V)<E$ by Lemma \ref{lem:prod coord}.  Since the tuple $(\pi_U(x))$ is $\theta= \theta(\mathfrak S)$-consistent, it follows that one of $d_U(\rho^Z_U, \rho^V_U)<\theta + E$ or $d_V(\rho^Z_V, \rho^U_V)< \theta + E$ must hold, so we are done in this case.
\end{itemize}
\item Next, suppose that $\calW_U \cap \calW_V = \emptyset$.  Now let $W_U \in \calW_U$ and $W_V \in \calW_V$, and choose $x \in \PP_{W_U} \cap \PP_{W_V}$, which is possible by Lemma \ref{lem:double prod}.  Then $d_U(x, \rho^{W_U}_U)< E$ and $d_V(x, \rho^{W_V}_V)<E$, and one of $d_U(\rho^{W_U}_U, \rho^V_U)<\theta + E$ or $d_V(\rho^{W_V}_V, \rho^U_V)< \theta + E$ must hold, and so we are done in this case.
\end{itemize}

The next subcase of (1.a) is when $\calW_U \neq \emptyset$ but $\calW_V = \emptyset$.  Then there exists some (unique) $Z \in \supp(\lambda)$ so that $V \nest Z$, and $\lambda_V = \lambda^Z_V$, while $\lambda_U = \bigcup_{W \in \calW_U} \rho^W_U$.  Since $U \pitchfork V$, we cannot have that $Z \perp U$ and hence $Z \in \calW_U$.  But then $d_U(\lambda_U, \rho^Z_U) < E$ and so $d_U(\lambda_U, \rho^V_U)<2E$ since $V \nest Z$.  This proves this case.

The final subcase of (1.a) is when $\calW_U = \calW_V = \emptyset$.  Then our assumptions that $U \pitchfork V$ and $U,V \notin \supp^{\perp}(\lambda)$ provide a domain $W \in \supp(\lambda)$ with $U,V \nest W$.  Let $\gamma_U, \gamma_V \subset \calC(W)$ be $(1,20\delta)$-quasi-geodesic rays based at $\rho^U_W, \rho^V_W$, respectively, which are both  asymptotic to $\lambda_W \in \partial \calC(W)$.  Since they have the same limit, they are eventually $100\delta$-close.  Using $E$-coarse surjectivity of $\pi_W:\calX \to \calC(W)$ (Definition \ref{defn:normalized}), we can choose $x \in \calX$ so that $\pi_W(x)$ is as far away from $\rho^U_W, \rho^V_W$ as we want and within $E + 100\delta$ of both $\gamma_U,\gamma_V$.  The BGIA \ref{ax:BGIA} then implies that $d_U(x,\lambda_U)<E$ and $d_V(x,\lambda_V) < E$, while $\theta$-consistency of the tuple $(\pi_Z(x))$ implies that one of $d_U(\lambda_U, \rho^V_U)< E + \theta$ or $d_V(\lambda_V, \rho^U_V)<E + \theta$.  So we are done in this case.

This completes the proof of subcase (1.a).

\underline{(1.b) $U \nest V$ and $U,V \notin \supp^{\perp}(\lambda)$}

For this subcase, we note that if $Z \in \calW_U$, then $Z \in \calW_V$ since $U \nest V$ and therefore we cannot have either $Z \perp V$ or $V \nest Z$.  Hence we have two cases, when $\calW_U \neq \emptyset$ and when $\calW_U = \calW_V = \emptyset$.

So first suppose that $Z \in \calW_U \cap \calW_V$.  There are three subcases:
    \begin{itemize}
        \item If $Z \nest U$, then $Z \nest U \nest V$, in which case $d_V(\lambda_V, \rho^U_V) < 2E$ because $d_V(\rho^U_V, \rho^Z_V) < E$, and we are then done.
        \item If $Z \pitchfork U$ and $Z \nest V$, then let $x \in \PP_Z$.  By definition of the projections, we have $d_V(\lambda_V, \rho^Z_V)<E$ and $d_U(\lambda_U, \rho^Z_U)<E$, while $x \in \PP_Z$ implies that $d_U(x, \rho^Z_U)<E$.  Now if $d_V(\lambda_V, \rho^U_V)>20E + \theta$, then $\theta$-consistency of the tuple $(\pi_W(x))$  provides that $d_U(\rho^V_U(x_V),x_U)<\theta$ while Lemma \ref{lem:rho project} implies that $d_U(\rho^V_U(\rho^Z_V), \rho^Z_U)<3E$.  An application of the BGIA \ref{ax:BGIA} gives us that $d_U(\rho^V_U(\lambda_V), \rho^V_U(\rho^Z_V))<E$, and hence $d_U(\rho^V_U(\lambda_V), \lambda_U)<\theta + E$, completing the proof of this subcase.
        \item Suppose $Z \pitchfork U$, $Z \pitchfork V$, so that $d_U(\lambda_U, \rho^Z_U)<E$ and $d_V(\lambda_V, \rho^Z_V)<E$.  Suppose $d_V(\rho^U_V, \lambda_V)>10E$ and $x \in \PP_Z$.  Then $d_V(\lambda_V, \rho^Z_V)<E$ by definition of $\lambda_V$, while $d_V(\rho^Z_V, x)<\theta$ because $x \in \PP_Z$ (Lemma \ref{lem:prod coord}).  Now $\theta$-consistency says $d_U(\rho^V_U(x),x)<\theta$, while $d_U(x, \rho^Z_U)<\theta$ because $x \in \PP_Z$.  It follows then that $d_U(\lambda_U, \rho^V_U)<10\theta + 10E$. as required.
    \end{itemize}

The final subcase of (1.b) is when $\calW_U = \calW_V = \emptyset$.  Since $U \nest V$ and $\calW_V = \emptyset$, there must be a (unique) domain $W \in \supp(\lambda)$ so that $U \nest V \nest W$.  Note that $d_W(\rho^V_W, \rho^U_W)<E$.  Using $E$-surjectivity of $\pi_W:\calX \to \calC(W)$, we can choose a point $x \in \calX$ so that $\pi_W(x)$ is as far from $\rho^V_W,\rho^V_W$ as necessary and within $E+100\delta$ of two $(1,20\delta)$-quasi-geodesic rays limiting to $\lambda_W$ and based at $\rho^U_W,\rho^V_W$, respectively (which eventually $100\delta$-fellow travel).  By definition of $\lambda_U,\lambda_V$ and the BGIA \ref{ax:BGIA}, we have $d_U(\rho^W_U(x),\lambda_U)<2E$ and $d_V(\rho^W_V(x),\lambda_V)<2E$.  On the other hand $\theta$-consistency and the fact that $d_W(x,\rho^V_W)>10E$ implies that $d_V(x, \rho^W_V(x))<\theta$ and $d_U(x,\rho^W_U(x))<\theta$.  So if $d_V(\lambda_V, \rho^U_V)>10E + \theta$, then $d_U(\rho^V_U(\pi_V(x)),\rho^V_U(\lambda_V))<E$ and $d_U(\rho^V_U(\pi_V(x)), \pi_U(x))<\theta$.  Hence $d_U(\lambda_U, \rho^V_U(\lambda_V))<2E + \theta$, as required.

\underline{(2) $U \notin \supp^{\perp}(\lambda)$ and $V \in \supp^{\perp}(\lambda)$}:  There are three main subcases:

\begin{itemize}
    \item $U \pitchfork V$
    \begin{itemize}
        \item If $\calW_U \neq \emptyset$, then there is some $W \in \supp(\lambda)$ so that $W \pitchfork U$ or $W \nest U$, and $d_U(\lambda_U, \rho^W_U)<E$.  But since $V \in \supp^{\perp}(\lambda)$, we must have $W \perp V$.  Hence $d_U(\rho^V_U, \rho^W_U)<E$, and it follows that $d_U(\lambda_U, \rho^V_U)<2E$, and we are done.
        \item On the other hand, we cannot have $\calW_U = \emptyset$, for then there is some $W \in \supp(\lambda)$ with $U \nest W$, which implies either $V \pitchfork W$ or $V \nest W$, both of which are impossible because $V \perp W$.
    \end{itemize}
    \item It is not possible for $U \nest V$ under the current assumptions, which imply that there is some $W \in \supp(\lambda)$ satisfying one of the following:
    \begin{itemize}
        \item $W \pitchfork U$ and hence $W \nest V$ or $W \pitchfork V$, contradicting $V \in \supp^{\perp}(\lambda)$;
        \item $W \nest U$ and hence $W \nest V$, another contradiction;
        \item Or $U \nest W$, in which case either $V \nest W$ or $W \nest V$, both of which are contradictions.
    \end{itemize}
    \item Finally, if $V \nest U$, then there must be some $W \in \supp(\lambda)$ satisfying $W \pitchfork U$ or $W \nest U$, since if $U \nest W$ then $V \nest U \nest W$, which contradicts $V \in \supp^{\perp}(\lambda)$.  So when $W \pitchfork U$ or $W \nest U$, we have $d_U(\lambda_U, \rho^V_U)<2E$ since $d_U(\lambda_U, \rho^W_U) < E$ and $d_U(\rho^W_U, \rho^V_U)<E$ because $V \perp W$.
\end{itemize}

This completes the proof of (2).

\underline{(3) $U, V \in \supp^{\perp}(\lambda)$}: This subcase has two main subcases, where $\lambda \in \partial \calX$ and when $\lambda \in \calX^{ray}$.

In fact, when $\lambda \in \partial \calX$, we have $\lambda_U = \pi_U(\go)$ and $\lambda_V = \pi_V(\go)$ for our fixed basepoint $\go \in \calX$.  Hence consistency in this case follows from consistency of the projections of $\go$.

It remains to consider when $\lambda \in \calX^{ray}$.  Let $t_U,t_V \in [0,\infty)$ be times so that $d_U(\lambda(0), \lambda(t_U))+1 \geq \diam_U(\overline{\lambda})$ and similarly for $t_V$, as in Definition \ref{defn:ray projection}.

Suppose first that $U \pitchfork V$, and we may assume that $t_U \geq t_V$.  Then Lemma \ref{lem:no backtracking hp} implies that $d_V(\lambda_V, \lambda(t_U)) = d_V(\lambda(t_U), \lambda(t_V))< M_1 = M_1(\mathfrak S, L)$.  So if $d_U(\lambda(t_U),\rho^V_U) + E \geq d_U(\lambda_U, \rho^V_U)>\theta + E$, then $d_V(\lambda_V, \rho^U_V)<M_1+\theta+ E$, a required.

Finally, suppose that $U \nest V$ and $d_V(\lambda_V, \rho^U_V)>10E + 10M_1 + 10\theta$.  There will be two last subcases, depending on whether $t_U \geq t_V$ or $t_V \geq t_U$.

If $t_V \geq t_U$, then $d_U(\rho^V_U(\lambda_V), \rho^V_U(\lambda(t_V))) < E$ and $d_U(\lambda(t_V), \rho^V_U(\lambda(t_V))<\theta$, from which it follows that $d_U(\rho^V_U(\lambda_V), \lambda(t_V))< \theta + E$.  But then $d_U(\lambda_U, \rho^V_U(\lambda(V)) < E + \theta + M_1$, as required.

On the other (and final) hand, if $t_U \geq t_V$, then $d_V(\lambda(t_V), \lambda_V) < M_1$ and $d_V(\lambda(t_V), \lambda(t_U)) < M_1$.  Thus $d_V(\lambda(t_V), \rho^U_V)>E$, then the BGIA \ref{ax:BGIA} implies that $d_U(\rho^V_U(\lambda_V), \lambda_U) < 2E$, as required.

This completes the proof of (3) and of the lemma.

\end{proof}

\section{Hierarchical accounting techniques} \label{sec:qpu}

In this section, we establish some useful statements about the basic structure of HHSes.  They all elaborate on the Passing-Up Lemma \ref{lem:passing-up}, which is a core property of HHSes that follows from finite complexity and the Large Link Axiom \ref{ax:LL}.  In fact, we use the passing-up lemma to prove a strong version of the large links axiom in Proposition \ref{prop:strong ll} below, showing that the former can replace the latter as an axiom.

One of the main complications in this section is that it takes a bit of work to generalize Lemma \ref{lem:passing-up} to allow for hierarchy rays instead of just internal points of $\calX$.  While this makes the narrative a bit circuitous, it is necessary for the rest of the paper. 

Our first main goal is to prove a generalization of Lemma \ref{lem:passing-up} to hierarchy rays, which we do in Lemma \ref{lem:PU rays} below.  This proof will rely on a stronger version of Lemma \ref{lem:passing-up}, namely Strong Passing-up Proposition \ref{prop:strong pu}, which we will also want to extend to allow for hierarchy rays.  In order to minimize repetition of statements, we will begin by proving an extension of the Petyt-Spriano Lemma \ref{lem:PS} assuming the extended version of Lemma \ref{lem:passing-up}, namely Lemma \ref{lem:PU rays}.  Using the unextended versions of Lemmas \ref{lem:passing-up} and Lemma \ref{lem:PS}, we will prove the unextended version of the Strong Passing-up Proposition \ref{prop:strong pu}.  This will allow us to prove Extended Passing-up, Lemma \ref{lem:PU rays}, completing the loop and the actual proof of the extended version of the Petyt-Spriano Lemma \ref{lem:PS}, which then gives a proof of Extended Strong Passing-up Proposition \ref{prop:SPU}.

After this proof cycle is complete, we proceed with proving some other new facts.  These include a strong version of the Large Links property \ref{ax:LL} in Proposition \ref{prop:strong ll}, and some other useful consequences, namely: the Covering Lemma \ref{lem:covering} which we use frequently; the Bounding Containers Proposition \ref{prop:bounding containers}, which we need for various distance estimates; and finally an independent account of why changing thresholds in the distance formula works, Proposition \ref{cor:DF accounting}.

\subsection{No dense clusters, interior points}

In addition to Lemma \ref{lem:passing-up}, we need the following useful consequence due to Petyt-Spriano \cite[Lemma 5.4]{PS20}.  Roughly, it says that whenever a large number of relevant domains for a pair $x,y$ nests maximally into another relevant domain $W$, then their $\rho$-points cannot cluster into a small area along any geodesic $[x,y] \subset \calC(W)$.  In Lemma \ref{lem:PS} below, we will extend this to allow for hierarchy rays and boundary points with the same conclusion, and essentially the same proof just with appropriately modified arguments.

\begin{lemma}\label{lem:PS interior}
Let $(\calX, \mathfrak S)$ be an HHS and $D>50E$.  Let $x,y \in \calX$ and $W \in \mathfrak S$ have $d_W(x,y)>D$.  Suppose that \begin{enumerate}
    \item $U_1, \dots, U_n \sqsubsetneq W$ satisfy $d_{U_i}(x,y)>3E$.
    \item If $U_i \sqsubsetneq V \sqsubsetneq W$ for some $i$, then $d_V(x,y)\leq D$.
\end{enumerate}
Then if $n \geq P(D + 2E)$, then
$$\diam_W (\bigcup_{i=1}^n \rho^{U_i}_W) > D - 30E.$$
\end{lemma}

\subsection{Strong passing-up}\label{subsec:strong pu}

Our first goal is to prove a strengthened version of the Passing-Up Lemma \ref{lem:passing-up}, which will allow us to generalize that lemma to allow for hierarchy rays.

Roughly, the stronger version says that not only does having a very large number of relevant domains between $x,y \in \calX \cup \calX^{ray}$ guarantee the existence of a very large domain $W$ between them, but we can both force arbitrarily many of the relevant domains to be contained in $W$ while also forcing the $\rho$-sets of this colelction to spread out in $\calC(W)$.

The second conclusion also requires some notation and a discussion.  The idea is that not only can we force the $\rho$-sets to spread out, but we can force them to spread out in a relatively efficient manner.  For this, we need a notion of subdivision of an efficient path between points in $\calC(W) \cup \partial \calC(W)$.

\begin{definition}[$\sigma$-subdivision] \label{defn:sigma subdivision}
Let $\gamma:I \to \calC(W)$ be a $(1,20\delta)$-quasi-geodesic path (i.e., segment, ray, or bi-infinite) in $\calC(W)$ between $\pi_W(x),\pi_W(y)$ where $x,y\in \calX \cup \calX^{ray}$.  For $\sigma>0$, we say that a subdivision $\{x_i\}$ of $I$ is a $\sigma$-\emph{subdivision} of $\gamma$ if the $x_i$ decompose $I$ into a collection of subintervals $[x_i, x_{i+1}]$ so that for all but at most one $i$ we have $|x_{i+1}-x_i| = \sigma$, with the (possibly nonexistent) extra subinterval for which $|x_{i+1} - x_i| < \sigma$.
\end{definition}

We note that when one or both of $x$ or $y$ is a ray with $W \in \supp(x) \cup \supp(y)$, then a $\sigma$-subdivision does not determine an extra possibly shorter subinterval.

Suppose now that $x,y \in \calX \cup \calX^{ray}$, $K_1 \geq 50E$, $\calV' \subset \Rel_{K_1}(x,y)$ and $W \in \Rel_{K_1}(x,y)$ so that $V \nest W$ for all $V \in \calV'$.  Let $\gamma:I \to \calC(W)$ be a $(1,20\delta)$-quasi-geodesic path between $x,y$ in $\calC(W)$.

Given a $\sigma$-subdivision of $\gamma$, let $\calW_i$ denote the set of domains $V \in \calV'$ so that $p_{\gamma}(\rho^V_W) \cap [x_i, x_{i+1}] \neq \emptyset$, where $p_{\gamma}: \calC(W) \to \gamma$ is closest point projection.  Note that any given $V \in \calV'$ belongs in at most two $\calW_i$ when $\sigma \geq 10E + 20\delta$.  See Figure \ref{fig:qpu}.

\begin{figure}
    \centering
    \includegraphics[width=\textwidth]{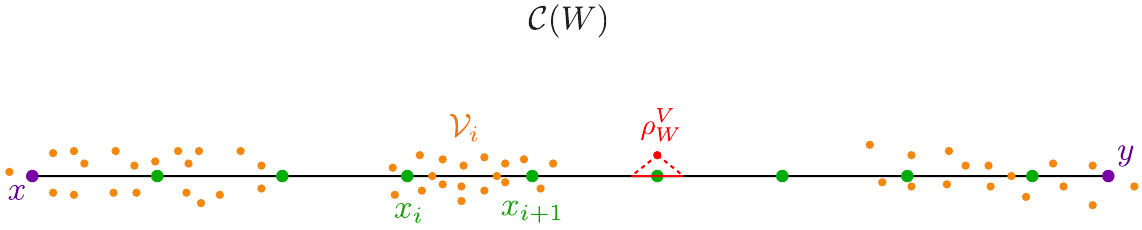}
    \caption{A $\sigma$-subdivision: In the proof of Proposition \ref{prop:strong pu}, subdividing $\gamma$ allows us to iteratively apply Lemma \ref{lem:PS interior}.  The $\rho$-points (in orange) must lie $E$-close to $\gamma$, and have closest point projections to $\gamma$ with diameter bounded by $2E$, as highlighted by $\rho^V_W$ (in red).  Each such projection therefore intersects at most two of the $[x_i,x_{i+1}]$, and hence the $
    \#\calW_i$ coarsely account for $\#\calV'$.}
    \label{fig:qpu}
\end{figure}

The second conclusion of the proposition says that, given any $n$, we can increase the number of domains in $\calV$ so that at least $n$ of the $\calV_i$ are nonempty.  This specific conclusion will be necessary in Subsection \ref{subsec:bipartite}.

The proof of this stronger proposition proceeds via an inductive argument using the Petyt-Spriano Lemma \ref{lem:PS interior} as a hammer.

\begin{proposition}[Strong passing-up]\label{prop:strong pu}
For any $K_2 \geq K_1 \geq 50E$, there exists $P_1 = P_1(K_1, K_2)>0$ so that for any $x, y\in \calX$, if $\calV \subset \Rel_{K_1}(x,y)$ with $\#\calV > P_1$, then there exists $W \in \Rel_{K_2}(x,y)$ and $\calV' \subset \calV$ so that $V \nest W$ for all $V \in \calV'$ and 
$$\diam_W\left(\bigcup_{V \in \calV'} \rho^V_W\right) > K_2.$$

Moreover, for any $\sigma\geq 10E + 20\delta$ and $n \in \mathbb{N}$, there exists $P_2(K_1,K_2,\sigma, n)>0$ so that if $\#\calV > P_2$, then we can arrange the following to hold:
\begin{itemize}
    \item If $\gamma:I \to \calC(W)$ is a $(1,20\delta)$-quasi-geodesic in $\calC(W)$ between $x,y$ and $\{x_i\}$ is a $\sigma$-subdivision of $\gamma$ determining sets $\calW_i$ as above, then 
    $$\#\{1 \leq i \leq k| \calW_i \neq \emptyset\} \geq n.$$
\end{itemize}
\end{proposition}

\begin{proof}

In the proof, we first set $\calU = \Rel_{K_2+30E}(x,y)$.  Let $\calU_0$ be the set of $\nest_{\calU}$-maximal domains in $\calU$.  Let $\calV'$ denote the set of $V \in \calV$ so that $V$ does not nest into any domain in $\calU_0$.  Note that
\begin{itemize}
    \item $\#\calU_0 < P(K_2)$, otherwise the Passing-Up Lemma \ref{lem:passing-up} would produce a domain in $\calU$ containing some domain in $\calU_0$, violating our $\nest$-maximality assumption on $\calU_0$;
    \item Similarly, $\#\calV'< P(K_2)$, otherwise Lemma \ref{lem:passing-up} would produce some domain in $\calU$ containing some domain in $\calV'$, violating its definition.
\end{itemize}

Hence there is some $U_0 \in \calU_0$ so that if $\calV_0 = \{V \in \calV | V \nest 
U_0\}$, then we can take $\#\calV_0$ to be as large as necessary.

Now either $U_0$ satisfies both conclusions of the proposition, or not.  If not, then let $\calU_1$ be the set of $\nest_{\calU_0}$-maximal domains in $\calU_0$.  As above, we have $\#\calU_1 < P(K_2)$ and all but at most $P(K_2)$ domains in $\calV_0$ must nest into some domain in $\calU_1$.  Hence there exists some $U_1 \in \calU_1$ so that if $\calV_1 = \{V \in \calV_0 | V \nest U_1\}$, then we can take $\#\calV_1$ to be as large as necessary.

We can repeat this argument until arriving at the second level from the bottom of the $\nest_{\calU}$-lattice, say at step $m$.  This produces for us some $U_m \in \calU_m$ and $\calV_m = \{V \in \calV_{m-1}| V \nest U_m\}$.  But now all $V \in \calV_m$ satisfy $V \nest_{\calU_m} U_m$, and hence we may apply Lemma \ref{lem:PS interior}, which says that if $\#\calV_m > P(K_2 + 30E + 2E)$, then
$$\diam_W\left( \bigcup_{V \in \calV_n} \rho^V_{U_m}\right) > K_2 + 30E - 30E = K_2.$$

To complete the proof of the first conclusion, we observe that at each step of the argument, we only ignored at most $P(K_2)$-many domains from $\calV_i$ to guarantee they each nested into some domain in $\calU_{i+1}$, and then only needed to divide the number of domains in $\calV_i$ by at most $P(K_2)$ to find some very large subset all of which nested into $U_{i+1}$.  Finally, the number of steps was bounded by the number of levels in the $\nest_{\calU}$-lattice, which is bounded above by the number of levels in the $\nest_{\mathfrak S}$-lattice, which is controlled solely by $\mathfrak S$.

The ``moreover'' part is essentially now a consequence of Lemma \ref{lem:PS interior} as follows.  Let $\calV_m$ and $U_m$ be as above in this proof, and the $\calW_i$ be defined as the discussion before the proposition.  Each $V \in \calV_m$ is contained in at most two $\calW_i$, while we have $\calV_m = \bigcup_i \calW_i$.  Since each $V \in \calV_m$ is $\nest_{\calU}$-maximal in $U_m$, we can use Lemma \ref{lem:PS interior} to bound $\calW_i$ by $2P(\max\{\sigma, 52E\})$ for each $i$.  Hence given some $n$, we can increase $\#\calV_m$ as necessary to force at least $n$-many indices $i$ so that $\calW_i \neq \emptyset$, as required.  This completes the proof. 
\end{proof}

\subsection{Passing-up for hierarchy rays}

We now prove the hierarchy ray version of the Passing-up Lemma \ref{lem:passing-up}.

The proof involves approximating the rays by points along the hierarchy rays which have coarsely the same projections as the rays to the small domains in question, similar to the proof of Lemma \ref{lem:PS} above.  The idea is then to use some sort of passing-up type statement to produce a large domain $W$ for the approximation points which contains at least one of the small domains for the rays.  We need Proposition \ref{prop:strong pu}, instead of Lemma \ref{lem:passing-up}, to deal with the case where $W$ is not in the support of either ray, since it allows us to force the geodesic between the projections of the rays to be long because it must pass close to a wide set of $\rho$-sets from our small domains.

\begin{lemma}\label{lem:PU rays}
There exists $PU:\mathbb N \to \mathbb N$, depending only on $\calX$ and $L$ such that the following holds:  Let $x,y \in \calX \cup \calX^{ray}$, and let $D > 50E$.  If $V_1, \dots, V_N \in \Rel_{5E}(x,y)$ with $N> PU(D)$, then there exists $W \in \Rel_D(x,y)$ and $i$ with $V_i \nest W$.  Moreover, if $V_i \nest Z$ for all $i$, then we can take $W \nest Z$.
\end{lemma}

\begin{proof}
Observe that if both $x,y \in \calX$, then this is just Lemma \ref{lem:passing-up}.  We deal with the case where both $x,y \in \calX^{ray}$, and indicate how to deal with the other cases afterwards.

Since the support sets of $x,y$ are pairwise-orthogonal and hence have cardinality bounded by $\xi = \xi(\mathfrak S)>0$, we may assume that the $V_i \notin \supp(x) \cup \supp(y)$ for all $i$ by increasing $PU(D)$ by $2\xi$.

By Lemma \ref{lem:no backtracking hp}, there exist times $t_{V_1}, \dots, t_{V_N}$ so that $d_{V_i}(x(t), x) < E$ and $d_{V_i}(y(t), y)<E$.  Choosing $t > \max\{t_{V_1}, \dots, t_{V_N}\}$ then gives
$$d_{V_i}(x(t), y(t)) - 2E > d_V(x,y) > 3E.$$

We now want to apply a passing-up type statement.  The effect of this will be to produce a larger $W$ which contains at least one of the $V_i$, and for which we can make $d_W(x(t),y(t))$ as large as necessary to conclude that $d_W(x,y)> D$.  When $W \in \supp(x) \cup \supp(y)$, the original Passing-up Lemma \ref{lem:passing-up} suffices.  The difficulty will come when $W \notin \supp(x) \cup \supp(y)$, and this is where we need the main part of Proposition \ref{prop:strong pu}.

 The Strong Passing-up Proposition \ref{prop:strong pu} implies that if $N > P_1(D + 10E)$, then there exists $W \in \Rel_{2D + 10E}(x(t),y(t))$ so that $V_i \nest W$ for all $i$ and
 $$\diam_W(\bigcup_{i}\rho^{V_i}_W) < 2D+10E.$$

There are two main cases: (1) $W \in \supp(x) \cup \supp(y)$, and (2) $W \notin \supp(x) \cup \supp(y)$.

In case (1), if $W \in \supp(x) - \supp(y)$, then $d_W(x,y) = \infty$ and we are done, and similarly when $W \in \supp(y) - \supp(x)$.  On the other hand, if $W \in \supp(x) \cap \supp(y)$, then we are also done if $[x] \neq [y] \in \partial \calC(W)$ for then $d_W(x,y) = \infty$.  But if $[x] = [y]$, then by definition of $\pi_{V_i}(x), \pi_{V_i}(y)$, we have $d_{V_i}(x,y) < E$, which is a contradiction of the assumption that $V_i \in \Rel_{5E}(x,y)$.  This takes care of case (1).

For case (2), both $\pi_W(x), \pi_W(y)$ are interior points of $\calC(W)$.  Since $V_i \in \Rel_{5E}(x,y)$ for all $i$, the BGIA \ref{ax:BGIA} implies that any geodesic $[x,y]_W$ between $x,y$ in $\calC(W)$ must pass within $E$ of every $\rho^{V_i}_W$.  But the diameter of these $\rho$-sets is $2D+10E$, and it follows then that the length of $[x,y]_W$ is at least $D$, and hence $W \in \Rel_D(x,y)$, as required.  We note that we only needed the approximators $x(t),y(t)$ to force the wide spread of these $\rho$-sets

Finally, if all domains $V_i\nest Z$ for some $Z \in \mathfrak S$, then we can choose $W \nest Z$ because Proposition \ref{prop:strong pu} allows it.  This completes the proof.
\end{proof}

\subsection{No dense clusters for rays}

We are now ready to prove the extended version of the Petyt-Spriano Lemma \ref{lem:PS interior} for rays:

\begin{lemma}[No dense clusters]\label{lem:PS}
Let $(\calX, \mathfrak S)$ be an HHS and $D>50E$.  Let $x,y \in \calX \cup \calX^{ray}$ and $W \in \mathfrak S$ have $d_W(x,y)>D$.  Suppose that \begin{enumerate}
    \item $U_1, \dots, U_n \sqsubsetneq W$ satisfy $d_{U_i}(x,y)>5E$.
    \item If $U_i \sqsubsetneq V \sqsubsetneq W$ for some $i$, then $d_V(x,y)\leq D$.
\end{enumerate}
Then if $n \geq P(D + 2E) + 2\xi$, then
$$\diam_W (\bigcup_{i=1}^n \rho^{U_i}_W) > D - 30E.$$
\end{lemma}

\begin{proof}

First, suppose the $U_i \nest W$ are as in the statement, but $\diam_W(\bigcup_{i=1}^n \rho^{U_i}_W) \leq D - 30E$.

Next, observe by ignoring the at-most $2\xi$-many domains contained in $\supp(x) \cup \supp(y)$ when $x,y \in \calX^{ray}$, we may assume that $\pi_{U_i}(x), \pi_{U_i}(y) \in \calC(U_i)$ for each $i$.

There are three cases: (1) $W \in \supp(x) \cap \supp(y)$, (2) $W \in \supp(x) - \supp(y)$ (and same with $x,y$ switching roles), and (3) $W \notin \supp(x) \cap \supp(y)$.  The arguments are similar and all use a localization trick from Petyt-Spriano's proof, and in fact their proof works essentially verbatim in case (3).  We will only present the mixed case (2), since it covers the other two.

Suppose then that $W \in \supp(x) - \supp(y)$.  Then $\pi_W(y) \in \calC(W)$ is an interior point, while $\pi_W(x) \in \partial \calC(W)$ determines a boundary point.  By Lemma \ref{lem:no backtracking hp}, there exists times $\{t_{U_1}, \dots, t_{U_n}\}$ so that if $T > t_{U_i}$ for all $i$ and $t \geq T$, then $d_{U_i}(x(t), x) < E$ for all $i$ (since we are assuming that $E > M_1$ from that lemma; see Section \ref{sec:constants}).  It follows from our assumptions and the first condition then that $d_{U_i}(x(t),y) > 3E$ when $t >T$.  In addition, since $\pi_W(x)$ is an unparameterized $(L,L)$-quasigeodesic, we may choose $t>T$ that $d_W(x(t), \rho^{U_i}_W) > 10D$ for all $i$. 

For the Petyt-Spriano localization argument, we will choose two points $z^-,z^+ \in \calX$ as local representatives of $x(t),y$, respectively, as follows.  If $d_W(y, \rho^{U_i}_W) < 10E$ for some $i$, then set $z^+ = y$.  Otherwise, let $y'$ be a point on $[x(t),y]_W \subset \calC(W)$ with $d_W(y', \rho^{U_i}_W) \leq D - 20E$ and $d_W(y',y) < d_W(\rho^{U_i}_W, y) - 5E$ for all $i$.  Since $\calX$ is normalized (Definition \ref{defn:normalized}), $\pi_W$ is $E$-coarsely onto, so we can choose $z^+ \in \calX$ with $d_W(z^+, y')\leq E$.  Similarly define $z^-$ by using $x(t)$ instead of $y$.

By assumption, $y, x(t)$ cannot simultaneously be $10E$-close to the union of the $\rho^{U_i}_W$, so $d_W(z^-,z^+)\leq D$ by construction.  On the other hand, the BGIA \ref{ax:BGIA} implies that $d_{U_i}(x(t),z^-)\leq E$ and $d_{U_i}(y,z^+)\leq E$, and hence $d_{U_i}(z^-,z^+)> E$ for all $i$.

Thus if $n \geq P(D+2E) + 2\xi$, then the Extended Passing-up Lemma \ref{lem:PU rays} provides some $V \nest W$ with $d_U(z^-,z^+)>D+2E$ and $U_i \nest V$ for some $i$.  Since $d_W(z^-,z^+)\leq D$, we have $V \sqsubsetneq W$ and thus $d_W(\rho^V_W, \rho^{U_i}_W) \leq E$ for each $i$.  The BGIA \ref{ax:BGIA} then implies that $d_U(x,z^-)\leq E$ and $d_U(y,z^+) \leq E$, and hence $d_U(x,y)>D$, which contradicts our assumptions.  This completes the proof.

\end{proof}

\begin{remark}
Besides the fact that it works for rays, the above statement is different from the original in \cite{PS20} in that it has a slightly larger smaller threshold, $5E$ versus $3E$, and a slightly larger lower bound for $n$, namely $P(D+ 2E) + 2\xi$.  The latter is to allow us to ignore the at-most $2\xi$-many domains in the support of the rays.
\end{remark}

\subsection{Strong passing-up for rays}

Finally, we get our full version of Strong Passing-up, in which we use the same notation:

\begin{proposition}\label{prop:SPU}
    For any $K_2 \geq K_1 \geq 50E$, there exists $P_1 = P_1(K_1, K_2)>0$ so that for any $x, y\in \calX \cup \calX^{ray}$, if $\calV \subset \Rel_{K_1}(x,y)$ with $\#\calV > P_1$, then there exists $W \in \Rel_{K_2}(x,y)$ and $\calV' \subset \calV$ so that $V \nest W$ for all $V \in \calV'$ and 
$$\diam_W\left(\bigcup_{V \in \calV'} \rho^V_W\right) > K_2.$$

Moreover, for any $\sigma\geq 10E$ and $n \in \mathbb{N}$, there exists $P_2(K_1,K_2,\sigma, n)>0$ so that if $\#\calV > P_2$, then we can arrange the following to hold:
\begin{itemize}
    \item If $\gamma$ is any geodesic in $\calC(W)$ between $x,y$ and $\{x_i\}$ is a $\sigma$-subdivision of $\gamma$ determining sets $\calV_i$ as before, then 
    $$\#\{i| \calV_i \neq \emptyset\} \geq n.$$
\end{itemize}

\end{proposition}

\begin{proof}
The proof of Proposition \ref{prop:strong pu} works essentially verbatim, replacing Lemma \ref{lem:passing-up} with Lemma \ref{lem:PU rays}.

\end{proof}

As a quick corollary, we observe the following useful Covering Lemma, which first appeared as \cite[Lemma 2.15]{DMS20} for interior points; see also \cite[Lemma 11.13]{CRHK} for a different proof of the interior point case:

\begin{lemma}[Covering]\label{lem:covering}
For any HHS $(\calX, \mathfrak S)$, there exists $N = N(\mathfrak S)>0$ so that the following holds.  For any $x,y \in \calX \cup \calX^{ray}$ and all $U \in \Rel_{50E}(x,y)$, there are at most $N$ domains $V \in \Rel_{50E}(x,y)$ so that $U \nest V$.
\end{lemma}

\begin{proof}
Let $U \in \Rel_{50E}(x,y)$ and $\calV = \{V \in \Rel_{50E}(x,y)| U \nest V\}$.  If $\#\calV > P_1(50E,50E)$, then the Strong Passing-Up Proposition \ref{prop:SPU} provides $W \in \Rel_{50E}(x,y)$ and a subcollection $\calV' \subset \calV$ so that $\diam_W \left(\bigcup_{V \in \calV'} \rho^{V'}_W\right)>50E$.  But this is impossible, since $U \nest V$ for all $V \in \calV'$, which forces that $d_W(\rho^V_W, \rho^{V'}_W) < 2E$, for any $V,V' \in \calV'$.  Hence $\#\calV < P_1(50E,50E)$, which depends only on $\mathfrak S$, as required.
\end{proof}

\subsection{Detecting rays from collections of domains}

In Sections \ref{sec:walls in Q} and \ref{sec:boundary compare}, we will need to be able to convert infinite collections of domains in the relevant set of a finite number of interior points and hierarchy rays into the data associated to one of the rays.  Since basic ideas required for such statements are contained in this section, we include them as the following lemma.

The first is an iterative applciation of Strong Passing-Up \ref{prop:SPU} for rays:

\begin{lemma}\label{lem:detecting rays}
Let $x, y \in \calX \cup \calX^{ray}$.  For any $K_1 \geq 50E$, if $\calV \subset \Rel_{K_1}(x,y)$ has $\#\calV = \infty$, then there exists an infinite subcollection $\calV' \subset \calV$ and a domain $W \in \supp(x) \cup \supp(y)$ so that $V \nest W$ for all $V \in \calV'$ and $\diam_W(\bigcup_{V \in \calV'} \rho^V_W) = \infty$.
\end{lemma}

\begin{proof}

The proof is inductive.  At each step, we either construct some domain in the support of $x$ or $y$, in which case we are done, or otherwise we construct a properly nested chain of relevant domains for $x,y$, which must have bounded length by finite complexity, Axiom \ref{item:dfs_complexity}.

For each $n>0$, apply Strong Passing-Up \ref{prop:SPU} to obtain a collection $\calV^1_n \subset \calV$ and a domain $W^1_n \in \Rel_{n}(x,y)$ so that $V \nest W^1_n$ for all $V \in \calV^1_n$ and 
$$\diam_{W^1_n}(\bigcup_{V \in \calV^1_n} \rho^V_{W^1_n}) > 50En.$$

Consider the sequence $W^1_1, W^1_2, \dots \in \Rel_{50E}(x,y)$ of domains produced by this process.  If there exists a domain $W$ so that $W  = W^1_n$ for infinitely-many $n$, then we must have that $W \in \supp(x) \cup \supp(y)$, since $d_W(x,y) \succ n$ for infinitely-many $n$, with the constants of $\succ$ depending only on $\calX$.  Moreover, we can accomplish this infinite spread in $\calC(W)$ with domains in $\calV$ as follows.  First, throw out the $n$ for which $W^1_n \neq W$ and similarly relabel the $\calV^1_n$, then set $\calZ_1 = \bigcup_n \calV^1_n$.  Then it follows that $\diam_W(\bigcup_{V \in \calZ_1} \rho^Z_W)= \infty$, as required.

Suppose then that no such $W$ exists after this first iteration.  Hence we have an infinite distinct collection $\calV^1 = \{W^1_1, W^1_2, \dots\} \subset \Rel_{50E}(x,y)$.  Repeating this process, for each $n$, we obtain a subcollection $\calV^2_n \subset \calV^1$ and a domain $W^2_n$ so that $V \nest W^2_n$ for all $V \in \calV^2_n$ and $\diam_{W^2_n}(\bigcup_{V \in \calV^2_n} \rho^V_{W^2_n}) > 50En$.

This time, however, the domains in the $\calV^2_n$ need not be in $\calV$.  Nonetheless, we can find \emph{$\calV$-representatives} for each $\calV^2_n$ as follows.  Observe that for each $n$, there exists $V_2 \in \calV^2_n$ with $V_2 \sqsubsetneq W^2_n$.  However, by construction, $V_2 = W^1_k$ for some $k$, and hence there is some $V_1 \in \calV^1_k \subset \calV$ so that $V_1 \sqsubsetneq V_2 \sqsubsetneq W^2_n$.  This domain $V_1 \in \calV$ represents $W^2_n$ because $V_1 \nest W^2_n$ implies that $d_Z(\rho^{V_1}_Z, \rho^{W^2_n}_Z) < E$ for any domain $Z$ into which both $V_1$ and $W^2_n$ nest.  

Again, if there exists some $W$ so that $W = W^2_n$ for infinitely-many $n$, then necessarily $W \in \supp(x) \cup \supp(y)$, as before.  Moreover, as observed above, we can realize this infinite diameter in $\calC(W)$ by a spread of domains in $\calV$ by taking their $\calV$-representatives.  More specifically, we throw out the $n$ for which $W^2_n \neq n$, relabel, and set $\calZ_2 = \bigcup_n \calV^2_n$.  Then we can consider the collection $\calV'_2 \subset \calV$ of $\calV$-representatives for the domains in $\calZ_2$, and we get that
$$\diam_W(\bigcup_{V \in \calV'_2} \rho^V_W) = \infty,$$
as required.

If no such $W$ exists, then we can repeat this process.  By finite complexity (Axiom \ref{item:dfs_complexity}), this process can only be repeated boundedly-many times (controlled by $\calX$), because at each step $k$ and each $n$, we create a properly nested chain $V_1 \sqsubsetneq V_2 \sqsubsetneq \cdots \sqsubsetneq W^k_n$, where $V_1$ is the $\calV$-representative of $W^k_n$.  Hence at some finite step $k$, we must have that there exists some $W \in \mathfrak S$ so that we have $W = W^k_n$ for infinitely-many $n$.  Setting $\calZ = \bigcup_n \calV^k_n$, it follows then
$$\diam_W\left(\bigcup_{V \in \bigcup \calZ} \rho^V_W\right) = \infty,$$

while these $\rho$-sets must also be $(E+20\delta)$-close to any $(1,20\delta)$-quasi-geodesic ray (or biinfinite path) between $x,y$ in $\calC(W)$ by the BGIA \ref{ax:BGIA}.

Finally, if we set $\calV'$ to be the $\calV$-representatives for the domains in $\calV^k = \{W^k_1, W^k_2, \dots\}$, then our observation above forces 
$$\diam_W\left(\bigcup_{V \in \bigcup \calV'} \rho^V_W\right) = \infty,$$
as required.  This completes the proof.
\end{proof}

\subsection{Bounding containers}

Our next goal is to prove some statements which allow us to manage multiple thresholds in distance formula-like estimates.  The first lemma says that the number of domains nesting maximally into $W$ is coarsely bounded above by $d_W(x,y)$, thereby reversing the inequality.  Its proof is essentially an iterative application of Lemma \ref{lem:PS}.

\begin{lemma}[Quantitative passing-up]\label{lem:qpu}
For any $K_1\geq 3E$ and $K_2 \geq K_1 + 47E$ the following holds.  Let $x,y \in \calX \cup \calX^{ray}$, and set $\calW \subset \Rel_{K_2}(x,y)$ and $\calU \subset \Rel_{K_1}(x,y)$.  Suppose that $W \in \calW$ and let $\calV = \{V \in \calU  | V \nest_{\calU} W\}$.  Then $\#\calV \prec d_W(x,y)$, with constants depending only on $\mathfrak S$ and $K_1$.

\end{lemma}

\begin{proof}

It suffices to consider the cases where either $W \notin \supp(x) \cup \supp(y)$, since otherwise $d_W(x,y) = \infty$ and the statement is vacuous.

Let $\gamma \subset \calC(W)$ be a geodesic between $\pi_W(x)$ and $\pi_W(y)$, and let $p_{\gamma}: \calC(W) \to \gamma$ denote closest point projection.  By the BGIA \ref{ax:BGIA}, if $V \in \calU$, then $d_
W(\rho^V_W, \gamma) < E$.  Since also $\diam_W(\rho^V_W) < E$, it follows that $\diam_W(p_{\gamma}(\rho^V_W))< 2E$.

Let $x = x_0, x_1, x_2, \dots, x_n = y$ be a subdivision of $\gamma$ into subintervals $[x_i,x_{i+1}]$ of length at most $10E$, so that $d_W(x_i, x_{i+1}) = 10E$ for all $i < n-1$.  Note that $n \geq 5$ by our assumption on $K_2$.

As in the proof of Proposition \ref{prop:strong pu}, for each $i$, let $\calW_i = \{V \in \calV| p_{\gamma}(\rho^V_U) \cap [x_i, x_i+1] \neq \emptyset\}$.  Since any given $V \in \calV$ belongs to at most $2$ of the $\calW_i$ while also $\calV = \bigcup_{i=1}^n \calW_i$, it suffices to bound the cardinality of each $\calW_i$ since $10En = \lfloor d_W(x,y) \rfloor$.

For this, we observe that each $\calV_i$ satisfies the assumptions of Lemma \ref{lem:PS}, and hence must have cardinality at most $P(52E)$.  This completes the proof.

\end{proof}

\begin{remark}
The need for the above proof, instead of using Lemma \ref{lem:passing-up} directly, comes from the fact that the passing-up function $P$ is likely to be exponential in the general setting.  If it were linear, then we could directly apply it to obtain the conclusion of Lemma \ref{lem:qpu}.
\end{remark}

\begin{remark}
We note that Lemma \ref{lem:qpu} is not unrelated to \cite[Lemma 5.1]{RS09} about certain anti-chains in relevant sets in the mapping class group setting.  The main difference is that the roles of their double-thresholds are reversed: where we are trying to control medium-sized domains via large domains, they were controlling large domains via medium-sized domains.  Nonetheless, we suspect that an general HHS version of their anti-chain statement can be derived from the ideas in this section, though we have not done so presently. 
\end{remark}

The proof of Lemma \ref{lem:qpu} has a useful corollary:

\begin{lemma}\label{lem:rel sets countable}
Suppose that $x, y \in \calX \cup \calX^{ray}$.  Then for any $K>50E$, we have $\Rel_K(x,y)$ is a countable set.
\end{lemma}

\begin{proof}
Note that if $x,y \in \calX$, then $\Rel_K(x,y)$ is finite by Corollary \ref{cor:rel sets are finite}.

Suppose that $\Rel_{K}(x,y)$ is infinite.  By Lemma \ref{lem:qpu}, only finitely-many domains in $\Rel_K(x,y)$ can nest into any domain not in $\supp(x) \cup \supp(y)$.  On the other hand, if $W \in \supp(x) \cup \supp(y)$, then Lemma \ref{lem:PS} implies that only boundedly-many $\rho$-sets for domains can cluster near any bounded-diameter segment of a $(1,20\delta)$-quasi-geodesic path between $\pi_W(x), \pi_W(y)$, as in the argument in Lemma \ref{lem:qpu}.  Hence we can decompose $\gamma$ into an infinite union of bounded length segments $[x_i,x_{i+1}]$, and realize the set of domains in $\Rel_K(x,y)$ which nest into $W$ as the union of the sets $\calV_i$ of domains whose $\rho$-sets are close to $[x_i, x_{i+1}]$.  Hence this set is countable, and this works for any $W \in \supp(x) \cup \supp(y)$, and hence all of $\Rel_K(x,y)$ is countable, completing the proof. 
\end{proof}

We can now prove the following useful proposition, which will allow us to control medium-size domains in various distance estimates and double threshold arguments.  First we set some notation.

Choose thresholds $K_1 > 50E$ and $K_2 > K_1 + 100E$.  Let $x,y \in \calX\cup \calX^{ray}$.  Let $\calU_1 \subset \Rel_{K_1}(x,y)$ and $\calU_2 \subset \Rel_{K_2}(x,y)$ and $\calV = \calU_1 - \calU_2$.  The $\calV$ are ``medium-size'' domains which we will want to control.

The proposition says two things.  First, it says that all but boundedly-many $V \in \calV$ nest into some $W \in \calU_2$.  Second, the lemma bounds the number of $V \in \calV$ which can nest into any given $W \in \calU_2$ terms of $d_W(x,y)$, $\mathfrak S$ and $K_2$.

\begin{proposition}[Bounding large containers] \label{prop:bounding containers}
Let $K_1 >50E$, $K_2 \geq K_1$, and $x,y \in \calX^{ray}$ with $\calU_1 \subset \Rel_{K_1}(x,y), \calU_2 \subset \Rel_{K_2}(x,y),$ and $\calV \subset \calU_1 - \calU_2$.  Then the following hold:
\begin{enumerate}
    \item For all but $PU(K_2)-1$ domains $V \in \calV$, there exists $W \in \calU_2$ with $V \nest W$.
    \item For any $W \in \calU_2$, if $\calW = \{V \in \calV| V \nest W\}$ then we have $\#\calW \prec d_W(x,y)$ with constants depending only on $\mathfrak S$ and $K_2$.
\end{enumerate}
\end{proposition}

\begin{proof}
For part (1), we use an iterative argument via Lemma \ref{lem:PU rays} as follows: Let $n > PU(K_2)$ and enumerate $\calV$, which is possible since $\calV \subset \calU_1$ is countable by 
Lemma \ref{lem:rel sets countable}.  By applying Lemma \ref{lem:PU rays} to $V_1, \dots, V_n$, we obtain $W_1 \in \calU_2$ so that, up to relabeling, $V_1 \nest W_1$.  Now apply the lemma again to $V_2, \dots, V_{n+1}$ to obtain $W_2 \in \calU_2$ so that, up to relabeling, $V_2 \nest W_2$.  Repeating this process finds containers in $\calU_2$ for all but $PU(K_2)-1$ domains (when $\calU_1$ is finite), and these only impact the additive error in the desired estimate from part (2) of the statement.

For part (2), now fix $W \in \calU_2$.  We may assume that $W \notin \supp(x) \cup \supp(y)$, for otherwise $d_W(x,y) = \infty$ and the statement is vacuous.  We will bound $\#\calW$ by iteratively applying Lemma \ref{lem:qpu}.

For ease of notation within this proof, let $A \nest_{\calB} C$ denote the case where $A \neq C$ and $D \notin \calB$ whenever $A \sqsubsetneq D \sqsubsetneq C$, noting that we are not requiring that $A \in \calB$ for this notation.

Proceeding iteratively, if $\calV_1 = \{V \in \calV| V \nest_{\calU_1} W\}$, then $\#\calV_1 \prec d_W(x,y)$ by Lemma \ref{lem:qpu}.  Now if $V \in \calV_1$, then $\#\{Z \in \calV| Z \nest_{\calU_1} V\} \prec d_V(x,y) < K_2$, by definition of $\calV$ and again Lemma \ref{lem:qpu}.  We can then proceed down the $\nest_{\calU_1}$-lattice through $\calV_1$, which has boundedly-many levels by finite complexity of $\mathfrak S$ (Axiom \ref{item:dfs_complexity}).  This bounds $\#\{Z \in \calV | Z \nest V\}$ for each $V \in \calV_1$ in terms of $\mathfrak S$ and $K_2$, whereas we already bounded $\#\calV_1$ in terms of $d_W(x,y)$.  Combining these gives the desired bound on $\#\calW$.
\end{proof}

\subsection{Strong Large Links}

We briefly observe that the standard Passing-Up Lemma \ref{lem:passing-up} implies the statement of the Large Links Axiom \ref{ax:LL}.

The proof of Lemma \ref{lem:PS} from \cite{PS20} only involves the BGIA \ref{ax:BGIA}, the Passing-Up Lemma \ref{lem:passing-up}, and the assumption that the ambient HHS is normalized (namely that $\pi_U:\calX \to \calC(U)$ is coarsely surjective for each $U \in \mathfrak S$).  Hence the flow of logic from the Passing-Up Lemma \ref{lem:passing-up} to Lemma \ref{lem:qpu} via Lemma \ref{lem:PS} does not depend on the Large Link Axiom \ref{ax:LL}.  The following corollary then allows us to, for instance, replace the Large Links Axiom \ref{ax:LL} with the Passing-Up Lemma \ref{lem:passing-up} as an axiom in the definition of an HHS.

Note that the conclusion is actually stronger than the Large Links Axiom, in three ways.  First, we get that the containers $T_i$ are relevant domains themselves.  Second, we get that the $T_i$ are $\nest$-maximal in $W$ with respect $\Rel_{E'}(x,y)$.  Third, we can force close proximity of $\rho^{T_i}_W$ to any geodesic between $\pi_W(x)$ and $\pi_W(y)$, and not just rough proximity to $x$.  We will make use of these additional features.

\begin{proposition} [Strong large links]\label{prop:strong ll}
There exist $\lambda\geq1$ and $E' = E'(\mathfrak S)>50E$ such that the following holds.
Let $W\in\mathfrak S$ and let $x,x'\in\calX$ and set $\calW = \Rel_{E'}(x,x')$.  Let
$N=\lambda d_{W}(\pi_W(x),\pi_W(x'))+\lambda$.  Then there exists $\{T_i\}_{i=1,\dots,\lfloor
N\rfloor}\subseteq\mathfrak S_W-\{W\}$ such that for all $T\in\mathfrak
S_W-\{W\}$, either $T\in\mathfrak S_{T_i}$ for some $i$, or $d_{
T}(\pi_T(x),\pi_T(x'))<E'$.  Moreover, we have:
\begin{enumerate}
    \item $T_i \in \calW$ for each $i$.
    \item $T_i \nest_{\calW} W$ for each $i$;
    \item $d_{W}([x,x']_W,\rho^{T_i}_W)\leq E'$ for each $i$ and any geodesic $[x,x']_W$ between $\pi_W(x), \pi_W(x')$;
\end{enumerate}
\end{proposition}

\begin{proof}
Let $W \in \mathfrak S$ and $x,y \in \calX$.  Let $\calU = \Rel_{50E}(x,y)$ and $\calV = \{V \in \calU | V \nest_{\calU} W\}$.  By the Quantitative Passing-Up Lemma \ref{lem:qpu}, there exists $\lambda = \lambda(50E, \mathfrak S)>0$ so that $\#\calV < \lambda d_W(x,y) + \lambda$.

Hence if $T \in \mathfrak S_W - \{W\}$ and $d_T(x,y) \geq 50E$, then $T \nest V$ for some $V \in \calV$ by definition of $\calV$.  The fact that $d_W(x,\rho^V_W)\leq N$ for each $V \in \calV$ follows from the BGIA \ref{ax:BGIA} and by increasing $\lambda$ a bounded amount, as necessary.  This completes the proof.
\end{proof}

\subsection{Changing thresholds}

We observe a final foundational application of Proposition \ref{prop:bounding containers}, which we believe has expositional value.

In \cite{HHS_2} and this paper, the proof of the Distance Formula \ref{thm:DF} is intertwined with the proof of the existence of hierarchy paths.  As a consequence, it is less than transparent that changing the relevance threshold should only affect the distance estimate up to coarse equality, because changing the relevance threshold requires one to build a new hierarchy path.

The following corollary explains why this sort of basic domain accounting works, regardless of whether the distance formula holds or hierarchy paths exist.  The proof is a straight-forward application of the ideas in this section, which the reader can take as a warm-up for its more sophisticated analogues later in the paper.

\begin{corollary}\label{cor:DF accounting}
For any $K_1 > 50E$, $K_2 \geq K_1 + 100E$, and any $x,y \in \calX$, we have
$$\sum_{U \in \mathfrak S} [d_U(x,y)]_{K_1} \asymp \sum_{U \in \mathfrak S} [d_U(x,y)]_{K_2},$$
where the constants for the coarse equality $\asymp$ depend only on $K_2, \mathfrak S$.
\end{corollary}

\begin{proof}
Since the right-hand sum (RHS) is smaller than the left-hand sum (LHS), it suffices to coarsely bound the LHS by the RHS.  Note that each side is finite by Corollary \ref{cor:rel sets are finite} to the Passing-up Lemma \ref{lem:passing-up}.

Denote the domains determining terms on the LHS but not the RHS by $\calV = \Rel_{K_1}(x,y) - \Rel_{K_2}(x,y)$, i.e $d_V(x,y) < K_2$ if $V \in \calV$.

Then Proposition \ref{prop:bounding containers} says that up to ignoring all but $P(K_2)$-many domains in $\calV$, for each $V \in \calV$ there exists some domain $W \in \Rel_{K_2}(x,y)$ determining a term on the RHS with $V \nest W$.  Moreover, it says that the number of domains $V \in \calV$ for which each such $W$ accounts in this way satisfies $\#\calW \leq A\cdot d_W(x,y) + A$, for some $A = A(K_2, \mathfrak S)>0$.

Hence, we have the desired conclusion, namely that
$$\sum_{U \in \mathfrak S} [d_U(x,y)]_{K_1}\leq A \cdot \sum_{U \in \mathfrak S} [d_U(x,y)]_{K_2} + A + P(K_2).$$

\end{proof}

\section{Coralling constants and standing assumptions}\label{sec:constants}

In this section, we briefly pause to set up some constants and standing assumptions which will be used throughout the rest of the paper.

\subsection{The constant $E$}

There are many constants which appear in the various basic statements associated to an HHS, and it is convenient to choose one which rules them all.  We set this constant to be $E$, and we choose $E> L, M_1, \theta_L$, where:

\begin{itemize}
\item $L=L(\calX)>0$ is the hierarchy path constant from Lemma \ref{lem:hp exist}.
\item $M_1 = M_1(\calX, L)>0$ is the constant from Lemma \ref{lem:no backtracking hp}.
\item $\theta_L = \theta_L(\calX, L)>0$ is the constant so that all points $x \in \calX$, boundary points in $\partial \calX$, and $L$-hierarchy rays are $\theta_L$-consistent, as in Definition \ref{defn:extended consistency} (via Proposition \ref{prop:extended consistency}).
\end{itemize}.  Also, fixing $K$ to be as large as necessary.

\subsection{The largeness threshold $K$ and the set $\calU$}\label{subsec:K}

For the rest of the paper, we fix a constant $K = K(\mathfrak S, |F \cup \Lambda|)>0$, whose size will be determined at various points along the way.  Importantly, we will always be able to choose $K$ to be larger than any of the other constants in the construction.

We set $\calU = \Rel_K(F \cup \Lambda) = \{U \in \mathfrak S | \diam_U(F \cup \Lambda)> K\}$.  This will be the main set of domains of focus in the rest of the paper, and our task in this section is to reduce our analysis to studying the Gromov modelings trees of $F \cup \Lambda$ in this set of domains.

\section{Hulls of rays: reducing to products of trees}\label{sec:reduce to tree}

In this section, we introduce the notion of a hierarchical hull for a finite collection of hierarchy rays, and make a first reduction of our analysis of these hulls to a coarse subset of the product of the associated Gromov modeling trees coming from each relevant domain.

This quasi-hierarchical setup replaces the hulls $H_U$ of each $\pi_U(F \cup \Lambda)$ for each $U \in \calU = \Rel_K(F \cup \Lambda)$ with Gromov modeling trees $T_U$, along with appropriate notions of projections and relative projections.  The main technical statements are Lemma \ref{lem:tree control}, which encodes the important HHS features, like the BGIA \ref{ax:BGIA} into this setup, and Lemma \ref{lem:consist to tree}, which says that the image of $H$ and $\Lambda$ into this product of the $T_U$ satisfies a notion of consistency.

\subsection{Gromov trees associated to the hull of rays}\label{subsec:ray trees}

Let $(\calX, \mathfrak S)$ be an HHS.  Fix $F \subset \calX$ and $\Lambda \subset \calX^{ray}$, so that if $\lambda \in \Lambda$ is an $L$-hierarchy ray, then $\lambda(0) \in F$.

The first step toward understanding the hierarchical hull of $F \cup \Lambda$ in $\calX$ is understanding the hyperbolic hull of $\pi_U(F \cup \Lambda)$ for each $U \in \mathfrak S$.

Gromov \cite{Gromov1987} proved that the hyperbolic hull of a finite set of points in a hyperbolic space, i.e. the union of all geodesics between the points, is quasi-isometric to a tree.  We need a generalization of this that allows for points at infinity.

\begin{definition}\label{defn:hyp hull infty}
If $X$ is $\delta$-hyperbolic with $F \subset X$ and $\Lambda \subset \partial X$ both finite, then we let $\hull_X(F \cup \Lambda)$ denote the union of the following:
\begin{itemize}
\item all geodesics between pairs of points in $F$,
\item all $(1,20\delta)$-quasi-geodesic rays between points in $F$ and points in $\Lambda$,
\item all bi-infinite $(1,20\delta)$-quasi-geodesics between points in $\Lambda$. 
\end{itemize}
\end{definition}

The following is an exercise in $\delta$-hyperbolic geometry whose proof we sketch for the reader's convenience:

\begin{lemma}\label{lem:ray trees exist}
Let $X$ be $\delta$-hyperbolic, with $F \subset X$ and $\Lambda \subset \partial X$ both finite.  Then there exists a tree $T$ and a $(1,C)$-quasi-isometric embedding $\phi:T \to X$ which is within Hausdorff distance $C$ of $\hull_X(F \cup \Lambda)$, where $C = C(\delta, |F \cup \Lambda|)>0$.
\end{lemma}

\begin{proof}
The idea is to start with the hyperbolic hull of $F$ and then to add each point in $\Lambda$, one at a time.  In particular, each $\lambda \in \Lambda$ closest point projects to a bounded diameter set $\pi_{H_F}(\lambda)$ on $H_F = \hull_X(F)$, since $H_F$ is uniformly quasiconvex.  Moreover, any geodesic ray representing $\lambda$ which is based at point in $F$ must pass through a bounded neighborhood of $\pi_{H_F}(\lambda)$, with the bound depending on $|F|$ and $\delta$.  This allows us to choose a basepoint on the standard Gromov tree $T_F$ for $H_F$ and a representative ray for $\lambda$ based on $T$ near $\pi_{H_F}(\lambda)$, thereby extending $T_F$ to a new tree $T$ representing $\hull_X(F \cup \{\lambda\})$.  One can repeat this process for each ray in $\Lambda$, with the bounds increasing at each step.  The details are a straightforward exercise in $\delta$-hyperbolic geometry and we leave them to the reader.
\end{proof}

\begin{remark}
In what follows, we will only need that the map $\phi:T \to X$ is a $(C,C)$-quasi-isometric embedding for $C = C(\delta, |F \cup \Lambda|)>0$, so that is what we will assume.  In fact, this assumption is necessary for some applications, since the stable trees from \cite{DMS20} are not almost isometrically embedded.
\end{remark}

We refer to any tree $\phi:T\to X$ satisfying the conclusions of Lemma \ref{lem:ray trees exist} as a \emph{Gromov modeling tree} for $\hull_X(F \cup \Lambda)$.

\subsection{Hierarchical hulls with their retracts and induced HHS structures} \label{subsec:hier hull}

The notion of a \emph{hierarchical hull} for a finite set of interior points in the mapping class group was introduced in \cite{BKMM} (as $\Sigma$-hulls), was extended to Teichm\"uller space in \cite{EMR_rank}, and then fully realized in the HHS setting in \cite{HHS_2}.  The definition we will give extends this to finite sets of interior points, hierarchy rays, and boundary points, but first we fix some notation.

Hierarchical hulls are built out of consistent tuples of points whose coordinates are contained in the hyperbolic hulls in each domain.  More specifically, for each, $U \in \mathfrak S$, we set $H_U = \hull_U(\pi_U(F \cup \Lambda))$.

\begin{definition}[Hierarchical hull]\label{defn:hier hull}
Given a constant $\theta>0$, we define the \emph{$\theta$-hull} of $F \cup \Lambda$ to be all $x \in \calX$ so that $d_U(x, H_U)<\theta$ for all $U \in \mathfrak S$.
\begin{itemize}
\item We denote the $\theta$-hull of $F \cup \Lambda$ by $H_{\theta} = \hull_{\theta}(F \cup \Lambda)$.
\end{itemize}
\end{definition}

As with the hyperbolic hull in a hyperbolic space, hierarchical hulls are hierarchically quasi-convex, in the sense of Definition \ref{defn:hqc}.  The following is a definition chase, of which we provide a sketch:

\begin{lemma}\label{lem:hull is hqc}
For any $\theta>0$, there exists $k:[0,\infty) \to [0,\infty)$ depending only on $\mathfrak S, \theta$ so that $H_{\theta}$ is $k$-hierarchically quasiconvex.
\end{lemma}

\begin{proof}
The first property is an easy consequence of the definitions and Lemma \ref{lem:no backtracking hp}.  For the second, observe that if $\theta'>\theta$, then $\hull_{\theta}(F \cup \Lambda) \subset \hull_{\theta'}(F \cup \Lambda)$, and in fact they are in bounded Hausdorff distance of each other depending only on $\theta, \theta', \mathfrak S$ by the Distance Formula \ref{thm:DF}.  Hence using the Realization Theorem \ref{thm:realization}, we can realize any tuple in the second condition, which then must be contained in $\hull_{\theta'}(F \cup \Lambda)$ for $\theta'$ depending only on $\theta, \mathfrak S$.  We leave the details to the reader.
\end{proof}

Thus $H_{\theta}$ admits a gate map $\gate_{H_{\theta}}: \calX \to H_{\theta}$ via Lemma \ref{lem:gates exist}, and \cite[Proposition 6.3]{HHS_2} provides says that this map is a coarse Lipschitz retract:

\begin{lemma}\label{lem:gate retract}
For any $\theta>0$, the gate map $\gate_{H_{\theta}}:\calX \to H_{\theta}$ is a $(J,J)$-coarse Lipschitz retract, where $J = J(L, \theta, \mathfrak S)>0$.
\end{lemma}

\begin{remark}
In our proof of the distance formula, we use the existence of such a retract to the hull of two interior points to construct the map from the cubical model of the hull back to the hull (Section \ref{sec:Q to Y}), which in turn is necessary to prove that combinatorial geodesics in the model are sent to hierarchy paths in the hull (Section \ref{sec:HP and DF}).  However, the existence of such a retraction does not require the distance formula, see \cite[Lemmas 4.6 and 4.15]{HHS_2}.
\end{remark}

Finally, since each hull $H$ is hierarchically quasi-convex, it has an induced HHS structure as in Lemma \ref{lem:hqc induce}, which we spell out in more detail here:

\begin{lemma}\label{lem:hull induced structure}
For any sufficiently large $K = K(\calX, |F\cup \Lambda|)>0$, the hierarchical hull $\hull_{\calX}(F \cup \Lambda)$ admits an HHS structure with the following data:
\begin{itemize}
\item The index set is $\calU = \Rel_K(F \cup \Lambda)$.
\item The sets of hyperbolic spaces is $\{\hull_U(F \cup \Lambda)| U \in \calU\}$.
\item All relations, projections, and relative projections are those coming from the ambient structure $(\calX, \mathfrak S)$, where the projections involve a further composition with the retraction $r_U:\calC(U) \to \hull_U(F \cup \Lambda)$ provided by quasiconvexity of $\hull_U(F \cup \Lambda)$ in $\calC(U)$. 
\end{itemize}
\end{lemma}

This structure will come up in two places: when we discuss the coarse median structure of $H$ in Section \ref{sec:medians}, and when we analyze the HHS boundary of $H$ in Subsection \ref{subsec:H simplex}.

Since the hull $H$ is actually a coarsely-defined set made of consistent tuples over the index set $\mathfrak S$, it is natural to consider $H$ with its induced structure as a set of consistent tuples over the index set $\calU$.  Each tuple of the latter type is only a partial tuple of the formal type, but we can complete such a partial tuple $(x_U)_{U \in \calU}$ by setting $x_V = \pi_V(f)$ for some fixed $f \in F$.  This defines a map:

$$I_H: \prod_{U \in \calU} \hull_U(F \cup \Lambda) \to \prod_{V \in \mathfrak S} \calC(V).$$

The following lemma says that this map defines a quasi-isometric embedding $H \to \calX$.  The proof is a straight-forward exercise in the definitions, Realization \ref{thm:realization}, and the Distance Formula \ref{thm:DF}, which we leave to the reader:

\begin{lemma}\label{lem:H qie}
For any sufficiently large $K=K(\calX, |F \cup \Lambda)>0$ and $\theta>0$, there exists $\theta_K = \theta_K(K, \theta, \calX, |F\cup \Lambda|)>0$ so that if $(x_U)_{U \in \calU}$ is a $\theta$-consistent tuple in the induced HHS structure on $H$, then $I_H((x_U)_{U \in \calU})$ is a $\theta_K$-consistent tuple in $\calX$.  Moreover, this identification is a $(\theta_K, \theta_K)$-quasi-isometric embedding.
\end{lemma}

\subsection{Marked points on modeling trees}

We briefly want to highlight a slightly complicating observation.

In what follows, the projections of the points of $F$ and $\Lambda$ will play an unsurprisingly pivotal role in our arguments.  And while often one thinks of the points being modeled by a Gromov modeling tree $T$ in a hyperbolic space $X$ as being leaves of $T$---either abstractly or in $X$---this need not be the case.

For instance, if one chooses $F$ to be three points which lie on a geodesic in $X$, then a perfectly fine modeling tree is the (sub-)geodesic between the outer-two most points, and so the middle point can be arbitrarily far from any leaf, both in the modeling tree $T$ and its image in $X$.

This motivates the following terminology:

\begin{definition}[Marked points on trees]\label{defn:marked}
Let $\calX$ be an HHS, $F \subset \calX$ a finite set of internal points, and $\Lambda$ a finite set of $L$-hierarchy rays, with $F \neq \emptyset$.  Let $U \in \mathfrak S$, and let $\phi_U:T_U \to \calC(U)$ any $(C,C)$-approximating model tree for $\pi_U(F \cup \Lambda)$.

\begin{itemize}
\item Each element of $F \cup \Lambda$ determines either an internal point on $T_U$ or a boundary point of $\partial T_U$.  When $x \in F \cup \Lambda$ determines an internal point, we call it a \emph{marked point}.  Otherwise, we call it a \emph{ray} for convenience.
\end{itemize}
\end{definition}

Note that if $U \notin \supp(\lambda)$, then $\lambda_U \in T_U$ is a marked point.  In the end, the fact that marked points need not be leaves of our modeling trees $T_U$ will not play a significant role, but it does complicate our arguments in a few places.

\begin{remark}[Geodesics to and between rays in $T_U$] \label{rem:geodesic rays in trees}
Another point to note is that if $\lambda \in \Lambda$, $U \in \supp(\lambda)$, and $ x \in T_U$ is another point, then the ``geodesic between'' $x$ and $\lambda$ is actually the unique geodesic ray based at $x$ in the equivalence class of $\lambda_U$ in the Gromov boundary $\partial T_U$.  Similarly, if $\lambda_U, \lambda'_U \in \Lambda$ satisfy $U \in \supp(\lambda)\cap \supp(\lambda')$, then the ``geodesic between'' $\lambda_U, \lambda'_U$ will refer to the unique bi-infinite geodesic ray in $T_U$ whose ends are in the equivalences classes of $\lambda_U, \lambda'_U$ in $\partial T_U$.  Notably, in the above cases we have $d_{T_U}(x,\lambda_U) = \infty$ and $d_{T_U}(\lambda_U, \lambda'_U) = \infty$.  In what follows, we will not make a fine point of this.
\end{remark}

\subsection{Tree projections and shadows}\label{subsec:tree proj}

For each $U \in \calU$, let $H_U = \hull_{\calC(U)}(F \cup \Lambda)$.  Let $\phi_U:T_U \to \calC(U)$ be the $(C,C)$-quasi-isometric embedding with $d_{Haus}(\phi_U(T_U), H_U)<C$ for $C = C(\mathfrak S, |F \cup \Lambda|)>0$ provided by Lemma \ref{lem:ray trees exist}, and let $\phi_U^{-1}$ be its quasi-inverse.  Let $p_U:\calC(U) \to T_U$ be coarse closest-point projection to $T_U$.  Set $\beta_U = p_U \circ \pi_U: \calX \to T_U$.

We begin by using the hierarchical projections between the $\calC(U)$ to define analogous projections between the trees $T_U$:

\begin{definition}[Tree projections]\label{defn:tree projections}
For $U,V \in \calU$ with $U,V$ not orthogonal, we define \emph{tree projections} $\delta^U_V$ and $\delta^V_U$ as follows:

\begin{enumerate}
    \item If $U \pitchfork V$, then we set $\delta^U_V = \hull_{T_V}(\phi_V^{-1} \circ p_V(\rho^U_V)) \subset T_V$, and define $\delta^V_U$ similarly.
    \item If $U \nest V$, then
    \begin{itemize}
        \item We define $\delta^U_V = \hull_{T_V}(\phi_V^{-1} \circ p_V(\rho^U_V)) \subset T_V$.
        \item We define a map $\delta^V_U:T_U \to T_V$ by $\delta^V_U(x)= \phi_V^{-1} \circ p_V \circ \rho^U_V \circ \phi_U(x)$ for any $x \in T_U$.
    \end{itemize}
\end{enumerate}

\begin{itemize}
    \item When $U \nest V$, we call $\delta^U_V \subset T_V$ the \emph{shadow} of $U$ on $T_V$.
\end{itemize}
\end{definition}

\begin{figure}
    \centering
    \includegraphics[width=.6\textwidth]{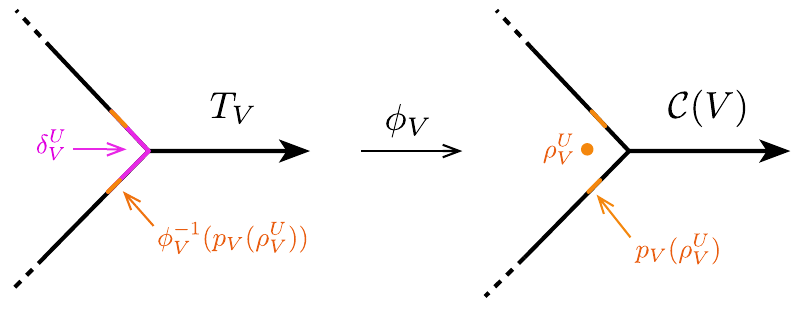}
    \caption{When $U \nest V$, then $\rho^U_V$ (in orange) is close to the image of $\phi_V:T_V \to \calC(V)$, and hence $p_V(\rho^U_V)$ (orange) defines a bounded diameter subset of $\phi_V(T_V)$, possibly with multiple components.  Pulling this back gives bounded diameter subsets $\phi^{-1}_V(p_V(\rho^U_V))$ of $T_V$ (orange).  Their convex hull of these is the \emph{shadow} $\delta^U_V$ of $U$ on $T_V$ (the union of the pink and orange), which is then a bounded diameter subtree of $T_V$.}
    \label{fig:shadow}
\end{figure}

This following is a basic lemma which allows us to control the tree projections like the projections in an HHS:

\begin{lemma}\label{lem:tree control}
There exist $R= R(\mathfrak S, |F \cup \Lambda|)>0$ and $E' = E'(\mathfrak S, |F \cup \Lambda|)>0$ so that the following hold:

\begin{enumerate}

    \item If $U \pitchfork V \in \calU$, then both $\delta^V_U \subset T_U$ and $\delta^U_V \subset T_V$ are sets of diameter at most $R$ which are within $R$ of some marked point in $F \cup \Lambda$ of $T_U,T_V$, respectively.  Moreover, $d_U(\phi_U(\delta^V_U), \rho^V_U)<R$, and similarly for $\delta^U_V$.
    \item If $U \nest V \in \calU$, then $\delta^U_V \subset T_V$ has diameter bounded by $R$.  Moreover, $d_V(\phi_V(\delta^U_V), \rho^U_V)<R$.
     \item If $U, V \nest W \in \calU$ with $U,V$ not transverse, then $\diam_{T_W}(\delta^U_W \cup \delta^V_W) < R$.
    \item If $U \pitchfork V$ and $W \nest V$ and $W,U$ not orthogonal, then $\diam_{T_U}(\delta^V_U \cup \delta^W_U)<R$.
    \item(Bounded geodesic image) If $U \nest V \in \calU$, then any component $C \subset T_V - \calN_{E'}(\delta^U_V)$ satisfies $\diam_{T_U}(\delta^V_U(C))<E'$.
    \item If $U,V,W \in \calU$ with $V,W \nest U$, $V \pitchfork W$, and $d_{T_U}(\delta^V_U,\delta^W_U)>10E'$, then $d_{T_V}(\delta^U_V(\delta^W_U), \delta^W_V)< E'$.
\end{enumerate}
\end{lemma}

\begin{proof}

Item (1): By choosing $K=K(\mathfrak S, |F \cup \Lambda|)>0$ sufficiently large and the fact that $|F \cup \Lambda|>1$, for at least one $f \in F \cup \Lambda$, we have $d_U(f, \rho^V_U) > E$.  Hence $d_V(f, \rho^U_V)<E$ by the $\pitchfork$-consistency Axiom \ref{eq:transverse consistency} or Proposition \ref{prop:extended consistency} when $f \in \Lambda$, and so $f \in T_U$ and $\phi_U^{-1} \circ p_U(\rho^V_U)$ are within $R = R(\mathfrak S, |F\cup \Lambda|)>0$, as required.  A similar argument for $\phi_V^{-1} \circ p_V(\rho^U_V)$ completes the proof.

Item (2): This follows immediately from the observations in item (1), plus the fact that $\phi_U$ and $\phi_V$ are $(C,C)$-quasi-isometries with $C= C(\mathfrak S, |F \cup \Lambda|)>0$.

Items (3) and (4) are straight-forward applications of items (1) and (2), plus the facts that the analogous bounds hold for their $\rho$-versions.

Item (5): The fact that $\phi_U$ is a $(C,C)$-quasi-isometry and the Morse lemma imply that there exists an $E' = E'(\mathfrak S, |F \cup \Lambda|)>0$ so that if $\gamma \subset T_V$ is a geodesic segment which avoids the $E'$-neighborhood of $\delta^U_V$, then $\phi_U(\gamma)$ is a $(C,C)$-quasi-geodesic whose geodesic straightening avoids the $E$-neighborhood of $\rho^U_V$.  The BGIA \ref{ax:BGIA} then implies that $\diam_{\calC(U)}(\rho^V_U(\phi_U(\gamma)) < E$, and hence $\delta^V_U(\gamma) = \phi^{-1}(\rho^V_U(\phi_U(\gamma))$ has diameter bounded solely in terms of the quasi-constants of $\phi^{-1}_U$, and hence in terms of $\mathfrak S$ and $|F \cup \Lambda|$.

Item (6): This is an easy consequence of Lemma \ref{lem:rho project} and the other items in this lemma.  We leave the proof for the reader.

This completes the proof.

\begin{figure}
    \centering
    \includegraphics[width=.8\textwidth]{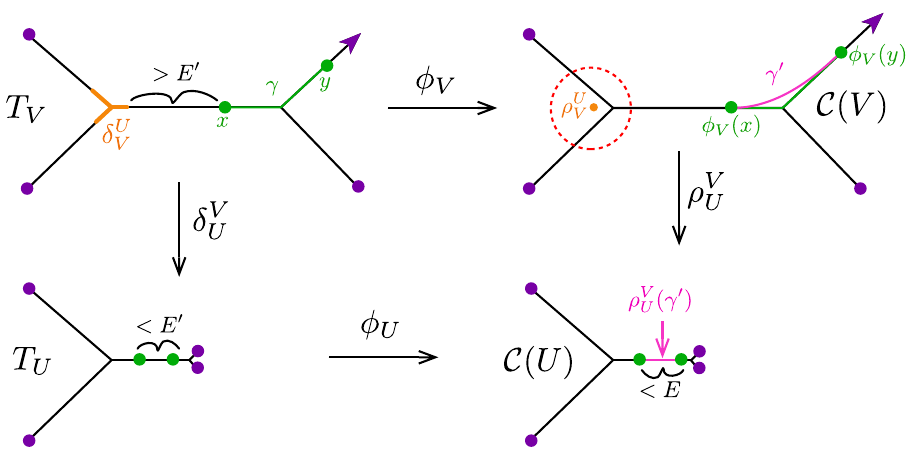}
    \caption{A cartoon of the BGI property for the tree projections, which constrains the map $\delta^V_U: T_V \to T_U$ when $U \nest V$.  Choosing $x,y$ (green) sufficiently far from $\delta^V_U$ in $T_V$ means that any straightening $\gamma'$ (pink) of $\phi_V(\gamma)$ (green) in $\calC(V)$ will avoid the $E$-neighborhood of $\rho^U_V$ (red).}
    \label{fig:tree_bgia}
\end{figure}

\end{proof}

\begin{remark}
While the largeness constant $K$ does appear in item (1) of the above lemma, it's role is simply forcing the existence of the point $f \in F \cup \Lambda$.  Any sufficiently large $K$ will do, and so our ability to independently increase $K$ in what follows is not affected.
\end{remark}

\subsection{Extended tree consistency}

One of our ultimate goals is to think of points in the hierarchical hull of finites sets of interior points and hierarchy rays as $0$-consistent tuples in some product of trees.  Our first reduction toward that is replacing points in the hull (including at infinity) with points in the product of associated trees $T_U$.  Toward that end, we need the following notion of (extended) consistency for points in $\prod_{U \in \calU} T_U \cup \partial T_U$, which is a direct analogue of Definition \ref{defn:consistency} and Definition \ref{defn:extended consistency}:

\begin{definition}[Extended tree consistency]\label{defn:tree consistency}
Given $\alpha\geq 0$, we say that a tuple $(x_U) \in \prod_{U \in \calU} T_U \cup \partial T_U$ is \emph{$\alpha$-consistent} if the following hold:

\begin{enumerate}
    \item If $U,V\in \calU$ and $U \pitchfork V$, then we have $$\min\{d_U(x_U,\delta^V_U),d_V(x_V,\delta^U_V)\}\leq\alpha.$$
    \item If $U, V \in \calU$ and $U \nest V$, then we have
$$\min\{d_V(x_V,\delta^U_V),\diam_U(x_U\cup(\delta^V_U(x_V))\}\leq\alpha.$$
\end{enumerate}

Given some $\alpha\geq 0$, we set
\begin{itemize}
    \item $\calZ_{\alpha}$ to be the set of $\alpha$-consistent tuples in $\prod_{U \in \calU}T_U$.
    \item $\oZ_{\alpha}$ to be the set of extended $\alpha$-consistent tuples in $\prod_{U \in \calU}T_U \cup \partial T_U$.
    \item Set $\calZ^{\infty} = \oZ - \calZ$, namely the set of tuples for which at least one component .

\end{itemize}
\end{definition}

There is a map $\Psi: H \to \prod_{U \in \calU} T_U$ given by the product of the compositions $\psi_U: = \phi^{-1}_U \circ p_U\circ \pi_U:H \to T_U$ for each $U \in \calU$, which associates a tuple of points in $\prod_{U \in \calU} T_U$ to any point in $H$.  More generally, given an $L$-hierarchy ray $\lambda \in \Lambda$, we can associate a tuple $(\lambda_U) \in \prod_{U \in \calU} T_U \cup \partial T_U$ as follows:
\begin{itemize}
\item If $U \in \supp(\lambda)$, then $\lambda_U$ corresponds to the geodesic ray in $\partial T_U$ in the class of the quasi-geodesic ray $\psi_U(\lambda)$.
\item If $U \notin \supp(\lambda)$, then we take $\lambda_U = \psi_U(\lambda)$.
\end{itemize}

First, we observe that $\Psi(H) \subset \calZ_{\alpha}$ and $\Psi(\Lambda) \subset \calZ^{\infty}_{\alpha}$ for an appropriate choice of $\alpha$:

\begin{lemma}\label{lem:consist to tree}
There exists $\alpha_0 = \alpha_0(\mathfrak S, |F \cup \Lambda|)>0$ such that for any $x \in H$, we have $\Psi(x) \in \calZ_{\alpha_0}$, and any $\lambda \in \Lambda$, we have $\Psi(\lambda) \in \calZ^{\infty}_{\alpha_0}$.
\end{lemma}

\begin{proof}
Let $x \in H$.  The corresponding tuple $(\pi_U(x_U))$ is $E$-consistent by the Consistency Axioms \ref{item:dfs_transversal} and our choice of $E$ in Section \ref{sec:constants}.  Moreover, by definition of $H = H_{\theta}(F \cup \Lambda)$, we have $d_U(\pi_U(x_U), H_U) < \theta = \theta(\mathfrak S, |F \cup \Lambda|)$, and hence $d_U(\pi_U(x_U), \phi_U(T_U))<\theta'$ for some $\theta' = \theta'(\mathfrak S, |F \cup \Lambda|)>0$.  A similar conclusion holds when replacing $x$ with $\lambda \in \Lambda$ by Proposition \ref{prop:extended consistency}, where the proximity of $\pi_U(\lambda)$ for $U\notin \supp(\lambda)$ is guaranteed by the fact that we have defined $H_U$ to be the hull of $\pi_U(F \cup \Lambda)$.  At this point, a straight-forward argument using Lemma \ref{lem:tree control} gives the required consistency constant $\alpha_0$.  We leave the details to the reader.
\end{proof}

We also want to be able to move in the opposite direction.

The family of maps $\phi_U:T_U \to \calC(U)$ for $U \in \calU$ combine into a component-wise map $\prod_{U \in \calU} T_U \to \prod_{U \in \calU} \calC(U)$.  We can extend this to a map $\Phi:\prod_{U \in \calU} T_U \to \prod_{U \in \mathfrak S} \calC(U)$ by choosing a point $f \in F \neq \emptyset$, setting $T_V = \{f\}$ for each $V \in \mathfrak S - \calU$, and defining $\phi_V:T_V \to \calC(V)$ by $\phi_V(f) = \pi_V(f)$.  The idea here is that $\diam_{\calC(V)}(F \cup \Lambda)<K$ if $V \notin \calU$.

The following proposition says that the tuples in $\Phi(\calZ_{\alpha})$ are once again consistent.  It is a simple application of Lemma \ref{lem:tree control} and the definitions.  In the next subsection, we will use this fact to extend $\Phi$ to a map to $H$.

\begin{proposition}\label{prop:consist from tree}
For any $\alpha \geq 0$, there exists $\beta= \beta(\mathfrak S, |F \cup \Lambda|, K, \alpha)>0$ so that if $x = (x_U) \in \calZ_{\alpha}$, then $\Phi(x)$ is $\beta$-consistent.
\end{proposition}

\begin{proof}

Observe that if $V \in \mathfrak S - \calU$ and $U \in \calU$ with either $U \pitchfork V$ or $U \nest V$, then $\delta^U_V \cap \mathcal{N}_A(f_V)$ for some $A = A(\mathfrak S, |F \cup \Lambda|, K)>0$ and any $f \in F$.  Hence, by choosing $f \in F$ and setting $x_V = f_V$ for each $V \in \mathfrak S - \calU$, we need only concern ourselves with the domains in $\calU$.  Notably, for the domains in $\calU$, we will get a consistency constant independent of $K$.

For each $U \in \calU$, let $x'_U = \phi_U(x_U)$.  There are two main cases to consider:

\underline{$U \pitchfork V$}: Without loss of generality, $\alpha$-consistency of $(x_U)$ says that $d_{T_U}(x_U, \delta^V_U)<\alpha$.   Item (1) of Lemma \ref{lem:tree control} says that $d_U(\phi_U(\delta^V_U), \rho^V_U)< R$, while $\phi_U$ is a $(C,C)$-quasi-isometric embedding, and so $d_U(\phi_U(x_U), \rho^V_U)<R + C\alpha + C$, which depends only on $\mathfrak S$, $|F\cup \Lambda|$, and $\alpha$, as required.

\underline{$U \nest V$}: By $\alpha$-consistency, we have $\min \{d_{T_V}(x_V, \delta^V_U), \diam_{T_U}(x_U \cup \delta^V_U(x_V))\}<\alpha$.  The case where $d_{T_V}(x_V, \delta^V_U)<\alpha$ follows similarly from Item (2) of Lemma \ref{lem:tree control} as argued above.  The remaining case is when $\diam_{T_U}(x_U \cup \delta^V_U(x_V))<\alpha$ and $d_{T_V}(x_V, \delta^V_U)\geq \alpha$.

For this, observe that $\phi_U \circ \delta^V_U = \phi_U \circ \phi^{-1}_U \circ \rho^V_U$, where $\phi_U$ and $\phi^{-1}_U$ are $C$-coarse inverses.  Since $\phi_U$ is a $(C,C)$-quasi-isometric embedding, it follows that $\diam_{U}(\phi_U(x_U) \cup \phi_U(\delta^V_U(x_V))) < C \alpha + C$.  Hence, $\diam_U(\phi_U() \cup \rho^V_U(\phi_V(x_V))) < C\alpha + 2C$.  This completes the proof.

\end{proof}

\section{Collapsed trees and the model space $\calQ$}\label{sec:Q}

In this section, we will develop the basic structure of our cubical models.  Our first goal is to explain how to modify the collection of trees $\{T_U\}_{U \in \calU}$ from the last section by collapsing certain subtrees arising from the shadows of nested large domains.  The idea is that these shadows encode redundant information on the trees, and so collapsing them removes this redundancy.  In particular, removing these redundancies removes the coarseness from relative projections and consistency, giving us an exact model for the hull.

The main technical statements in this section are Lemma \ref{lem:collapsed tree control}, which is a collapsed version of Lemma \ref{lem:tree control} and explains how the hierarchical structure is encoded in these collapsed trees, as well as Lemma \ref{lem:encoding PU}, which explains how the structural information provided by Strong Passing-up \ref{prop:SPU} is encoded in the collapsed data.

\subsection{Abstracting the tree reduction setup}\label{subsec:tree axioms}

For this section, we let $F\cup \Lambda$ and $\calU = \Rel_K(F \cup \Lambda)$ be as in Section \ref{sec:reduce to tree}.  In what follows, it will only be necessary that the rest of the setup satisfies the main outputs of the last section, namely Lemma \ref{lem:tree control}, Lemma \ref{lem:consist to tree}, and Proposition \ref{prop:consist from tree}.  For this, we only need the following slightly more general setup which is useful for applications, see in particular Subsection \ref{subsec:reduce tree for stable trees} below:

\begin{definition}[Reduced tree system]\label{defn:tree axioms}
A \emph{reduced tree system} for $F \cup \Lambda$ and $\calU$ is the following set of data:
\begin{enumerate}
\item For each $U \in \calU$, we have a tree $T_U$ and a $C$-quasi-median $(C,C)$-quasi-isometric embedding $\phi_U:T_U \to \calC(U)$ so that $d^{Haus}_U(\phi_U(T_U), H_U)<C$.
\item For each $a \in F \cup\Lambda$ and $U \in\calU$, there exists a marked point or ray $a_U \in T_U \cup \partial T_U$ with $\phi_U(a_U) = \pi_U(a) \in \calC(U) \cup \partial \calC(U)$.
\item There exists $R = R(\calX, |F \cup \Lambda|)>0$ and for any $V, U \in \calU$ with $V \nest U$ or $V \pitchfork U$, there exists $\delta^V_U \subset T_U$ so that items (1)--(4) and (6) of Lemma \ref{lem:tree control} hold.
\item For any $V, U \in \calU$ with $U \nest V$, the composition $\rho^V_U = \phi_U^{-1} \circ \rho^V_U \circ p_V \circ \phi_V: T_V \to T_U$ satisfies item (5) of Lemma \ref{lem:tree control}.
\end{enumerate}
\end{definition}

See Definition \ref{defn:quasi-median} below for the definition of quasi-median, which we include here for future applications of this setup; see also Lemma \ref{lem:median hyp}.

In particular, we have the following:

\begin{proposition}\label{prop:convert to tree system}
Any family of Gromov modeling trees $\{T_U\}_{U \in \calU}$ for $F \cup \Lambda$ can be converted into a reduced tree system which satisfies the conclusions of Lemma \ref{lem:consist to tree} and Proposition \ref{prop:consist from tree}.
\end{proposition}

\subsection{Clusters and the tree decomposition} \label{subsec:shadows}

In what follows, we let $F, \Lambda$ be finite with $F \neq \emptyset$, and $\calU = \Rel_K(F \cup \Lambda)$ be the set of relevant domains for $F \cup \Lambda$ with $K$ as in Subsection \ref{subsec:K}.  Fix a reduced tree system for this setup as defined in Definition \ref{defn:tree axioms}, for which we can freely assume the relevant statements from Section \ref{sec:reduce to tree} by Proposition \ref{prop:convert to tree system}.  We will now proceed to explain how to cluster the $\delta$-set data as our first step for producing our collapsed trees.

Let $U,V \in \calU$.  In Lemma \ref{lem:tree control}, we saw that when $V \nest U$ or $V \pitchfork U$, that the shadow $\delta^V_U$ of $\rho^V_U$ on $T_U$ has diameter bounded by $R = R(\mathfrak S, |F \cup \Lambda)>0$.  Moreover, in Definition \ref{defn:tree projections}, we constructed these shadows to be subtrees of $T_U$.

We now fix $r = r(\mathfrak S, |F \cup \Lambda|)>0$, called the \emph{cluster separation constant}, to be determined later (see Lemma \ref{lem:collapsed tree control}). In particular, our threshold constant $K$ can be chosen independently of $r$.

Now let $T^c_U$ to be the union of 
\begin{enumerate}
    \item the $2R$-neighborhoods in $T_U$ of the shadows $\delta^V_U$ of the $V \in \mathfrak S$ for $V \nest U$,
    \item the $2R$-neighborhoods in $T_U$ of the marked points of $T_U$ (i.e., the points of $F$ and $\Lambda$ which are not rays, Definition \ref{defn:marked}), and
    \item The convex hull of any collection of sets from (1) and (2) which are within distance $r$ of each other in $T_U$. 
\end{enumerate}

\begin{remark}
The reason for the choice of $2R$ above becomes apparent in the proof of Proposition \ref{prop:Q-consistent} below, in which we build a map from our model $\oQ$ (Definition \ref{defn:Q consistent} below), back into $\calY$.
\end{remark}

We call the components of $T^c_U$ from item (3) \emph{clusters}\footnote{We borrow the term ``cluster'' from \cite{DMS20}, though with a slightly different usage.}.  Our flexibility in choosing $r$ says that we can force clusters to be as far apart as we need.  This is important in Theorem \ref{thm:Q distance estimate} below, in which we compare the distance in our model to the Distance Formula \ref{thm:DF} sum in the hull $H$.

Set $T^e_U = T_U - T^c_U$.  We call the components of $T^e_U$ \emph{edge components}.  Notice that since each component of $T^c_U$ is convex by construction, we have that $T^c_U$ and $T^e_U$ are both disjoint unions of subtrees of $T_U$.

\begin{remark}[Cluster philosophy]
Roughly speaking,the components of $T^c_U$ encode parts of $T_U$ which are already encoded in other domains not orthogonal to $U$, while $T^e_U$ encodes the parts of $T_U$ which are unique to $U$.  That is, $T^c_U$ contains redundant information, and $T^e_U$ contains information unique to $U$.  This is the same principle for the subdivision points in the original cubulation construction from \cite[Section 2.1]{HHS_quasi}, as discussed in Subsections \ref{subsec:BHS cube} and \ref{subsec:Dur cube} of the Introduction \ref{sec:intro}.
\end{remark}

\subsection{A reduced tree system for stable trees} \label{subsec:reduce tree for stable trees}

The stable tree construction from \cite[Theorem 3.1]{DMS20} produces trees with exactly this sort of decomposition $T_U = T^c_U \cup T^e_U$, where the relative projection data for nested domains is encoded in the components of $T^c_U$.  More specifically, there is a uniform $C$-quasi-median $(C,C)$-quasi-isometric embedding $\phi_U:T_U \to \calC(U)$ satisfying
\begin{itemize}
\item For each component $D \subset T^c_U$, there is an associated collection of domains $\calU_U(D) \subset \calU$ so that
\item For each $V \in \calU$ with $V \nest U$, there exists a unique component $D_V \subset T^c_U$ so that $V \in \calU_U(D_V)$,
\item $V \nest U$ for each $V \in \calU_U(D)$, and
\item $d^{Haus}_U(\phi_U(D), \bigcup_{V \in \calU_U(D)} \rho^V_U)<C$.
\end{itemize}

As such, given $V \nest U \in \calU$, we can define $\delta^V_U = \phi^{-1}_U(p_{D_V}(\rho^V_U))$, and this family of $\delta$-sets, along with the $\delta$-sets and -maps for $V \pitchfork U$ and $U \nest V$ will satisfy Definition \ref{defn:tree axioms}.  Moreover, \cite[Theorem 3.1]{DMS20} also provides a cluster separation constant $r = r(\calX, |F \cup \Lambda|)>0$ so that the components of $T_U = T^c_U \cup T^e_U$ satisfying the properties discussed in Subsection \ref{subsec:shadows} above.

We have made these observations to set ourselves up for future work, in which we will want to be able to plug in constructions like these stable trees into the cubulation machine we develop in this article.  We will elaborate on how these different machines work together this in a forthcoming expository paper \cite{Dur_expo}.

\subsection{Collapsed trees}

Fixing the setup as before, in particular at the level of generality of a reduced tree system for our finite sets $F,\Lambda$, we now define our main trees of interest:

\begin{definition}[Collapsed trees and rays]\label{defn:collapsed tree}
Let $q_U:T_U \to \hT_U$ be the quotient map obtained by collapsing all components of $T^c_U$ to points.  Since each shadow is convex, the quotient space $\hT_U$ is a tree, which we call the \emph{collapsed tree} associated to $U$.  The points in $F \cup \Lambda$ determine \emph{marked cluster points} or rays in $\hT_
    U$.  In particular, when $\lambda \in \Lambda$ with $U \in \supp(\lambda)$, then either
    \begin{itemize}
    	\item There exists a unique cluster $C_{\lambda_U} \subset T^c_U$ containing a geodesic representative $\gamma_U$ of $\lambda_U \in \partial T_U$, in which we set $q_U(\lambda_U)\in \hT^c_U$ to be equal to the cluster point $q_U(C_{\lambda_U})$.
	\item Otherwise, any geodesic representative of $\lambda_U$ projects to an unparameterized geodesic ray in $\hT_U$, and we set $q_U(\lambda_U) \in \partial \hT_U$ to be the class of that ray. 
	\end{itemize}
\end{definition}

The following is clear from this discussion:

\begin{lemma}\label{lem:component preserve}
The quotient map $q_U: T_U \to \hT_U$ preserves the decomposition $\hT_U = \hT^c_U \sqcup \hT^e_U$, into clusters whose images in $\hT_U$ are points, and edge components, on each of which $q_U$ restricts to an isometry. 
\end{lemma}

In what follows, this decomposition---especially the fact that all relative projections are now collapsed to points---will simplify many of our hierarchical arguments.  See Figure \ref{fig:cluster_collapse} for a schematic.

\begin{figure}
    \centering
    \includegraphics[width=.7\textwidth]{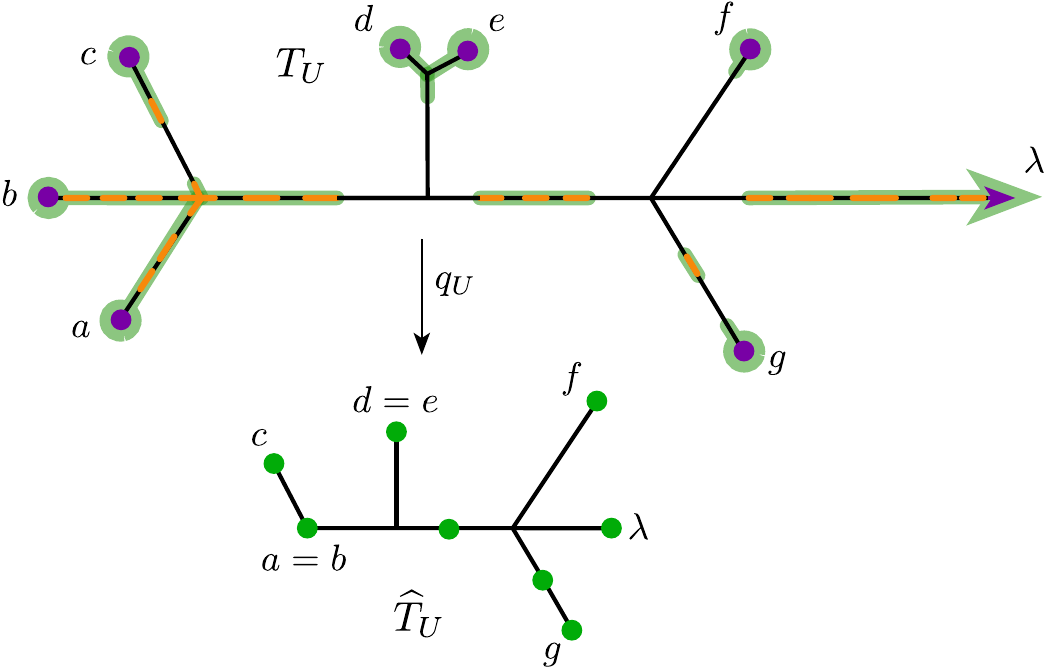}
    \caption{The shadows $\delta^V_U$ (orange) combine together with leaves (purple) to form clusters (green).  This can involve collapsing nearby leaves to a point (for $d$ and $e$), far away leaves to a point (for $a$ and $b$), or even a ray covered with shadows to a point (as with $\lambda$).  Other leaves (like $f$ and $g$) are essentially unaffected.  }
    \label{fig:cluster_collapse}
\end{figure}

\subsection{Relative projections for collapsed trees}

We finish our initial discussion by defining new relative projections for our collapsed trees.

Consider the product space $\hY = \prod_{U \in \calU} \hT_U$, as well as the extended product $\oY = \prod_{U \in \calU} \hT_U \cup \partial \hT_U$.  These spaces inherit a ``quasi-hierarchical'' structure, in the sense that the index set $\calU$ inherits the relations $\nest, \pitchfork, \perp$, and we have natural projections as follows:

\begin{definition}[Collapsed tree projections] \label{defn:Q project}
The projections and relative projections between $\oY$ and the collapsed trees $\hT_U \cup \partial \hT_U$ for $U \in \calU$ are as follows:
\begin{itemize}
    \item The domain projections are just the natural component-wise projections $\pi_{\hT_U}: \oY \to \hT_U \cup \partial \hT_U$.
    \item When $U \pitchfork V$, the relative projection from $\hT_U$ to $\hT_V$ is $\hd^U_V = q_V(\delta^U_V)\in \hT_V$.
    \item When $U \nest V$, the relative projection from $\hT_U$ to $\hT_V$ is $\hd^U_V = q_V(\delta^U_V)$.
    \item Finally, when $U \nest V$, the relative projection from $\hT_V$ to $\hT_U$ is $\hd^V_U = q_U \circ \delta^V_U \circ q^{-1}_V$.
\end{itemize}
\end{definition}

\begin{remark}
The above structure is not a genuine hierarchical structure because it does not satisfy the Large Links Axiom \ref{ax:LL} or, equivalently, the Passing-Up Lemma \ref{lem:passing-up}.  In some sense, that is the whole point of this construction, namely collapsing redundant information contained up the $\nest$-lattice.
\end{remark}

\subsection{$0$-consistency and the model space $\calQ$}

Since we have a set of relative projections for $\oY$, we now can define $0$-consistency:

\begin{definition}[$0$-consistency and $\calQ$]\label{defn:Q consistent}
A tuple $(\hx_U) \in \oY$ is $0$-\emph{consistent} if we have
\begin{enumerate}
    \item If $U \pitchfork V \in \calU$, then either $\hx_U = \hd^V_U \in \hT_U$ or $\hx_V = \hd^U_V \in \hT_V$.
    \item If $U \nest V \in \calU$, then either $\hx_V = \hd^U_V \in \hT_V$ or $\hx_U \in \hd^V_U(\hx_V) \subset \hT_U$.
\end{enumerate}
\begin{itemize}
\item We set $\oQ$ to be the set of $0$-consistent tuples in $\oY$.
    \item (Canonicality) We let $\calQ \subset \oY$ denote the set of tuples $\hx \in \oQ$ which are also \emph{canonical} in the sense that 
    \begin{itemize}
    \item There does not exist $U \in \calU$ so that $\hx_U \in \partial \hT_U$ represents a ray, and
    \item $\#\{U \in \calU| \hx_U \neq \ha_U\} < \infty$ for all $a \in F$.
	\end{itemize}
\end{itemize}

\end{definition}

\begin{remark}\label{rem:canonicality}
Canonicality is a finiteness condition, and plays a similar role to the canonical condition on sets of hyperplane orientations in a cube complex, hence the name.  It is only necessary in the definition of $\calQ$ when $\Lambda \neq \emptyset$, since otherwise each $\hT_U$ is a finite tree, $\calU$ is finite, and therefore $\calQ = \oQ$.  Its purpose is to pick out ``interior'' points in $\oY$, and it is strictly necessary because one can build a variety of examples of an HHS $\calX$ and a hierarchy ray $h \subset \calX$ whose projection to $\calC(S)$ (where $S$ is $\nest$-maximal) has infinite diameter, while the quasi-geodesic ray $\pi_S(h)$ is coarsely covered by $\calV = \{\rho^U_S| U \in \calS - \{S\}\}$, with each $U \in \calV$ being $\nest$-minimal.  Consequently, $\hT_S$ would be a point, but tuples representing the basepoint of $h$ and $h$ itself would both be $0$-consistent and yet differ on infinitely-many domains.  See Subsection \ref{subsec:ray example} for an explicit example, as well as \cite[Proposition 7.21]{DZ22}.  See also \cite{LLR, BLMR, CMW} for explicit examples in Teichm\"uller space.
\end{remark}

As a consequence of the canonical condition in Definition \ref{defn:Q consistent}, we get that the $\ell^1$-sum on $\calQ$
$$d_{\calQ}(x,y) = \sum_{U \in \calU} d_{\hT_U}(\hx_U,\hy_U)$$
defines a metric on $\calQ$, where $\hx,\hy \in \calQ$.  In particular, the canonical condition guarantees that the sum has finitely-many non-zero terms, which, in turn, guarantees that the triangle inquality holds.

\subsection{The map $\hPsi:H \to \calQ$} \label{subsec:hPsi defined}

Our next goal is to prove a consistency result for points in $H \cup \Lambda$, which will complete the reduction of the coarseness of consistency to nothing.  First, we need some notation and a lemma.

Recall from Definition \ref{defn:tree consistency}, when $\alpha>0$, then $\oZ_{\alpha}$ denotes the set of $\alpha$-consistent tuples in $\oY = \prod_{U \in \calU} T_U \cup \partial T_U$.  Moreover, in Lemma \ref{lem:consist to tree}, we saw that there exists $\alpha_0 = \alpha_0(\mathfrak S, |F \cup \Lambda|)>0$ so that for any $x \in H \cup \Lambda$, we had that $\Psi(x) \in \oZ_{\alpha_0}$.  We can extend this to a map to 
$$\hPsi:H\cup \Lambda \to \oY$$
by applying the quotient maps $q_U$ coordinate-wise.  In particular, for each $U \in \calU$, we have a map $\psi_U:\calX \to T_U$ given by $\psi_U = \phi^{-1}_U\circ p_U \circ \pi_U$.  We can extend this to a map $\hpsi_U:\calX \to \hT_U$ to the collapsed trees by composing with the quotient maps, namely $\hpsi_U:\calX \to \hT_U$, where $\hpsi_U = q_U \circ \psi_U$.

For a point $x \in H \cup \Lambda$, we will write $\psi_U(x) = x_U$ and $\hpsi_U(x) = \hx_U$.  Thus when $x,y \in H \cup \Lambda$, we will write $d_{T_U}(x_U,y_U)$ and $d_{\hT_U}(\hx_U,\hy_U)$ instead of $d_{T_U}(\psi_U(x), \psi_U(y))$ and $d_{\hT_U}(\hpsi_U(x), \hpsi_U(y))$, respectively.

Before our consistency proof, we need to record some basic features of the construction. It will sometimes be useful to highlight coordinates $\hx_U$ of $\hx \in \oQ$ which live in edge components of $\hT_U$.  For a given $\hx \in \oQ$, we denote the corresponding set of domains by $E(\hx) = \{U \in \calU | \hx_U \in \hT^e_U\}$.

\begin{lemma}\label{lem:ct basics}
The following hold:
\begin{enumerate}

\item If $U \pitchfork V \in \calU$, then both $\hd^V_U \in \hT_U$ and $\hd^U_V \in \hT_V$ coincide with marked cluster points.
    
    \item If $U \nest V \in \calU$, then $\hd^U_V$ coincides with a cluster point.
    
    \item For any $U \in \calU$ and any $x,y \in H$, we have $d_U(x,y) \asymp d_{T_U}(x_u,y_U)\geq d_{\hT_U}(\hx_U,\hy_U)$, with the constants in $\asymp$ depending only on $\mathfrak S$ and $|F \cup \Lambda|$.
    
    \item  If $\hx \in \oQ$, then $\{U \in \calU| \hx_U \neq \ha_U \text{ for some } a \in F \cup \Lambda\}$ is pairwise non-transverse, and hence has cardinality bounded in terms of $\mathfrak S$ and $|F \cup \Lambda|$.
\end{enumerate}
\end{lemma}

\begin{proof}
Item (1) follows from item (2) of Lemma \ref{lem:tree control}.  Item (2) follows directly from the construction.  Item (3) is an immediate consequence of the definition of the hull $H$ and the construction in this subsection.

Finally, for item (4): Suppose $U, V$ are transverse and $\hx_U \neq \ha_U$ for any $a \in F \cup \Lambda$.  Since $\hd^V_U$ coincides with a marked cluster point by item (1), it follows from $0$-consistency of $\hx$ that $\hx_V = \hd^U_V$, and hence $\hx_V$ coincides with a marked cluster point again by item (1).  This completes the proof.

\end{proof}

With this in hand, we can prove our consistency result:

\begin{proposition}\label{prop:hPsi defined}
If $R > \alpha_0 + E'$, then for any $(x_U)\in \oQ_{\alpha_0}$, we have $(q_U(x_U)) \in \oQ$.
\end{proposition}

\begin{proof}
Let $(x_U) \in \oZ_{\alpha_0}$.  We will show that $(q_U(x_U)) \in \oY$ is $0$-consistent (Definition \ref{defn:Q consistent}).

First suppose that $U \pitchfork V \in \calU$.  Since $(x_V)$ is $\alpha_0$-consistent by Lemma \ref{lem:consist to tree}, we have, without loss of generality, that $d_{T_U}(x_U, \delta^V_U)< \alpha_0 < R$.  Hence $q_U(x_U) = \hd^V_U$.

Next suppose that $U \nest V \in \calU$.  If $d_{T_V}(x_V, \delta^U_V)< \alpha_0 < R$, then $q_V(x_V) = \hd^U_V$, and we are done.  Otherwise, we have $\diam_{T_U}(x_U \cup \delta^V_U(x_V))<\alpha_0$ while also $d_{T_V}(x_V, \delta^U_V)>\alpha_0$.  By taking $\alpha_0 > E'$, we have that $\diam_{T_U}(\delta^V_U(x_V)\cup f_U)< E'$ for some $f \in F \cup \Lambda$ by the BGI property in item (5) of Lemma \ref{lem:tree control}.  In particular, both $x_U$ and $\delta^V_U(x_V)$ are within $\alpha_0 + E'< R$ of $f_U$, and hence $q_U(x_U) = q_V(\delta^V_U(x_V))$.  Finally, observe that $q^{-1}_V(q_V(x_V))$ is contained in the component of $T_V - \mathcal{N}_{E'}(\delta^U_V)$ containing $x_V$, and hence $q_V(\delta^V_U(x_V)) = \hd^V_U(q_V(x_V))$.  This completes the proof.

 \end{proof}

\begin{remark}
It is not clear, a priori, that the map $\hPsi: H \to \calQ$ is coarsely-dense.  This will follow from Theorem \ref{thm:coarse inverse}, in which we build a map $\hPhi:\calQ \to H$ and show that $d_{\calQ}(x, \hPhi \circ \hPsi(x))$ is bounded independent of $x \in \calQ$.  Unsurprisingly, this will depend crucially on a version of coarse surjectivity for $\Psi:H \to \calZ_{\alpha_0}$, which we prove first below in Proposition \ref{prop:coarsely surjective, trees}.
\end{remark}

\subsection{Collapsed tree control}

We finish this section with the analogue of the Tree Control Lemma \ref{lem:tree control} for our collapsed trees, which we will cite many times.  In order to make citations easier in the proofs that follow, we have made it an omnibus lemma by including some of the conclusions of Lemma \ref{lem:ct basics} and Proposition \ref{prop:hPsi defined}.

\begin{lemma}[Controlling the collapsed trees]\label{lem:collapsed tree control}
The following hold:
\begin{enumerate}
    \item For any $x \in H \cup \Lambda$, the tuple $(\hx_U)$ is $0$-consistent.
    
    \item If $U \pitchfork V \in \calU$, then both $\hd^V_U \in \hT_U$ and $\hd^U_V \in \hT_V$ coincide with marked cluster points.
    
    \item If $U \nest V \in \calU$, then $\hd^U_V$ coincides with a cluster point.
    
    \item For any $U \in \calU$ and any $x,y \in H$, we have $d_U(x,y) \asymp d_{T_U}(x_u,y_U)\geq d_{\hT_U}(\hx_U,\hy_U)$, with the constants in $\asymp$ depending only on $\mathfrak S$ and $|F \cup \Lambda|$.

    \item There exists a cluster separation constant $r = r(\mathfrak S, |F \cup \Lambda|)>0$ so that if $x,y \in H$ and $U \in \calU$ and $\hpsi_U(x), \hpsi_U(y)$ lie in distinct components of $\hT^c_U$, then $d_U(x,y) > 50E$.
    
    \item If $U \in \calU$ is $\nest_{\calU}$-minimal and $a,b \in F \cup \Lambda$ have $d_U(a,b) > K$, then $d_{\hT_U}(\ha_U,\hb_U) \succ K$, with the constant in $\succ$ depending only on $\mathfrak S$ and $|F \cup \Lambda|$.
    
    \item (Bounded geodesic image) Suppose $V \nest U$, $C \subset \hT_U - \hd^V_U$ is a component, then $\hd^U_V(C)$ is a point.  In particular, if $f \in F \cup \Lambda$ with $\hf_U$ contained in $C$, then $\hd^U_V(C) = \hf_V.$
    
    \item If $\hx \in \calQ^{\infty}$, then $E(\hx) = \{U \in \calU| \hx_U \in \hT^e_U\}$ is pairwise orthogonal.
    
    \item If $\hx \in \calQ^{\infty}$, then $\{U \in \calU| \hx_U \neq \ha_U \text{ for some } a \in F \cup \Lambda\}$ is pairwise non-transverse, and hence has cardinality bounded in terms of $\mathfrak S$ and $|F \cup \Lambda|$.
    \item If $U,V, W \in \calU$ with $V,W \nest U$, $V \pitchfork W$, and $\hd^V_U \neq \hd^W_U)$, then $\hd^U_V(\hd^W_U) = \hd^W_V$.
\end{enumerate}
\end{lemma}

\begin{proof}
Item (1) is a consequence of Proposition \ref{prop:hPsi defined}, in addition to Lemma \ref{lem:consist to tree}, which together show that $\Psi(H \cup \Lambda) \subset \oZ_{\alpha_0}$ and hence $\hPsi(F \cup \Lambda)\subset \oQ$.

Items (2)--(4) are items (1)--(3) of Lemma \ref{lem:ct basics}.

For item (5), if $\om_U(x),\om_U(y)$ are in distinct clusters, then $d_{T_U}(x,y) > r$, and hence $d_U(x,y) \succ r$ by item (4).  Hence choosing $r$ large enough gives the desired lower bound.

For item (6): When $U \in \calU$ is $\nest_{\calU}$-minimal with respect to $\calU$, then $q_U:T_U \to \hT_U$ only collapses the $2R$-neighborhoods of leaves in $T_U$.  Hence if $d_U(a,b) >K$ so that $d_{T_U}(a_U,b_U) \succ K$, then $\ha_U,\hb_U$ are only separated by at most $|F \cup \Lambda| - 2$ clusters, each of which has diameter at most $2R|F\cup \Lambda|$.  The claim follows.

For item (7): By item (1), if $\hf_U$ is a marked point or ray in $C$, then $0$-consistency $(\hf_V)$ implies that $\hd^U_V(\hf_U) = \hf_V$ because $d_{T_U}(\delta^V_U(f_U),f_V)<\alpha < R$.  On the other hand, since $C \cap \hd^V_U = \emptyset$, we have $d_{T_U}(q^{-1}_U(C), \delta^V_U)>R > E'$.  Hence item (5) of Lemma \ref{lem:tree control} says that $\diam_{T_V}(\delta^U_V(C))<E'$, while also $\delta^U_V(f_U) \subset \delta^V_U(C)$.  Hence $d_{T_V}(\delta^U_V(C), f_V)<\alpha + E'<R$, where we guarantee the last inequality by increasing $R$ as necessary.  It follows then that $\hd^U_V(C) = \hf_V$, as claimed.

For item (8): Supposing that $U,V$ are not orthogonal.  When 
$U \pitchfork V$ or $V \nest U$ and $\hx_U = \hd^V_U$, we are done by items (1) and (3).  When $V \nest U$ and $\hx_U \neq \hd^V_U$, item (7) implies there is some $f \in F \cup \Lambda$ so that $\hd^U_V(\hx_U) = \hf_V$, while $0$-consistency says that $\hx_V = \hd^U_V(\hx_U)$, giving $\hx_V = \hd^U_V(\hx_U) \in \hT^c_U$.

Finally, item (9) is item (4) of Lemma \ref{lem:ct basics}, while item (10) follows from item (6) of Lemma \ref{lem:tree control}.
\end{proof}

\subsection{(Non-)canonicality of $F$ and $\Lambda$}

We pause to observe the following:

\begin{lemma}\label{lem:F and Lambda}
We have $\hPsi(H) \subset \calQ$ and $\hPsi(\Lambda) \subset \calQ^{\infty}$.
\end{lemma}

\begin{proof}
We first show that $(\hx_U)$ is canonical in the sense of Definition \ref{defn:Q consistent} when $x \in H$.  Let $a \in F$.  Since $x \in H$, it follows that $\hx_U \not \partial \hT_U$ for any $U$.  So it suffices to prove that $\#\{U \in \calU | \ha_U \neq \hx_U\} < \infty$.  We may assume that $\Lambda \neq \emptyset$, since this is trivial otherwise.

To see this, assume for a contradiction that $\#\{U \in \calU | \ha_U \neq \hx_U\} = \infty$.  By item (4) of Lemma \ref{lem:ct basics}, $\hx_U$ coincides with a marked cluster point for all but boundedly-many $U \in \calU$.  Hence, there exists an infinite subset $\calV \subset \{U \in \calU | \ha_U \neq \hx_U\}$ so that $\hx_V, \ha_V$ are at distinct cluster points in $\hT_V$, and hence $d_V(x,a)>E$ for each $V \in \calV$.  Thus the Passing-up Lemma \ref{lem:passing-up} provides that for any $N>0$, if $V_1, \dots, V_N \in \calV$ are distinct, then there exists $W \in \calU$ with $d_W(x,a)>N$.  However, the Distance Formula \ref{thm:DF} provides that $\sum_{U \in \calU} [d_U(a,x)]_K \asymp d_{\calX}(a,x)< \infty$ is finite, which is a contradiction, since the previous fact says that the sum should be infinite.  This completes the proof that $(\hx_U) \in \calQ$.

Finally, we prove that $(\hlam_U) \in \calQ^{\infty}$ when $\lambda \in \Lambda$.  Since we are done if $\hlam_U \in \partial \hT_U$ for some $U$, suppose not.  Let $U \in \supp(\lambda)$, and since $\hlam_U \notin \partial \hT_U$, it follows that $\hlam_U$ coincides with a cluster containing the collapsed relative projections $\hd^V_U$ for infinitely-many domains $V \nest U$.  Moreover, we have that the corresponding tree relative projections $\delta^V_U$ form an $r$-dense subset of any ray in $T_U$ corresponding to $\lambda_U$.  Thus we can choose an infinite collection of domains $\calV\subset \calU$ so that for each $V \in \calV$ we have
\begin{itemize}
\item $V \nest U$,
\item $V$ is $\nest_{\calU}$-minimal,
\item $\hd^V_U = \hlam_U$, and
\item $\delta^V_U$ is at least $100K$-away from any branched point or marked point of $T_U$.
\end{itemize}

By construction, we have that $d_{T_V}(a_V,b_V) < E'$ for each $V \in \calV$ and any $b \in F \cup \Lambda - \{\lambda\}$.  Since $V \in \calU = \Rel_K(F \cup \Lambda)$, it follows that there exists some $b \in F \cup \Lambda - \{\lambda\}$ so that $V \in \Rel_K(b, \lambda)$, and so we must have $d_V(a,\lambda) > K - E$.  The proof of item (6) of Lemma \ref{lem:collapsed tree control} then implies that $d{\hT_V}(\ha_V, \hlam_V) \succ K$, meaning $\ha_V \neq \hlam_V$ for all $V \in \calV$.  Since $\calV$ is infinite by construction, we are done.  This completes the proof.

\end{proof}

\begin{remark}
We invoked the Distance Formula \ref{thm:DF} in the proof of Lemma \ref{lem:F and Lambda}.  This is only necessary when $\Lambda \neq \emptyset$, and hence does not affect our proof of the lower bound of the distance formula in Corollary \ref{cor:DF lower bound}.
\end{remark}

\subsection{Encoding passing-up}

In this final subsection, we observe how Strong Passing-Up \ref{prop:SPU} is actually encoded into the collection of domains $\calU$ and their collapsed trees.  The following structural lemma is necessary for studying $\calQ^{\infty}$, in particular as a cubical boundary.  Its proof is actually an application of Lemma \ref{lem:detecting rays}.

Roughly, it says that any infinite collection of domains in $\calU$ can be converted into a collection of domains whose $\hd$-sets (i.e., collapsed relative projections) run out a ray end of one of the trees.

\begin{lemma}\label{lem:encoding PU}
For any infinite collection $\calV \subset \calU$, there exists an infinite subcollection of domains $\{V_1, V_2, \dots\} \subset \calV$, elements $a \in F$ and $\lambda \in \Lambda$, so that the following hold:
\begin{enumerate}
\item $V_j \pitchfork V_i$ for $i \neq j$;
\item For $j>i$, we have $\hd^{V_j}_{V_i} = \hlam_{V_i}$ and $\hd^{V_i}_{V_j} = \ha_{V_j}$;
\item For $k<i$, we have $\hd^{V_k}_{V_i} = \ha_{V_i}$ and $\hd^{V_i}_{V_k} = \hlam_{V_k}$. 
 \end{enumerate}
\end{lemma}

\begin{proof}
Let $\calV \subset \calU$ be infinite.  Up to passing to an infinite subsequence, we may assume that there are $a,b \in F \cup \Lambda$ with $\calV \subset \Rel_K(a,b)$, and note that at least one of $a \in \Lambda$ or $b \in \Lambda$, and possibly both.

By Lemma \ref{lem:detecting rays}, there exists an infinite subcollection $\calV'$ and a support domain $W \in \supp(a) \cup \supp(b)$ so that $\diam_W(\bigcup_{V \in \calV'} \rho^V_W) = \infty$.  By the BGIA \ref{ax:BGIA}, we have that $d_W(\rho^V_W, [a,b])<E$, where $[a,b]$ is any $(1,20\delta)$-quasi-geodesic between $a,b$ in $\calC(W)$.  Up to passing to a further infinite subcollection of $\calV'$ if necessary, we may also assume that $d_W(\rho^{V}_W, \rho^{V'}_W) > 100K$ for any $V \neq V' \in \calV'$.  In particular, we may assume that $\calV'$ is pairwise transverse.

Without loss of generality, assume that $b = \lambda \in \Lambda$.  We now explain how to reduce to the case where $a \in F$.  To see this, observe that we can again pass to a further infinite subcollection of $\calV'$ so that if $f \in F$, then $\rho^V_W$ is at least $K$-far away from $f$ and within $100E\delta$ of any $(1,20\delta)$-quasi-geodesic ray between $c$ and $\lambda$ in $\calC(W)$.  As such, the BGIA \ref{ax:BGIA} implies that $d_V(a,c) < E$ for any $c \in F \cup \Lambda - \{\lambda\}$, which implies that $V \in \Rel_{K-E}(c,\lambda)$ for all $V \in \calV'$.

Thus we may assume that $a \in F$ and $\lambda \in \Lambda$.  Since $\calU$ is countable (Lemma \ref{lem:rel sets countable}), enumerate $\calV' = \{V_1, V_2, \dots\}$, where we are taking $\calV'$ to be the subcollection as refined above.  Moreover, up to reordering, we may assume that the $\rho^{V_i}_W$ run monotonically (with $i$) out the end of any $(1,20\delta)$-quasi-geodesic ray from $a$ to $\lambda$ in $\calC(W)$, with, as above, $d_W(\rho^{V_i}_{W}, \rho^{V_j}_W)>100K$ for each $i \neq j$.

Item (1) of the lemma is immediate from the above.  Items (2) and (2) of the lemma are now straight-forward to verify using $0$-consistency of $a$ and $\lambda$ (Proposition \ref{prop:hPsi defined}), and the fact that the $\rho^{V_i}_W$ escape out toward $\lambda \in \partial \calC(W)$ monotonically.  This completes the proof.
\end{proof}

We will also later need the following lemma, in which $\supp^{\nest}(\lambda)$ denotes the set of domains in $\calU$ which nest into some domain in $\supp(\lambda)$.

\begin{lemma}\label{lem:forced support}
For any $\lambda \in \Lambda$ and $a \in F$, we have $\# \{U \in \calU| \hlam_U \neq \ha_U\} - (\supp(\lambda) \cup \supp^{\nest}(\lambda))< \infty$.
\end{lemma}

\begin{proof}
Let $\calV$ be the set in the statement.  Then for any $V \in \calV$, the fact that $\hlam_V \neq \ha_V$ implies that $V \in \Rel_{50E}(a, \lambda)$ since our cluster separation constant satisfies $r \gg 50E$.  Thus if $\#\calV = \infty$, Lemma \ref{lem:detecting rays} provides $W \in \supp(\lambda)$ and an infinite subcollection $\calV' \subset \calV$ so that $V \nest W$ for all $V \in \calV'$, which contradicts the definition of $\calV$.  This proves the lemma.
\end{proof}

\subsection{Equivariance for $\calQ$} \label{subsec:auto model}

In this subsection, we briefly observe that the construction of our models can be done equivariantly with respect to HHS automorphisms (Subsection \ref{subsec:HHS auto}).  The basic observation is that the construction depends only on the fundamental parts of an HHS, which any automorphism preserves.

\begin{corollary}[Equivariant models]\label{cor:auto model}
If $g \in \mathrm{Aut}(\calX, \mathfrak S)$ is an HHS automorphism, then one can construct a cubical model $\calQ_g$ for $g \cdot (F \cup \Lambda)$ so that $g$ induces an isometry $\calQ \to \calQ_g$, and moreover a bijection $\calQ^{\infty} \to \calQ_g^{\infty}$.
\end{corollary}

\begin{proof}
By Definition \ref{defn:HHS auto}, $g$ preserves projection sizes (up to relabeling), and hence $g \cdot \Rel_K(F \cup \Lambda) = \Rel_K(g \cdot (F \cup \Lambda))$.   Moreover, given $U \in \calU$, any Gromov modeling tree $\phi_U:T_U \to \calC(U)$ for $\pi_U(F \cup \Lambda)$ can be turned into a modeling tree $T_{gU}$ for $\pi_{gU}(g \cdot (F \cup \Lambda))$ in the obvious way, by composing $T_U \to \calC(U) \to \calC(g U)$, with the last map the isometry $\calC(U) \to \calC(gU)$ induced by $g$.  Similarly, since $g$ preserves the relative projections, one can construct cluster and the collapsed tree $\hT_{gU}$ from $\hT_U$ in a similar fashion.  Hence $g$ induces an isometry between $\hT_U \to \hT_{gU}$ for each $U \in \calU$ which preserves labels of marked points and rays, in that $g\cdot \ha_U = \widehat{g \cdot a}_{g U}$ for each $a \in F \cup \Lambda$, while also preserving cluster points, in that $g \cdot \hd^V_U = \hd^{g \cdot V}_{g \cdot U}$.

As such, $g$ induces a bijection between $0$-consistent sets $\oQ \to \oQ_g$, which preserves canonicality.  We leave the details to the reader.  This completes the proof.
\end{proof}

This proves part of the equivariance statement of Theorem \ref{thmi:main model} from the introduction.  We will show in Corollary \ref{cor:auto boundary} that the map $\calQ \to \calQ_g$ is a cubical isomorphism and induces a simplicial isomorphism between their simplicial boundaries.

\section{Distance estimates for $\calQ$} \label{sec:distance estimates}

In this section, we prove the following theorem, which allows us to estimate distances in $H = \hull_{\calX}(F \cup \Lambda)$ by distances in $\calQ$.  Notably, the sum on the left is the sum from the Distance Formula (Theorem \ref{thm:DF}), though our proof is completely independent of it and, in fact, will be used to give the upper bound in the distance formula in Corollary \ref{cor:DF lower bound} in the Section \ref{sec:HP and DF}.

In the following, for numbers $A,K$ the notation $[A]_K$ mean $0$ when $A <K$ and $A$ when $A \geq K$.  For $x \in H$ and $U \in \calU$, we will write 

\begin{itemize}
    \item $x = \pi_U(x) \in \calC(U)$;
    \item $x_U = \psi_U(x) \in T_U$;
    \item $\hx_U = \hpsi_U(x) \in \hT_U$.
\end{itemize}

\begin{theorem}\label{thm:Q distance estimate}
There exists $K_1 = K_1(\mathfrak S, |F \cup \Lambda|)>0$ so that for any $K_2 \geq K_1$ and any $x, y \in H$, we have that
$$\sum_{U \in \mathfrak S} [d_U(x,y)]_{K_2} \asymp \sum_{U \in \calU} d_{\hT_U}(\hx_U,\hy_U)$$
where the constants in $\asymp$ only depend on $\mathfrak S, |F \cup \Lambda|$ and the choice of $K_2$.
\end{theorem}

Before beginning the proof, we make a couple of simplifying observations:
\vspace{.1in}

First, note that if $x,y \in H$ and $V \in \Rel_K(x,y)$ but $V \notin \calU$, then $d_V(x,y) < K + \epsilon$ for some $\epsilon = \epsilon(\mathfrak S, |F \cup \Lambda|)>0$.  Hence by taking $K_2 \geq K_1 > K + \epsilon$, we may increase the threshold to remove any extraneous domains.  That is:
$$\sum_{U \in \mathfrak S} [d_U(x,y)]_{K_2} = \sum_{U \in \calU} [d_U(x,y)]_{K_2}.$$

We next observe that for any $x,y \in H$, we have $d_U(x,y) \asymp d_{T_U}(x_U, y_U)$, with constants depending only on $\mathfrak S$ and $|F \cup \Lambda|$.

We prove the two bounds in Theorem \ref{thm:Q distance estimate} separately.  For the upper bound, we only need the basic details from the construction of the $\hT_U$ and the passing-up type arguments from Section \ref{sec:qpu}.  Note that this proposition is used in proving the existence of hierarchy paths and the upper bound of the distance formula, which were summarized in Corollary \ref{cori:DF and HP}.

\begin{proposition}[Upper bound]\label{prop:upper bound}
There exists $K_1= K_1(\mathfrak S, |F \cup \Lambda|)>0$ so that for any $K_2 >K_1$ and any $x, y \in H$ we have
\begin{equation}\label{eqn:upper}
    \sum_{U \in \calU} [d_U(x,y)]_{K_2} \succ \sum_{U \in \calU} d_{\hT_U}(\hx_U,\hy_U)
\end{equation}
with the constants in $\succ$ depending only on $\mathfrak S,|F \cup \Lambda|,$ and $K_2$.
\end{proposition}

\begin{proof}

First, item (4) of Lemma \ref{lem:collapsed tree control} says that $d_U(x,y) \succ d_{\hT_U}(\hx_U,\hy_U)$ for all $U \in \calU$, where the constants in $\succ$ depend only on $\calX, |F \cup \Lambda|$.

Now note that if $U \in \calU$ satisfies $\hx_U = \hy_U$, then the corresponding term disappears on the right-hand side.  Hence it suffices to account for the domains in $\calV = \{V \in \calU | d_V(x,y) < K_2 \textrm{ and } \hpsi(x) \neq \psi(y)\}$, which appear on the left side but not the right.

For this, item (8) of Lemma \ref{lem:collapsed tree control} says that $\hx_U$ is not contained in some cluster point of $\hT_U$ for only boundedly-many $U$, and the same for $\hy_U$.  This means that, up to ignoring boundedly-many domains, if $V \in \calV$ then $\hx_V$ and $\hy_V$ are at distinct cluster points, and so $d_V(x,y) > 50E$ by item (5) of Lemma \ref{lem:collapsed tree control}.

We are now done by Proposition \ref{prop:bounding containers}, which says that up to ignoring boundedly many such $V$, there is some $W \in \calU$ with $d_W(x,y) > K_2$ and $V \nest W$ for each $V \in \calV$.  Moreover, for any given $W \in \calU - \calV$ with $d_W(x,y) > K_2$, we have $\#\{V \in \calV|V \nest W\} \prec d_W(x,y)$.  That is, each $V \in \calV$ is accounted for by some $W \in \calU - \calV$, and each such $W$ accounts for boundedly-many $V \in \calV$.  Hence we can adjust the multiplicative constant to allow the terms in the distance estimate in \eqref{eqn:upper} associated to the collection of such $W$ to account for the missing $\calV$ terms, and the additive constants to account for the boundedly-many terms associated to the domains in $\calV$ which were left out of this process.

\end{proof}

It remains to prove the following:

\begin{proposition}[Lower bound]\label{prop:lower bound}
There exists $K_1= K_1(\mathfrak S, |F \cup \Lambda|)>0$ so that for any $K_2 > K_1$ and any $x, y \in H$ we have
\begin{equation}\label{eqn:lower}
    \sum_{U \in \calU} [d_U(x,y)]_{K_2} \prec \sum_{U \in \calU} d_{\hT_U}(\hx_U,\hy_U)
\end{equation}
with the constants in $\prec$ depending only on $\mathfrak S,|F \cup \Lambda|,$ and $K_2$.
\end{proposition}

For the lower bound, we have to work a bit harder.  Observe that since $d_U(x,y) \prec d_{T_U}(x_U,y_U)$ for each $U$, it suffices to show that 
$$\sum_{U \in \calU} [d_{T_U}(x_U,y_U)]_{K_2} \prec \sum_{U \in \calU} d_{\hT_U}(\hx_U,\hy_U).$$

The trouble comes from domains $U \in \calU$ with $d_{T_U}(x_U,y_U) > K_2$ but there is a discrepancy $d_{T_U}(x_U,y_U) > d_{\hT_U}(\hx_U,\hy_U)$.  Since $\hT_U$ is obtained from $T_U$ by collapsing clusters, some analysis of these collapsed clusters is needed.  In particular, any difference in distances occurs because of distinct clusters of shadows $C_1, \dots, C_n$ occurring in order along the unique geodesic $[x_U,y_U]_{T_U}$.  To each such cluster $C_i$, we will associate a collection $\calV_i$ of $\nest_{\calU}$-minimal domains $V \in \calU$ with $\delta^V_U \in C_i$, and then use the Bounding Large Containers Proposition \ref{prop:bounding containers} (and a localization argument) to conclude that $\#\calV_i \asymp \diam_{T_U}(C_i \cap [x_U,y_U]_{T_U})$ for each $i$.  The fact that the $V \in \calV_i$ are $\nest_{\calU}$-minimal will say that, up to ignoring boundedly many such $V$, we have $d_V(x,y) \succ K$ by item (6) of Lemma \ref{lem:collapsed tree control}.  Hence each such minimal domain makes a large contribution to the distance sum for $\calQ$, and since each can only nest into boundedly-many larger domains by  Covering Lemma \ref{lem:covering}, it follows that the collection of minimal domains can account for the distance in each $U$ lost by collapsing clusters, up to increasing the multiplicative constant a bounded amount.

We note that the proof uses an auxiliary lemma, which we state and prove after the proof since it is notation-dependent.

\begin{proof}[Proof of Proposition \ref{prop:lower bound}]
Continuing with the above notation, for each $i$, let $C'_i = C_i \cap [x_U,y_U]{T_U} - \left(B_{\alpha_0 + E'}(x_U) \cup B_{\alpha_0 + E'}(y_U)\right)$.

Observe that for each $i$, if we set $\calV_i = \{V \in \calU|V \nest U \text{ and } \delta^V_U \subset C_i\}$ then $\{\delta^V_U| V \in \calV_i\}$ is $(r+2R)$-dense in $C_i$.  Hence we can find a subcollection $\calV_i'$ so that
\begin{itemize}
    \item Each $V \in \calV_i'$ is $\nest_{\calU}$-minimal;
    \item For each $V \in \calV'$, we have $\delta^V_U \subset C_i'$;
    \item $\{\delta^V_U| V \in \calV_i\}$ is a $(2r+4R, r)$-net on $C_i'$ (see Definition \ref{defn:net} below).
\end{itemize}
Note that the increase in the net constants in the third item is required to guarantee the second item.

In particular, we have $\#\calV_i \succ \diam_{T_U}(C_i') \geq \diam_{T_U}(C_i)- 2E'$.

Let $V \in \calV_i'$.  Then, by construction, $d_{T_U}(\delta^V_U, x_U)> \alpha_0 + E'$ and $d_{T_U}(\delta^V_U, y_U)>\alpha_0 + E'$.  Moreover, up to ignoring at most $N_1=N_1(\mathfrak S, |F \cup \Lambda|)>0$ domains $V \in \calV$, Lemma \ref{lem:big minimal} provides $a,b \in F \cup \Lambda$ where $C'$ separates $a_U$ from $b_U$ in $T_U$ and $V \in \Rel_K(a,b)$.  Hence the BGI property (5) of Lemma \ref{lem:tree control} says that $d_{T_V}(\delta^U_V(a_V),\delta^U_V(x_V))<E'$ and $d_{T_V}(\delta^U_V(b_V),\delta^U_V(y_V))<E'$. 

Lemma \ref{lem:consist to tree} says that the tuples $(a_U), (b_U), (x_U), (y_U)$ are all $\alpha_0$-consistent.  Thus since $d_{T_U}(a_U, \delta^V_U)>E' + \alpha_0$ and similarly for $b_U,x_U,y_U$, $\alpha_0$-consistency forces that $d_{T_V}(a_V, \delta^U_V(a_U)) < \alpha_0$, and similarly for $b_V,x_V,y_V$.  As such, we get that $d_{T_V}(x_V,a_V) < \alpha_0 + E'$ and $d_{T_V}(y_V,b_V) < \alpha_0 + E'$, and so that $d_{T_V}(x_V,y_V) \geq d_{T_V}(a_V,b_V) - 2\alpha_0 - 2E'$, from which we can conclude that $\ha_V = \hx_V$ and $\hb_V = \hy_V$ by requiring $R > \alpha_0 + E'$.

On the other hand, since each $V \in \calV_i'$ is $\nest_{\calU}$-minimal and $d_V(a,b) > K$, item (4) of Lemma \ref{lem:collapsed tree control} says that $d_{\hT_V}(\hat{a}_V,\hat{b}_V) \succ K$.  Hence we obtain that $d_{\hT_V}(\hx_V,\hy_V)  = d_{\hT_V}(\ha_V,\hb_V) \succ K$.

Thus for each $U \in \calW$, we have
\begin{eqnarray*}
d_{T_U}(x_U,y_U) & = & d_{\hT_U}(\hx_U,\hy_U) + \sum_{i=1}^n \diam_{T_U}(C_i \cap [x_U,y_U]_{T_U})\\
&\prec& d_{\hT_U}(\hx_U,\hy_U) + \sum_{i=1}^n \#\calV_i'\\
&<& d_{\hT_U}(\hx_U,\hy_U) + \sum_{i=1}^n K\cdot \#\calV_i'\\
&\prec& d_{\hT_U}(\hx_U,\hy_U) + \sum_{i=1}^n \left(\sum_{V \in \calV_i'} d_{\hT_V}(\hx_V,\hy_V)\right).
\end{eqnarray*}

Finally, the Covering Lemma \ref{lem:covering} says that any domain $V \in \calU$ can nest into only boundedly-many domains in $\calU$, with bounds determined by $\mathfrak S, |F \cup \Lambda|$ since Lemma \ref{lem:covering} applies to each pair $a,b \in F \cup \Lambda$.  Hence, each $\nest_{\calU}$-minimal domain $V \in \calU$ must account for at most boundedly-many $U \in \calW$ in the above way.  Thus we can account for the distance lost by collapsing clusters in each domain in $\calW$ with these large $\nest_{\calU}$-minimal domains by increasing the multiplicative constant for the estimate in \eqref{eqn:upper} by a bounded amount.  This completes the proof.

\begin{figure}
    \centering
    \includegraphics[width=.6\textwidth]{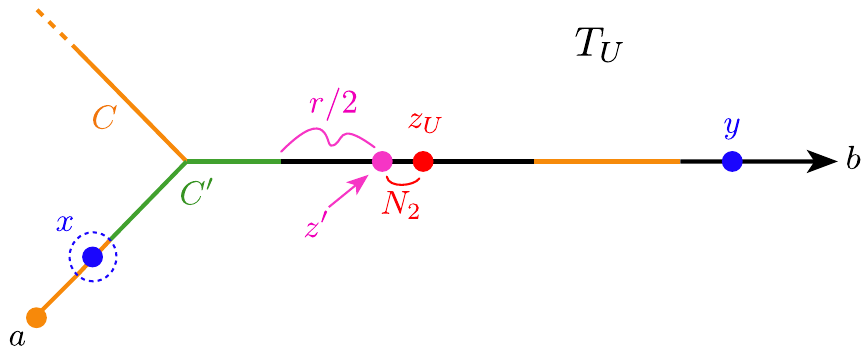}
    \caption{Proof of Proposition \ref{prop:lower bound}: The subcluster $C'$ isolates the portion of $C$ between the $R$-neighborhood of $x_U$, $\calN_R(x_U)$, and $y_U$.  The point $z \in H$ coarsely projects the same as $y$ and the rest of the ray $b$ to all domains with shadows in $C'$.}
    \label{fig:kappa_nbhd}
\end{figure}

\end{proof}

The following definition was used in the proof of Proposition \ref{prop:lower bound}:

\begin{definition}[Net]\label{defn:net}
Given constants $a,A\geq 0$, a $(a,A)$-\emph{net} on a subspace $Y \subset X$ of a metric space is a collection $Z \subset Y$ of points so that:
\begin{enumerate}
\item $Z$ is $a$-coarsely dense in $Y$, and
\item For any $z,z' \in Z$, we have $d_X(z,z') \geq A$.
\end{enumerate}
\end{definition}

The following lemma was also used in the proof, and we use the same notation:

\begin{lemma}\label{lem:big minimal}
There exists $N_1 = N_1(\mathfrak S, |F\cup \Lambda|)>0$ so that for each $i$ and for all but $N_1$-many domains $V \in \calV_i'$, there exists $a,b \in F \cup \Lambda$ so that $d_{T_V}(a_V,b_V)\succ K$.
\end{lemma}

\begin{proof}
By construction, the set of $\delta^V_U$ for $V \in \calV_i'$ is a $(2r+4R, r)$-net in $C'$ (Definition \ref{defn:net}), and hence only boundedly-many $\delta^V_U$ sets are within $E'$ of a branch point of $T_U$ depending on the uniform bound on branching of $T_U$.  Each remaining $\delta^V_U$ lies along an edge of $T_U$, and hence separates $T_U = T_x \cup T_{z}$ into two components, the component containing $x_V$ and the component containing $z_V$.  Let $l_x$ be the marked points and rays of $T_U$ in $T_x$, and similarly for $l_{z}$.  Since each such $\delta^V_U$ is $E'$-away from any branch point, the BGI property in item (5) of Lemma \ref{lem:tree control} implies that the geodesics in $T_U$ between any pair of marked points or rays in $p,q \in l_x$ must avoid the $E'$-neighborhood of $\delta^V_U$, and thus $d_{T_V}(p_V,q_V) < E'$, and similarly for marked points and rays of $l_{z}$.

Yet $V \in \calU = \Rel_K(F \cup \Lambda)$, and so there exists $a,b \in F \cup \Lambda$ so that $d_{T_V}(a_V,b_V)\succ K$.  By the above argument, we must have $a \in l_x$ and $b \in l_{z}$, as required.
\end{proof}

\begin{remark}
In the proof above, we used the observation that all but boundedly many shadows for $V \in \calV'$ avoided the branch points of $T_U$.  In fact, this something much stronger is true, namely that all but boundedly-many domains in $U \in \calU$ have the property that $\delta^U_W$ avoids the branch points of $T_W$ for each $W \in \calU$ with $U \nest W$.  See Subsection \ref{subsec:bipartite} below.
\end{remark}

\section{The map $\Omega: \calQ \to \calY$}\label{sec:Q to Y}

The purpose of this section is to define a map $\Omega: \calQ \to \calY$ and analyze its features.  In particular, the main goal is to prove the following:

\begin{proposition}\label{prop:Q-consistent}
There exist $\alpha_{\omega}=\alpha_{\omega}(\mathfrak S, |F \cup \Lambda|)>0$ and a map $\Omega: \calQ \to \calY$ so that for any $\hx \in \calQ$, the tuple $\Omega(\hx) \in \calZ_{\alpha_{\omega}}$.  Moreover, if $\Omega(\hx) = (x_U)$, then $q_U(x_U) = \hx_U$ for each $U \in \calU$.
\end{proposition}

With $\Omega:\calQ \to \calZ_{\alpha_{\omega}}$ in hand, we will be able to use our realization result (Proposition \ref{prop:consist from tree}) to define the map $\hO:\calQ \to H$.  In Section \ref{sec:coarsely surjective}, we will show that $\hPsi$ is a coarse inverse of $\hO$, and therefore $\hPsi$ is surjective.  Combining this with the distance estimates from Section \ref{sec:distance estimates} will prove that $\hPsi:H \to \calQ$ is a quasi-isometric embedding.

The proof of Proposition \ref{prop:Q-consistent} involves analyzing the structure of clusters in depth.  This is arguably the most technical part of the construction of the cubical model, and undoubtedly where cubical arguments first appear in spirit though not explicitly.

\subsection{Bipartite domains}\label{subsec:bipartite}

The first step is to make the observation that for all but boundedly-many domains in $\calU$, we have that $\hT_U$ is an interval.

\begin{definition}\label{defn:bipartite}
A domain $U \in \calU$ is called $\chi$-\emph{bipartite} if there is some $V \in \calU$ with $U \nest V$ so that $T_V - \calN_{\chi}(\delta^U_V)$ has two components which contain all marked points and branch points in $T_V$.
\begin{itemize}
    \item We denote the set of bipartite domains by $\calB(\chi)$.
    \item If $U \in \calB(\chi)$ and $V \in \calU$ are as above, then we refer to $V$ as a \emph{witness} for $U$.
\end{itemize}
\end{definition}

The following lemma explains the bipartite terminology, since it says their collapsed trees are bipartite graphs, i.e. intervals:

\begin{lemma}\label{lem:bipartite properties}
For any $\chi>E' + \alpha_0$, if $U \in \calB(\chi)$ and $V\in \calU$ is a witness for $U$, then the complement $T_V - \calN_{E'}(\delta^U_V) = C \sqcup C'$ has two components.  Moreover, the collapsed tree $\hT_U$ is a finite-length interval whose endpoints $l_{C}, l_{C'}$ coincide with the projections under $\hd^U_V$ of the marked points and rays in $q_V(C)$ and $q_V(C')$, respectively.  In particular, if $U \in \calB(\chi)$, then $U \notin \supp(\lambda)$ for any $\lambda \in \Lambda$.
\end{lemma}

\begin{proof}
By definition of $U \in \calB(\chi)$ and $V$ being a witness for $U$, the complement $T_V - \calN_{\chi}(\delta^U_V)$ has two components $C' \sqcup C''$.  By the BGI property in item (5) of Lemma \ref{lem:tree control} and the assumption that $\chi>E'$, the marked points and leaves of $C'$ which are marked points or rays of $T_V$, denoted $l_{C'}$, project $E'$-close in $T_U$, and hence coincide in $\hT_U$, and similarly for $C''$.  Hence $\hT_U$ is an interval with endpoints $q_U(\delta^V_U(l_{C'}))$ and $q_U(\delta^V_U(l_{C''}))$, as required.

To see that the diameter of $\hT_U$ is finite, suppose not.  Without loss of generality, we may assume that the endpoint $q_U(\delta^V_U(l_{C'}))$ coincides with $\hlam_U$ for some $\lambda \in \Lambda$.  Hence the only marked point in $C'$ is $\lambda_V$.  But by assumption $d_{T_V}(\delta^U_V, \lambda_V)>E' + \alpha_0$, and this violates $\alpha_0$-consistency of $(\lambda_U)$ (Lemma \ref{lem:consist to tree}), which requires that $d_{T_V}(\delta^U_V, \lambda_V) < \alpha_0$ since $\lambda_U \in \partial T_U$.  This completes the proof.
\end{proof}

\begin{figure}
    \centering
    \includegraphics[width=.6\textwidth]{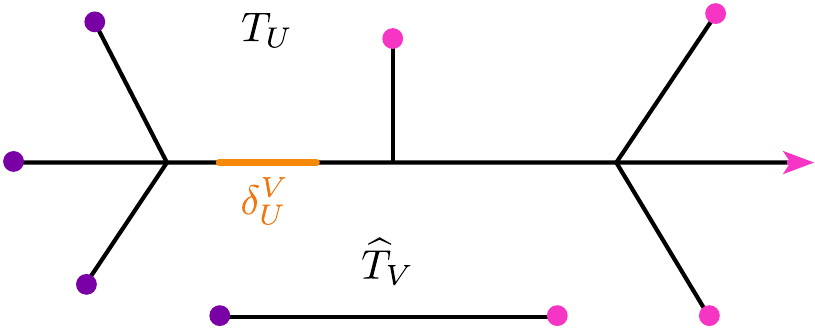}
    \caption{$V$ is a bipartite domain with witness $U$ because $\delta^V_U$ is far away from all branch and marked points of $T_U$.  The two endpoints of $\hT_V$ correspond to the way that $\delta^V_U$ partitions the marked points of $T_U$.} 
    \label{fig:bipartite}
\end{figure}

We are focusing on bipartite domains is because almost every domain in $\calU$ is bipartite:

\begin{proposition}
\label{prop:bipartite}
For any $\chi>E'$, there exists $B_{\chi} = B_{\chi}(\mathfrak S, |F \cup \Lambda|, \chi)>0$ so that $\#(\calU - \calB(\chi)) < B_{\chi}$.
\end{proposition}

\begin{proof}
Let $\calV \subset \calU$.  We will prove that if $\#\calV$ is sufficiently large, then $\calV$ contains a bipartite domain.

Up to dividing $\#\calV$ by at most $2^{|F \cup \Lambda|}$, we may assume that $\calV \subset \Rel_K(a,b) \subset \Rel_{50E}(a,b)$ for some fixed $a,b \in F \cup \Lambda$ while maintaining $\#\calV$ to be as large as necessary.

Let $\sigma = \sigma(\mathfrak S, |F \cup \Lambda|)>0$ and $\nu = \nu(\mathfrak S, |F \cup \Lambda|)>0$ be as in Lemma \ref{lem:tree subdivision} below.  Let $\mathbf{b} = \mathbf{b}(\mathfrak S, |F \cup \Lambda|)>0$ be the uniform bound on the number of branch points of $T_U$ for any $U \in \calU$.  Let $a= a_0, a_1, a_2, \dots, a_k = b$ be a $\sigma$-subdivision of $\gamma = [a,b]_W$.

Now we invoke Proposition \ref{prop:strong pu}, which provides $P_2(50E, K_2, \sigma, 100\mathbf{b}\nu|F \cup \Lambda|, \chi)>0$ (which depends only on $\mathfrak S, |F \cup \Lambda|$, and $\chi$) so that if $\#\calV > P_2$, then there exists $W \in \Rel_{K_2}(a,b)$ and a subcollection $\calV' \subset \calV$ so that 
$$\diam_W(\bigcup_{V \in \calV'} \rho^V_W) >K_2.$$

Moreover, if $\calW_i' = \{V \in \calV'| p_{\gamma}(\delta^V_U) \cap [a_i, a_{i+1}] \neq \emptyset\}$, then the proposition also implies that $\#\{1 \leq i \leq k-1| \calW'_i \neq \emptyset\} \geq 100\mathbf{b}\nu|F \cup \Lambda|\chi$.

Since the $\delta^V_W$ for $V \in \calV'$ can only intersect at most two subintervals in the $10E'\chi$-subdivision of $[a_W,b_W]_{T_W}$, Lemma \ref{lem:tree subdivision} and the lower bound above implies that the number of subintervals containing some $\delta^V_U$ far exceeds $2\mathbf{b}|F \cup \Lambda|$, which bounds the number of subintervals of $[a_W,b_W]_{T_W}$ containing branch points and marked points.  Hence we were able to force $\calV$ to contain at least one $\chi$-bipartite domain only by requiring it to have at most boundedly-many domains, as required.

\end{proof}

The following lemma was used in the proof above:

\begin{lemma}\label{lem:tree subdivision}
For any $\chi>10E'$, there exist $\sigma = \sigma(\mathfrak S, |F \cup \Lambda|,\chi)>0$ and $\nu = \nu(\mathfrak S, |F \cup \Lambda|,\chi)>0$ so that the following holds:

Suppose $a,b \in F \cup \Lambda$ and $W \in \Rel_{50E}(a,b)$.  Let $a = a_0, a_1, a_2, \dots, a_k = b$ be a $\sigma$-subdivision of a $(1,20\delta)$-quasi-geodesic $\gamma$ between $a_U,b_U$ in $\calC(W)$. Set $a_i' = \phi_{W}^{-1}\circ p_W (a_i)$ for each $i$, where $p_W:\calC(W) \to \phi_W(T_W)$ is closest point projection, and $\phi_W:T_W \to \calC(W)$ is any uniform quality tree modeling the hull $H_W$ of $\pi_W(F \cup \Lambda) \subset \calC(W)$.

Then for any $10E'\chi$-subdivision of $[a_W,b_W]_{T_W} \subset T_W$, each $a_i', a_j'$ are contained in distinct subintervals when $i\neq j$.  Moreover, we have that $d_{T_W}(a'_i, a'_{i+1}) < \nu$ for each $0 \leq i \leq k-1$.
\end{lemma}

\begin{proof}
The proof is a straight-forward application of the techniques from Section \ref{sec:reduce to tree}.  The one subtle point is choosing $\sigma$ to be as small enough to guarantee that the subdivision points of $\gamma$ land in separate subintervals along the tree geodesic, but large enough to guarantee that their images are at a bounded distance controlled by $\mathfrak S, |F \cup \Lambda|$.  The point is that the subdivision width of $[a_W,b_W]_{T_W} \subset T_W$ is fixed (at $10E'\chi$), while we get to control both the subdivision width $\sigma$ in $\calC(W)$ and its separation constant $\nu$ under the uniform quality quasi-inverse $\phi^{-1}_W:\phi_W(T_W) \to T_W$.  We leave the details to the interested reader.
\end{proof}

The reason for the extra variable $\chi$ will become clear in what follows.  That is, we need to be able to choose $\chi$ sufficiently large, depending on the relevant context, but always depending only on $\mathfrak S, |F \cup \Lambda|$.

Hence for the rest of this section, we fix $\chi = \chi(\mathfrak S, |F \cup \Lambda|)>0$, to be determined below.  In the end, taking $\chi> 5E' + \alpha_0$ works.

\subsection{Honing clusters}\label{subsec:honing clusters}

We will now explain how to define a map $\Omega:\calQ \to \calY$, so that:

\begin{itemize}
    \item $\Omega(\calQ) \subset \calZ_{\omega}$ for some $\omega = \omega(\mathfrak S, |F \cup \Lambda|)>0$, and
    \item when $\hx  = (\hx_U)\in \calQ$ and $(x_U) = \Omega(\hx)$, then $q_U(x_U) = \hx_U$ for each $U \in \calU$.
\end{itemize}

In this subsection, we set notation and give a description of the map $\Omega$.  We will prove that it has the desired properties in Proposition \ref{prop:Q-consistent} below.  Unsurprisingly, the definition is domain-wise in the $U \in \calU$.

Let $\hx = (\hx_U) \in \calQ$.  When $U\in \calU$ has $\hx_U \in \hT^e_U$, then the definition is simple: $x_U = q_U^{-1}(\hx_U)$, which is a unique point.  For the other domains, we need some more setup.

Suppose that $U \in \calU$ and $\hx_U \notin \hT^e_U$.  The first step is to associate a collection of marked points $l_V\subset F \cup \Lambda$ to each $\nest_{\calU}$-minimal bipartite domain $V \in \calB(\chi)$.  We then use these $l_V$ to hone the clusters of the larger $U \in \calU$ in which such $V$ are involved.

Suppose then that $U \in \calU$ is $\nest_{\calU}$-minimal.  Since $\hT_U$ has no internal cluster points, $\hat{x}_U$ is either a marked cluster point or a point in $\hT^e_U$.  If the former, let $l_U$ be the set of marked points coming from points in $F \cup \Lambda$ identified in the cluster associated to $\hat{x}_U$.  If the latter, set $l_U = \emptyset$.

Now suppose that $U \in \calU$ is not $\nest_{\calU}$-minimal.  We have that either $\hat{x}_U \in \hT^e_U$, in which case we set $l_U = \emptyset$ and $x_U = q_U^{-1}(\hat{x}_U)$.  Otherwise, $\hat{x}_U$ is contained in a collapsed cluster $q_U(C)$ for a cluster $C \subset T_U$.  Let $\calB_{\min}(C)$ be the domains $V \in \calB(\chi)$ so that
\begin{itemize}
    \item $V \nest U$ with $\delta^V_U \subset C$, and
    \item $V$ is $\nest_{\mathcal U}$-minimal.
    \item $U$ is a witness for $V$.
\end{itemize}

Observe that if $\BM(C) = \emptyset$, then $\diam_{T_U}(C)$ is bounded in terms of $R,r$ and the total branching of $T_U$, and hence in terms of $\mathfrak S$ and $|F \cup \Lambda|$.

We now explain the honing process for domains in $\BM(C)$, which will associate to $\hx_U$ a bounded diameter subset of the cluster $C$ in $T_U$ (Lemma \ref{lem:bdd int total}).

Let $V \in \BM(C)$.  Then we set 
$$C^V_U = \hull_C(C \cap (\calN_{r}(\delta^V_U) \cup l_V)),$$
which is a subtree of $C$, see Figure \ref{fig:C^V_U}.  Finally, let $B_U(\hx) = \bigcap_{V \in \BM(C)} C^V_U$ be the total intersection of the $C^V_U$.

\begin{figure}
    \centering
    \includegraphics[width=.6\textwidth]{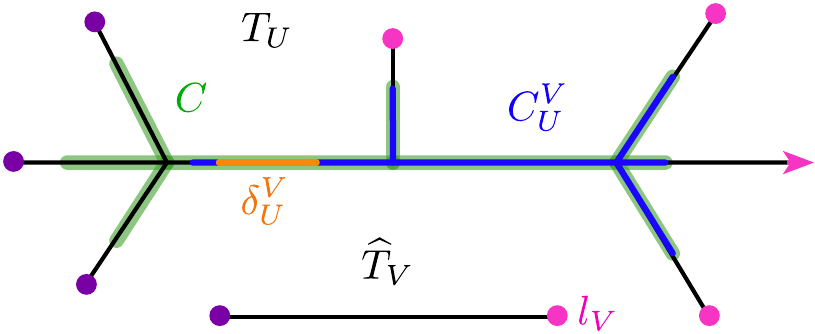}
    \caption{The cluster $C$ (green) in $T_U$ is identified with $\hx_U$ in $\hT_U$.  Since $V \in \BM(C)$, its shadow $\delta^V_U$ (orange) avoids all branches of $T_U$.  Lemma \ref{lem:leaves choose half-space} implies all marked points and rays in $T_U$ are identified to the two marked clusters of $\hT_V$.  The marked points and leaves $l_V$ (pink) chosen by $\hx_V$ determine the half-subtree $C^V_U$ (blue) of $C$ chosen by $V$.} 
    \label{fig:C^V_U}
\end{figure}

In Lemma \ref{lem:bdd int pair} we show that if $V, W \in \BM(C)$, the $C^V_U \cap C^W_U \neq \emptyset$.  Since each $C^V_U$ is a subtree, we can make an argument in Lemma \ref{lem:extended helly} using the Helly property for trees to conclude that $B_U(x)$ is nonempty, though this takes some care because $C$ and the $C^V_U$ can be infinite-diameter when $\Lambda \neq \emptyset$.  Finally, in Lemma \ref{lem:bdd int total} we prove that this intersection has diameter bounded in terms of $\mathfrak S$ and $|F \cup \Lambda|$.

Thus we can define a tuple $\Omega(\hx) = (x_U)_{U \in \calU}$ as follows:
\begin{itemize}
    \item For $U \in \calU$ with $U \in E(\hx)$, we set $x_U = q_U^{-1}(\hx_U)$.
    \item For $U \in \calU$ with $U \notin E(\hx)$ and $U \nest_{\calU}$-minimal, we let $x_U$ be any point in $q_U^{-1}(\hx_U)$.
    \item Otherwise, we let $x_U \in B_U(\hx)$.
\end{itemize}

The following is immediate from the construction:

\begin{lemma}\label{lem:B_U in cluster}
For every $\hx \in \calQ$ and $U \in \calU$, we have $q_U(x_U) = \hx_U$.
\end{lemma}

In Proposition \ref{prop:Q-consistent} below, we prove that $(x_U)$ is a $\omega$-consistent tuple for $\omega = \omega(\mathfrak S, |F \cup \Lambda|)>0$.

\subsection{Bounding intersections}

We carry forward the notation from the Subsection \ref{subsec:honing clusters}.

We begin with two simple but important observations.  The first is a direct application of Lemma \ref{lem:bipartite properties} and the definition of $C^V_U$:

\begin{lemma}\label{lem:leaves choose half-space}
Let $U,V \in \calU$ and $C\subset T_U$ a cluster with $V \in \BM(C)$.  Let $T_U - \calN_{E'}(\delta^V_U) = C' \cup C''$, and let $l_{C'}, l_{C''}$ denote the marked points and rays contained in $C',C''$ respectively.

Then $l_V \in \{l_{C'}, l_{C''}\}$, and if $l_V = l_{C'}$ then $C^V_U = \calN_{r}(\delta^V_U) \cup (C' \cap C)$, and similarly when $l_V = l_{C''}$.
\end{lemma}

The second follows directly from the construction of the clusters (Subsection \ref{subsec:shadows}), the definition of $\BM(C)$, and the uniform bound on branching of the trees $T_U$ and the size of $|F \cup \Lambda|$:

\begin{lemma}\label{lem:BM dense}
There exists $r' = r'(\mathfrak S, |F \cup \Lambda|)>0$ so that the following holds:  Suppose that $C \subset T^c_U$ is a cluster for some $U \in \calU$, and $C' \subset C$ is a (connected) subtree.  Then either:
\begin{itemize}
    \item $\BM(C') = \emptyset$ and $\diam_{T_U}(C') < 4r'$, or
    \item $\BM(C') \neq \emptyset$ and then there exists $\calV \subset \BM(C')$  so that $\{\delta^V_U|V \in \calV\}$ forms a $(2r',r')$-net on $C'$ (Definition \ref{defn:net}).
\end{itemize}
\end{lemma}

We can now prove that the $C^V_U$ for $V \in \BM(C)$ have pairwise nonempty intersection:

\begin{lemma}\label{lem:bdd int pair}
If $V, W \in \BM(C)$, then $C^V_U \cap C^W_U \neq \emptyset$.  Moreover, if $d_{T_U}(\delta^V_U, \delta^W_U)>r$, then we have the following:

\begin{enumerate}
    \item At least one of $\delta^V_U \subset C^W_U$ or $\delta^W_U \subset C^V_U$ holds.
    \item If $\delta^V_U$ is not a subset of $C^W_U$, then $\delta^V_U \cap C^W_U = \emptyset$, and similarly for $\delta^W_U$.
\end{enumerate}
\end{lemma}

\begin{proof}

Item (2) is an immediate consequence of the main part of the statement, and we shall observe item (1) during the course of the proof of the statement.

By definition of $C^V_U = \hull_C((\calN_{r}(\delta^V_U) \cup l_V)\cap C)$, we may assume that $\calN_{r}(\delta^V_U) \cap \calN_{r}(\delta^W_U) = \emptyset$, and hence that $V \pitchfork W$ by item (3) of Lemma \ref{lem:tree control}.

Let $\gamma_V$ be the marked points and rays of $T_U$ which are also the marked points and rays in the component of $T_U - \calN_{E'}(\delta^W_U)$ containing $\delta^V_U$.  Define $\gamma_W$ similarly.

A standard argument using Lemma \ref{lem:leaves choose half-space} and the BGI property in item (7) of Lemma \ref{lem:collapsed tree control} show that $\hT_V$ and $\hT_W$ are intervals with the labels of the endpoints of $\hT_V$ coinciding with $\gamma_V - \gamma_W$ and $\gamma_W$, while the labels of the endpoints of $\hT_W$ coincide with $\gamma_W - \gamma_V$ and $\gamma_V$.

Now if $a \in l_V \cap \gamma_W \subset F \cup \Lambda$, then $\hx_V = \ha_V$, and thus $l_V = \gamma_W$.  Hence the labels in $l_V$ choose the half-tree of $T_U$ containing $\delta^W_U$ precisely when $l_V \cap \gamma_W \neq \emptyset$, and similarly for $l_W$.  Note that this proves item (1).

Hence, by the Lemma \ref{lem:leaves choose half-space}, we may assume, for a contradiction, that $l_V \cap \gamma_W = l_W \cap \gamma_V = \emptyset$.

Choose labels $a \in \gamma_V - \gamma_W$ and $b \in \gamma_W - \gamma_V$.  Since $a,b$ determine distinct marked clusters or rays in all of $\hT_U, \hT_V, \hT_W$, their distances in the trees $T_U,T_V,T_W$ are at least $r$, and hence by choosing $r$ sufficiently large depending only on $\mathfrak S, |F \cup \Lambda|$, we can arrange that $d_{T_U}(a_U,b_U) > 10E'$, and similarly for $V,W$.  Item (6) of Lemma \ref{lem:tree control} then says that $d_{T_V}(\delta^U_V(\delta^W_U), \delta^W_V) < E'$, which implies that $\hd^W_V = \hb_V = q_V(\gamma_W)$, and similarly $\hd^V_W = \ha_W = q_W(\gamma_V)$.

Since $V \pitchfork W$, item (8) of Lemma \ref{lem:collapsed tree control} says that we cannot have both of $V,W \in E(\hx)$.  Hence without loss of generality, there are now two cases (1) $W \in E(\hx)$ and $V \notin E(\hx)$, and (2) when $V, W \notin E(\hx)$.

In case (1), then $l_V \neq \emptyset$, so that $a \in l_V = \gamma_V - \gamma_W$, and hence $\hx_V = \ha_V = q_V(l_V)$.  But the above says that $\hx_V \neq \hd^W_V$, while $\hx_W$ is not in a marked cluster, and this contradicts $0$-consistency of $(\hx_U)$.

In case (2), now $a \in l_V$ and $b \in l_W$, and the above forces that $\hx_V = \ha_V$ is in a different marked cluster than $\hd^W_V = \hb_V$, so that $\hx_V \neq \hd^W_V$.  Similarly $\hx_W \neq \hd^V_W$.  Together, these also contradict $0$-consistency.

This completes the proof.
\end{proof}

\begin{figure}
    \centering
    \includegraphics[width=.6\textwidth]{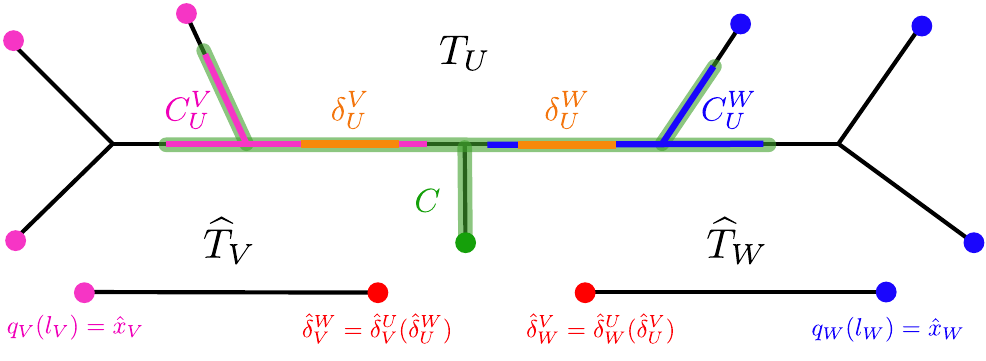}
    \caption{The impossible situation in which $V\pitchfork W \in \BM(C)$ choose non-intersecting sides $C^V_U$ (pink) and $C^W_U$ (blue) of $C$ (green).  This forces inconsistency of $\hx_V$ and $\hx_W$.} 
    \label{fig:cluster_int}
\end{figure}

The next step is to conclude that the total intersection is nonempty, which uses but does not follow directly from the Helly property for finite trees:

\begin{lemma}\label{lem:extended helly}
For each $U \in \calU$, we have
$$B_U = \bigcap_{V \in \BM(C)} C^V_U \neq \emptyset.$$
\end{lemma}

\begin{proof}
If $C$ has bounded diameter, then $C$ and all of the $C^V_U$ are finite diameter trees, and the conclusion is immediate from the Helly property for finite trees.

Suppose instead that $C$ has infinite diameter.  We will produce a finite diameter subtree $C_{fin} \subset C$ satisfying $C^V_U \cap C^W_U \cap C_{fin} \neq \emptyset$ for all $V,W \in \BM(C)$, while also $B_U \subset C_{fin}$.  It will then follow from the Helly property for trees that $B_U \neq \emptyset$. 

Let $\Lambda_C = \{\lambda \in \Lambda| U \in \supp(\lambda) \text{ and } \lambda_U \in C\}$ denote all of the rays in $\Lambda$ which determine rays in $T_U$ whose ends are contained in $C$.  For each $\lambda \in \Lambda_C$, choose $V_{\lambda} \in \BM(C)$ so that:
\begin{itemize}
\item $l_{V_{\lambda}} = (F \cup \Lambda) - \{\lambda\}$;
\item $\calN_r(\delta^{V_{\lambda'}}_U) \subset C^{V_{\lambda}}_U$ for all $\lambda \neq \lambda' \in \Lambda(C)$
\item If $W \in \BM(C)$ so that some point in $\delta^W_U$ separates $\delta^{V_{\lambda}}_U$ from $\lambda_U$ in $T_U$, then $l_W = l_{V_{\lambda}} = (F \cup \Lambda) - \{\lambda\}$.
\end{itemize}

Note that such a $V_{\lambda}$ exists for each $\lambda \in \Lambda_C$ because of the canonical property in Definition \ref{defn:Q consistent} for the tuple $(\hx_U)$, since otherwise there would be an infinite sequence $\{V_i\} \subset \BM(C)$ with the sets $\delta^{V_i}_U$ running out the end of $T_U$ corresponding to $\lambda_U$ with $\hx_{V_i} = \hlam_{V_i}$ for all $i$, thereby violating the finiteness condition.  Hence we may choose such a $V_{\lambda}$ by running far out the end of $C$ corresponding to $\lambda_U$ for each $\lambda \in \Lambda(C)$.

Set $C_{fin} = \bigcap_{\lambda \in \Lambda_C} C^{V_{\lambda}}_U$, see Figure \ref{fig:inf_Helly}.  Observe that $B_U \subset C_{fin}$.  It suffices to prove the following claim:

\begin{claim}\label{claim:C_fin}
For all $V, W \in \BM(C)$, we have $C^V_U \cap C^W_U \cap C_{fin} \neq \emptyset$.
\end{claim}

\begin{figure}
    \centering
    \includegraphics[width=.4\textwidth]{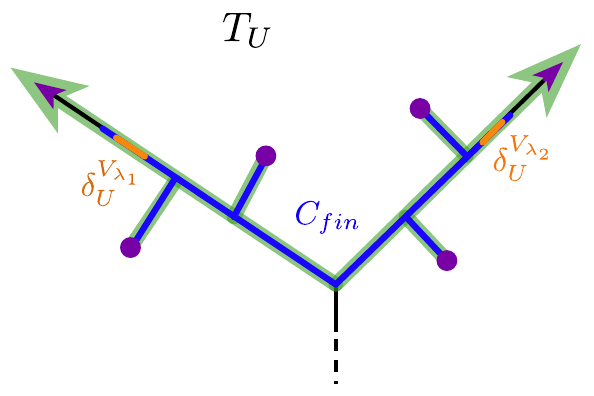}
    \caption{The proof of Lemma \ref{lem:extended helly}: For each of the finitely-many $\lambda \in \Lambda$ determining rays in $T_U$, we choose a domain $V_{\lambda}$ which is located past every branch point along that ray, then take the intersection of the $C^{V_{\lambda}}_U$ across all such $\lambda$.}
    \label{fig:inf_Helly}
\end{figure}

To see the claim, let  $E(\Lambda, C)$ denote the set of $W \in \BM(C)$ so that $\delta^W_U$ contains a point which separates $\lambda_U$ from $\delta^{V_{\lambda}}_U$ for some $\lambda \in \Lambda(C)$.

Observe the following:
\begin{enumerate}
\item Each $W \in E(\Lambda, C)$ uniquely determines such a $\lambda$, and that $l_W = l_{V_{\lambda}} = (F \cup \Lambda) - \{\lambda\}$ by construction.  Hence by convexity of the $\delta$-sets, we have that $C_{fin} \subset C^{V_{\lambda}}_U \subset C^W_U$.

\item If $W \notin E(\Lambda, C)$, then for every $\lambda \in \Lambda(C)$, some point of $\delta^{V_{\lambda}}_U$ separates $\delta^W_U$ from $\lambda$ in $T_U$.  It follows from the definition of $C^{V_{\lambda}}_U$ that $N_r(\delta^W_U) \subset C^{V_{\lambda}}_U$ for all $\lambda$, and hence that $N_r(\delta^W_U) \subset C_{fin}$.
\end{enumerate}

Following these observations, we can now argue by cases.  If $V,W \in E(\Lambda, C)$, then observation (1) shows that $C_{fin} \subset C^V_U \cap C^W_U$, and we are done.  If instead $V \in E(\Lambda, C)$ while $W \notin E(\Lambda,C)$, then combining both observations provides that $\delta^W_U \subset C_{fin} \subset C^V_U$, and we are done.

Finally, if $V,W \notin E(\Lambda, C)$, then one of two things holds.  The first possibility is that $d_{T_U}(\delta^V_U,\delta^W_U)>r$, in which case Lemma \ref{lem:bdd int pair} says, with loss of generality, that $\delta^V_U \subset C^W_U$ while the second observation says that $\delta^V_U \subset C_{fin}$, and hence $\delta^V_U \subset C^V_U\cap C^W_U \cap C_{fin}$, as required.   Otherwise, $d_{T_U}(\delta^V_U,\delta^W_U)<r$, and we may also assume that neither $\delta^V_U \subset C^W_U$ nor $\delta^W_U \subset C^V_U$.  But then $C^W_U \cap C^V_U \subset \calN_r(\delta^V_U) \cap \calN_r(\delta^W_U)$, which is contained in $C_{fin}$ by the second observation above.  This completes the proof.
\end{proof}

Finally, we prove that $B_U = \bigcap_{V \in \BM(C)} C^V_U$ has bounded diameter:

\begin{lemma}\label{lem:bdd int total}
There exists $B_1 = B_1(\mathfrak S, |F\cup \Lambda|)>0$ such that for all $U \in \calU$, we have 
$\diam_{T_U}(B_U) < B_1$.
\end{lemma}

\begin{proof}
By Lemma \ref{lem:BM dense}, there exists $r' = r'(\mathfrak S, |F \cup \Lambda|)>0$ so that either $\BM(C) = \emptyset$ and $\diam_{T_U}(C)<2r'$, or for there exists a collection $\calV \subset \BM(C)$ so that $\{\delta^V_U| V \in \calV\}$ forms a $(2r',r')$-net on $C$ (Definition \ref{defn:net}).  Since we are done in the former case, we may assume the latter holds.

We may further suppose that $B_U$ contains an edge of $T_U$ with diameter more than $8r'$ that does not contain any marked or branch points, for we are done if it does not.  Then the net provides distinct $V,W \in \calV$ with $r'< d_{T_U}(\delta^V_U, \delta^{W}_U)< 2r'$ with $\delta^V_U, \delta^W_U \subset B_U$.

By Lemma \ref{lem:bdd int pair}, we have that either $\delta^V_U \subset C^{W}_U$ or $\delta^W_U \subset C^{V}_U$, and possibly both.  Supposing only the first containment holds, then in fact $\delta^{W}_U \cap C^V_U = \emptyset$ by item (1) of Lemma \ref{lem:bdd int pair}, and hence $\delta^{W}_U$ is not a subset of $B_U$, which contradicts our assumption that $\delta^W_U \subset B_U$.

Hence both containments must hold.  Then and moreover $C^V_U \cap C^W_U \subset \hull_{C}(\calN_{r}(\delta^V_U) \cup \calN_{r}(\delta^W_U))$, which has diameter bounded by $2r' + 2r$ since $d_{T_U}(\delta^V_U, \delta^W_U) < 2r'$.

Either way, the diameter of $B_U$ is bounded and we are done.
\end{proof}

\subsection{Lemmas supporting consistency of $\Omega(\hx)$}

We record the following lemmas which will be useful in the proof of consistency in Proposition \ref{prop:Q-consistent} below.

In the following, if $A \subset F \cup \Lambda$, then we let $A_U$ denote the corresponding collection of marked points and rays of $T_U$.  If in addition $C$ is a cluster, then we let $A_U(C)$ be the collection of points and rays obtained by closest-point projecting $A_U$ to $C$ in $T_U$.  The first lemma is used in the arguments for transverse consistency, and the second lemma for nested consistency.

\begin{lemma}\label{lem:pointing at leaf}
Suppose $\hx_U$ is contained in a cluster $C$ in $\hT_U$.  If $A = \bigcap_{V \in \BM(C)} l_V \neq \emptyset$, then $\hull_C(A_U(C)) \subset B_U$.  In particular, $\diam_{T_U}(\hull_C(A_U(C)) < B_1$.
\end{lemma}

\begin{proof}

By assumption, $A \subset l_V$ for all $V \in \BM(C)$.  Moreover, since $A \neq \emptyset$, we have that $\hull_C(A_U(C)) \subset C^V_U$ for all $V \in \BM(C)$, and hence $\hull_C(A_U(C)) \subset B_U$, which has diameter bounded by $B_1$ by Lemma \ref{lem:bdd int total}.
\end{proof}

\begin{lemma}\label{lem:cluster choice contain}
Let $\hx \in \calQ$, and $W,U \nest V \in \calU$ with $d_{T_V}(\delta^W_V, \delta^U_V)>4E'$.  If $C\subset T_V$ is the cluster with $\delta^U_V \subset C$ and $q_V(C) = \hx_V$, and $W \in \BM(C)$ with $\hx_W = \hd^U_W$, then $\delta^U_V \subset C^W_V$.
\end{lemma}

\begin{proof}
Since $W \in \BM(C)$, Lemma \ref{lem:leaves choose half-space} gives that $\hT_W$ is an interval with two endpoints labeled by the marked points and rays $l_W$ of $T_V$ in the two complementary components of $T_V - \calN_{E'}(\delta^V_U)$.  By assumption, $\delta^U_V$ is contained in one of those complementary components, and item (6) of Lemma \ref{lem:tree control} gives that $\hd^V_W(\hd^U_V) = \hd^U_W$.  Since $\hx_W = \hd^U_W$ by assumption, it follows from item (5) of Lemma \ref{lem:tree control} that $\delta^U_V$ is contained in the same component $C'$ of $T_V - \calN_{E'}(\delta^W_V)$ chosen by $\hx_W$.  It follows that $\delta^U_V \subset C^W_V = \hull_{C}\left(C \cap \calN_r(\delta^W_U)) \cup l_W\right)$, as required.
\end{proof}

\subsection{Proof of Proposition \ref{prop:Q-consistent}}

Our goal is to prove that there is a consistency constant $\alpha_{\omega} = \alpha_{\omega}(\mathfrak S, |F \cup \Lambda|)>0$ so that $\Omega(\calQ) \subset \calZ_{\alpha_{\omega}}$.  The ``moreover'' part of the statement, namely $q_U(x_U) = \hx_U$ for each $U \in \calU$, where $\hx \in \calQ$ and $\Omega(\hx) = (x_U)$, is the conclusion of Lemma \ref{lem:B_U in cluster}.

We check the consistency inequalities (Definition \ref{defn:tree consistency}) by a case-wise analysis.  Let $U, V \in \calU$.

\medskip

\underline{(1) $U \pitchfork V$ and $\hx_U \neq \hd^V_U$}:  Then we must have that $\hx_V = \hd^U_V$ by $0$-consistency.  Let $C$ denote the marked point cluster of $T_V$ containing $\delta^U_V$, so that $B_V \subset C$ by construction.  We will prove that $B_V$ and $\delta^U_V$ are at a bounded distance.

If $V$ is $\nest$-minimal in $\calU$, then $C$ is a marked point cluster with diameter bounded in terms of $\mathfrak S, |F \cup \Lambda|$, so we may assume otherwise.  Moreover, we may assume that $\BM(C) \neq \emptyset$ by Lemma \ref{lem:BM dense}, otherwise again $C$ has bounded diameter.

Let $a \in F \cup \Lambda$ be some marked point or ray with $d_{T_V}(a_V, \delta^U_V)<E'$, which exists by item (1) of Lemma \ref{lem:tree control}.  Note that since $\delta^V_U \subset C$, we have $a_V \in C$ by construction of the clusters.

If $W \in \BM(C)$, then $d_{T_V}(\delta^W_V, a_V) > \chi > 5E'+ \alpha_0$, by choosing the bipartite threshold $\chi = \chi(\mathfrak S, |F \cup \Lambda|)>5E' + \alpha_0$.  So $d_{T_V}(\delta^W_V, \delta^U_V) > 3E'$, and thus $W \pitchfork U$.  Since $W \nest V$, we have $d_{T_U}(\delta^W_U, \delta^V_U)<E'$, so $\hx_U \neq \hd^W_U = \hd^V_U$, forcing that $\hx_W = \hd^U_W$ by $0$-consistency of $(\hx_U)$.

We claim that $a \in l_W$.  The reason for this is that $\hT_W$ is an interval whose endpoints partition $F \cup \Lambda$ into two sets by Lemma \ref{lem:leaves choose half-space}, where $l_W$ labels the points in $F \cup \Lambda$ coinciding with $\hx_W$.  As such, if $a \notin l_W$ and hence $\ha_W \neq \hx_W$, then $d_{T_W}(a_W, \delta^V_W(a_V))>r > \alpha_0$, while already $d_{T_V}(\delta^W_V, a_V)>5E' + \alpha_0$, which contradicts the fact that $(a_U)$ is an $\alpha_0$-consistent tuple by Lemma \ref{lem:consist to tree}.  Thus we must have $a \in l_W$, and so $\ha_W = \hd^U_W = \hx_W$.  Hence $a \in l_W$ for all $W \in \BM(C)$.  We are now done by Lemma \ref{lem:pointing at leaf}, which implies that $a_V \in B_V$, while $d_{T_V}(\delta^U_V,a_V)<E'$ and $\diam_{T_V}(B_V) < B_1 = B_1(\mathfrak S, |F \cup \Lambda)$ by Lemma \ref{lem:bdd int total}.
\medskip

\underline{(2) $U \pitchfork V$ with both $\hx_U = \hd^V_U$ and $\hx_V = \hd^U_V$}: Suppose instead that $\hx_U$ is contained in the cluster $C$ of $T_U$ containing $\delta^V_U$.  Moreover, assume that $B_U$ (and hence $x_U$) is at least $10r'\cdot B_1\alpha_0$ away from any projection of a marked point or ray of $T_U$ to $C$, where $B_1$ is as in Lemma \ref{lem:bdd int total} and $r'$ is from Lemma \ref{lem:BM dense}.  We will bound the distance between $B_V$ (and hence $x_V$) and $\delta^U_V$ by producing a domain $Z \nest U$ with $Z \pitchfork V$ and $\hx_Z \neq \hd^V_Z$, thereby allowing us to apply the previous case.

\begin{figure}
    \centering
    \includegraphics[width = .8\textwidth]{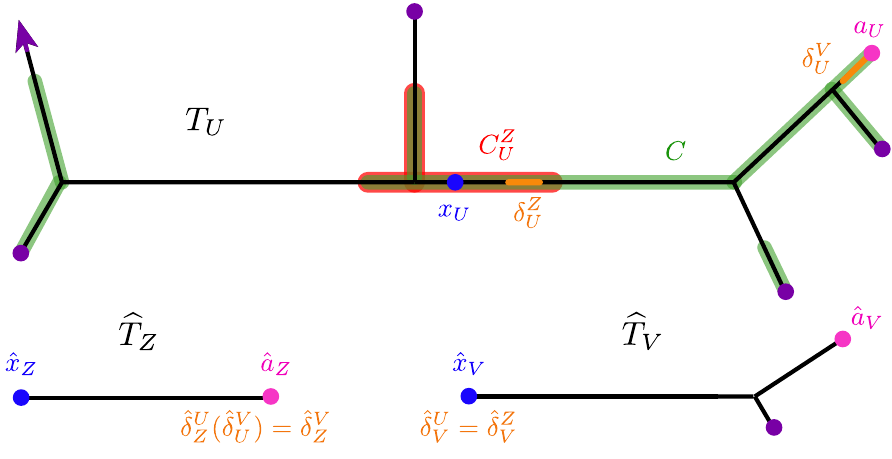}
    \caption{The proof of case (2) of Proposition \ref{prop:Q-consistent}: By assuming that $x_U$ is far from $\delta^V_U$, we can produce a bipartite domain $Z \in \BM(C)$ which chooses the half-tree not containing $\delta^V_U$, thereby forcing $\hx_Z \neq \hd^V_Z$, allowing us to invoke case (1) with $V,Z$ playing the roles of $U,V$, respectively.}
    \label{fig:Q-consistency transverse}
\end{figure}

Let $A\subset F \cup \Lambda$ be the set of marked points and rays $a$ satisfying $d_{T_U}(\delta^V_U, a_U) < E'$, with $A \neq \emptyset$ by item (1) of Lemma \ref{lem:tree control}.  Let $C'  = \hull_C(B_U, A_U) \subset C$.  By assumption $\diam_{T_U}(C') > 10r'B_1$, so Lemma \ref{lem:BM dense} provides a bipartite domain $Z \in \BM(C')$ so that $d_{T_U}(\delta^Z_U, B_U) > r'$ and $d_{T_U}(\delta^Z_U, A_U) > r'$, and so that $\delta^Z_U$ separates $B_U$ from $a_U$ for all $a \in A$.  Since $B_U \subset C^V_U$ by definition, it follows that $l_Z \cap A = \emptyset$, and hence $\hx_Z \neq \ha_Z$ for all $a \in A$.

Moreover, item (6) of Lemma \ref{lem:tree control} implies that $d_{T_Z}(\delta^V_Z, \delta^U_Z(\delta^V_U))<E'$ while item (5) of that lemma implies $d_{T_Z}(\delta^U_Z(\delta^V_U), \delta^U_Z(a_U)) < E'$.  On the other hand, $\alpha_0$-consistency of $(a_U)$ and the assumption that $d_{T_U}(\delta^Z_U, a_U)> 10r'B_1\alpha_0 > \alpha_0$ implies that $d_{T_Z}(a_Z, \delta^V_Z)<\alpha_0$.  Hence we have $d_{T_Z}(\delta^V_Z, a_Z) < 2E' + \alpha_0 < R$, and it follows that $\hd^V_Z = \ha_Z$.  In particular, we have $\hx_Z \neq \hd^V_Z$.

Finally, since $\hx_V = \hd^U_V$ by assumption, we let $C''$ be the marked point cluster of $T_V$ containing $\delta^U_V$, and hence also $\delta^Z_U \subset C''$ since $Z \nest U$.  But now we are in the situation from the first case, where $V,Z$ play the roles of $U,V$ there, since $Z \pitchfork V$ and $\hd_Z \neq \hx_Z$.  The conclusion of that case was that $d_{T_V}(B_V, \delta^Z_V)$ was bounded distance, while we have $d_{T_V}(\delta^Z_V, \delta^U_V)<E'$ since $Z \nest U$, and we are done.

\medskip
\underline{(3) $U \nest V$ and $\hx_U \neq \hd^V_U(\hx_V)$}: By $0$-consistency, we have $\hx_V = \hd^U_V$.  We will prove that $C^W_V \cap \delta^U_V \neq \emptyset$ for all $W \in \BM(C)$.  It will follow then that $d_{T_V}(x_V, \delta^U_V) < 4r$.

Let $C$ denote the cluster of $T_V$ containing $\delta^U_V$ and $x_V$.  First observe that if $W \in \BM(C)$ and $d_{T_V}(\delta^W_V, \delta^U_V) < 4E'$, then $\delta^U_V \cap C^W_V \neq \emptyset$ by definition of $C^W_V$.

On the other hand, if $W \in \BM(C)$ with $4E' < d_{T_V}(\delta^W_V, \delta^U_V)$, then $W \pitchfork U$.  Then item (6) of Lemma \ref{lem:tree control} gives that $d_{T_W}(\delta^U_W, \delta^V_W(\delta^U_V))<3E'$.  Hence $\hd^W_U = \hd^V_U(\hd^W_V) = \hd^V_U(\hx_V) \neq \hx_U$, and so since $\hx_U \neq \hd^W_U$, $0$-consistency implies that $\hx_W = \hd^U_W$.  Lemma \ref{lem:cluster choice contain} gives that $\delta^U_V \subset C^W_V$.  Consequently, we have $C^W_V \cap \delta^U_V \neq \emptyset$ for all $W \in \BM(C)$, and it follows that $d_{T_V}(x_V, \delta^U_V) < 4r$.

\medskip

\underline{(4) $U \nest V$ and $\hx_V \neq \hd^U_V$}:  In this final case, $0$-consistency provides $\hd^V_U(\hx_V) = \hx_U$.  We will bound $\diam_{T_U}(x_U \cup \delta^V_U(x_V))$.

We may assume that $d_{T_V}(x_V, \delta^U_V)> \alpha_0 + E'$, otherwise we are done.  Let $a \in F \cup \Lambda$ be such that $a_V$ is contained in the same component of $T_V - \calN_{\alpha_0+E'}(\delta^U_V)$ as $x_V$.  By the BGI property in item (5) of Lemma \ref{lem:tree control}, we get $d_{T_U}(\delta^V_U(x_V), \delta^V_U(a_V))< E'$.  But since $d_{T_V}(a_V, \delta^U_V)>\alpha_0$ and $(a_U)$ is $\alpha_0$-consistent by Lemma \ref{lem:consist to tree}, we get $d_{T_U}(a_U, \delta^V_U(a_V))<\alpha_0$.  Hence $d_{T_U}(\delta^V_U(x_V), a_U) < \alpha_0 + E'$.

Since $\hx_U = \hd^V_U(\hx_V)$ by assumption, we get that $a_U, x_U$ are in the same cluster $C$ in $T_U$.  We may assume, for a contradiction, that there is some $W \in \BM(C)$ such that $a_U, x_U$ are in separate components of $T_U - \calN_{E'}(\delta^W_U)$, for otherwise we are done.

Since $W \in \BM(C)$, Lemma \ref{lem:leaves choose half-space} says that $\hT_W$ is an interval with two endpoints labeled by the marked points and rays in $T_U$ in the two complementary components of $T_U - \calN_{E'}(\delta^W_U)$.  Since $\delta^W_U$ separates $a_U$ from $x_U$ in $T_U$, it follows that $\ha_W, \hx_W$ are not in the same endpoint cluster.  Moreover, since $x_U \in C^W_U$ by definition of $B_U$, we have $\hx_W \neq \ha_W$.

But this contradicts $0$-consistency of $\hx$, since $\hx_V \neq \hd^U_V = \hd^W_V$ (because $W \nest U \nest V$) while $\hx_W \neq \ha_W = \hd^V_W(\hx_V)$, with the latter equality holding by a similar argument for why $\hd^V_U(\hx_V) = \ha_U$.  This completes the proof of this case, and hence of the proposition.

\begin{figure}
    \centering
    \includegraphics[width=.8\textwidth]{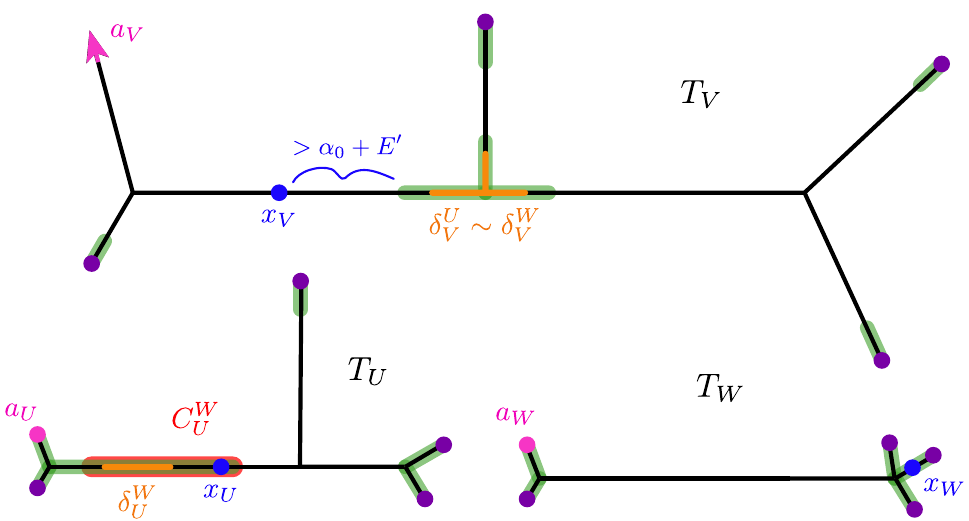}
    \caption{The proof of case (4) of Proposition \ref{prop:Q-consistent}: If $x_V$ (blue) is far enough away from $\delta^U_V$ (orange), then we can find a marked point or ray $a \in F \cup \Lambda$ (pink) so that $a_U \sim \delta^V_U(x_V)$.  If also the cluster $C \subset T_U$ containing $x_U$ (and $a_U$) is wide enough, we can find a domain $W \in \BM(C)$ so that $\delta^W_U$ separates $x_U$ from $a_U$, forcing $a_W$ and $x_W$ to be at opposite ends of $T_W$, because $C^W_U$ (red) has to contain $x_U$ by definition of $B_U$. }
    \label{fig:Q-consistency}
\end{figure}

\section{Tree trimming} \label{sec:tree trimming}

In this section, we give a general construction that allows us to trim the trees $\hT_U$ in various ways while preserving the metric on $\calQ$ up to quasi-isometry.  Applications of these trimming techniques include
\begin{itemize}
\item surjectivity of $\hPsi:H \to \calQ$ (Theorem \ref{thm:coarse inverse}),
\item the ability to convert all of the $\hT_U$ to simplicial trees while preserving $\calQ$ up to quasi-isometry (Corollary \ref{cor:simplicial structure}), which is crucial for proving that $\calQ$ is quasi-median quasi-isometric to a cube complex (Theorem \ref{thmi:main model}), and
\item modifying the cubical model construction to encode curve graph distance in their hyperplanes (Theorem \ref{thmi:curve graph}).
\end{itemize}

As we will see in the Sections \ref{sec:walls in Q} and \ref{sec:Q is cubical}, the hierarchical geometry of $H$ is directly encoded in the cubical geometry of $\calQ$, and so tree trimming will also allow to directly alter the cubical structure of $\calQ$.

Roughly speaking, the proposition says that if we collapse boundedly-many subtrees of bounded diameter in each component of $\hT^e_U$ for each $U \in \calU$, then the resulting collapsing map induces a quasi-isometry from $\hY$ to the resulting product of trees, and this map restricts to a map between the $0$-consistent sets which preserves (non-)canonicality.  First we need some notation.

Let $K$ be our largeness threshold for the construction (Subsection \ref{subsec:K}) and suppose that $B = B(\mathfrak S, |F\cup \Lambda|)>0$ is some other constant with respect to which we can make $K$ as large as we like (as is the case in applications).

Suppose for each $U \in \calU$ and each component $E \subset \hT^e_U$ that $A_1, \dots, A_n \subset E$ is a pairwise disjoint collection of subtrees, with $\diam_{E}(A_i) < B$ and $n < B$ for each $i$.  Let $\hT'_U$ be the tree obtained by collapsing each of the $A_i$ to a point across all components $E \subset \hT^e_U$, and let $\Delta_U: \hT_U \to \hT'_U$ be the quotient map.  We think of this process as ``trimming'' the trees $\hT_U$, even though the collapsed sets can be in the interior of edges of the trees.

As with the original quotients $q_U:T_U \to \hT_U$, we get induced relative projections $(\hd')^V_U$ as follows:
\begin{itemize}
    \item If $V \nest U$ or $V \pitchfork U$, then $(\hd')^V_U = \Delta_U(\hd^V_U)$.
    \item If $U \nest V$, then $(\hd')^V_U = \Delta_U \circ \hd^V_U \circ \Delta_V^{-1}: \hT'_V \to \hT'_U$.
\end{itemize}

We get a notion of $0$-consistency relative to these projections in the product $\oY' = \prod_{U \in \calU} \hT'_U \cup \partial \hT'_U$, and can define $\oQ' \subset \oY'$ to be the set of  $0$-consistent tuples in the sense of Definition \ref{defn:Q consistent}.  Moreover, it is clear that if $\Delta:\oY \to \oY'$ is the corresponding quotient map, then the restriction $\Delta|_{\oQ}: \oQ \to \oY'$ satisfies $\Delta(\oQ) \subset \oQ'$, for trimming a tree will not change the fact that a coordinate coincides with a relative projection.  With these definitions in place, we observe that the conclusions of the Collapsed Tree Control Lemma \ref{lem:collapsed tree control} hold for this trimmed setup.

For each $a \in F$ and $U \in \calU$, we set $\ha'_U = \Delta_U(\ha_U)$.  For $\lambda \in \Lambda$, we can similarly define $\hlam'_U =\Delta_U(\hlam_U)$ for all $U \in \calU$, but the following remark is a word of caution:

\begin{remark}[Collapsing rays]\label{rem:collapsing rays}
Notice that when $\Lambda \neq \emptyset$, then the component quotient maps need not extend to bijections $\partial \Delta_U: \partial \hT_U \to \partial \hT'_U$.  This can occur when, for some $\lambda \in \Lambda$ with $U \in \supp(\lambda)$, any ray in $\hT_U$ representing $\hlam_U \in \partial \hT_U$ is contains an infinite sequence of cluster points with consecutive pairs separated by an edge (in $\hT^e_U$) of diameter less than $B^2$.  As a result, any such ray can be collapsed to a bounded diameter set.  In particular, it follows that there is some infinite collection of domains $\calV \subset \calU$ with $V \nest U$ for all $V \in \calV$, and so that $(\hd^V_U)' = (\hd^W_U)'$ for all $V, W \in \calV$.  In particular, $\hlam'_U = (\hd^V_U)'$ for all $V \in \calV$.  As in the example in Subsection \ref{subsec:ray example} from the introduction, the set $\{\hd^V_U| V \in \calV\}$ represents $\hlam_U \in \partial \hT_U$.
\end{remark}

We can now state our tree trimming theorem:

\begin{theorem}[Tree Trimming] \label{thm:tree trimming}

Using the above notation, there exists $B_{tt} = B_{tt}(\mathfrak S, |F \cup \Lambda|, K)>0$, so that for any $B < B_{tt}$ there exists $L = L(\mathfrak S, |F \cup \Lambda|, B)>0$ so that the quotient map $\Delta: \oY \to \oY'$ satisfies the following:

\begin{enumerate}
\item $\Delta(\oQ) = \oQ'$, $\Delta(\calQ) = \calQ'$, and $\Delta(\calQ^{\infty}) = (\calQ^{\infty})'$.
\item The restriction $\Delta|_{\calQ}:\calQ \to \calQ'$ is an $(L,L)$-quasi-isometry.
\item If the number of domains $U \in \calU$ with collapsed subtrees is bounded by $B'$, then $\Delta|_{\calQ}:\calQ \to \calQ'$ is a $(1,BB')$-quasi-isometry.
\end{enumerate}
\begin{itemize}
\item  Moreover, the map $\Delta|_{\calQ}:\calQ \to \calQ'$ is $0$-median.
\end{itemize}
\end{theorem}

\begin{proof}

The median ``moreover'' is proven in Corollary \ref{cor:tt median} below, see that section for the relevant definitions.

We begin by proving item (1).  As noted before the statement, the fact that $\Delta$ preserves consistency follows from the fact that we are only collapsing subtrees, so coordinates which were relative projections remain relative projections.  Hence $\Delta(\oQ) \subset \oQ'$.

\vspace{.1in}

\textbf{\underline{(Non-)canonicality is preserved}}: To see that $\Delta(\calQ) \subset \Delta(\calQ')$, we confirm canonicality (as in Definition \ref{defn:Q consistent}) is also preserved.  Suppose that $\hx \in \calQ$.  Then $\hx_U \notin \partial \hT_U$ for all $U$, and it follows that $\Delta_U(\hx_U) \notin \partial \hT'_U$, since interior points of $\hT_U$ cannot become boundary points of $\hT'_U$ by collapsing subtrees.  On the other hand, since $\#\{U| \hx_U \neq \ha_U\}< \infty$ for every $a \in F$ and we are only collapsing subtrees, it follows that the same property holds for $\hx'$.

To see that $\Delta(\calQ^{\infty}) \subset (\calQ^{\infty})'$, let $\hx \in \calQ^{\infty}$ be non-canonical.  Then either (a) there exists $U \in \calU$ so that $\hx_U \in \partial \hT_U$, or (b) there exists an infinite collection of domains $\calV_U$ and $a \in F$ so that $\hx_V \neq \ha_V$ for all $V \in \calV_U$.

In case (a), set  $\hx' = \Delta(\hx)$ and suppose that $\hx'_U \notin \partial \hT'_U$, for we are done otherwise.  Hence we are in the situation discussed before the statement of the proposition.  Let $\lambda \in \Lambda$ be so that $\hx_U = \hlam_U \in \partial \hT_U$.  Since the map $\Delta_U:\hT_U \to \hT'_U$ collapses at most $B$-many subtrees of diameter at most $B$ in each component of $\hT^e_U$, it follows that there exists an infinite sequence of $\nest_{\calU}$-minimal bipartite domains $V_i \nest U$ so that 
\begin{itemize}
\item the sequence $\hd^{V_n}_U$ monotonically escapes out the end of $\hT_U$ corresponding to $\hlam_U$,
\item each consecutive pair $\hd^{V_i}_U, \hd^{V_{i+1}}_U$ satisfies $d_{\hT_U}(\hd^{V_i}_U, \hd^{V_{i+1}}_U)< B^2$,
\item $(\hd^{V_i}_U)' = \hx'_U$ for all $i$, and
\item $\hlam_{V_i} = \hx_{V_i}$ for all $i$, with $\lambda$ the unique element of $F \cup \Lambda$ labeling $\hx_{V_i}$ for all $i$.
\end{itemize}

As a consequence of the last item, we see that $\hx'_{V_i} \neq \ha'_{V_i}$ for any $a \in F$, and hence $\hx'$ is not canonical, as required.

For the rest of item (1), it remains to prove that $\Delta|_{\calQ}:\calQ \to \calQ'$ and $\Delta|_{\calQ^{\infty}}: \calQ^{\infty} \to (\calQ^{\infty})'$ are surjective.

\vspace{.1in}

\textbf{\underline{Surjectivity of $\Delta$}}:  Let $\hy \in \oQ'$.  We will build $\hx \in \oQ$ with $\Delta(\hx) = \hy$ by a domain-wise argument, by a construction closely related to the construction of $\hO:\calQ \to \calY$ in Section \ref{sec:Q to Y}.  As in that construction, the idea is to pick out a point of $C_U \subset \hT_U$ by using $\nest_{\calU}$-minimal domains, and then show that the corresponding tuple is $0$-consistent.  Set $C_U = \Delta^{-1}_U(\hy_U)$.

We begin by dealing with two easy cases:
\begin{itemize}
\item If $U \in \calU$ be such that $\hy_U = \hlam'_U \in \partial \hT'_U$, then we set $\hx_U = \hlam_U \in \partial \hT_U$.
\item If $\hy_U \in \hT'_U$ is not a marked or cluster point, then $C_U$ contains no marked or cluster points, and we let $\hx_U$ be any point therein.
\end{itemize}

The general case is when $\hy_U \in \hT'_U$ is a marked or cluster point, and we construct $\hx_U$ by induction on $\nest$-level in $\calU$.

The base case is when $U$ is $\nest_{\calU}$-minimal, and hence $\hy_U$ is a marked point of $\hT'_U$ (since $\hT'_U$ has no cluster points).  Note that $\diam_{\hT_U}(C_U)<B^2$, but it is possible that $\Delta_U:\hT_U \to \hT'_U$ identifies distinct marked points of $\hT_U$.  We need to choose one of them.

\begin{claim}\label{claim:choose one}
If $W,Z \in \{V \in \calU| V \pitchfork U, \mathrm{ and } \hspace{.1in} \hy_V \neq (\hd^U_V)'\}$, then $\hd^W_U = \hd^Z_U$.
\end{claim}

\begin{proof}[Proof of Claim \ref{claim:choose one}]
Suppose instead that $\hd^W_U \neq \hd^Z_U$.  Then $W \pitchfork Z$, and Lemma \ref{lem:cluster same} (which appears after the proof of the theorem) implies that $\hd^U_W = \hd^Z_W$ and $\hd^U_Z = \hd^{W}_Z$.  But since $\hy_W\neq \hd^U_W = \hd^Z_W$ and $\hy_Z \neq \hd^{U}_Z = \hd^W_Z$, we get a contradiction of $0$-consistency of $\hy$.
\end{proof}

Observe that we did not require that $U$ was $\nest_{\calU}$-minimal for the above claim, which will be useful later.

What Claim \ref{claim:choose one} allows us to do is define $\hx_U$ for $\nest_{\calU}$-minimal $U$ as follows:

\begin{itemize}
\item If $\calW_U^{\pitchfork} = \{V \in \calU| V \pitchfork U \hspace{.1in} \mathrm{ and } \hspace{.1in} \hy_W \neq (\hd^V_U)'\} \neq \emptyset$, then we set $\hx_U = \hd^W_U$ for any (and hence all, by the claim) $W \in \calW_U^{\pitchfork}$.
\item If $\calW_U^{\pitchfork} = \emptyset$, then we let $\hx_U$ be any marked point in $C_U$.
\end{itemize}

At this point, we also want to define $l_U \in F \cup \Lambda$ to be the collection of elements $a \in F \cup \Lambda$ so that $\hx_U = \ha_U$.  We set $l_U = \emptyset$ when $\hy_U$ (and hence $\hx_U$) is not a marked point.  These labels will play a similar role as they did in the construction of $\hO: \calQ \to H$, see especially Subsection \ref{subsec:honing clusters}.

For the inductive step, suppose that $U$ is not $\nest_{\calU}$-minimal and that we have defined $\hx_V$ for all $V \nest U$ and, in particular, for all $\nest_{\calU}$-minimal $V \in \calU$.

Now if $V \nest U$ and $\hx_V = \hlam_V \in \partial \hT_V$ for some $\lambda \in \Lambda$, then we set $\hx_U = \hd^V_U$.  Similarly, if there exists $V \nest U$ so that $\hx_V \in \hT_V$ is not a marked, ray, or cluster point, then we set $\hx_U = \hd^V_U$.  Note that in these cases, we either have that $\hy_V \in \partial \hT_V'$ or $\hy_V$ is not a marked point, ray, or cluster point, and so $\hy_U = (\hd^V_U)'$ by $0$-consistency of $\hy$, meaning that $\hd^V_U \in C_U$, so that we have $\hx_U \in C_U$, as required.

In the general case, $\hx_V \in \hT_V$ is either a marked or cluster point (and not a ray) for every $V \nest U$.  Set $C = C_U$ and let $\calV_U$ denote the set of $\nest_{\calU}$-minimal domains $V \nest U$ so that $\hd^V_U \in C$ (or equivalently, $(\hd^V_U)' = \hy_U$).  For each $V \in \calV_U$, set $C^V_U = \hull_C(\hd^V_U \cup \pi_C(l_V))$, namely the convex hull in $C$ of $\hd^V_U$ and the projections of the marked points or rays labeled by the elements of $l_V$ down to the convex subset $C$.  The following claim is an ``exact'' analogous version of Lemma \ref{lem:bdd int pair}:

\begin{claim}\label{claim:cluster overlap}
For every $V, W \in \calV_U$, we have $C^V_U \cap C^W_U \neq \emptyset$.  Moreover:
\begin{enumerate}
\item At least one of $\hd^V_U \in C^W_U$ or $\hd^W_U \in C^V_U$ holds.
\item If $\hd^V_U \in C^W_U$ but $\hd^W_U \notin C^V_U$, then $l_V \subset l_W$ and $C^V_U \subset C^W_U$ (and hence $C^V_U = C^V_U \cap C^W_U$).

\end{enumerate}
\end{claim}

\begin{proof}[Proof of Claim \ref{claim:cluster overlap}]
We may assume that $\hd^V_U \neq \hd^W_U$, and hence that $V \pitchfork W$.  If $C^V_U \cap C^W_U = \emptyset$, then $\hd^V_U \notin C^W_U$ and $\hd^W_U \notin C^V_U$, and it follows that $l_V \cap l_W = \emptyset$.  The BGI property of Lemma \ref{lem:collapsed tree control} and Item (6) of Lemma \ref{lem:tree control} imply that $\hd^W_V \notin \Delta^{-1}_V(\hy_V)$ and $\hd^V_W \notin \Delta^{-1}_W(\hy_W)$, and hence $\hy_V \neq (\hd^W_V)'$ and $\hy_W \neq (\hd^V_W)'$, which is a contradiction.  This proves the main part of the claim and also item (1).

For item (2), suppose that $\hd^V_U \in C^W_U$ but $\hd^W_U \notin C^V_U$.  It follows that $\hd^V_U$ is contained in one of the components of $\hT_U - \{\hd^W_U\}$ which is chosen by $l_W$, while the components of $\hT_U-\{\hd^V_U\}$ chosen by $l_V$ do not contain $\hd^W_U$.  Suppose that $a \in l_V - l_W$.  Then $\ha_U$ is a marked point in a component of $\hT_U- \{\hd^V_U\}$ not containing $\hd^W_U$, and so $\hd^V_U$ separates $\ha_U$ from $\hd^W_U$.  Since every leaf of $\hT_U$ is a marked point, it follows that there exists some $b \in l_W$ so that $\hb_U$ is contained in the component of $\hT_U - \{\hd^W_U\}$ containing $\hd^V_U$, with $\hd^V_U$ separating $\hb_U$ from $\hd^W_U$.  It follows then that  the BGI property of Lemma \ref{lem:collapsed tree control} implies that $\ha_W = \hb_W = \hx_W$, which is a contradiction.  Hence $l_V \subset l_W$.  The fact that $C^V_U \subset C^W_U$ now follows from the facts that $l_V \subset l_W$ and $\hd^V_U \in C^W_U$, completing the proof of (2) and the lemma.

\end{proof}

With this claim in hand, we have two main subcases, whose specifics we state as the following claim, part of which is analogous to Lemma \ref{lem:extended helly}:

\begin{claim}\label{claim:cluster intersect}
The following hold:
\begin{enumerate}
\item If $\bigcap_{V \in \calV_U} l_V \neq \emptyset$, then either
	\begin{enumerate}
	\item The diameter of $C_U$ is infinite and there exists some $\lambda \in \Lambda$ so $\hlam_U \in \partial C_U$ and $\bigcap_{V \in \calV_U} l_V = \{\lambda\}$, or
	\item Otherwise, every $a \in \bigcap_{V \in \calV_U} l_V$ labels a marked point of $C_U$ (i.e., does not label an infinite end of $C$).
	\end{enumerate}
\item If $\bigcap_{V \in \calV_U} l_V = \emptyset$ and $\diam (C_U) = \infty$, then there exist $\lambda_1, \dots, \lambda_n\in \Lambda$ labeling the infinite ends of $C_U$, and $V_1, \dots, V_n \in \calV_U$, so that
\begin{enumerate}
\item $l_{V_i} = F \cup \Lambda - \{\lambda_i\}$ for each $i$,
\item If $C_{fin}= \bigcap_i C^{V_i}_U$, then $C_{fin} \neq \emptyset$ and $\diam (C_{fin})< \infty$, and
\item $C^V_U \cap C^W_U \cap C_{fin} \neq \emptyset$ for all $V,W \in \calV_U$.
\end{enumerate}
\end{enumerate}
\begin{itemize}
\item Consequently, combining the above with the Helly property for trees provides that $\bigcap_{V_U} C^V_U \neq \emptyset$ and has finite diameter.
\end{itemize}
\end{claim}

\begin{proof}[Proof of Claim \ref{claim:cluster intersect}]
Suppose first that $\bigcap_{V \in \calV_U} l_V \neq \emptyset$.  Suppose that the diameter of $C_U$ is infinite and that $\lambda \in \Lambda$ with $\hlam_U \in \partial C_U$ (by which we mean that $\hlam_U$ represents a ray in the Gromov boundary of the tree $C_U$), with $\lambda \in \bigcap_{V \in \calV_U} l_V$.  We will prove that $\bigcap_{V \in \calV_U} l_V = \{\lambda\}$.

Since we are assuming that $\hy_U \notin \partial \hT'_U$, there exists a sequence $V_n \in \calV_U$ with $\hd^{V_n}_U$ monotonically escaping out the end of $C_U$ corresponding to $\hlam_U$, as earlier in the proof of this theorem.  Taking an element $V_n$ of this sequence sufficiently far out the end corresponding to $\hlam_U$ gives that $\hd^{V_n}_U$ partitions $\hlam_U$ away from every other marked point and branched point of $\hT_U$.  As a consequence $C^{V_n}_U$ contains only $\hlam_U$, and hence $l_{V_n} = \{\lambda\}$, which forces $\bigcap_{V \in \calV_U} l_V = \{\lambda\}$, as required.

As a consequence, if (1)(a) does not hold, then the only other possibility is that every $a \in \bigcap_{V \in \calV_U} l_V$ labels a marked point of $C_U$ (i.e., $a \in F \cup \Lambda$ but $\ha_U \notin \partial \hT_U$), proving the dichotomy of items (1)(a) and (1)(b).

To deduce (2), suppose that $\bigcap_{V \in \calV_U} l_V = \emptyset$ and $\diam (C_U) = \infty$.  Let $\lambda_1, \dots, \lambda_n \in \Lambda$ label the infinite ends of $C_U$.  For each $i$, there exists a sequence of elements of $\calV_U$ whose $\hd$-sets monotonically escape out the end corresponding to $(\hlam_i)_U$.  Since $\lambda_i \notin \bigcap_{V \in \calV_U} l_V = \emptyset$, it follows that infinitely-many of the elements $V$ in this sequence must satisfy $l_{V_i} = F \cup \Lambda - \{\lambda_i\}$.  Choose such a $V_i \in \calV_U$ for each $i$.  Observe that $C_{fin} = \bigcap_i C^{V_i}_U \neq \emptyset$, as $\lambda_j \in l_{V_i}$ for all $i \neq j$, and also $\diam (C_{fin}) <\infty$, as every infinite end of $C_U$ is cut off in this intersection.

As in the proof of Claim \ref{claim:C_fin}, set $E(\Lambda, C_U)$ to be the set of $W \in \calV_U$ so that $\hd^W_U$ separates  $(\hlam_i)_U$ from $\hd^{V_i}_U$ for some $i$.  As in that proof, we have the following observations:
\begin{enumerate}
\item Each $W \in E(\Lambda, C_U)$ uniquely determines an $i$, and we necessarily have $l_W = F \cup \Lambda - \{\lambda_i\}$ and so $C_{fin} \subset C^{V_i}_U \subset C^W_U$.  \item On the other hand, if $W \notin E(\Lambda,C_U)$, then the definition of each $V_i$ implies that $\hd^W_U \in C^{V_i}_U$ for each $i$, and hence $\hd^W_U \in C_{fin}$.
\end{enumerate}

With these observations in hand, let $V,W \in \calV_U$ and we prove that $C^V_U \cap C^W_U \cap C_{fin} \neq \emptyset$ for all $V, W \in \calV_U$, which will complete the proof of the claim.  If $V, W \in E(\Lambda, C_U)$, then observation (1) implies that $C_{fin} \subset C^V_U \cap C^W_U$.  If $V \in E(\Lambda,C_U)$ while $W$ is not, then combining both observations gives $\hd^W_U \in C_{fin} \subset C^V_U$, which is what we wanted.  Finally, if $V, W \notin E(\Lambda, C_U)$, then Claim \ref{claim:cluster overlap} implies that, without loss of generality, $\hd^V_U \in C^W_U$, while observation (2) says that $\hd^V_U \in C_{fin}$, implying $\hd^V_U \in C^V_U \cap C^W_U \cap C_{fin}$.  This completes the proof of part (2) and hence of the claim.
\end{proof}

Observe that $\bigcap_{V \in \calV_U} C^V_U$ may contain many marked points and cluster points, and we want to be able to choose a unique coordinate $\hx_U$ inside of this intersection.  To do so, we need to use the other information encoded in these marked and cluster points.  Claim \ref{claim:choose one} allows us to decide between the marked points, and the following claim allows us to choose between the cluster points:

First, set $\calW_U^{\nest} = \{V \nest U| (\hd^U_V)'(\hy_U) \neq \hy_V\}$.  We need a claim which is the nesting version of Claim \ref{claim:choose one}:

\begin{claim}\label{claim:choose two}
Suppose $V, W \in \calW_U^{\nest}$, then $\hd^V_U = \hd^W_U$.
\end{claim}

\begin{proof}[Proof of Claim \ref{claim:choose two}]
Suppose for a contradiction that $\hd^W_U \neq \hd^V_U$, so that $V \pitchfork W$.  Then item (10) of Lemma \ref{lem:collapsed tree control} implies that $\hd^V_W = \hd^U_W(\hd^V_U)$ and $\hd^W_V = \hd^U_V(\hd^W_U)$.  But then $\hy_W \neq (\hd^V_W)' = (\hd^U_W)'((\hd^V_U)')$ and $\hy_V \neq (\hd^W_V)' = (\hd^U_V)'((\hd^W_U)')$, which contradicts $0$-consistency of $\hy$, completing the proof.
\end{proof}

What Claim \ref{claim:choose two} allows us to do is choose between cluster points in $\bigcap_{V \in \calV_U} C^V_U$, in the same way Claim \ref{claim:choose one} will allow us to choose between marked points.

We need one final claim, which will allow us to deal with the situation when $\bigcap_{V \in \calV_U} C^V_U$ contains both marked points and cluster points, since it says that we are forced into a consistent choice:

\begin{claim}\label{claim:good choice}
The following holds:
\begin{enumerate}
\item If there exists $V \pitchfork U$ so that $\hy_V \neq (\hd^U_V)'$, then for all $W \in \calV_U$ with $\hd^W_U \neq \hd^V_U$, we have $\hx_W = \hd^V_W$.
\item If there exists $W \in \calV_U$ with $\hy_W \neq (\hd^U_W)'(\hy_U)$, then for all $V \pitchfork U$ with $\hd^W_U \neq \hd^V_U$, we have $\hy_V = (\hd^U_V)'$.
\end{enumerate}
\begin{itemize}
\item In particular, if $V \in \calW_U^{\pitchfork}$ and $W \in \calW_U^{\nest}$, then $\hd^V_U = \hd^W_U$.
\end{itemize}
\end{claim}

\begin{proof}
Suppose $V \pitchfork U$ as in item (1), and let $W \in \calV_U$ with $\hd^W_U \neq \hd^V_U$.  Then $W \pitchfork V$, and since $W \nest U$, we have $(\hd^W_V)' = (\hd^U_V)' \neq \hy_V$, and hence $\hx_W = \hd^V_W$, by definition of $\hx_W$.

Now suppose $W \in \calV_U$, as in item (2), and that $V \pitchfork U$ with $\hd^W_U \neq \hd^V_U$. Then $W \pitchfork V$.  By $0$-consistency of $\hy$, we have that $\hy_U = (\hd^W_U)'$.  But since $\hd^V_U \neq \hd^W_U$, we have $(\hd^V_U)' \neq (\hd^W_U)'$, and so $\hy_V = (\hd^U_V)' = (\hd^W_V)'$ since $W \nest U$.  This completes the proof.
\end{proof}

We are finally ready to define our tuple $\hx$, which we do domain-wise for each $U \in \calU$:

\begin{enumerate}
\item If $\hy_U  = \hlam'_U \in \partial \hT'_U$, then we set $\hx_U = \hlam_U \in \hT_U$.
\item If $\hy_U$ is not a marked or cluster point of $\hT'_U$, then we let $\hx_U \in C_U$ be any point in $\Delta^{-1}_U(\hy_U)$.
\item For any $U$ of types (1) or (2): If $V \pitchfork U$ or $U \nest V$, we set $\hx_V = \hd^U_V$, and if $U \nest V$, we set $\hx_V = \hd^U_V(\hx_U)$.
\item If $U$ is $\nest_{\calU}$-minimal and $\hy_U$ is a marked point of $\hT'_U$, then either \label{item:minimal choice}
\begin{enumerate}
\item The set $\calW_U^{\pitchfork} = \{W \in \calU| W \pitchfork U \hspace{.1in}  \textrm{and } \hspace{.1in} (\hd^U_W)' \neq \hy_W\}$ is nonempty, and we set $\hx_U = \hd^W_U$ for any (and hence all by Claim \ref{claim:choose one}) $W \in \calW_U^{\pitchfork}$;
\item $\calW_U^{\nest} = \{V \nest U| (\hd^U_V)'(\hy_U) \neq \hy_V \hspace{.1in} \textrm{and} \hspace{.1in} (\hd^U_W)'(\hy_U) \neq \hy_W\}$ is nonempty, and we set $\hx_U = \hd^V_U$ for any (and hence all by Claim \ref{claim:choose two}) $V \in \calW_U^{\nest}$; or 
\item $\calW_U^{\pitchfork}$ and $\calW_U^{\nest}$ are both empty, and we set $\hx_U$ to be any point of $\hT^e_U$ contained in $C_U$, in particular a point which is not a cluster or marked point. 
\end{enumerate}
\item Proceeding inductively, if $U$ is not in cases (1)--(4) above, then Claim \ref{claim:cluster intersect} gives the following possibilities: \label{item:general choice}
\begin{enumerate}
\item The diameter of $C_U$ is infinite and there is some $\lambda \in \Lambda$ so that $\hlam_U \in \partial C_U$ and $\bigcap_{V \in \calV_U} l_V = \{\lambda\}$.  In this case, we set $\hx_U = \hlam_U \in \partial \hT_U$.
\item We have $\bigcap_{V \in \calV_U} C^V_U \neq \emptyset$ and has finite diameter.  Then there are three cases:
\begin{enumerate}
\item $\hx_U = \hd^W_U$ for any $W \in \calW_U^{\pitchfork}$ if $ \calW_U^{\pitchfork} \neq \emptyset$, \label{item:choose transverse}
\item $\hx_U = \hd^V_U$ for any $V \in \calW_U^{\nest}$ when $\calW_U^{\nest} \neq \emptyset$, or \label{item:choose nested}
\item Otherwise, both $\calW_U^{\pitchfork}$ and $\calW_U^{\nest}$ are empty, and we set $\hx_U$ to be any point in $\hT^e_U \cap C_U$.
\end{enumerate}
\end{enumerate}
\end{enumerate}

Note that the choice in item \eqref{item:general choice} above is well-defined by Claim \ref{claim:good choice}.

Finally, we prove that $\hx = (\hx_U)$ is $0$-consistent by a case-wise analysis, similar to the proof of Proposition \ref{prop:Q-consistent}.  Showing this will complete the proof of surjectivity of $\Delta: \oQ \to \oQ'$.

\vspace{.1in}

\underline{$U \pitchfork V$:} We may assume that $\hx_V \neq \hd^U_V$, for we are done otherwise.  Then $\hy_V \neq (\hd^U_V)'$, and so $\hy_U = (\hd^V_U)'$ by $0$-consistency of $\hy$.  It follows that $\hd^V_U \in C_U = \Delta^{-1}_U(\hy_U)$, and note that $\hd^V_U = \ha_U$ for some $a \in F \cup \Lambda$ by Lemma \ref{lem:collapsed tree control}, with $\ha_U$ not representing a ray in $\hT_U$.  If $U$ is $\nest_{\calU}$-minimal, then our choice of $\hx_U$ in \eqref{item:minimal choice} says that $\hx_U = \hd^V_U$, as required.

Now suppose that $U$ is not $\nest_{\calU}$-minimal.  Set $\calV_U$ to be the set of $\nest_{\calU}$-minimal $W \nest U$ with $\hd^W_U \in C_U$.  Let $W \in \calV_U$ be so that $\hd^W_U \neq \hd^V_U$.  Then $W \pitchfork V$ and $(\hd^W_V)' = (\hd^U_V)' \neq \hy_V$, and so $\hx_W = \hd^V_W = \ha_W$, with the first equality following from our choice in \eqref{item:minimal choice} and the last equality by $0$-consistency of $a$.  But this says that $a \in l_W$, and hence $\ha_U \in C^W_U$ for all $W \in \calV_U$ (including those for which $\hd^W_U = \hd^V_U = \ha_U$).

It follows then that $\ha_U \in \bigcap_{W \in \calV_U}C^W_U$ and so $\hx_U = \ha_U$ by definition of $\hx$ in item \eqref{item:choose transverse} above.

\vspace{.1in}

\underline{$V \nest U$:} We may assume that $\hx_V \neq \hd^U_V(\hx_U)$, for we are done otherwise.  Then $\hy_V \neq (\hd^U_V)'(\hy_U)$, and $0$-consistency of $\hy$ implies that $\hy_U = (\hd^V_U)'$, and thus $\hd^V_U \in C_U = \Delta^{-1}_U(\hy_U)$.  We claim that for all $W \in \calV_U$, we have $\hd^V_U \in C^W_U$.

To see this, is suffices to consider $W \in \calV_U$ with $\hd^W_U \neq \hd^V_U$.  Then $W \pitchfork V$.  But $\hx_V \neq \hd^U_V(\hx_U)$, so if also $\hd^V_U \notin C^W_U$, then $\hx_V \neq \hd^U_V(\hd^W_U) = \hd^W_V$, while also $\hx_W \neq \hd^U_W(\hd^V_U) = \hd^V_W$, which both use item (10) of Lemma \ref{lem:collapsed tree control}.  But then $\hy_V \neq (\hd^W_V)'$ and $\hy_W \neq (\hd^V_W)'$, which contradicts $0$-consistency of $\hy$.

Hence $\hd^V_U \in C^W_U$ for all $W \in \calV_U$.  It follows that $\hd^V_U \in \bigcap_{W \in \calV_U} C^W_U$, and then \eqref{item:choose nested} says that $\hx_U = \hd^V_U$.

\vspace{.1in}

Hence $(\hx)$ is $0$-consistent.  Since $\Delta(\hx) = \hy$ because $\hx_U \in \Delta^{-1}_U(\hy_U)$ for all $U \in \calU$, we are done with the proof of surjectivity of $\Delta: \oQ \to \oQ$.

\vspace{.1in}

\textbf{\underline{$\Delta|_{\calQ}: \calQ \to \calQ'$ is a quasi-isometry}}: Since $\Delta|_{\calQ}$ is surjective by the above, it suffices to relate distances.  For that, note that $d_{\hY}\geq d_{\hY'}$, so we only need to prove that $d_{\calQ} \prec d_{\calQ'}$, and for this it suffices to bound the distances in each $U \in \calU$.

Let $\hx, \hy \in \calQ$ with $\hx' = \Delta(\hx)$ and $\hy' = \Delta(\hy)$.  We may assume that $\hx_U \neq \hy_U$ for any $U\in \calU$ that we will consider.  We begin by making a useful reduction.

Observe that item (8) of Lemma \ref{lem:collapsed tree control} implies that there are at most boundedly-many domains $U \in \calU$ for which $\hx_U,\hy_U$ are in the same component of $\hT^e_U$, where the bound is controlled by $\mathfrak S$ and $|F \cup \Lambda|$.  For such $U$, we have $d_{\hT_U}(\hx_U, \hy_U) \leq d_{\hT'_U}(\hx'_U, \hy'_U) + B^2$ because at most $B$-many subtrees of diameter $B$ have been deleted from that edge.  Hence we may deal with these boundedly-many domains with an additive constant in the distance estimate comparing $d_{\calQ}(\hx,\hy)$ and $d_{\calQ'}(\hx',\hy')$. Hereafter, we assume that at least one of $\hx_U, \hy_U$ coincides with a cluster, or they are separated by at least one cluster.

\begin{figure}
    \centering
    \includegraphics[width=.8\textwidth]{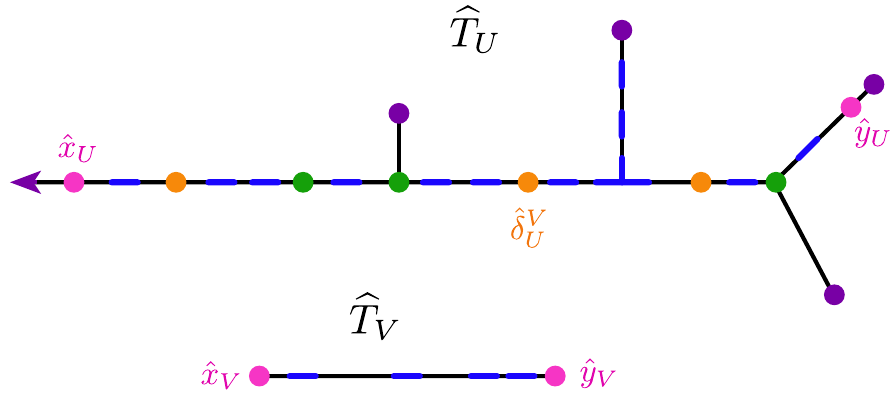}
    \caption{The proof of Theorem \ref{thm:tree trimming}: In case (1), we consider a $\nest_{\calU}$-minimal domain like $V$, which is guaranteed to have a large diameter while also only containing a bounded number of trimming subtrees, and hence trimming only affects the distance between $\hx_V, \hy_V$ by a bounded amount. In case (2), the geodesic in $\hT_U$ between $\hx_U,\hy_U$ (pink) must cross a number of marked cluster points and hence potentially arbitrarily many trimmed subtrees (blue).  Some of these may be sporadic (green), but there are only boundedly-many of these across all $\hT_U$.  Hence the cluster points labeled by $\nest_{\calU}$-minimal bipartite domains decompose $[\hx_U, \hy_U]_{\hT_U}$ into components which contain only boundedly-many trimmed subtrees, allowing for these domains to account for the lost distance by a bounded multiplicative constant.} 
    \label{fig:tree_trim}
\end{figure}

There are two cases, (1) $U$ is $\nest_{\calU}$-minimal and (2) otherwise.

\underline{Case (1)}: $\hT^e_U$ has boundedly-many components, where the bound depends on $|F \cup \Lambda|$, as each collapsed point is a marked point.  So there are at most $B|F \cup \Lambda|$-many subtrees of diameter at most $B$ in $\hT^e_U$ being collapsed.  By item (8) of Lemma \ref{lem:collapsed tree control} and Proposition \ref{prop:bipartite}, there are at most boundedly-many (in $\mathfrak S, |F \cup \Lambda|$) such $\nest_{\calU}$-minimal $U$ for which either one of $\hx_U, \hy_U$ is not a cluster point or $U$ is not bipartite.  We may deal with such $U$ by increasing the additive constant in the distance estimate.

The general case is where $U$ is bipartite and $\hx_U, \hy_U$ are the distinct endpoint clusters of the interval $\hT_U$ (Lemma \ref{lem:bipartite properties}).  In this case, 
$$d_{\hT'_U}(\hx'_U,\hy'_U)  \geq d_{\hT_U}(\hx_U,\hy_U) - B^2 \geq K - B^2 - N,$$ where $N = N(\mathfrak S, |F \cup \Lambda|)>0$ by item (6) of Lemma \ref{lem:collapsed tree control}.   Hence by choosing $K = K(\mathfrak S, |F \cup \Lambda|)>0$ sufficiently large, we can guarantee that
$$d_{\hT_U}(\hx_U,\hy_U) \leq d_{\hT'_U}(\hx'_U,\hy'_U) + B^2 \leq (1+B^2)d_{\hT'_U}(\hx'_U,\hy'_U)$$
and thus account for these domains via a bounded multiplicative error.  Moreover, by choosing $K$ sufficiently large, we can guarantee that $\diam(\hT'_U)$ is as large as necessary, which will allow us to use a multiplicative error and these $\nest_{\calU}$-minimal bipartite domains to account for domains up the $\nest$-lattice in the other case.

\underline{Case (2)}:  Suppose that $U \in \calU$ is not $\nest_{\calU}$-minimal.  Since $\hx_U \neq \hy_U$, there are some $k \leq 2$, intervals $E_1, \dots, E_k \subset \hT^e_U$, and cluster points $C_1, \dots, C_{k-1}$ with $E_i, E_{i+1}$ connected at $C_i$ for $1 \leq i \leq k-1$, and each $C_i$ separating $\hx_U$ from $\hy_U$.

Note that it is possible that each $C_i$ is \emph{sporadic}, namely $C_i$ either is
\begin{itemize}
\item a marked point cluster, or
\item an internal cluster containing no $10E'$-bipartite domains for which $U$ is a witness (Definition \ref{defn:bipartite}).
\end{itemize}

However, the set of such sporadic clusters across $\hT_U$ for all $U \in \calU$ is bounded in terms of $|F \cup \Lambda|$ and $10E'$, which is bounded in terms of $\mathfrak S, |F \cup \Lambda|$. In particular, $k < N_s = N_s(\mathfrak S, |F \cup \Lambda|)$.  In particular, the number of edge components adjacent to such clusters is bounded, and so the total contributed distance loss between $\hx,\hy$ and $\hx',\hy'$ from collapsing at most $B$-many subtrees in any of these components is bounded in terms of $\calX, |F \cup \Lambda|$, and $B$.  Thus we can deal with this lost distance with an additive constant. 

Hence, for every $U \in \calU$ which is not $\nest_{\calU}$-minimal, the geodesic $[\hx_U,\hy_U] \subset \hT_U$ can be decomposed along non-sporadic cluster points $C'_1, \dots, C'_n$ so that each edge in the complement of the $C'_i$ has only boundedly-many bounded diameter subintervals collapsed.  In other words, there exists $B'' = B''(\calX, |F\cup \Lambda|,B)>0$ so that
$$d_{\hT_U}(C'_i, C'_{i+1}) \leq d_{\hT'_U}(\Delta_U(C'_i),\Delta_U(C'_{i+1})) + B''.$$

As the $C'_i$ are non-sporadic clusters, each coincides with $\hd^{V_i}_U$ for some $\nest_{\calU}$-minimal $10E'$-bipartite domain $V_i \in \calU$ for which $U$ is a witness.  This means (Definition \ref{defn:bipartite}) that the $10E'$-neighborhood of $\delta^{V_i}_U$ avoids all marked points and branched points of $T_U$.  Since $\hx_U,\hy_U \neq \hd^{V_i}_U$ for each $i$, we have that $\hx_{V_i}, \hy_{V_i}$ must coincide with the marked point clusters contained in the two components of $T_U - \delta^{V_i}_U$.  In particular, $\hx_{V_i} \neq \hy_{V_i}$ are at distinct marked point clusters, and thus $d_{\hT_{V_i}}(\hx_{V_i}, \hy_{V_i}) \succ K$ because $V_i$ is bipartite and $\nest_{\calU}$-minimal (this uses item (6) of Lemma \ref{lem:collapsed tree control}).  Moreover, by choosing $K$ sufficiently large, we can guarantee that $d_{\hT'_{V_i}}(\hx'_{V_i}, \hy'_{V_i}) > 1$.  As a consequence, we have
\begin{eqnarray*}
d_{\hT_U}(C'_i, C'_{i+1}) &\leq& d_{\hT'_U}(\Delta_U(C'_i),\Delta_U(C'_{i+1}))+ B''\\
&\leq& d_{\hT'_U}(\Delta_U(C'_i),\Delta_U(C'_{i+1})) + B''d_{\hT'_{V_i}}(\hx'_{V_i}, \hy'_{V_i}).
\end{eqnarray*}

Hence we can account for the lost distance along $[\hx_U, \hy_U]$ between each pair of successive non-sporadic clusters $C'_i, C'_{i+1}$ by distances in the $\nest_{\calU}$-minimal bipartite $V_i$.  Similarly, we can use $V_1, V_k$ to account for any lost distance between $\hx_U, C'_1$ and $C'_n, \hy_U$, respectively.  Since each domain in $\calU$ nests into at most boundedly-many domains in $\calU$ by the Covering Lemma \ref{lem:covering}, the proof of case (2) is complete.

Finally, the ``moreover'' part is clear: if only boundedly-many trees are trimmed, then the global distance is only decreased by the sum of all of the diameters of the collapsed subtrees.  Hence only an additive error is necessary.  This completes the proof.
\end{proof}

The following lemma was used in the proof:

\begin{lemma}\label{lem:cluster same}
Suppose that $U, V, W \in \calU$ are pairwise transverse and $\hd^V_U \neq \hd^W_U$.  Then $\hd^U_V = \hd^W_V$ and $\hd^U_W = \hd^V_W$.
\end{lemma}

\begin{proof}
By choosing the cluster separation constant $r$ sufficiently large, we can guarantee that $d_U(\rho^V_U, \rho^W_U)> 100E$ since $\hd^V_U \neq \hd^W_U$.  Thus if also $\hd^U_V \neq \hd^W_V$, we also have $d_V(\rho^U_V, \rho^W_V)>100E$.  Now choose $x \in \PP_W$ the product region in $\calX$ for $W$ (Subsection \ref{subsec:product region}), making $d_U(x,\rho^W_U)<E$ and $d_V(x, \rho^W_V)<E$.  It follows that $d_U(x,\rho^V_U)>E$ and $d_V(x, \rho^U_V)>E$, which contradicts the $E$-consistency of $x$.  This completes the proof.
\end{proof}

\subsection{Replacing $\hT_U$ with simplicial trees}

In this subsection, we quickly observe how to use Tree Trimming \ref{thm:tree trimming} to replace all $\hT_U$ with simplicial trees which encode the relative projection data into the vertices while only changing $\calQ$ up to quasi-isometry with controlled constants.  This refinement will be necessary produce the cubical structure of $\calQ$ via a direct argument later in Section \ref{sec:Q is cubical}.

The goal is to replace each $\hT_U$ with a simplicial tree (i.e., a 1-dimensional cube complex), where all cluster points and marked points are vertices.  The issue is that the intervals between cluster and marked points can have non-integer lengths, and the idea is to trim any such interval by collapsing a subinterval of length less than $1$.  Tree Trimming \ref{thm:tree trimming} will tell us that the resulting quotient map is a uniform quasi-isometry.

To setup notation, fix $U \in \calU$.  Observe that each component $E \subset \hT^e_U$ is a tree with valence bounded by $|F \cup \Lambda|$ and $\calX$, with leaves all cluster points or marked points of $\hT_U$, and with no internal marked or cluster points.  By choosing the cluster separation constant $r$ sufficiently large, we we may choose intervals $J_1, \dots, J_n$, with 
\begin{itemize}
\item $n$ controlled by $\calX$ and $|F \cup \Lambda|$;
\item each $J_i$ adjacent to a distinct leaf of $E$;
\item $|J_i|<1$ for each $i$;
\item the distances between each pair of leaves in $E$ is larger than $2$;
\item and if we collapse each $J_i$ to a point, the resulting tree $E \to E'$ has the property that all distances between leaves in $E'$ are integers.
\end{itemize}

Let $\Delta_U:\hT_U \to \hT'_U$ denote the map that quotients each $J_i$ to a point for all such intervals across all edge components $E \subset \hT^e_U$ for every $U \in \calU$.  Observe that each collapsed interval $J$ is identified with the cluster point to which it was adjacent.  We again refer to the images of marked and cluster points in the $\hT'_U$ using the same terminology.  Note that by construction, every pair of cluster points or marked points in each $\hT'_U$ are at integer distance, and thus we can subdivide the edges of each $\hT'_U$ to make it a simplicial tree with marked and cluster points at vertices.

Finally, set $\calY' = \prod_{U \in \calU} \hT'_U$, let $\oQ'$ be the set of $0$-consistent tuples, and let $\Delta:\calY \to \calY'$ all be as defined in the previous subsection.  It follows from Tree Trimming \ref{thm:tree trimming} that the map $\Delta:\calQ \to \calQ'$ is a uniform quasi-isometry, since only boundedly-many (in $\calX$ and $|F \cup \Lambda|$) subtrees of diameter at most $1$ are collapsed in each edge component.

The following summarizes this discussion:

\begin{corollary}\label{cor:simplicial structure}
For any finite collection $F \cup \Lambda$ of interior points, boundary points, and hierarchy rays in an HHS $\calX$, let $\{\hT_U\}_{U \in \calU}$ and $\calQ$ be the family of collapsed trees and $0$-consistent set as produced in Section \ref{sec:Q}.  Then up to a uniform quasi-isometry of $\calQ$ (in $\calX$ and $|F \cup \Lambda|)$, we may assume that every $\hT_U$ is a simplicial tree with all collapsed points and marked points being vertices.
\end{corollary}

\section{Coarse surjectivity of $\hPsi:H \to \calQ$}\label{sec:coarsely surjective}

In this section, we prove that $\hPsi:H \to \calQ$ is coarsely surjective, which we do by proving that the composition $\hPsi \circ \hO: \calQ \to \calQ$ is coarsely constant.  Namely:

\begin{theorem}\label{thm:coarse inverse}
There exists $B_0 = B_0(\mathfrak S, |F \cup \Lambda|, K)>0$ and a map $\hO:\calQ \to H$ so that for any $\hx \in \calQ$, we have $d_{\calQ}(\hPsi \circ \hO(\hx), \hx)<B_0$.
\end{theorem}

As a corollary of this and Theorem \ref{thm:Q distance estimate}, we get the following:

\begin{corollary}\label{cor:Q qi to H}
There exists $L' = L'(\mathfrak S, |F \cup \Lambda|, K)>0$ so that $\Phi:H \to \calQ$ is an $(L',L')$-quasi-isometry.  Moreover, $\hPsi$ and $\hO$ are quasi-inverses.
\end{corollary}

\begin{proof}
The fact that $\Phi$ is a quasi-isometric embedding follows from Theorem \ref{thm:Q distance estimate} plus the Distance Formula Theorem \ref{thm:DF}, while Theorem \ref{thm:coarse inverse} says that $\Phi$ is coarsely-dense, making it a quasi-isometry.
\end{proof}

\begin{remark}
Note that our proof of the Distance Formula \ref{thm:DF} (in Corollary \ref{cor:DF lower bound} below) does not rely on either of the above results.
\end{remark}

We will break the proof of Theorem \ref{thm:coarse inverse} into two steps, first analyzing the composition at the level of the trees $T_U$, and then the collapsed trees $\hT_U$.

\subsection{Coarse surjectivity of $\Psi$}\label{subsec:coarse surject tree}

We first prove that $\Psi: H \to \calZ_{\alpha_0}$ is coarsely surjective, where $\Psi(H) \subset \calZ_{\alpha_0}$ for $\alpha_0 = \alpha_0(\mathfrak S, |F \cup \Lambda|)>0$ by Lemma \ref{lem:consist to tree}.   To do this, we need to define a map from $\Phi_{\alpha}:\calZ_{\alpha} \to H$ for any $\alpha \geq \alpha_0$, which we do as follows:

Recall that Proposition \ref{prop:consist from tree} provides a $\beta = \beta(\calX, |F\cup \Lambda|, K, \alpha)>0$ so that $\Phi(\calZ_{\alpha}) \subset \prod_{U \in \mathfrak S} \calC(U)$ consists of $\beta$-consistent tuples.  Thus we can define a map from $\Phi_{\alpha}:\calZ_{\alpha} \to H$ by
$$\Phi_{\alpha} = \ret_H \circ \Real_{\beta} \circ \Phi: \calZ_{\alpha} \to H,$$
where $\Real_{\beta}$ assigns a point of $\calX$ to any $\beta$-consistent tuple, and $\ret_H:\calX \to H$ is the coarse Lipschitz retraction provided by Lemma \ref{lem:gate retract}, whose quality is controlled by $\calX$ and $|F \cup \Lambda|$.

We now state and prove our first main step of the coarse inverse theorem.  One interpretation is that, with a distance formula-like ``metric'' on $\calZ_{\alpha}$, the composition $\Psi \circ \Phi_{\alpha}$ is coarsely the identity.  But in fact, the proposition gives much more information, since we get explicit bounds on component-wise distances, including a global bound on all component-wise distances (depending on $K$), and a bound on the number of domains satisfying a medium-size distance bound (with this size independent of $K$).

We note that, in the statement below, we need the extra flexibility of dealing with $\calZ_{\alpha}$ for $\alpha \geq \alpha_0$, because Proposition \ref{prop:Q-consistent} provides an $\alpha_{\omega} = \alpha_{\omega}(\mathfrak S, |F \cup \Lambda|)>0$ so that $\Omega: \calQ \to \calZ_{\alpha_{\omega}}$, and likely $\alpha_{\omega}> \alpha_0$.

\begin{proposition}\label{prop:coarsely surjective, trees}
For any $\alpha \geq \alpha_0$, there exist $D_1 = D_1(K,\alpha, \mathfrak S, |F \cup \Lambda|)>0$ so that for all $x = (x_V) \in \calZ_{\alpha}$, we have
\begin{enumerate}
    \item For all $U \in \calU$, we have $d_{T_U}(x_U, \psi_U \circ \Phi_{\alpha}(x))<D_1$
    \item $\#\{U \in \calU| d_{T_U}(x_U, \psi_U \circ \Phi_{\alpha}(x)) > 2\alpha + 3E'\} < D_1$.
\end{enumerate}

\begin{itemize}
    \item In particular, by choosing $L> 2\alpha+E'$, we get
$$\sum_{U \in \calU} [d_{T_U}(x_U, \psi_U \circ \Phi_{\alpha}(x)]_L< D_1^2.$$
\end{itemize}
\end{proposition}

\begin{proof}
The proof of (1) is essentially a definition chase: Let $x = (x_U) \in \calZ_{\alpha}$, let $(x'_V) = \Phi_{\alpha}(x)$ be the $\beta=\beta(\calX, |F \cup \Lambda|, K, \alpha)>0$ consistent tuple provided by Proposition \ref{prop:consist from tree}, and let $y' \in \Real_{\beta}(\Phi_{\alpha}(x))$ be a realization.

Let $\theta_{\beta} = \theta_{\beta}(\alpha_0, \mathfrak S)>0$ be the constant produced by the Realization Theorem \ref{thm:realization} for the $\beta$-consistent tuple $(x'_V)$, so that $d_V(y', x'_V)<\theta_{\beta}$ for each $V \in \mathfrak S$, where we recall that $\mathfrak S$ is the full domain index set for the ambient HHS $(\calX, \mathfrak S)$.

Observe that since the image of $\phi_V:T_V \to \calC(U)$ is within Hausdorff distance $C = C(\calX, |F \cup \Lambda|)>0$ of $H_V = \hull_V(F \cup \Lambda)$, we get that $x'_V \in \calN_C(H_V)$ for each $V \in \mathfrak S$.  It follows that $y' \in \hull^{\calX}_{\theta_{\beta} + C}(F \cup \Lambda)$ (see Definition \ref{defn:hier hull}).  In particular, if $y \in \ret_H(y')$, then $d_V(y,y') < \theta_0 = \theta_0(\beta, \mathfrak S)$ for all $V \in \mathfrak S$.  It follows that $d_V(y, x'_V)< \theta_0 + \theta_{\beta} + C$ for all $V \in \mathfrak S$.

On the other hand, $y = \ret_H(y') \in H$, and so $d_U(y, H_U) < \theta$, where $\theta= \theta(\calX, |F \cup \Lambda|)$ is the hull constant defining $H$.  Then since $d_{Haus}(\phi_U(T_U),H_U)<C$ and $\phi^{-1}_U$ is a $(C,C)$-quasi-isometry, we see that $d_{T_U}(x_U, \psi_U \circ \Phi_{\alpha}(x))< d_{T_U}(x_U, \phi^{-1}_U \circ p_U(y))) < C(\theta_0 + \theta_{\beta} + \theta + 2C) + C$, giving us the desired bound for (1).

For the proof of (2), let $y = (y_U) \in \Psi \circ \Phi_{\alpha}(x)$.  Let $\calV = \{U \in \calU| d_T(x_U, y_U) > 2\alpha + 3E'\}$.  Set $K_1 = \alpha+ E'$ and choose $K_2 = K_2(\mathfrak S, |F|, D_1, \alpha)>0$ to be determined shortly.  By assuming, for a contradiction, that $\#\calV$ is sufficiently large, we will be able to employ Strong Passing-up \ref{prop:SPU} to produce a domain $W \in \calU$ for which $d_{T_W}(x_U, y_U) > D_1$, violating the first conclusion of the current proposition.

By passing to a subset if necessary, we may assume that $\calV \subset \Rel_{K}(a,b)$ for some fixed $a,b \in F \cup \Lambda$ and that $\#V > P_1(K_1,K_2)$, where $P_1$ is the function provided by Proposition \ref{prop:strong pu}, since otherwise $\#\calV$ is bounded as desired.

Hence by Proposition \ref{prop:SPU}, there exists a subset $\calV' \subset \calV$ and $W \in \Rel_{K_2}(a,b)$ so that $V \nest W$ for all $V \in \calV'$ and 
$$\diam_{W} \left(\bigcup_{V \in \calV'} \rho^V_W\right) > K_2.$$

\begin{figure}
    \centering
    \includegraphics[width=.7\textwidth]{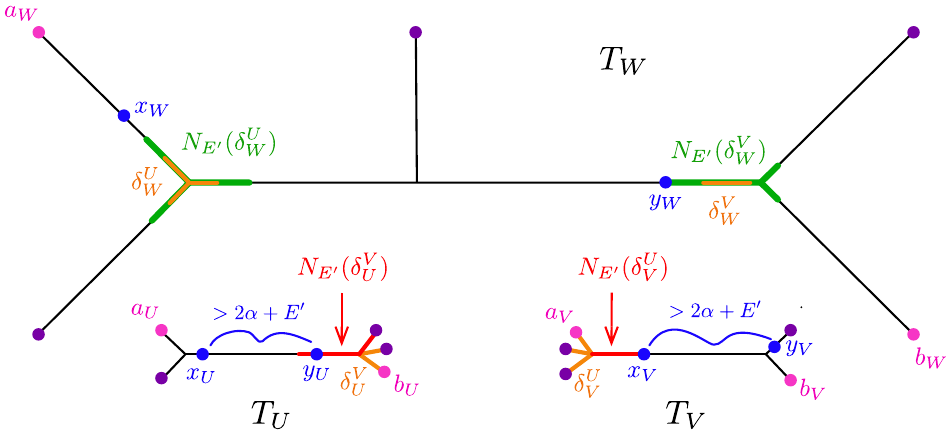}
    \caption{The proof of case (2) of Proposition \ref{prop:coarsely surjective, trees}: Not both of $x_U$ and $y_U$ (both blue) can be close to $\delta^V_U$ in $T_U$, and similarly for $V$.  The BGI property then forces $x_W$ and $y_W$ to be relatively close to $\delta^U_W$ and $\delta^V_W$, respectively, which then forces $d_{T_W}(x_W,y_W)$ to be larger then $D_1$, violating the bound from item (1).}
    \label{fig:coarse_surject_tree}
\end{figure}

Observe that we may choose $K_2 = K_2(\mathfrak S, |F|, D_1, \alpha)>0$ sufficiently large so that $\diam_{T_W}\left(\bigcup_{V \in \calV'} \delta^V_W\right) > D_1 + 2E'$.  Let $U, V \in \calV'$ realize this diameter, hence forcing $U \pitchfork V$.

Now if $d_{T_U}(x_U, \delta^V_U) > \alpha$ and $d_{T_U}(y_U, \delta^V_U)> \alpha$, then $d_{T_V}(x_V, y_V) < 2\alpha$ since both are $\alpha$-consistent, contradicting the assumption that $d_{T_V}(x_V, y_V) > 2\alpha + 3E'$.  Consequently, exactly one of $x_U$ and $y_U$ is $\alpha$-close to $\delta^V_U$ in $T_U$, and similarly for $V$.  Without loss of generality, assume that $d_{T_V}(x_V, \delta^U_V) < \alpha$ and $d_{T_V}(y_V, \delta^U_V) < \alpha$.

Also observe that the geodesic $[a_W,b_W]$ in $T_W$ (which is possibly a bi-infinite geodesic) between $a_W,b_W$ must pass within $E'$ of both $\delta^U_W, \delta^V_W$ by the BGI property (item (5) of Lemma \ref{lem:tree control}).  Since $d_{T_W}(\delta^U_W, \delta^V_W) > K_2$, we may assume, without loss of generality, that $[a_W,b_W]$ first passes close to $\delta^U_W$ then $\delta^V_W$.  It follows then from item (6) of Lemma \ref{lem:tree control} that $d_{T_U}(b_U, \delta^V_U) < 2E'$ and $d_{T_V}(a_V, \delta^U_V)< 2E'$.

It follows then that $d_{T_U}(x_U, b_U) > 2\alpha + 3E' - 2E' > E'$ and $d_{T_V}(y_V, a_V)>2\alpha+ 3E' - 2E' > E'$.  Hence the BGI property (item (5) of Lemma \ref{lem:tree control}) says that the geodesics in $T_W$ between $x_W,b_W$ and $y_W,a_W$ must pass within $E'$ of $\delta^U_W$ and $\delta^V_W$, respectively.  It follows then that $d_{T_W}(x_W,y_W) > d_{T_W}(\delta^U_W, \delta^V_W) - 2E' \geq D_1 + 2E' - 2E' = D_1$, which is a contradiction.

\end{proof}

\subsection{The map $\hO:\calQ \to H$ and coarse surjectivity} \label{subsec:proof of coarsely surjective}

We are now ready to define the map $\hO:\calQ \to H$ using our work from Section \ref{sec:Q to Y} as follows.  Given a tuple $(\hx_U) \in \calQ$, let $(x_U)_{U \in \calU} \in \calZ_{\alpha_{\omega}}$ be a $\alpha_{\omega}$-consistent partial tuple, where $x_U \in B_U$ as defined in Subsection \ref{subsec:honing clusters} for each $U \in \calU$ and $\alpha_{\omega} = \alpha_{\omega}(\calX, |F \cup \Lambda|)>0$ is provided by Proposition \ref{prop:Q-consistent}.  Since $(x_U) \in \calZ_{\alpha_{\omega}}$, we can apply the map $\Phi_{\alpha_{\omega}}: \calZ_{\alpha_{\omega}} \to H$ as defined above in Subsection \ref{subsec:coarse surject tree}.  Hence by composing these we obtain a map:
$$\hO = \Phi_{\alpha_{\omega}} \circ \Omega: \calQ \to H.$$

Since $\hPsi(H) \subset \calQ$ by Lemma \ref{lem:F and Lambda}, we get that $\hPsi \circ \hO: \calQ \to \calQ$.  We are now in a position to prove that this composition is coarsely the identity:

\begin{proof}[Proof of Theorem \ref{thm:coarse inverse}]

We need to prove that there exists $B_0= B_0(\mathfrak S, |F \cup \Lambda|, K)>0$ so that for any $\hx \in \calQ$, if $\hy \in \hPsi \circ \hO(\hx)$, then we have

$$d_{\calQ}(\hx, \hy) < B_0.$$

For this, we will use Proposition \ref{prop:coarsely surjective, trees}, which provides $D_1 = D_1(K, \alpha_{\omega}, \mathfrak S, |F \cup \Lambda)>0$ so if $(x_U) = \Omega(\hx)$ and $(y_U) \in \Psi \circ \hO(\hx)$ then
\begin{enumerate}
    \item $d_{T_U}(x_U, y_U)< D_1$ for all $U \in \calU$, and
    \item $\#\{U \in \calU| d_{T_U}(x_U, y_U) > 2\alpha_{\omega} + 3E'\} < D_1$.
\end{enumerate}

Let $\calV = \{U \in \calU| \hx_U \neq \hy_U\}$.  Item (9) of Lemma \ref{lem:collapsed tree control} says that the number of domains $U \in \calU$ so that $\hx_U$ and $\hy_U$ are both not contained in marked point clusters is bounded solely in terms of $\mathfrak S, |F \cup \Lambda|$.  Hence by the global bound $D_1$ in property (1) above, we may assume that for every $V \in \calV$, both $\hx_V, \hy_V$ are at marked point clusters.

Moreover, by Proposition \ref{prop:bipartite}, all but boundedly-many domains in $\calV$ are bipartite, so we may again use property (1) to further reduce to the case where every $V \in \calV$ is bipartite.

Hence if we consider $\calV_{min}$ the set of $\nest_{\calU}$-minimal bipartite domains $V \in \calV$, then each $\hT_V$ is a non-degenerate interval by Lemma \ref{lem:leaves choose half-space}.  Thus if $V \in \calV_{min}$, then the assumption that $\hx_V \neq \hy_V$ are at distinct marked cluster points implies that $d_{\hT_V}(\hx_V,\hy_V) \succ K$ by item (6) of Lemma \ref{lem:collapsed tree control}.  Since $2\alpha_{\omega} + E'$ is independent of our choice of $K$, we can choose $K$ large enough to guarantee that $\#\calV_{min} < D_1$.

The next step is to bound the number of domains $V$ higher in the $\nest$-lattice where $\hx_V \neq \hy_V$, though we will not do so directly.  Instead, we observe that in all but boundedly-many domains in $\calV$, we have that $d_{T_V}(x_V, y_V) < 2\alpha_{\omega} + 3E'$, which is a constant independent of $K$.  The idea is to now use the Tree Trimming Theorem \ref{thm:tree trimming} as follows.

Set $B = 2\alpha_{\omega} + 3E'$.  Let $\calV' \subset \calV - \calV_{min}$ be the set of domains $V\in \calU$ where $\hx_V \neq \hy_V$ are at distinct marked cluster points and $d_{T_V}(x_V, y_V) < B$, and so that at least one of the cluster points contained in $[\hx_V, \hy_V]$ is labeled by a $\nest_{\calU}$-minimal bipartite domain not in $\calV_{min}$.  Observe that all but boundedly-many cluster points satisfy this last property, by the bound on $\#\calV_{min}$ and the Covering Lemma \ref{lem:covering}, which bounds the number of domains in $\calU$ into which the domains in $\calV_{min}$ can nest.  In particular, $\#(\calV - \calV')$ is bounded in terms of $D_1, \mathfrak S, |F \cup \Lambda|$ by the above reductions. 

For each $V \in \calV'$, let $A_V$ denote the set of these segments and let $\Delta_V: \hT_V \to \hT'_V$ denote the map which collapses each segment to a point.  Theorem \ref{thm:tree trimming} says that the combined map $\Delta:\calQ \to \calQ'$ is a quasi-isometry, with constants depending only on $\mathfrak S, |F \cup \Lambda|$ and $B$.  Moreover, by its construction we have $d_{\hT'_U}(\Delta_U(\hx_U), \Delta_U(\hy_U)) \leq d_{\hT_U}(\hx_U, \hy_U)$ for each $U \in \calU$.

Observe that $\Delta_V(\hx_V) = \Delta_V(\hy_V)$ by construction.  Hence we have
\begin{eqnarray*}
d_{\calQ}(\hx, \hy) &=& \sum_{U \in \calU} d_{\hT_U}(\hx_U,\hy_U)\\
&\asymp& \sum_{U \in \calU} d_{\hT'_U}(\Delta_U(\hx_U), \Delta_U(\hy_U))\\
&=& \sum_{U \in \calU - \calV'} d_{\hT'_U}(\Delta_U(\hx_U), \Delta_U(\hy_U)) + \sum_{V \in \calV'} d_{\hT'_V}(\Delta_V(\hx_V), \Delta_V(\hy_V))\\
&= &\sum_{U \in \calU - \calV'} d_{\hT'_U}(\Delta_U(\hx_U), \Delta_U(\hy_U))\\
\end{eqnarray*}
and this final sum is bounded since each term is at most $D_1$ by property (1) above, and $\#(\calU - \calV')$ is bounded by $D_1, \mathfrak S, |F \cup \Lambda|$, as observed above.

Hence we get that $d_{\calQ}(\hx, \hy)$ is bounded in terms of $\mathfrak S, |F \cup \Lambda|$, and $K$, as required, completing the proof.

\end{proof}

\section{The induced wall-structure on $\oQ$}\label{sec:walls in Q}

In this section, we analyze the combinatorial structure of $\oQ$, the $0$-consistent set in the product of simplicial trees $\oY = \prod_{U \in \calU} \hT_U \cup \partial \hT_U$, as produced in Corollary \ref{cor:simplicial structure}.

Since each of the $\hT_U$ is a tree and thus a cube complex, $\oY$ naturally admits the structure of a cube complex, and we can use its hyperplanes to induce walls on $\oQ$.  The main goal of this section is to study the resulting wall-space, in particular to prove that it forms a pocset (Proposition \ref{prop:pocset}).  Much of this discussion is straight-forward in the interior case (i.e., when $\Lambda = \emptyset$), so the bulk of the work happens in the infinite case.

\subsection{Hierarchical families of trees}\label{subsec:HFT}

It is useful to axiomatize our setup in the finite case, and indeed this proves useful later for the applications in Section \ref{sec:sep hyp}.  The following definition axiomatizes the properties we need for Proposition \ref{prop:pocset} and its supporting lemmas when $\Lambda = \emptyset$.

\begin{definition}[Hierarchical family of trees] \label{defn:HFT}
A \emph{hierarchical family of trees} is the following collection of objects and properties:

\begin{enumerate}
\item A finite index set $\calU$ of \emph{domains} with \emph{relations} $\nest, \pitchfork, \perp$ so that
\begin{itemize}
    \item $\nest$ is anti-reflexive and anti-symmetric, and gives a partial order with a bound on the length of chains;
    \item $\pitchfork$ is symmetric and anti-reflexive;
    \item There is a unique $\nest$-maximal element $S$;
    \item Any pair of domains $U,V \in \calU$ satisfies exactly one of the relations. 
\end{itemize}
\item There is a bound on the size of any pairwise $\perp$-comparable subset of $\calU$. \label{item:finite dim}
\item To each $U \in \calU$ there is an associated finite simplicial tree $\hT_U$. 
\item A finite collection of labels $F$ and, on each tree, a finite set of \emph{marked points}, which are labeled by $F$, with each element of $F$ labeling exactly one marked point of $\hT_U$.  Moreover, for each $U \in \calU$, each leaf of $\hT_U$ is a marked point.
\item A family of \emph{relative projections} determining:
\begin{itemize}
    \item For each $U, V \in \calU$ with $V \nest U$ or $V \pitchfork U$, there is a vertex $\hd^V_U \in \hT^{(0)}_U$.  Moreover, each component of $\hT_U - \hd^V_U$ contains a marked point.
    \item If $U,V,W \in \calU$ with $U \perp V$ and $V \nest W$ and either $U \nest W$ or $U \pitchfork W$, then $\hd^U_W = \hd^V_W$.
\end{itemize}
    \item(BGI) There exist projection maps $\hd^U_V: \hT_U \to \hT_V$ when $V \nest U$, satisfying the following bounded geodesic image property:\label{item:BGI HFT}
    \begin{itemize}
        \item If $C \subset \hT_U - \hd^V_U$ is a component, then $\hd^U_V(C)$ coincides with $\hf_V$ for any $f \in F$ with $\hf_U \in C$. 
    \end{itemize}
\end{enumerate}
\end{definition}

From the above set of data, one can define consistency equations and the \emph{$0$-consistent subset} $\calQ$ of the \emph{total space}  $\calY \subset \prod_{U \in \calU} \hT^{(0)}_U$ as in Definition \ref{defn:Q consistent}.

We observe that our hierarchical setup satisfies this set of properties:

\begin{lemma}\label{lem:Q is HFT}
The family of trees $\{\hT_U\}_{U \in \calU}$ with associated collapsed relative projections $\hd^V_U$, as produced in Section \ref{sec:Q} starting from a reduced tree system (Definition \ref{defn:tree axioms} associated to the hierarchical hull of a finite set of internal points in an HHS $\calX$ satisfies Definition \ref{defn:HFT}.
\end{lemma}

\begin{proof}
Items (1) and (2) follow from the Definition \ref{defn:HHS} of an HHS.  Items (3), (4), and (5) are contained in Definition \ref{defn:collapsed tree}, while item (6) is in Definition \ref{defn:Q project}.

\end{proof}

\begin{remark}[Infinite case, $\Lambda \neq \emptyset$]
While it is possible to axiomatize the infinite setting (when $\Lambda \neq \emptyset$), we were unable to agreeably and succinctly package the required extra axioms, which all essentially encode various structural applications of Strong Passing-Up \ref{prop:SPU}.  As with much of this article, the infinite setting asserts itself several times throughout the discussion, with the finite setting being relatively straight-forward.
\end{remark}

\begin{remark}[Other generalizations]
It seems possible that the construction in this section works in greater generality.  For instance, one could define a \emph{hierarchical family of cube complexes}, where the $\hT_U$ are higher dimensional cube complexes, and perhaps the relative projections are convex subcomplexes, not necessarily points.  As we do not have a direct application in mind, we have not pursued this more general version of this machinery.
\end{remark}

\begin{remark}[Proofs in this section]
We will proceed with the proofs in this section by using the various parts of Definition \ref{defn:HFT} to prove the (usually much simpler) finite case when $\Lambda \neq \emptyset$, and then supplement this setup with hierarchical facts (such as Strong Passing-Up \ref{prop:SPU}) when dealing with the infinite case.  The more general framework in the interior case is necessary for the proof of Theorem \ref{thmi:curve graph} in Section \ref{sec:sep hyp}.
\end{remark}

\subsection{A quick primer on CAT(0) cube complexes, wallspaces, and Sageev's construction}

As we have arrived at the point in the paper where we begin constructing cube complexes, we review some basics.  For a detailed introduction to cube complexes, see Sageev's notes \cite{Sageev_notes} as well as Hruska-Wise \cite{HW_CCC}.  Wallspaces were introduced by Haglund-Paulin \cite{HP_wallspace}, though we most closely follow Roller's approach \cite{Roller} to Sageev's cubulation machine \cite{Sageev_machine}.

A \emph{cube complex} is a CW complex obtained by gluing a collection of unit Euclidean cubes of various dimensions along their faces by a collection of (Euclidean) isometries.  A cube complex $\Sigma$ is \emph{non-positively curved} when it satisfies Gromov's link condition, and every non-positively curved simply connected cube complex admits a unique CAT(0) metric.

A \emph{hyperplane} in a CAT(0) cube complex $\Sigma$ is a connected subspace whose intersection with each cube $\sigma = [0,1]^{n}$ is either $\emptyset$ or the subspace obtained by restricting exactly one coordinate to $\frac{1}{2}$.  In particular, every hyperplane partitions $\Sigma$.  This motivates the definition of a half-spaces and pocsets, following Roller \cite{Roller}.

\begin{definition}[Pocset]\label{defn:pocset}
Let $(\Sigma,<)$ be a partially-ordered set.  A \emph{pocset} structure on $\Sigma$ has the additional structure of an order-reversing involution $A \mapsto A^*$ satisfying
\begin{enumerate}
\item $A \neq A^*$ and $A, A^*$ are $<$-incomparable;
\item $A < B \implies B^* < A^*$.
\end{enumerate}
\begin{itemize}
\item Given $A<B$, the \emph{interval} between $A$ and $B$ is $[A,B] = \{C|A<C<B\}$.
\item The pocset $(\Sigma,<)$ is \emph{locally finite} if every interval is finite.
\item Given $A,B \in \Sigma$, we say $A$ is \emph{transverse} to $B$ if none of $A<B, A<B^*, A^*<B, A^*<B^*$ holds.  The pocset $(\Sigma,<)$ has \emph{finite width} if the maximal size of a pairwise transverse subset is finite.
\end{itemize}
\end{definition}

In a CAT(0) cube complex, every vertex is contained in exactly one half-space associated to each hyperplane.  The following notion abstracts the combinatorial structure of such a family of half-spaces:

\begin{definition}[Ultrafilters and the DCC]\label{defn:ultrafilter}
Given a locally finite pocset $(\Sigma, <)$, an \emph{ultrafilter} $\alpha$ on $\Sigma$ is a subset of $\Sigma$ satisfying
\begin{enumerate}
\item (Choice) Exactly one element of each pair $\{A,A^*\}$ is an element of $\alpha$;
\item (Consistency) $A \in \alpha$ and $A< B \implies B \in \alpha$.
\end{enumerate}
\begin{itemize}
\item An ultrafilter $\alpha$ satisfies the \emph{descending chain condition} (DCC) if every descending chain of elements terminates.
\end{itemize}
\end{definition}

We can now define the dual cube complex to a locally finite, discrete, finite-width pocset:

\begin{definition}[Dual cube complex]\label{defn:dual CCC}
Given a locally finite, discrete, finite-width pocset $(\Sigma,<)$, we define a CW complex $X(\Sigma)$ as follows:
\begin{enumerate}
\item The $0$-cells $X^{(0)}$ of $X$ are the DCC ultrafilters on $(\Sigma, <)$.
\item Two $0$-cells $\alpha, \beta$ are connected by an edge in $X^{(1)}$ if and only if $|\alpha \triangle \beta| = 2$, i.e. $\alpha = (\beta - \{A\}) \cup \{A^*\}$ for some $A \in \Sigma$.
\item Define $X^{(n)}$ inductively by adding an $n$-cube whenever the boundary of one appears in the $X^{(n-1)}$-skeleton.
\end{enumerate}
\end{definition}

See Sageev's notes \cite{Sageev_notes} for a thorough discussion of the above.

The following is a theorem of Roller \cite{Roller}, building on Sageev \cite{Sageev_machine}:

\begin{theorem}[Sageev's machine]\label{thm:Sageev}
Given a locally finite, finite-width pocset $(\Sigma,<)$, the CW complex $X(\Sigma)$ is a finite dimensional CAT(0) cube complex, where the dimension is equal to the width of $(\Sigma, <)$. 
\begin{itemize}
\item We call $X(\Sigma)$ the \emph{dual cube complex} to $(\Sigma, <)$.
\end{itemize}
\end{theorem}

\subsection{Inducing walls on $\oQ$} \label{subsec:walls and halfspaces}

For the rest of this section, fix a finite set $F \cup \Lambda$ of internal points and hierarchy rays in an HHS $\calX$, along with their associated structures, such as the $K$-relevant set $\calU$ and its associated simplicial trees $\hT_U$, as we have been working with in most of this paper.  Our goal is to induce a wallspace structure on the $0$-consistent set $\oQ$, and then prove that it gives a pocset.  In Section \ref{sec:Q is cubical}, we prove that the dual cube complex to this wallspace on $\calQ$ is isometric to $\calQ$, as described in Theorem \ref{thm:dual}.

To begin, observe that $\calY = \prod_{U \in \calU} \hT_U$ is a product of cube complexes.  It is a general fact that a product of cube complexes admits a natural cubical structure, where the global hyperplanes come from the hyperplanes in each component in a natural way.  Using this observation, we can induce a wallspace structure on $\oQ$ as follows.

For each $U\in \calU$, a \emph{component hyperplane} $h_U$ in $\hT_U$ is simply the midpoint of an edge in $\hT_U$.  Each such hyperplane $h_U$ then determines a \emph{global hyperplane} $\bar{h}_U \subset \oY$ of the following form:

$$\bar{h}_U = h_U \times \prod_{V \in \calU - \{U\}} \hT_V \cup \partial \hT_V.$$

Set $\hh_U = \bar{h}_U \cap \oQ$, which we refer to as a $\oQ$-\emph{hyperplane}.  We make some basic observations about these $\oQ$-hyperplanes:

\begin{lemma}\label{lem:hyperplane char}
Let $U \in \calU$, and $h_U \in \hT_U$ be a hyperplane.  The following hold:

\begin{enumerate}
\item If $\hz \in \hh_U$, then the coordinates $\hz_V$ of $\hz$ satisfy:
\begin{enumerate}
\item If $U \nest V$ or $U \pitchfork V$, then $\hz_V = \hd^U_V$.
\item If $V \nest U$, then $\hz_V = \hf_V$ for some $f \in  F \cup \Lambda$.
\end{enumerate}

\item If $h_V \subset \hT_V$ is another component hyperplane, then $\hh_U \cap \hh_V \neq \emptyset$ if and only if $U \perp V$.
\end{enumerate}
\end{lemma}

\begin{proof}
For (1), first assume that $U \pitchfork V$.  Then $h_U = \hz_U$ is not a marked or cluster point, and hence $0$-consistency implies that $\hz_U \neq \hd^V_U$.  Similarly, if $U \nest V$ but $\hz_V \neq \hd^U_V$, then the BGI property \eqref{item:BGI HFT} of Definition \ref{defn:HFT} implies that $\hd^V_U(\hz_V) = \hf_V$ for some $f \in F \cup \Lambda$.  But $\hz_U = h_U$ is not a marked point, which is a contradiction, forcing $\hz_V = \hd^U_V$.

Now if $V \nest U$, then $\hz_U = h_U$ is not a marked or cluster point, and so $\hz_U \neq \hd^V_U$, and so $\hz_U = \hd^U_V(\hz_U)$ by $0$-consistency, while the latter equals $\hf_U$ for some $f \in F \cup \Lambda$ by the BGI property \eqref{item:BGI HFT} of Definition \ref{defn:HFT}.

Finally, Now (2) follows from (1), for (1) says that if $U \notperp V$ and $\hz \in \hh_U \cap \hh_V$, then one of $\hz_U,\hz_V$ must be at a marked or cluster point, which is a contradiction.  This completes the proof.
\end{proof}

We observe the following immediate corollary of item (2) of Lemma \ref{lem:hyperplane char} and \eqref{item:finite dim} of Definition \ref{defn:HFT}, namely finite dimensionality of $\oQ$:

\begin{corollary}\label{cor:dimension}
The number of pairwise intersecting $\oQ$-hyperplanes is bounded by the number of pairwise orthogonal domains in $\calU$, and is hence finite.
\end{corollary}

\subsection{Half-spaces in $\oQ$}\label{subsec:Q hs}

We next observe that the $\oQ$-hyperplanes form a wallspace structure on $\oQ$.

If $h_U \in \hT_U$ is a component hyperplane and $\hpi_U:\calY \to \hT_U$ is the coordinate projection, then $\hh_U = \hpi_U^{-1}(h_U)$.  Since $\hT_U - \{h_U\}$ has two components $\hT_U^+, \hT_U^-$, we get a partition 

$$\oQ^+_{\hh_U} \sqcup \oQ^-_{\hh_U} \sqcup \hh_U = \oQ$$

where $\oQ^{\pm}_{\hh_U} = \hpi_U^{-1}(\hT_U^{\pm}) \cap \oQ$, which we call the $\oQ$-\emph{half-spaces} associated to $\hh_U$.

\begin{notation}
In what follows, we will usually not need to refer to both a half-space and its dual.  Thus we will usually refer to the $\oQ$-half-space associated to some tree hyperplane $h_U$ as $\oQ_U$, and the corresponding half-tree of $\hT_U - \{h_U\}$ as $\hT'_U$.
\end{notation}

\vspace{.1in}

The following lemma gives a description of these half-spaces, and in particular shows both are nonempty:

\begin{lemma}\label{lem:partition}
For each $\oQ$-hyperplane $\hh_U$, the following description holds:
\begin{enumerate}
\item If $U \perp V$, then $\hpi_V(\hh_U) = \hpi_V(\oQ_U) = \hT_V$.
\item If $U \pitchfork V$, then
\begin{itemize}
\item $\hpi_V(\hh_U) = \hd^U_V$.
\item If $\hd^V_U \in \hT'_U$, then $\hpi_V(\oQ_U) = \hT_V$.  Otherwise, $\hpi_V(\oQ_U) = \hd^U_V$.
\end{itemize}
\item If $V \nest U$, then 
\begin{itemize}
\item $\hpi_V(\hh_U) = \hd^U_V(h_U)$.
\item If $\hd^V_U \in \hT'_U$, then $\hpi_V(\oQ_U) = \hT_V$.  Otherwise, $\hpi_V(\oQ_U)  = \hd^U_V(\hT'_U)$. \end{itemize}
\item If $U \nest V$, then
\begin{itemize}
\item $\hpi_V(\hh_U) = \hd^U_V$.
\item If $\hpi_V(\oQ_U) = \hull_{\hT_V}(\hd^U_V \cup l_U)$, where $l_U$ are the labels of the marked points in $\hT_U$ contained in $\hT'_U$.
\end{itemize}
\end{enumerate}

\end{lemma}

\begin{proof}
The proof of the description of the component projections of $\oQ_{U}$ is a straightforward exercise in the definitions utilizing Lemma \ref{lem:hyperplane char} and we leave most of it to the reader.  The one exception is item (4), which deserves a comment.

When $U \nest V$, then $\hT_V - \{\hd^U_V\}$ consists of some number of components, each of which contains some marked point or ray.  By the BGI property in item (5) of Definition \ref{defn:HFT}, each of these components projects to a marked cluster point in $\hT_U$.  By construction, $h_U$ partitions all marked points and rays in $\hT_U$, and so each half-space $\hT'_U$ contains some (nonempty) set of marked points or rays, whose labels we denote by $l_U$.  Hence each component of the subtree $\hull_{\hT_V}(\hd^U_V \cup l_U)$, namely the convex hull in $\hT_V$ of $\hd^U_V$ and $l_U$, projects to a marked point contained in $\hT_U$ by the BGI property \eqref{item:BGI HFT} of Definition \ref{defn:HFT}.  This completes the proof.
\end{proof}

\subsection{Intersection and nesting of half-spaces}\label{subsec:hs nest}

In this subsection, we prove some useful descriptions of half-space intersections and nesting.  We note that the first lemma is an exact version of \cite[Lemma 4.6]{DMS20}.

\begin{lemma}\label{lem:hs int}
Let $h_U \in \hT_U$ and $h_V \in \hT_V$ be component hyperplanes with corresponding $\oQ$-hyperplanes $\hh_U, \hh_V$, and $\oQ_U, \oQ_V$ be choices of half-spaces corresponding to half-trees $\hT'_U \subset \hT_U - \{h_U\}$ and $\hT'_V \subset \hT_V - \{h_V\}$, respectively.  Then $\oQ_U \cap \oQ_V \neq \emptyset$ if and only if one of the following holds:

\begin{enumerate}
\item $U \perp V$;
\item $U = V$ and $\hT'_U \cap \hT'_V \neq \emptyset$;
\item $U \pitchfork V$ and either $\hd^V_U \in \hT'_U$ or $\hd^U_V \in \hT'_V$;
\item $U \nest V$ and either $\hd^U_V \in \hT'_V$ or $\hd^V_U(\hT'_V) \in \hT'_U$.
\end{enumerate}

\end{lemma}

\begin{proof}
Item (1) follows from item (1) of Lemma \ref{lem:partition}.  Item (2) follows from the definition of the half-spaces as inverse images under component-wise projection.  Item (3) is a consequence of item (2) of Lemma \ref{lem:partition}, for if $\hd^V_U \notin \hT'_U$, then $\hpi_V(\oQ_U) = \hd^U_V$, so if $\hd^V_U \notin \hT'_V$, we cannot have $\oQ_U \cap \oQ_V \neq \emptyset$.  Finally, item (4) is a consequence of items (3) and (4) of Lemma \ref{lem:partition}, for if $\hd^U_V \notin \hT'_V$, then $\hpi_U(\oQ_V) = \hd^V_U(\hT'_V)$ by item (3) of that lemma.  Hence if $\oQ_U \cap \oQ_V \neq \emptyset$, then we must have that $\hd^V_U(\hT'_V) \in \hT'_U$, as required.  This completes the proof.
\end{proof}

A similar analysis gives a description of when half-spaces are nested:

\begin{lemma}\label{lem:hs nest}
Using the same notation as above, we have $\oQ_U \subset \oQ_V$ if and only if one of the following holds:

\begin{enumerate}
\item $U=V$ and $\hT'_U \subset \hT'_V$;
\item $U \pitchfork V$ and $\hd^U_V \in \hT'_V$ and $\hd^V_U \notin \hT'_U$;
\item $U \nest V$ and $\hd^U_V \in \hT'_V$ and $\hT'_U \subset \hd^V_U(\hT'_V)$;
\item $V \nest U$ and both $\hd^V_U \notin \hT'_U$ and $\hd^U_V(\hT'_U) \in \hT'_V$.
\end{enumerate}

\begin{itemize}
\item In particular, if $\oQ_U \neq \oQ_V$ and $U,V$ are not orthogonal or equal, then one of $\oQ_U \subset \oQ_V$, $\oQ_U \subset \oQ^*_V$, $\oQ^*_U \subset \oQ_V$, or $\oQ^*_U \subset \oQ^*_V$ holds.  On the other hand, if $\oQ_U = \oQ_V$ or $U,V$ are orthogonal, then none of the above relations holds.
\end{itemize}
\end{lemma}

\begin{proof}
Item (1) follows from the definition of $\oQ_U, \oQ_V$ as preimages under component-wise projection to $\hT_U = \hT_V$.  Items (2)--(4) are direct consequences of items (2)--(4) of Lemma \ref{lem:partition}, respectively.

We now prove the last statement.  The statement follows immediately from (1) when the domain labels for $\oQ_U, \oQ_V$ are the same.  Suppose then that $\oQ_U, \oQ_V$ correspond to half-trees $\hT'_U, \hT'_V$ for tree-hyperplanes $h_U,h_V$ in $\hT_U,\hT_V$, respectively.  When $U \pitchfork V$, if $\oQ_U \nsubset \oQ_V$, then  $\hd^V_U$ is contained in the complement of $\hT'_U$, and hence item (2) above implies that $\oQ_U \subset \oQ^*_V$.  For $U \nest V$, if $\oQ_U \nsubset \oQ_V$, then $\hd^U_V \notin \hT'_V$.  It follows that $\hd^U_V \in \hT_V - \hT'_V$, and hence item (3) implies that $\oQ_U \subset \oQ^*_V$.  See Figure \ref{fig:hs_nest}.  Finally, the claim when $\oQ_U = \oQ_V$ follows immediately from the definitions, while when $U \perp V$, the claim also requires item (1) of Lemma \ref{lem:partition}.  This completes the proof.

\begin{figure}
    \centering
    \includegraphics[width = .8\textwidth]{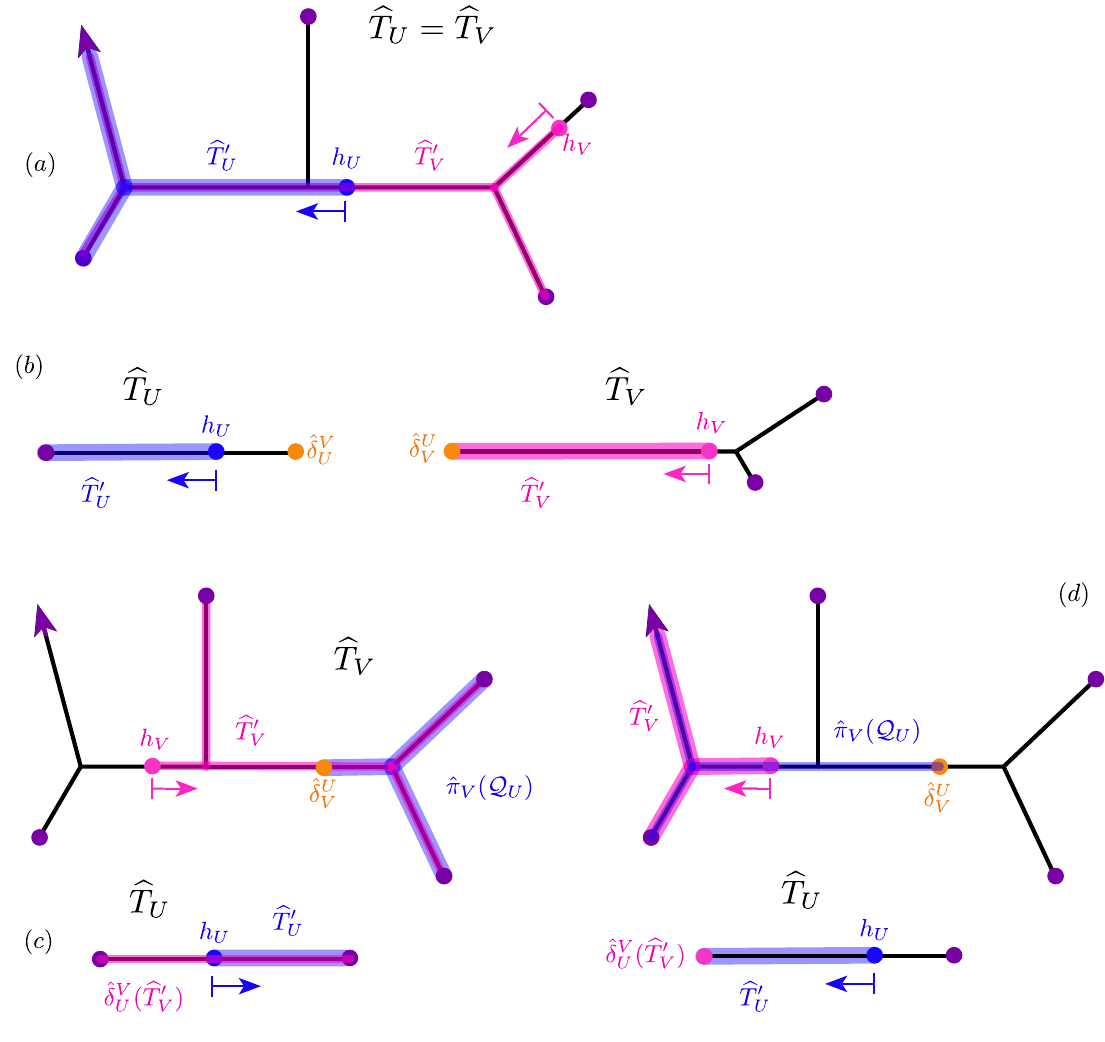}
    \caption{The four nesting possibilities in Lemma \ref{lem:hs nest}: In (a), the domain labels for $\oQ_U, \oQ_V$ are the same and we have $\oQ_U \subset \oQ_V$ because $\hT'_U \subset \hT'_V$.  In (b), $U \pitchfork V$ and $\oQ_U \subset \oQ_V$ because $\hd^U_V \in \hT'_V$ while $\hd^V_U \notin \hT'_U$.  In (c), we have $U \nest V$ and $\oQ_U \subset \oQ_V$ because $\hT'_U \subset \hd^V_U(\hT'_V)$.  In (d), we have $U \nest V$ and $\oQ_V \subset \oQ_U$ because $\hd^V_U(\hT'_V) \subset \hT'_U$.}
    \label{fig:hs_nest}
\end{figure}
\end{proof}

\subsection{Half-spaces are nonempty} \label{subsec:hs nonempty}

In this subsection, we prove that every $\oQ$-half-space contains a point in $\calQ$.  This fact is almost immediate when $\Lambda = \emptyset$, but requires some work in the infinite case.  It is crucial for proving local finite-ness in the proof that the induced wall-space on $\oQ$ forms a pocset (Proposition \ref{prop:pocset}), as well as proving that only points in $\calQ^{\infty}$ are contained in infinite descending chains of half-spaces (Proposition \ref{prop:DCC}).

\begin{lemma} \label{lem:hs nonempty}
For any $\oQ$-half-space $\oQ_U$, we have that $\oQ_U \cap \calQ$ and $\oQ^*_U \cap \calQ$ are nonempty.  Moreover, if $\hT'_U$ is the half-tree of $\hT_U$ corresponding to $\oQ_U$, then the following hold:
\begin{enumerate}
\item If $a \in F$ with $\ha_U \in \hT'_U$, then $\ha \in \oQ_U \cap \calQ$.
\item If $\lambda \in \Lambda$ with $\hlam_U \in \hT'_U \cup \partial \hT'_U$, then $\hlam \in \oQ_U \cap \calQ^{\infty}$.
\end{enumerate}
\end{lemma}

\begin{proof}
Let $\hT'_U$ be the half-tree of $\hT_U$ corresponding to $\oQ_U$.  Note that there is some $z \in F \cup \Lambda$ so that $\hz_U \in \hT'_U \cup \partial \hT'_U$.  First suppose that we can take $z = a \in F$ so that $\ha_U \in \hT'_U$, then $\ha \in \oQ_U\cap \calQ$.  This is because $\left(\oQ_U \sqcup \oQ^*_U \sqcup \hh_U \right)\cap \calQ= \calQ$ is indeed a partition of $\calQ$, and by Lemma \ref{lem:partition} and our assumption that $\ha_U \in \hT'_U$, we cannot have $\ha \in \oQ^*_U \cup \hh_U$.  This proves part (1) of the statement, and the whole statement when $\Lambda = \emptyset$.  Also, observe that part (2) is immediate from this observation as well.
 
The bulk of the work is in when we must take $z = \lambda \in \Lambda$ with $\hlam_U \in \hT'_U \cup \partial \hT'_U$.  Observe that when this happens, the opposite half-tree of $\hT'_U$ contains some marked point for an element of $F \neq \emptyset$, and hence $\oQ^*_U \cap \calQ \neq \emptyset$.  For showing $\oQ_U \cap \calQ \neq \emptyset$, there are a three main cases:

\begin{enumerate}
\item $U \in \supp(\lambda)$ and $\hlam_U \in \partial \hT_U$ represents a ray;
\item $U \in \supp(\lambda)$ and $\hlam_U \in \hT_U$ is a cluster point;
\item $U \notin \supp(\lambda)$.
\end{enumerate}

In each case, the idea is to construct a tuple $\hw \in \oQ_U$ and prove that it is canonical (Definition \ref{defn:Q consistent}).  The arguments are related, and each invokes Strong Passing-Up \ref{prop:SPU} at a key moment.

\underline{Case (1)}: In this case, there are infinitely-many vertices in $\hT'_U$ which are not marked, branched, or cluster points.  Choose a vertex $w \in \hT'_U$ sufficiently out the ray for $\hlam_U$ so that no marked point or branch points separates $w$ from $\hlam_U$.

Define a tuple $\hw$ domain-wise as follows:
\begin{itemize}
\item $\hw_U = w$;
\item If $V \pitchfork U$ or $U \nest V$, then $\hw_V = \hd^U_V = \hlam_V$;
\item If $V \nest U$, then we set $\hw_V = \hd^U_V(h_U) = \hlam_V$;
\item Otherwise, we fix $a \in F$ and for any $V \perp U$, we set $\hw_V = \ha_V$.
\end{itemize}

Observe that $\hw \in \oQ_U$ by construction, in particular it is $0$-consistent.  To see canonicality, let $b \in F$ and set $\calV_b = \{V \in \calU| \hb_V \neq \hw_V\}$.  If $\#\calV_b < \infty$ for all $b \in F$, then $\hw$ is canonical and we are done.

Let $\supp(\lambda) = \{U_1, \dots, U_n\}$, which we recall are pairwise orthogonal (Lemma \ref{lem:big orth}).  Then Lemma \ref{lem:forced support} implies that there is a cofinite subset $\calV' \subset \calV$ which partitions into subsets $\calV'_1, \dots, \calV'_n$ so that $V \nest U_i$ for each $V \in \calV'_i$.

Assuming that $U = U_1$, we observe that when $i>2$, then by Definition \ref{defn:ray projection} we have $\hw_{U_i} = \hlam_{U_i} = \ha_{U_i}$ for our fixed $a \in F$.  Since $\#\Rel_{50E}(a,b)<\infty$ (Corollary \ref{cor:rel sets are finite}), it follows that $\#\calV_i< \infty$ for each $i>2$.

On the other hand, observe that if $V \in \calV'_1$, then $\delta^V_U$ separates $b_U$ from $q^{-1}_U(w)$ in $T_U$ by the BGI property of Lemma \ref{lem:tree control}.  On the other hand, if $\#\calV'_1 = \infty$, then Strong Passing-Up \ref{prop:SPU} implies that the spread of the $\rho^V_U$ in $\calC(U)$, and hence $\delta^V_U$ in $T_U$, must be infinite, which is impossible because $d_{T_U}(b_U, q^{-1}_U(w))<\infty$.  Hence $\#\calV'_1 < \infty$ and $\hw$ is canonical.  This proves case (1).

\vspace{.1in}

\underline{Case (2)}: Now we assume that $U \in \supp(\lambda)$ but $\hlam_U \in \hT_U$ is a cluster point.  By definition of $\hT_U$, since $d_{T_U}(b_U, \lambda_U) = \infty$ for all $b \in F$, it follows that $\#\{V \in \Rel_{50E}(b, \lambda) | V \nest U, \hspace{.1in} \hd^V_U = \hlam_U\} = \infty$.

Let $V \in \calU$ be any $\nest_{\calU}$-minimal domain with $V \nest U$ and $\delta^V_U$ at least sufficiently far past any branched or marked point of $T_U$ along the ray to $\lambda_U$, so that $[b,c] \cap \calN_{E}(\rho^V_U) = \emptyset$ for any geodesic $[b,c]$ between $b,c$ in $\calC(U)$, for all $b \in F$ and $c \in F \cup \Lambda - \{\lambda\}$.  The BGIA \ref{ax:BGIA} then implies that $d_V(b,c)<E$ for all such $b,c$ and hence $d_V(b, \lambda)>K-E$, since $V \in \calU = \Rel_K(F \cup \Lambda)$.

It follows from item (6) of Lemma \ref{lem:collapsed tree control} that $d_{\hT_V}(\hb_V, \hlam_V)\succ K$, and in particular that $\hb_V \neq \hlam_V$ for all $b \in F$.  Thus there exists a vertex $w \in \hT_V$ which is not a marked or cluster point (there are none of the latter) which separates $\hb_V$ from $\hlam_V$ for all $b \in F$.

Now define a tuple $\hw$ analogously to the above, where we now let $V$ control the related projections:

\begin{itemize}
\item $\hw_V = W$;
\item If $V \pitchfork Z$ or $V \nest Z$, then $\hw_Z = \hd^V_Z$;
\item Otherwise, we fix $a \in F$ and for any $V \perp U$, we set $\hw_V = \ha_V$.
\end{itemize}

Note that $\nest_{\calU}$-minimality says that no domain $Z \nest V$.  Hence the above gives a complete tuple and checking $0$-consistency is straight-forward, and $\hw \in \calQ_U$ since $\hw_U = \hd^V_U = \hlam_U \in \hT'_U$ by construction.

For canonicality, we argue in a similar fashion to case (1): If $b \in F$ and $\#\{Z \in \calU| \hb_Z \neq \hw_Z\} = \infty$, then Lemma \ref{lem:forced support} provides a cofinite subset which admits a partition into subsets of domains which nest into the finitely-many support domains $\{U_1, \dots, U_n\} = \supp(\lambda)$.  Since $V \nest U \in \supp(\lambda)$, the coordinate of $\hw_X$ for $X \in \supp(\lambda) - \{U\}$ is $\ha_X$, and hence the only possibly infinite subset corresponds to $U$.  

Let $\calV$ be this infinite subset.  By Strong Passing-Up \ref{prop:SPU}, we can choose $Z \in \calV$ so that $\delta^Z_U$ is much further out the ray toward $\lambda_U$ than $\delta^V_U$ in $T_U$.  In this case we have $V \pitchfork Z \nest U$ and it follows from item (6) of Lemma \ref{lem:tree control} that $d_{T_Z}(\delta^V_Z, \delta^U_Z(\delta^V_Z)) < E'$, while the BGI property of Lemma \ref{lem:tree control} also implies that $d_{T_Z}(b_Z,  \delta^U_Z(\delta^V_Z)<E'$, and so $\hb_Z = \hd^V_Z$, whereas $\hd^V_Z = \hw_Z$, by construction.  This contradicts the assumption that $Z \in \calV$, and hence $\#\calV < \infty$, as required.

\vspace{.1in}

\underline{Case (3)}: The last and final case uses similar ideas.  Since we are assuming that there does not exist $b \in F$ so that $\hb_U \in \hT'_U$, it follows that $U$ is not orthogonal to every domain in $\supp(\lambda)$, as the projection of $\lambda$ to those domains is determined by some point in $F$ (Definition \ref{defn:ray projection}).

Let $V \in \supp(\lambda)$ so that $U$ and $V$ are not orthogonal.  We argue case-wise depending on how $U$ relates to $V$.

When $V \pitchfork U$ or $V \nest U$, then it follows that $\hlam_U = \hd^V_U$.  In both cases, it follows from Lemma \ref{lem:partition} that $\hpi_V(\oQ_U) = \hT_V$.  Now we are in a similar position to cases (1) and (2), where either $\hlam_V$ represents a point in $\partial \hT_V$ or is a cluster point.  In the former case, we argue as in (1) by choosing a vertex $w \in \hT_V$ which is not at a cluster point or branched point, and then completing it to a tuple $\hw$.  Note that $\hd^U_V$ is a marked point by item (2) of Lemma \ref{lem:collapsed tree control}, and hence any vertex will do for the argument.  In the latter case, we again pass to a $\nest_{\calU}$-minimal domain $Z \nest V$ with $\hd^Z_V = \hlam_V$ in which we can enforce that $\hb_Z \neq \hlam_Z$ for any $b \in F \cup \Lambda -\{\lambda\}$.  Analogous arguments suffice to prove canonicality in both cases.

The final case is where $U \nest V$.  In this case, $\hlam_U = \hd^V_U(\hlam_V) = \hd^V_U(\hT^{\lambda}_V)$ where $\hT^{\lambda}_V \subset \hT_V - \{\hd^U_V\}$ is the component containing $\hlam_V$.  The same two subcases as above, namely where $\hlam_V$ is a ray or cluster point, happen again, and similar argument works.

This completes the proof of the lemma.

\end{proof}

\subsection{Descending chains of half-spaces} \label{subsec:hs DCC}

In this subsection, we prove a proposition which gives us control over infinite nested sequences of half-spaces, which will be useful to verify the descending chain condition (Definition \ref{defn:ultrafilter}) in the definition of the map from $\calQ$ to its dual cube complex in Theorem \ref{thm:dual} below, as well as later understanding the simplicial boundary of (the cubical dual of) $\calQ$ in Section \ref{sec:boundary compare}.  We note that this proposition only has content when $\Lambda \neq \emptyset$.

\begin{proposition}\label{prop:DCC}
Let $\hx \in \oQ$.  Then $\hx \in \calQ^{\infty} = \oQ - \calQ$ if and only if there exists a sequence $\{\calQ_{U_i}\}$ of $\oQ$-half-spaces with $\hx \in \oQ_{U_i}$ and $\oQ_{U_{i+1}} \subsetneq \oQ_{U_i}$ for all $i$.
\end{proposition}

\begin{proof}

Both directions of the proof proceed by a similar argument, either relying on the existence of an infinite domain $W \in \calU$ in which $\hx$ picks out a boundary point of $\hT_W$, and otherwise by using Lemma \ref{lem:encoding PU}, which is a consequence of Strong Passing-Up \ref{prop:SPU}, to violate the canonicality condition.

\medskip

\underline{\textbf{No DCC implies non-canonicality}}: First suppose that we have such a sequence for $\hx \in \oQ$.  We claim that $\hx$ is not canonical in the sense of Definition \ref{defn:Q consistent}.

 There are two cases: (1) up to passing to an infinite subsequence, there is a domain $W \in \calU$ so that the defining domains of the $U_i$ all equal $W$, or (2) there is no such infinite subsequence.
 
 In case (1), part (1) of Lemma \ref{lem:hs nest} implies that $\hT'_{U_{i+1}} \subsetneq \hT'_{U_i}$ for all $i$, where $\hT'_{U_i}$ is the half-tree of $\hT_{U_i}$ chosen by its defining tree-hyperplane $h_i$.  The fact that the sequence is infinite and the trees are properly nested then forces the tree-hyperplanes $h_i$ to run out an end of $\hT_W$ corresponding to some $\lambda \in \Lambda$.  Hence each $\hT'_{U_i}$ contains the end labeled by $\hlam_W$, forcing $\hx_W = \hlam_W$, which violates the first part of the definition of canonicality.
 
 In case (2), we use Lemma \ref{lem:encoding PU}, which allows us to pass to a (relabeled) subsequence $\{V_1, V_2, \dots\}$ and provides elements $a \in F$ and $\lambda \in \Lambda$ so that

\begin{enumerate}
\item $V_j \pitchfork V_i$ for $i \neq j$;
\item For $j>i$, we have $\hd^{V_j}_{V_i} = \hlam_{V_i}$ and $\hd^{V_i}_{V_j} = \ha_{V_j}$;
\item For $k<i$, we have $\hd^{V_k}_{V_i} = \ha_{V_i}$ and $\hd^{V_i}_{V_k} = \hlam_{V_k}$. 
 \end{enumerate}
 
 For each $i$, let $h_{V_i}$ be some tree-hyperplane in $\hT_{V_i}$ (which exists by property (1) above), and let $\hT'_{V_i}$ be the half-tree containing $\ha_{V_i}$, thus determining a $\oQ$-half-space $\oQ_{V_i}$.  Observe that, by construction, we have for each $i$ that $\hd^{V_{i+1}}_{V_i} = \ha_{V_i} \in \hT'_{V_i}$ while $\hd^{V_i}_{V_{i+1}} = \hlam_{V_{i+1}} \notin \hT'_{V_{i+1}}$ and similarly $\hd^{V_i}_{V_{i+1}} = \hlam_{V_{i+1}} \notin \hT'_{V_{i+1}}$.   Hence Lemma \ref{lem:hs nest} implies that $\oQ_{V_{i+1}}\subset \oQ_{V_i}$ for all $i$.  Note that part of this discussion is that $\ha_{V_i} \neq \hlam_{V_i}$ for all $i$.

Now item (2) of Lemma \ref{lem:partition} forces $\hx_{V_{i+1}} = \hd^{V_i}_{V_{i+1}} = \hlam_{V_{i+1}} \neq \ha_{V_i}$ for each $i$.  It follows that $\hx$ is not canonical as in Definition \ref{defn:Q consistent}, as required.

\medskip

\underline{\textbf{Non-canonicality implies no DCC}}: Now suppose that $\hx \in \calQ^{\infty} - \calQ$ is not canonical (Definition \ref{defn:Q consistent}).  Then either (a) there exists $U \in \calU$ so that $\hx_U \in \partial \hT_U$ or (b) no such $U$ exists but there does exist a sequence of domains $U_i \in \calU$ with $\hx_{U_i} = \hlam_{U_i} \in \hT_{U_i}$ for some fixed $\lambda \in \Lambda$, with $\hx_{U_i} \neq \ha_{U_i}$ for any $a \in F$ and each $i$.

\underline{Case (a)}: The distance between $\hx_U$ and $\ha_U$ for any $a \in F$ is infinite, and so we can find a sequence of tree-hyperplanes $h_{U,i} \in \hT_U$ each of which separate $\ha_{U}$ from $\hx_{U}$, and which monotically converge to $\hx_U$ in $\hT_U$.  If we define $\hT'_{U,i}$ to be the half-tree of $\hT_U- \{h_{U,i}\}$ containing $\hx_U$, then Lemma \ref{lem:hs nest} implies that $\oQ_{U,i+1} \subset \oQ_{U,i}$ for all $i$, providing the desired sequence of $\oQ$-half-spaces.

\underline{Case (b)}: This case follows from an almost identical application of Lemma \ref{lem:encoding PU}.  The only difference is now that conclusion is different, namely that the chain of $\oQ$-halfspaces $\oQ_{V_{i+1}} \sqsubsetneq \oQ_{V_i}$ produced by the argument is infinite. This completes the proof of the lemma.

\end{proof}

\subsection{$\calQ$ half-spaces form a pocset}

With these preliminaries in hand, we can now prove that the family of $\oQ$ half-spaces with the inclusion relation forms a pocset (Definition \ref{defn:pocset}):

\begin{proposition}\label{prop:pocset}
The set of $\oQ$-half-spaces together with the containment relation $\oQ_U \subset \oQ_V$ forms a locally finite, finite-width pocset where transverse half-spaces are labeled by orthogonal domains.
\end{proposition}

\begin{proof}
We verify the various items in Definition \ref{defn:pocset}.  Items (1) and (2) follow directly from the definitions and the fact that the relation is just inclusion.  The fact that transversality for the pocset corresponds to orthogonality of labels is Lemma \ref{lem:hs nest}.  Finite width is Corollary \ref{cor:dimension}.

The bulk of the work is local finiteness.  Note that local finiteness is immediate when $\Lambda = \emptyset$, because then the set of $\oQ$-hyperplanes is finite.  Hence we are already done when $\Lambda = \emptyset$, and more generally when our setup is that of a hierarchical family of trees (Definition \ref{defn:HFT}).

To see local finiteness when $\Lambda \neq \emptyset$, let $\oQ_U \subset \oQ_Z$ be nested $\oQ$ half-spaces, and $\calV$ be the set of domains labeling the half-spaces in $[\oQ_U, \oQ_Z]$.  Suppose for a contradiction that $[\oQ_U, \oQ_Z]$ is infinite.  Then either $\#\calV$ is finite or infinite.

\vspace{.1in}
\underline{$\#\calV$ is finite}: Then there exists some domain $W \in \calU$ labeling infinitely-many of the half-spaces in $[\oQ_U, \oQ_Z]$.  It follows that there is an infinite ordered subcollection of half-spaces $\{\oQ_{i}\} \subset [\oQ_U, \oQ_Z]$ corresponding to a sequence of tree-hyperplanes $h_i \in \hT_W$, so that the $h_i$ monotonically run out a ray end of $\hT_W$ corresponding to a distinct $\lambda \in \Lambda$, and we can even choose the $h_i$ to be at least $10$-away from any branched, marked, or cluster point of $\hT_W$.

By construction, we must have $\oQ_U \subset \oQ_{i} \subset \oQ_Z$ for each $i$.  Lemma \ref{lem:hs nest} says that $U, Z$ are not orthogonal to $W$, and moreover:
\begin{itemize}
\item If $U \pitchfork W$ or $U \nest W$, then items (2) or (3), respectively, of Lemma \ref{lem:hs nest} imply that $\hd^U_W \in \hT'_i$ for all $i$.
\item If $W \nest U$, then item (4) of Lemma \ref{lem:hs nest} implies that $\hd^U_W(\hT'_U) \in \hT'_i$ for each $i$, where $\hT'_U$ is the half-tree of $\hT_U$ corresponding to $\oQ_U$.  Moreover, $\hd^U_W(\hT'_U)$ is a marked point by the BGI property \eqref{item:BGI HFT} of Definition \ref{defn:HFT}.
\item If $W = U$, then $h_U \in \hT'_i$ for each $i$ by item (1) of Lemma \ref{lem:hs nest}, where $h_U$ is the tree-hyperplane in $\hT_U = \hT_W$ corresponding to $\oQ_U$, and $\hT'_i$ is the associated half-tree of $\hT_W - \{h_i\}$.
\end{itemize}

In each of these cases, the half-trees $\hT'_i$ are forced to nest $\hT'_1 \subsetneq \hT'_{2} \subset \cdots$ and by our choice of the $h_i$, we have that every marked or cluster point in $\hT_W$ is contained in each $\hT'_i$.  Note that this implies that $\oQ_i \subset \oQ_{i+1}$ for each $i$, which is an ascending chain and hence Proposition \ref{prop:DCC} does not apply.  However, we can now derive a contradiction based on where $\oQ_Z$ lies in $\hT_W$.

Observe that if $Z \pitchfork W, Z \nest W$, then $\hd^Z_W$ is at a marked or cluster point, and hence $\hd^Z_W \in \hT'_i$ for all $i$, which violates the assumption that $\oQ_i \subset \oQ_Z$ for all $i$.  Similarly, if $W \nest Z$, then $\hd^Z_W(\hT'_Z) \in \hT_W$ is a marked point, where $\hT'_Z$ is the half-tree of $\hT_Z$ corresponding to $\oQ_Z$, and hence is contained in each $\hT'_i$.  And finally, when $Z = W$, then the tree-hyperplane $h_Z$ corresponding to $\oQ_Z$ must be contained in $\hT'_i$ for sufficiently large $i$.  Applying the appropriate part of Lemma \ref{lem:hs nest} to each of these situations results in the conclusion that $\oQ_Z \subset \oQ_i$ for some $i$, which is a contradiction.

\vspace{.1in}

\underline{$\#\calV$ is infinite}: In this case, we apply Lemma \ref{lem:encoding PU} to obtain an infinite subcollection $\{V_1, V_2, \dots\} \subset [\oQ_U, \oQ_Z]$ of domains and elements $a \in F$ and $\lambda \in \Lambda$ so that:

\begin{enumerate}
\item $V_j \pitchfork V_i$ for $i \neq j$;
\item For $j>i$, we have $\hd^{V_j}_{V_i} = \hlam_{V_i}$ and $\hd^{V_i}_{V_j} = \ha_{V_j}$;
\item For $k<i$, we have $\hd^{V_k}_{V_i} = \ha_{V_i}$ and $\hd^{V_i}_{V_k} = \hlam_{V_k}$. 
 \end{enumerate}

Our goal is to use this family of domains to build a chain of $\oQ$-half-spaces, which is either descending and contains $\oQ_U$, in which case any $\hx \in \oQ_U \cap \calQ$ (which exists by Lemma \ref{lem:hs nonempty}) must actually be in $\calQ^{\infty}$ by Proposition \ref{prop:DCC}, or is ascending and is contained in $\oQ_Z$, in which case we can force $\oQ^*_Z\cap \calQ$ to contain a point in $\calQ^{\infty}$, with both outcomes being a contradiction.

We begin with restricting how $U$ relates to the $V_i$.  First, note that $U$ is not orthogonal to any $V_i$ by Lemma \ref{lem:hs nest} because $\oQ_U \subset \oQ_{V_i}$.  Second, the Covering Lemma \ref{lem:covering} says that $U$ only nests into finitely-many domains in $\calU$, so we may presume that $U$ does not nest into any of the $V_i$.  This allows us to reduce to the case where either (a) $V_i \nest U$ for all $i$ or (b) $U \pitchfork V_i$ for all $i$.  We treat case (a); the proof of case (b) requires replacing occurrences of $\hT^U_*(\hT'_U)$ with $\hT^U_*$, where $*$ is the contextually appropriate $V_i$.

For each $i$, let $h_i$ be the tree-hyperplane in $\hT_{V_i}$ corresponding to $\oQ_{V_i}$, and $\hT'_{V_i}$ the corresponding half-tree of $\hT'_{V_i} - \{h_{V_i}\}$.

We proceed by proving two claims:

\begin{claim}\label{claim:not the same}
If $i < j$, then we cannot have that both $\ha_{V_i}, \hlam_{V_i}$ are contained in the same half-tree of $\hT_{V_i} - \{h_{V_i}\}$ and that $\ha_{V_j}, \hlam_{V_j}$ are contained in the same half-tree of $\hT_{V_j} - \{h_{V_j}\}$.  In particular, we cannot have $\ha_{V_i} = \hlam_{V_i}$ and $\ha_{V_j} = \hlam_{V_j}$ for $i \neq j$.
\end{claim}

\begin{proof}[Proof of \ref{claim:not the same}]
Since $i \neq j$, we have $V_i \pitchfork V_j$, and hence Lemma \ref{lem:hs nest} implies that some pair of $\oQ_{V_i}, \oQ_{V_j}, \oQ^*_{V_i}, \oQ^*_{V_j}$ nests, which is impossible given the assumptions and the characterization of nesting for transversely labeled domains in Lemma \ref{lem:hs nest}.
\end{proof}

Hence we may assume that, for all $i$, we have $\ha_{V_i} \neq \hlam_{V_i}$ and exactly one of $\ha_{V_i} \in \hT'_{V_i}$ or $\hlam_{V_i} \in \hT'_{V_i}$ holds.  

\begin{claim} \label{claim:pick a side}
If $\ha_{V_i} = \hd^{U}_{V_i}(\hT'_U) \in \hT'_{V_i}$ for some $i$, then for any $j>i$, we have $\hd^{U}_{V_j}(\hT'_U) = \ha_{V_j}$.  In particular, we have $\oQ_{V_i} \subset \oQ_{V_{i+1}} \subset \cdots$.

\begin{itemize}
\item Otherwise, we have $\hlam_{V_i} = \hd^{U}_{V_i}(\hT'_U) \in \hT'_{V_i}$ for all $i$, and $\oQ_{V_i} \supset \oQ_{V_{i+1}} \supset \cdots$.
\end{itemize}
\end{claim}

\begin{proof}[Proof of Claim \ref{claim:pick a side}]
Given $\hx \in \oQ_U$ (which exists by Lemma \ref{lem:hs nonempty}), the fact that $\oQ_U \subset \oQ_{V_i}$ and Lemma \ref{lem:hs nest} imply that $\hx_{V_i} = \ha_{V_i}$, while the above reduction says that $\ha_{V_i} \neq \hlam_{V_i} = \hd^{V_j}_{V_i}$ for any $j > i$.  Hence $0$-consistency of $\hx$ gives $\hd^{U}_{V_j}(\hT'_U) = \hx_{V_j} = \hd^{V_i}_{V_j} = \ha_{V_j}$, as required.  The conclusion that $\oQ_{V_i} \subset \oQ_{V_{i+1}} \subset \cdots$ is immediate from Lemma \ref{lem:hs nest}.

If, on the other hand, we have $\hlam_{V_i} \in \hT'_{V_i}$ for all $i$, then Claim \ref{claim:not the same} implies that $\ha_{V_i} \notin \hT'_{V_i}$ for all $i$.  The conclusion that $\oQ_{V_i} \supset \oQ_{V_{i+1}} \supset \cdots$ again follows from Lemma \ref{lem:hs nest}.
\end{proof}

Hence we may assume that either $\oQ_{V_1} \subset \oQ_{V_2} \subset \cdots$ or $\oQ_{V_1} \supset \oQ_{V_2} \supset \cdots$.  In the former case, since $\oQ_{V_i} \subset \oQ_Z$ for all $i$, any $\hx \in \oQ^*_Z \cap \oQ$ (Lemma \ref{lem:hs nonempty}) must be contained in the descending chain $\oQ^*_{V_1} \supset \oQ^*_{V_2} \supset \cdots$ forcing $\hx \in \calQ^{\infty}$ by Proposition \ref{prop:DCC}, contradicting that $\hx \in \oQ = \oQ - \calQ^{\infty}$.  In the latter case, since $\oQ_{U} \subset \oQ_{V_i}$ for all $i$, it follows that any $\hy \in \oQ_U \cap \oQ$ is contained in the above descending chain, and hence Proposition \ref{prop:DCC} and Lemma \ref{lem:hs nonempty} give another contradiction.  This completes the proof of the proposition.

\end{proof}

\section{$\oQ$ is a cube complex}\label{sec:Q is cubical}

In this section, we analyze the geometry of $\calQ$ using the induced wall structure studied in the last section, Section \ref{sec:walls in Q}.  Our main goal is to prove that $\calQ$ is isometric to its cubical dual, via a map which identifies the $0$-cubes at infinity of the dual with $\calQ^{\infty}$.

We remark that it is possible to show that $\calQ$ is a complete connected median subalgebra of $\calY$, hence endowing it with a CAT(0) metric by work of Bowditch \cite[Theorem 1.1]{Bowditch_some}, or when $\Lambda = \emptyset$ then $\calY$ is finite dimensional thus making $\calQ$ quasi-median quasi-isometric to a cube complex \cite[Proposition 2.12]{HP_proj}.  In fact, one need not assume that the collapsed trees are simplicial for this argument, though the fact that it is quasi-isometric to a cube complex (via the work in this section) is useful for proving connectivity.  This version of $\calQ$ has a natural cuboid structure, in the sense of Beyrer-Fioravanti \cite[Subsection 2.4]{BF_cross}.  It is possible that this CAT(0) cuboid approach may be useful in settings where one wants to directly encode distances from the hyperbolic spaces $\calC(U)$ into the distances in the collapsed trees, as these distances are altered when passing to the simplicial versions of the trees.

On the other hand, our approach is more concrete and hands on, as we directly analyze the structure of $\calQ$ and then prove that it is \emph{isometric} to its cubical dual, thereby encoding hierarchical information into the cubical structure of $\calQ$.

\subsection{Separation and distance in $\calQ$}

We begin by describing the $\ell^1$-metric on $\calQ$ in terms of wall separation.  The following definition is standard:

\begin{definition}[Hyperplane separation]
We say a $\oQ$-hyperplane $\hh_U$ \emph{separates} $\hx, \hy \in \oQ$ if there is some choice of $\oQ$-half-spaces $\oQ_U, \oQ^*_U$ so that $\hx \in \oQ_U$ while $\hy \in \oQ^*_U$.
\begin{itemize}
\item For any $\hx,\hy \in \oQ$, we let $\calW(\hx|\hy)$ denote the set of $\oQ$-hyperplanes separating $\hx$ from $\hy$
\end{itemize}
\end{definition}

The following is immediate from the fact that the $\oQ$-hyperplanes are defined as preimages of component-wise projection of the tree-hyperplanes:

\begin{lemma}\label{lem:wall separation}
Let $U \in \calU$ and $h_U \in \hT_U$ be a component hyperplane with its corresponding $\oQ$-hyperplane $\hh_U$.  Then $\hh_U$ separates $\hx,\hy \in \oQ$ if and only if $h_U$ separates $\hx_U, \hy_U$.

\end{lemma}

For $\hx,\hy \in \calQ$ (which, recall, are $0$-consistent tuples of vertices in the $\hT_U$), we let $\calW(\hx|\hy)$ denote the sets of $\oQ$-hyperplanes $\hh_U$ separating $\hx,$ from $\hy$.  As an easy application of this setup, we have:

\begin{proposition}\label{prop:separation}
For any $\hx,\hy \in \calQ$, we have $d_{\calQ}(\hx,\hy) = \#\calW(\hx|\hy).$
\end{proposition}

\begin{proof}
By definition of the $\ell^1$-metric on $\calQ$, we have $d_{\calQ} (\hx,\hy) = \sum_{U \in \calU} d_{\hT_U}(\hx_U,\hy_U)$.  On the other hand, the $\ell^1$-metric on $\hT_U$ says that, for each $U \in \calU$, the distance $d_{\hT_U}(\hx_U,\hy_U)$ is the number of component hyperplanes $h_U$ separating $\hx_U,\hy_U$.  Hence we are done by Lemma \ref{lem:wall separation}.
\end{proof}

\subsection{The cubical dual to $\oQ$ (is $\oQ$)}\label{subsec:dual to Q}

Let $\calW(\oQ)$ denote the collection of all $\oQ$-half-spaces as described at the beginning of Subsection \ref{subsec:Q hs}.  By Proposition \ref{prop:pocset}, $(\calW(\oQ), \subset)$ is a locally finite, finite-width pocset structure (Definition \ref{defn:pocset}).

Let $\calD(\oQ)$ denote the Sageev dual cube complex provided by Theorem \ref{thm:Sageev}, where the internal $0$-cubes $\calD^{(0)}(\calQ)$ correspond to DCC ultrafilters on $(\calW(\calQ), \subset)$, and the $0$-cubes at infinity $\calD^{(0)}(\calQ^{\infty})$ correspond to ultrafilters which are not DCC (see Definition \ref{defn:ultrafilter}).  Equivalently, we think of ultrafilters as consistent choices of $\oQ$-half-spaces (Definition \ref{defn:ultrafilter}).  We endow $\calD(\calQ)$ with the path metric on its $1$-skeleton, i.e. the $\ell^1$- or combinatorial metric.

Our goal in this subsection is to prove that $\calQ$ and $\calD(\calQ)$ are isometric under a map which extends to a bijection at infinity, namely between $\calQ^{\infty}$ and $\calD^{(0)}(\calQ^{\infty})$.  In particular, this will show  that hierarchical consistency and cubical consistency coincide exactly in $\oQ$.

In order to define $\calD^{-1}$ in the case where $\Lambda \neq \emptyset$, we need the following lemma:

\begin{lemma} \label{lem:dual well-defined}
Let $\Sigma \in \calD(\oQ)$ be a consistent collection of $\oQ$-half-spaces.  For each $U$, let $\Sigma_U = \{\oQ_i\}$ denote the half-spaces labeled by $U$, with $\hT'_i$ the corresponding half-trees of $\hT_U$.  Then either
\begin{enumerate}
\item There exists a vertex $\hx_U \in \hT_U$ so that $\hx_U \in \bigcap_i \hT'_i$, or 
\item $\bigcap_i \hT'_i = \emptyset$ and there exists a unique $\lambda \in \Lambda$ so that $\hlam_U \in \hT'_i$ for all $i$.
\end{enumerate}
\end{lemma}

\begin{proof}
First, observe that if $\bigcap_i \hT'_i \neq \emptyset$, then it must contain a vertex.  Next, suppose that $\hx_U, \hy_U \in \hT_U$ are vertices as in the statement.  Then there exists a tree-hyperplane $h_U \in \hT_U$ separating them, while $h_U$ only determines one half-space in $\Sigma$, which is impossible.  Hence, if $\bigcap_i \hT'_i \neq \emptyset$, then it contains a unique vertex.

On the other hand, suppose that $\bigcap_i \hT'_i = \emptyset$.  It follows that $\partial \hT_U \neq \emptyset$, since otherwise $\hT_U$ is a finite tree, while consistency says that all half-spaces intersect, and then the Helly property for trees would imply that the intersection of all tree half-spaces is nonempty, a contradiction.

Let $\lambda^1, \dots, \lambda^n \in \Lambda$ be the labels which (bijectively) determine the points in $\partial \hT_U$.  For each $i$, let $h_i \in \hT_U$ be any tree hyperplane in $\hT_U$ which is at least $10$-away from any branch point along any geodesic from $\hf_U$ to $\hlam^i_U$ in $\hT_U$ for any $f \in F$.  Any particular, there is a unique geodesic ray from any $h_i$ to $\hlam^i_U$ for each $i$.

Now suppose that $J_i, J'_i \in \hT_U$ are some tree-hyperplanes which are further toward $\lambda_i$ than $h_i$, with $J'_i$ further along than $J_i$.  Observe that consistency of the ultrafilter $\Sigma$ implies that if the half-tree $\hT'_{J'}$ chosen by the $\oQ$-half-space in $\Sigma$ determined by $J'_i$ contains $\hlam^i_U$, then so does the half-tree corresponding to $J_i$, as the latter contains the former.  Hence either every such hyperplane escaping toward $\hlam^i_U$ chooses the infinite end corresponding to $\hlam^i_U$, or only finitely-many do.

We now claim that our assumption that $\bigcap_i \hT'_i = \emptyset$ implies that exactly one $\hlam^i_U$ exhibits this behavior, namely that every tree-hyperplane further out than $h_i$ toward $\hlam^i_U$ chooses the half-tree containing $\hlam^i_U$.

To see existence, suppose otherwise; we note the proof of this part is similar to the proof of Lemma \ref{lem:extended helly}.  Then, for each $i$, we can replace each $h_i$ by some further out hyperplane $h'_i$ along the ray toward $\hlam^i_U$, namely the first tree-hyperplane which chooses the half-tree not containing $\hlam^i_U$.  We can then consider the subtree $\hT^{fin}_U$ which is the convex hull of the $h'_i$ and $\hf_U$ for $f \in F$.  By construction, $\bigcap_i \hT'_i  = \bigcap_i (\hT'_i \cap \hT^{fin}_U)$, while also $\hT'_i \cap \hT'_j \cap \hT^{fin}_U \neq \emptyset$ for all $i,j$.  The Helly property for trees thus says that $\bigcap_i \hT'_i \neq \emptyset$, which is a contradiction.

Next, suppose that both $\hlam^i_U$ and $\hlam^k_U$ exhibit this behavior, for $i \neq k$.  We will derive a contradiction of (ultrafilter) consistency using Lemma \ref{lem:hs nest}.  By construction and consistency, we must have that the half-trees $\hT'_{J_i}, \hT'_{J_k}$ chosen by $\Sigma$ for $h_i$ and $h_k$ contain $\hlam^i_U$ and $\hlam^k_U$, respectively.  However, by construction, the opposite half-tree $\hT^*_{J_k}$ to $\hT'_{J_k}$ contains $\hT'_{J_i}$.  If $\oQ^*_{J_k}, \oQ_{J_k}, \oQ_{J_i}$ are their respective $\oQ$-half-spaces, then consistency implies that $\oQ_{J_i} \subset \oQ^*_{J_k}$, which contradicts the assumption that $\oQ_{J_k} \in \Sigma$.  Hence there can be only one end of $\hT_U$ which exhibits this behavior.

Finally, it remains to prove that $\hlam^i_U \in \hT'_j$ for all $j$, that is that every half-tree chooses the end containing $\hlam^i_U$.  In particular, we have that the tree-hyperplane $h_i$ chooses the half-tree $\hT'_{h_i}$ containing $\hlam^i_U$.  Now let $\hT^*_{h_i}$ be the opposite half-tree, and suppose that $h \in \hT^*_{h_i}$ is any other tree-hyperplane.  If $\hT_h$ is the half-tree of $\hT_U - \{h\}$ containing $\hlam^i_U$, then $\hT'_{h_i} \subset \hT_h$, and hence their respective $\oQ$-half-spaces $\oQ_{h_i}$ and $\oQ_{h}$ satisfy $\oQ_{h_i} \subset \oQ_{h}$ by Lemma \ref{lem:hs nest}.  Hence consistency of $\Sigma$ implies that $\hT_h$ must be the the half-tree corresponding to the $\oQ$-half-space in $\Sigma$.  This completes the proof. 
\end{proof}

With this lemma in place, we can define the maps:

\begin{definition}[Dual cubical maps]\label{defn:dual maps}
We define the maps $\calD: \oQ \to 2^{\calW(\calQ)}$ and $\calD^{-1}: \calD(\oQ) \to \oY$ as follows:
\begin{enumerate}
\item[($\calD$)] Given a tuple $\hx = (\hx_U) \in \oQ$, we set $\calD(\hx)$ to the set of orientations where, for all $U\in \calU$ and component hyperplanes $h_U \in \hT_U$, the corresponding $\calQ$-hyperplane $\hh_U$ chooses the $\calQ$-half-space $\oQ_U$ to correspond to the half-space of $\hT_U - \{h_U\}$ containing $\hx_U$.
\item[($\calD^{-1}$)] Given a consistent collection of half-spaces $\Sigma = (\oQ_U) \in \calD(\oQ)$, define $\calD^{-1}(\Sigma) = (\hy_U) \in \oY$ as follows: For each $U \in \calU$, let $\Sigma_U$ denote the set of half-spaces labeled by $U$.  Lemma \ref{lem:dual well-defined} provides two cases:
\begin{enumerate}
\item The intersection of the collection of half-trees of $\hT_U$ corresponding to the $\oQ$-half-spaces in $\Sigma_U$ contains a single vertex of $\hT_U$, which we define to be $\hy_U$.
\item Otherwise, the intersection is empty and Lemma \ref{lem:dual well-defined} produces some $\lambda \in \Lambda$ so that all corresponding half-trees in $\hT_U$ contain $\hlam_U$, and we set $\hy_U = \hlam_U$. 
\end{enumerate}
\end{enumerate}
\end{definition}

\begin{theorem}\label{thm:dual}
Both $\calD$ and $\calD^{-1}$ are well-defined with images contained in $\calD(\oQ)$ and $\oQ$, respectively.  Moreover, the following hold:
\begin{enumerate}
\item $\hx \in \calQ^{\infty}$ if and only if $\calD(\hx) \in \calD(\calQ^{\infty})$.
\item $\Sigma \in \calD(\calQ^{\infty})$ if and only if $\calD^{-1}(\Sigma) \in \calQ^{\infty}$.
\item  $\calD \circ \calD^{-1} = \id_{\calD(\oQ)}: \calD(\oQ) \to \calD(\oQ)$.
\item The restriction $\calD:\calQ \to \calD^{(0)}(\calQ)$ is an isometry.
\item The restriction $\calD:\calQ^{\infty} \to \calD(\calQ^{\infty})$ is a bijection.
\end{enumerate}
\end{theorem}

\begin{proof}

We prove the claims in order:
\medskip

\underline{\textbf{$\Dual$ is well-defined}}: Let $\hx \in \oQ$.  Consistency (item (2) of Definition \ref{defn:ultrafilter}) of the choice of half-spaces $\{\oQ_U\}$ defined by $\Dual(\hx)$ is immediate: if $\hx \in \oQ_U$ and $\oQ_U \subset \oQ_V$, then $\hx \in \oQ_V$.  When $\{\oQ_U\}$ satisfies the DCC, then the fact that $\hx \in \calQ$ is exactly the statement of Proposition \ref{prop:DCC}.  Similarly, when $\{\oQ_U\}$ does not satisfy the DCC, then $\hx \in \calQ^{\infty}$.

\vspace{.1in}

\underline{\textbf{$\Dual^{-1}$ is well-defined}}: Let $\{\oQ_U\}$ be a consistent choice of $\oQ$-half-spaces and let $(\hy_U) = \calD^{-1}(\{\oQ_U\})$ be as in Definition \ref{defn:dual maps}.  This tuple is well-defined by Lemma \ref{lem:dual well-defined}.

We claim that the tuple $\hy = (\hy_U)$ is contained in $\oQ$, and we check $0$-consistency by a domain-wise analysis.

\begin{itemize}
\item Let $U \pitchfork V$ and suppose that $\hy_V \neq \hd^U_V$.  Then there exists tree a tree hyperplane $h_V$ in $\hT_V$ with chosen half-space $\hT'_V$ with $\hd^U_V \notin \hT'_V$.  Let $h_U$ be any tree-hyperplane in $\hT_U$.  We claim that the corresponding chosen half-tree $\hT'_U$ contains $\hd^V_U$.  It will follow then that the intersection of all half-trees in $\hT_U$ contains $\hd^V_U$, and hence $\hy_U = \hd^V_U$.   

To see this, let $\hT^*_U, \hT^*_V$ denote the complementary half-trees to $\hT'_U, \hT'_V$, respectively.  The fact that $\hd^U_V \notin \hT'_V$ implies that $\hd^U_V \in \hT^*_V$.  So if also $\hd^V_U \notin \hT'_U$, then item (2) of Lemma \ref{lem:hs nest} implies that $\calQ_U \nsubset \calQ^*_V$.   But our consistent set of choices of half-spaces chose $\calQ_V$, so $\calQ^*_V$ is not in our set of half-spaces, which violates the assumption that this set was (ultrafilter) consistent, giving a contradiction.

\item Let $U \nest V$ and suppose that $\hy_V \neq \hd^U_V$.  Then there exists some $\calQ$-half-space $\calQ_V$ coming from a tree-hyperplane $h_V$ with corresponding half-space $\hT'_V$ which does not contain $\hd^U_V$.  We claim that for all tree-hyperplanes $h_U$, the chosen corresponding half-tree $\hT'_U$ must contain the point $\hd^V_U(\hT'_V)$.  If not, then part (4) of Lemma \ref{lem:hs nest} implies that $\calQ_V \subset \calQ^*_U$, and hence $\calQ_U \subset \calQ^*_V$, which contradicts the assumption that our chosen $\calQ$-half-space for $h_V$ is $\calQ_V$.
\end{itemize}

Thus we have established the first part of the statement.  It suffices to check the numbered items.

\vspace{.1in}
Item (1) is an immediate consequence of Proposition \ref{prop:DCC}, for if $\hx \in \calQ^{\infty}$, then it is contained in some infinite descending chain of half-spaces, and vice versa.

\vspace{.1in}

For (2), let $\Sigma \in \calD(\calQ^{\infty})$.  In view of Proposition \ref{prop:DCC}, it suffices to show that $\hy \in \oQ_U$ for all $\oQ_U \in \Sigma$.  But this is clear, since $\hy \in \oQ$, and $\hy_U \in \hT'_U$ by construction, where $\hT'_U$ is the half-tree of $\hT_U$ chosen by $\oQ_U$, where $U$ is the label of $\oQ_U$.

\vspace{.1in}

For (3), let $\{\oQ_U\} \in \calD(\oQ)$ be a consistent family of $\calQ$-half-spaces, and let $\calD^{-1}(\{\oQ_U\}) = \hx \in \oQ$ be as defined above.  Then $\calD(\hx)$ is the family of $\oQ$-half-spaces, where for each tree-hyperplane $h_U$, we choose the $\oQ$-half-space whose corresponding tree half-space $\hT'_U$ contains $\hx_U$.  By the coordinate-wise definition of $\calD^{-1}$---which, for each $U$, was either defined as an intersection over all half-trees, or as the unique boundary point in $\partial \hT_U$ which is contained in every half-tree---this $\oQ$-half-space must coincide with the original $\oQ$-half-space in the family $\{\oQ_U\}$.  This completes the proof that $\calD \circ \calD^{-1} = id_{\calD(\oQ)}$ and hence of the theorem.

\vspace{.1in}

For (4) and (5), observe that item (3) implies that $\calD:\oQ \to \calD(\oQ)$ is a bijection, while (1) and (2) imply that this restricts to a bijection $\calD: \calQ^{\infty} \to \calD(\calQ^{\infty})$, as required for (5).  On the other hand, Proposition \ref{prop:separation} implies that $\calD:\calQ \to \calD(\calQ)$ is an isometric embedding, thus making it an isometry.  This completes the proof of the theorem.
\end{proof}

\section{Cubical hierarchy paths and the distance formula} \label{sec:HP and DF}

In this section, we use the cubical structure of the cubical models to produce special hierarchy paths in any HHS $\calX$:

\begin{definition}\label{defn:cubical path}
Let $x \in \calX$ and $y \in \calX^{\infty}$ with $\hO:\calQ(x,y) \to \hull_{\calX}(x,y)$ any a cubical model.  If $\gamma$ is any combinatorial geodesic (ray) in $\calQ(x,y)$ between $\hPsi(x)$ and $\hPsi(y)$, we call $\hO(\gamma)$ a \emph{cubical path} in $\calX$ between $x,y$.
\end{definition}

The following is the main result of this section:

\begin{proposition}\label{prop:cube paths hp}
There exists $L =L(\calX)>0$ so that any cubical path between $x \in \calX$ and $y \in \calX^{\infty}$ is an $L$-hierarchy path.
\end{proposition}

The existence of hierarchy paths and the proof of the lower (and more difficult) bound in the Distance Formula \ref{thm:DF} are usually intertwined.  Notably, this is not the case for us, where we only need to know the existence of an $\ell^1$-geodesic in any cubical model $\calQ$ to obtain the latter as a consequence of Proposition \ref{prop:lower bound}, Corollary \ref{cor:Q qi to H}, and Theorem \ref{thm:dual}:

\begin{corollary}\label{cor:DF lower bound}
There exists $K_0\geq 0$ such that for all $K>K_0$ and any $x, y \in \calX$, we have
$$\sum_{U \in \mathfrak S} \left[d_U(x,y)\right]_K \prec d_{\calX}(x,y).$$
\end{corollary}

\begin{proof}
This statement is \cite[Lemma 4.19]{HHS_2}, whose relatively short proof hinges on the existence of a path in $\calX$ between $x,y$ whose length is coarsely the sum on the left (which they prove exists in \cite[Corollary 4.16]{HHS_2}).  Our proof proceeds as follows:

Let $\calQ$ be any cubical model for $x,y$.  Theorem \ref{thm:dual}, any combinatorial geodesic $\gamma$ between $\hx,\hy$ in the cubical dual $\calD(\calQ)$ to $\calQ$ is a geodesic in $\calQ$, while Corollary \ref{cor:Q qi to H} implies that $\hO(\gamma)$ is sent to a uniform quasi-geodesic between $x,y$.  Finally,  Proposition \ref{prop:lower bound} says that the length of $\hO(\gamma)$ is coarsely bounded below by the left-hand side of the sum in the statement, completing the proof.
\end{proof}

We will begin by describing a procedure due to Niblo-Reeves \cite{NR} for constructing combinatorial geodesics (or rays) between $0$-cells of a cube complex, which we include for completeness and summarize in Proposition \ref{prop:NR path construct}.  With this in hand, we prove that any combinatorial geodesic in a cubical model $\calQ$ for a hierarchical hull $H$ gets sent to a hierarchy path, proving Proposition \ref{prop:cube paths hp}.

\subsection{Niblo-Reeves normal paths}

In \cite{NR}, Niblo-Reeves provide a way of producing very nice combinatorial geodesics (and rays) in any CAT(0) cube complex which traverse the hyperplanes between any pair of $0$-cells in a maximally efficient fashion.  In this subsection, we describe their construction.

Let $X$ be a CAT(0) cube complex, and $x,y$ be $0$-cells in $X$, possibly at infinity.  Let $\calW(x|y)$ denote the hyperplanes in $X$ which separate $x$ from $y$.  For any hyperplanes $H,H' \in \calW(x|y)$, we write $H < H'$ if one half-space of $H$ contains $x$ and the other half-space of $H$ contains both $y$ and $H'$.  The following is straight-forward exercise:

\begin{lemma}\label{lem:NR po}
The relation $<$ is a partial order on $\calW(x|y)$.  Moreover, $H, H' \in \calW(x|y)$ are $<$-incomparable if and only if $H$ and $H'$ cross.
\end{lemma}

Assuming that $x \in X^{(0)}$ is a finite $0$-cell, and $y$ is any $0$-cell (possibly at infinity), we can now recursively define a sequence of vertices $(z_i) \subset X^{(0)}$ using this partial order as follows:

\begin{itemize}
\item We set $z_0 = x_0$.
\item Let $\calZ_1 \subset \calW(z_0|y) = \calW(x|y)$ be the set of hyperplanes which are $<$-minimal.  These are pairwise incomparable and hence pairwise cross by Lemma \ref{lem:NR po}.  We set $z_1$ to be the unique $0$-cell determined by flipping the orientations on the hyperplanes in $\calZ_1$ from the orientations in $z_0$, and leaving all other orientations the same.  Observe that $z_0$ and $z_1$ are opposite corners of a unique cube in $X$, and hence $d^{\infty}_X(z_0,z_1) \leq 1$, with equality when $\calZ_1 \neq \emptyset$.
\item We now define $\calZ_2 \subset \calW(z_1|y)$ to be the $<$-minimal hyperplanes in $\calW(z_1|y)$, and define $z_2$ similarly.
\item Preceding recursively, this process either terminates at $y$ when $y$ is a finite $0$-cell, or limits to $y$ when $y$ is at infinity.
\item For each $i$, let $\gamma_i$ be a combinatorial geodesic segment between $z_i$ and $z_{i+1}$.

\end{itemize}

The following summarizes Niblo-Reeves construction:

\begin{proposition}\label{prop:NR path construct}
For each $i$ for which $\calZ_i \neq \emptyset$, we have $d^{\infty}_X(z_i, z_{i+1}) = 1$.  Moreover, if $X$ is a finite dimensional cube complex, then concatenation of the $\gamma_i$ segments forms a combinatorial geodesic (ray) from $x$ to $y$.  In particular, $\calW(\gamma) = \calW(x|y)$.
\end{proposition}

We call the paths constructed as above \emph{Niblo-Reeves paths}, or sometimes \emph{normal cube paths}.

\begin{remark}
Note that any combinatorial geodesic (ray) between $\hx$ and $\hy$ satisfies $\calW(\gamma) = \calW(\hx|\hy)$, though having this statement in hand will be useful in the next section.
\end{remark}

\subsection{All combinatorial geodesics in $\oQ$ are hierarchical}

In this subsection, we prove that any combinatorial geodesic in $\calQ$ is sent to a hierarchy path in $H$.  We begin with the following lemma, which says that combinatorial geodesics are hierarchical in $\calQ$:

\begin{lemma}\label{lem:NR hp in Q}
Let $\hx \in \calQ$ and $\hy \in \oQ$.  Let $\gamma$ be any combinatorial geodesic (ray) in $\calQ$ from $\hx$ to $\hy$.  Then for all $U \in \calU$, the projection $\hpi_U(\gamma) \subset \hT_U$ is an unparameterized geodesic between $\hx_U$ and $\hy_U$.
\end{lemma} 

\begin{proof}
Let $U \in \calU$ and let $\hh, \hh' \in \calW(x|y)$ be labeled by $U$, corresponding to tree-hyperplanes $h,h' \in \hT_U$.  Without loss of generality, we may assume that $h<h'$ in $\calW_{\hT_U}(\hx_U, \hy_U)$, i.e. the unique (ray) in $\hT_U$ between $\hx_U$ and $\hy_U$ must cross $h$ before $h'$.  Then Lemma \ref{lem:hs nest} implies that $\hh < \hh'$ in $\calW(\hx|\hy)$.  Hence $\gamma$ crosses $\hh$ before $\hh'$, and thus $\hpi_U(\gamma)$ crosses $\hh$ before $\hh'$.  In particular, $\hpi_U(\gamma)$ must project onto the geodesic in $\hT_U$.  Moreover, since combinatorial geodesics do not cross a hyperplane twice, it follows that $\hpi_U(\gamma)$ does not backtrack in $\hT_U$ either.  This completes the proof.
\end{proof}

With this in hand, we prove our hierarchy path statement:

\begin{proposition} \label{prop:cube paths hp}
Let $x, y \in \calX$.  Set $F = \{x,y\}$, $\Lambda = \emptyset$, and let $\calQ$ be the cubical model for $\hx \in \calQ$ and $\hy \in \oQ$.  Let $\gamma$ be any combinatorial geodesic (ray) between $\hx$ and $\hy$.  Then $\hO(\gamma) \subset H$ is a hierarchy path between $\hO(\hx), \hO(\hy)$ with uniform constants depending only on $|F \cup \Lambda|$ and the constants involved in the construction of $\oQ$.
\begin{itemize}
\item Moreover, when $F = \{x,y\}$, $\Lambda = \emptyset$, $\hx = \hPsi(x)$, and $\hy = \hPsi(y)$, then the image in $\calX$ under $\hO$ of a combinatorial geodesic between $\hx,\hy$ in $\calQ$ is a hierarchy path between $x,y$ with constants depending only on $\calX$.
\end{itemize}
\end{proposition}

\begin{remark}
One could use the fact that $\hO:\calQ \to H$ is quasi-median (Theorem \ref{thm:quasi-median} below) plus the useful fact that median quasi-geodesics are hierarchy paths \cite{RST18} to conclude that the image of a combinatorial geodesic (ray) in $H$ is a hierarchy path (or ray).  But the work in \cite{RST18} requires the distance formula (Theorem \ref{thm:DF}), and one consequence of our construction is a new proof of the distance formula, so it is important that we have a direct proof.
\end{remark}

\begin{proof}
Note that the ``moreover'' statement is immediate.  Also, the image under $\hO: \calQ \to H$ of a geodesic is automatically a quasi-geodesic by Corollary \ref{cor:Q qi to H}.  So it suffices to prove that no backtracking occurs in any domain, and note that up to an error involving the choice of largeness constant $K$ (Subsection \ref{subsec:K}) and $|F \cup \Lambda|$, we need only consider domains in $\calU$.

Let $\gamma$ be a combinatorial geodesic (ray) between $\hx \in \calQ$ and $\hy \in \oQ$.  We want to show that for all $U \in \calU$, we have that $\pi_U(\hO(\gamma)) \subset \calC(U)$ is an unparameterized quasi-geodesic.  Recall that $\hO$ is defined coordinate-wise, first by defining a tuple of coordinates in the product of trees $\calY = \prod_{U \in \calU} T_U$, then by pushing this tuple in $\prod_{U \in \calU} \calC(U)$ and applying Realization \ref{thm:realization}; see Subsection \ref{subsec:honing clusters} for a full description.  The first step above involves a map $\Omega: \calQ \to \calZ_{\omega} \subset \calY$, where $\calZ_{\omega}$ is the $\omega$-consistent subset of $\calY$, and $\Omega$ is defined by coordinate-wise maps $B_U:\calQ \to T_U$.  Since the latter two steps only introduce a bounded error in each coordinate, it suffices to show that the coordinate-wise maps $B_U(\gamma) \subset T_U$ is an unparameterized quasi-geodesic in each $T_U$.

\smallskip

Let $(x_U) = (B_U(\hx))= \Omega(\hx)$ and $(y_U) = (B_U(\hy)) = \Omega(\hy)$ be the tuples in $\calY$ determined by $\hx,\hy$, respectively.  Let $r= r(\calX,|F\cup \Lambda|)>0$ be the cluster separation constant (Subsection \ref{subsec:shadows}).  The following claim is sufficient to complete the proof:

\begin{claim}\label{claim:dense shadow}
There exists a constant $\xi = \xi(\calX, |F \cup \Lambda|, K)>0$ such that for all $U \in \calU$, we have that $B_U(\gamma)$ satisfies the following two properties:
\begin{enumerate}
\item If $[x_U, y_U]_U \subset T_U$ is the geodesic in $T_U$ between $x_U, y_U$, then $d^{Haus}_{T_U}(B_U(\gamma), [x_U,y_U]_U)<\xi$, and
\item If $s<t \in [0, d_{\calQ}(\hx,\hy)]$, then $d_{T_U}(B_U(\gamma(s)), y_U) - d_{T_U}(B_U(\gamma(t)),y_U)<\xi$.
\end{enumerate}
\end{claim}

For each $U\in \calU$, we let $\gamma_U = \hpi_U(\gamma)$, and recall that Lemma \ref{lem:NR hp in Q} implies that $\gamma_U$ is an unparameterized geodesic in $\hT_U$.

The proof is by induction on the $\nest_{\calU}$-level of domains in $\calU$.  First, suppose that $U \in \calU$ is $\nest_{\calU}$-minimal.  Then $\hT_U$ contains no cluster points, and hence $B_U(\gamma) = q^{-1}_U(\gamma_U)$.  As such, (1) and (2) are automatic from Lemma \ref{lem:NR hp in Q} and the fact that $\phi_U:T_U \to \calC(U)$ is a uniform quasi-isometric embedding whose image is within uniform Hausdorff distance of $\hull_U(F \cup \Lambda)$ with constants depending only on $\calX$ and $|F \cup \Lambda|$.

Now suppose that $U \in \calU$ is not $\nest_{\calU}$-minimal.  For (1), observe that if $\hz \in \gamma$ so that $\hz_U \in \gamma_U$ is not at a cluster point in $\hT_U$, then $z_U := B_U(\hz) = q^{-1}_U(\hz_U)$ must lie on $[x_U,y_U]_U$.  This is because $q_U:T_U \to \hT_U$ isometrically preserves the edge components of the decomposition $T_U = T^e_U \cup T^c_U$, and this decomposition determines a decomposition of $[x_U,y_U]_U$, whose edge components are in isometric bijective correspondence (via $q_U$) with the components of $\gamma_U$ in the complement of the cluster points it contains.

Hence if $\hz \in \gamma$ with $d_{T_U}(z_U, [x_U,y_U]_U)>0$, then $\hz_U$ must be contained in a cluster point of $\hT_U$.  Suppose, for a contradiction, that $d_{T_U}(z_U, [x_U,y_U]_U) > 100r$.  Let $C \subset T^c_U$ be the cluster component containing $z_U$, and let $W \in \calU$ be any $\nest_{\calU}$-minimal bipartite domain with $d_{T_U}(\delta^W_U, z_U)>10r$ and $d_{T_U}(\delta^W_U, [x_U,y_U]_U)>10r$.  Since $W$ is $\nest_{\calU}$-minimal and bipartite, it is involved in the definition of $z_U, x_U$, and $y_U$.  And since $\delta^W_U$ separates $z_U$ from $x_U,y_U$, it follows that $\hz_W$ is one endpoint of interval $\hT_W$, and $\hx_W = \hy_W$ is the other, where $\hT_W$ is nondegenerate (with length coarsely $K$) by item (6) of Lemma \ref{lem:collapsed tree control}.  However this contradicts Lemma \ref{lem:NR hp in Q}, so (1) is proven.

For (2), suppose for a contradiction that there exist $s< t \in [0, d_{\calQ}(\hx,\hy)]$ and $U \in \calU$ so that $d_{T_U}(B_U(\gamma(s)), y_U) - d_{T_U}(B_U(\gamma(t)),y_U) > 100r$.  Observe that we must have that $\gamma_U(t), \gamma_U(s) \in \hT_U$ are at the same cluster point, otherwise we get a contradiction of Lemma \ref{lem:NR hp in Q}.  Set $s_U = B_U(\gamma(s))$ and $t_U = B_U(\gamma(t))$, and let $C \subset T^c_U$ be the cluster component containing $s_U,t_U$.  Our contradiction assumption implies that there exists a $\nest_{\calU}$-minimal bipartite $W \in \calU$ with $\delta^W_U \subset C$ so that one component of $T_U - \delta^W_U$ contains $x_U$ and $t_U$, while the other component contains $y_U$ and $s_U$, with $x_U,t_U,y_U,s_U$ all at least $10r$-away from $\delta^W_U$.

Since $W$ is involved in the definition of all of $x_U,s_U,y_U,t_U$, we get that $\hx_W = \hat{s}_W \neq \hy_W = \hat{t}_W$ are at opposite ends of the (nondegenerate) interval $\hT_W$, which again contradicts Lemma \ref{lem:NR hp in Q}.  This completes the proof of the claim and hence the proposition.
\end{proof}

\subsection{Hyperplanes and product regions}\label{subsec:active interval}

We finish this section with an observation about how cubically-defined paths in $\calX$ navigate its product regions (see Subsection \ref{subsec:product region}).

\begin{lemma}\label{lem:hyp PU}
Let $\calQ$ be a cubical model for any finite $F \subset \calX$ and finite set of hierarchy rays $\Lambda$.  Then there exists a constant $\nu'>0$ depending only on $\calX$, $|F \cup \Lambda|$, and the constants in the construction of $\calQ$, so that for any $\calQ$-hyperplane $\hh$, if $U \in \calU$ is the domain label for $\hh$, then $\hO(\hh) \subset \calN_{\nu'}(\PP_U)$.
\end{lemma}

\begin{proof}
Let $\hx \in \hh$ and $(x_U) = \Omega(\hx) \in \prod_{U \in \calU} T_U$.  Then Lemma \ref{lem:partition} implies that for any $U \nest V$ or $U \pitchfork V$, we have $\hpi_V(\hx) = \hd^U_V$.   Suppose, for a contradiction, that $d_{T_V}(x_V, \delta^U_V)>100r$, where $r= r(\calX,|F\cup \Lambda|)>0$ is the cluster separation constant (Subsection \ref{subsec:shadows}).

Since $\hpi_V(\hx) = \hd^U_V$, there exists a cluster $C \subset T^c_V$ containing both $x_V, \delta^U_V$.  Moreover, there exists a $\nest_{\calU}$-minimal bipartite $W \in \calU$ with $W \nest V$ so that $d_{T_V}(\delta^W_V, x_V)>10r$ and $d_{T_V}(\delta^W_V, \delta^V_V)>10r$.  Note that $W \pitchfork U$.  When $U \nest V$, an application of Lemma \ref{lem:rho project} implies that $d_{T_W}(\delta^U_W, \delta^V_W(\delta^U_V))<E'$ and hence $\hd^U_W$ is at one endpoint of the (nondegenerate) interval $\hT_W$, whereas $\hx_W$ is at the other endpoint, which contradicts Lemma \ref{lem:partition}.  Hence no such $W$ exists, and $d_{T_V}(x_V,\delta^U_V)$ is bounded in terms of the construction, showing that $\hO(\hh) \subset \calN_{\nu'}(\PP_U)$ for $\nu'>0$ controlled by the setup.

On the other hand, when $U \pitchfork V$, then there exists $a \in F \cup \Lambda$, so that $d_{T_V}(\delta^V_U, a_V)<E'$.  Since $\delta^W_V$ separates $x_V$ from $a_V$, it follows that $\hx_W$ and $\ha_W = \hd^V_W$ are opposite endpoints of the (nondegenerate) interval $\hT_W$, which contradicts Lemma \ref{lem:partition}.  This completes the proof.
\end{proof}

\section{Median structures, hierarchical and cubical}\label{sec:medians}

In this section, we complete the proof of our cubical model theorem, Theorem \ref{thmi:main model}, by explaining why the map $\hPsi:H \to \calQ$ is quasi-median.

The purpose of isolating this discussion in its own section is to make transparent the ways medians behave in the various parts of the construction, instead of having them sprinkled throughout the paper.  Incidentally, this shows how disconnected medians are from the underlying construction, as compared to Bowditch's construction \cite{Bowditch_hulls}, where the median property is a crucial part of the axiomatic setup.  To be sure, the median property of Behrstock-Hagen-Sisto's original construction \cite{HHS_quasi} has proven useful in essentially every application \cite{HHS_quasi, HHP, DMS20, DZ22}.

We begin the section with a very brief overview on (coarse) medians.  We then define hierarchical coarse medians and cubical medians, and then explain how the former winds itself to within bounded distance of the latter in the cubical model $\calQ$.

\begin{remark}
In a cube complex, one can take medians between triples of $0$-cubes which are possibly at infinity.  This requires the formalism of a median algebra, which we do not need here for our purposes.  In the same vein, one can define a median algebra structure on the set of extended $0$-consistent tuples $\oQ$ in any of our cubical models.  We leave this for future work.
\end{remark}

\subsection{(Coarse) median overview}

In \cite{Bowditch_coarsemedian}, Bowditch introduced the notion of a coarse median space, to generalize some of the properties that median spaces possess, see \cite{CDH10} and also Bowditch's book \cite{Bowditch_medianbook} for a thorough treatment of the latter.  We begin with defining a median metric space (see e.g. \cite{verheul1993multimedians, CDH10}):

\begin{definition}[Intervals and medians]
Let $x,y \in X$ be points in a metric space $(X,d_X)$.  The \emph{interval} between $x,y$ in $X$ is the set
$$I(x,y) = \{z \in X| d_X(x,y) = d_X(x,z) + d_X(z,y)\}.$$
\begin{itemize}
\item We say $X$ is a \emph{median space} if for all $x,y,z$ we have that
$$ I(x,y) \cap I(x,z) \cap I(y,z)$$
contains exactly one point, which is called the \textbf{\em{median}} of $x,y,z$ and denote by $\med_X(x,y,z)$.
\end{itemize}
\end{definition}

We get a particularly nice description of this property for graphs:

\begin{definition}[Median graph]\label{defn:median graph}
Let $\Gamma$ be a graph with unit edge lengths endowed with its path metric $d$.  We say $\Gamma$ is a \emph{median graph} if there exists a map $\med_{\Gamma}:\Gamma^3 \to \Gamma$ such that the following holds:
\begin{itemize}
\item If $x,y z \in \Gamma$, then $d(x,y) = d(x,\med_{\Gamma}(x,y,z)) + d(\med_{\Gamma}(x,y,z),y)$, and similarly for the pairs $x,z$ and $y,z$.
\end{itemize}  
\end{definition}

Note that when $\Gamma$ is a tree, then $\med_{\Gamma}(a,b,c)$ is precisely the center of unique tripod subgraph of $\Gamma$ with $a,b,c$ as leaves.

While the above definition may seem specialized, work of Chepoi \cite{Chepoi_median} shows that median graphs are in bijective correspondence with the $1$-skeleta of CAT(0) cube complexes.  This fact motivates the following definition due to Bowditch:

\begin{definition}[Coarse median]\label{defn:coarse median}
Let $(X,d)$ be a metric space and $\med:X^3 \to X$ be an operation satisfying the following properties:
\begin{enumerate}
\item (triples) There exists constants $\kappa, h(0)$ such that for all $a,a',b,b',c,c' \in X$, we have
$$d(\med(a,b,c),\med(a',b',c')) \leq \kappa(d(a,a') + d(b,b') + d(c,c')) + h(0).$$
\item (tuples) There exists a function $h:\mathbb N \cup \{0\} \to [0, \infty)$ so that for any finite $F \subset X$ with $|F| = n \leq \infty$, there exists a CAT(0) cube complex $\calF_n$ and maps $\alpha:F \to \calF^{(0)}_n$ and $\beta:\calF^{(0)}_n \to X$ so that $d(f, \beta(\alpha(f))<h(p)$ for all $f \in F$, and so that
$$d(\beta(\med_{\calF_n}(x,y,z)), \med_X(\beta(x),\beta(y),\beta(z)))\leq h(n)$$
for all $x,y,z \in \calF_n$, where $\med_{\calF_n}$ and $\med_X$ are the medians in $\calF_n$ and $X$, respectively.
\end{enumerate}
\end{definition}

\begin{remark}\label{rem:Bowditch median remark}
We note that the existence of a cube complex in the tuple condition for coarse median spaces is not unrelated to the cubical models constructed in this paper and in \cite{HHS_quasi, Bowditch_hulls}.  The key point is that the maps $\alpha, \beta$ above need only preserve the medians and not necessarily distances.
\end{remark}

Notably, every hyperbolic space is coarse-median, where we can take the approximating cube complexes to be Gromov modeling trees, and medians to be coarse centers of triangles.  To motivate our next observation, we need the following definition:

\begin{definition}[Quasi-median] \label{defn:quasi-median}
A map $\phi:X \to Y$ between coarse median spaces is \emph{$C$-quasi-median} if for any triple $a,b,c \in X$, we have $d_Y(\phi(\med_X(a,b,c)), \med_Y(\phi(a),\phi(b), \phi(c)))< C$.
\end{definition}

The following is an exercise in $\delta$-hyperbolic geometry along the lines of Lemma \ref{lem:ray trees exist}:

\begin{lemma}\label{lem:median hyp}
Let $X$ is $\delta$-hyperbolic, $F \subset X$ and $\Lambda \in \partial \calX$ are finite, and $\phi:T \to \hull_X(F \cup \Lambda)$ be any tree modeling the hull of $F \cup \Lambda$ in $X$.  Then $\phi$ is $C$-quasi-median for some $C = C(\delta)>0$.
\end{lemma}

\subsection{Coarse medians in HHSes}

The following construction was originally contained in \cite{BM_centroid} for mapping class groups, and was later generalized to every hierarchically hyperbolic space as in Definition \ref{defn:coarse median} by Bowditch by \cite{Bowditch_largemap}; see also \cite{HHS_2}.

Given an HHS $(\calX, \mathfrak S)$ and $a,b,c$, consider the tuple $(\med_U(a,b,c))_{U \in \mathfrak S} = (\med_U(\pi_U(a), \pi_U(b), \pi_U(c)))_{U \in \mathfrak S}$.  The following is contained in \cite[Lemma 2.6]{HHS_2}:

\begin{lemma}
The tuple $(\med_U(a,b,c))_{U \in \mathfrak S}$ is $\theta_{\med}$-consistent for some $\theta_{\med} = \theta_{\med}(\calX)>0$.
\end{lemma}

Hence we can define a map $\med_{\calX}:\calX^3 \to \calX$ by taking $\med_{\calX}(a,b,c)$ to be any realization point for the tuple $(\med_U(a,b,c))_{U \in \mathfrak S}$ as provided by Realization \ref{thm:realization}, and this map $\med_{\calX}$ gives a coarse median structure on $\calX$ \cite[Theorem 7.3]{HHS_2}.

\subsection{Media in hierarchical hulls}

For the rest of this section, we fix our usual setup of a finite set $F \cup \Lambda$ of internal points and hierarchy rays in an HHS $\calX$.  Fix the hierarchical hull threshold $\theta_0= \theta_0(\calX, |F \cup \Lambda|)>0$ provided by Proposition \ref{prop:ray replace boundary}, and let $H = \hull_{\theta_0}(F \cup \Lambda)$ be the hierarchical hull and $\ret_H:\calX \to H$ be the coarse retraction provided by Lemma \ref{lem:gate retract}.

Recall that $H$ has an induced HHS structure (Lemma \ref{lem:hull induced structure}), in which the index set is $\calU$ and the hyperbolic spaces are $\{\hull_U(F \cup \Lambda)| U \in \calU\}$, with projections being compositions $r_U \circ \pi_U:H \to \hull_U(F \cup \Lambda)$, where $r_U:\calC(U)\to \hull_U(F \cup \Lambda)$ is the hyperbolic retraction onto the quasi-convex set $\hull_U(F \cup \Lambda)$.  For any $a,b,c \in H$, let $\med_H(a,b,c)$ be the coarse median of $a,b,c$ in $H$, using this induced HHS structure.  Finally, let $i_H:H \to \calX$ be the inclusion, which is defined by completing the partial tuples defined by points in $H$ by setting, for  any $V \in \mathfrak S - \calU$ $f_V$ for fixed $f \in F$. 

The following lemma is straight-forward, and says that all of the various reasonable ways of building a median coincide:

\begin{lemma}\label{lem:med convex}
There exists $A = A(\calX, |F \cup \Lambda|, \theta_0)>0$ so that the following hold:
\begin{enumerate}
\item The map $i_H:H \to \calX$ is an $A$-median $(A,A)$-quasi-isometric embedding,
\item For any $a,b,c \in H$, each of $d_{\calX}(\med_{\calX}(a,b,c), \ret_H(\med_{\calX}(a,b,c))),$ and $i_H(\med_H(a,b,c))$ are at pairwise distance at most $A$.
\end{enumerate}
\end{lemma}

\begin{proof}
The fact that $i_H:H \to \calX$ is an $(A,A)$-quasi-isometric embedding for $A = A(\calX, |F \cup \Lambda|, \theta_0)>0$ is an application of the Distance Formula \ref{thm:DF}, since each of the hyperbolic hulls $H_V$ of $\pi_V(F \cup \Lambda) \subset \calC(V)$ for $V \in \mathfrak S - \calU$ has uniformly bounded diameter (controlled by our largeness constant $K$), and hence we can use a threshold so that all of those terms vanish. 

For the rest of the statement, note that by definition $\med_{\calX}(a,b,c) = (\med_U)_{U \in \mathfrak S}$ is a realization point for the tuple of hyperbolic medians of the projections of $a,b,c$, which are uniformly close to $H_U$ for each $U \in \mathfrak S$, by definition of $H$.  Since each $H_U$ is uniformly quasi-convex in the hyperbolic space $\calC(U)$ and hence median quasi-convex (with constants depending only on $\calX, |F \cup \Lambda|, \theta_0$), each coordinate $\med_U$ is uniformly close to the hyperbolic hull $H_U = \hull_{U}(a,b,c)$.  On the other hand, the retraction $\ret_H$ is defined by realizing the tuple obtained by coordinate-wise closest-point projecting $\pi_U(\med_{\calX}(a,b,c))$ to $H_U$.  Hence the coordinates of $\med_{\calX}(a,b,c)$ and $\ret_H(\med_{\calX}(a,b,c))$ are uniformly close, and hence the points themselves are uniformly close by the Distance Formula \ref{thm:DF}.

It remains to see that $d_{\calX}(\med_{\calX}(a,b,c), i_H(\med_H(a,b,c))<A$, but this again is essentially by construction, since each of the coordinates for $U \in \calU$ used to define $\med_H(a,b,c)$ and $\med_{\calX}(a,b,c)$ are uniformly close, while the coordinates for $V \in \mathfrak S - \calU$ are again bounded in terms of $K$.  This completes the proof.
\end{proof}

\begin{remark}
Median quasi-convexity of hulls is contained in \cite{HHS_2, RST18}, but we have included a simple proof for completeness.  Note that the proof uses the Distance Formula \ref{thm:DF}, which is not a problem for us, as our median considerations are essentially independent of the rest of the construction, and in particular of our proof of the upper bound of the distance formula (Corollary \ref{cori:DF and HP}). 
\end{remark}

In particular, Lemma \ref{lem:med convex} allows us to work with medians as defined in the hull $H= \hull_{\calX}(F \cup \Lambda)$.

\subsection{Median structure of $\calQ$}

In this subsection, we will prove that $\calQ$ admits a median structure directly from the definitions.  In this subsection, we need not assume that the $\hT_U$ are simplicial trees, as produced by Corollary \ref{cor:simplicial structure}.

We begin by describing intervals in $\calQ$:

\begin{lemma}\label{lem:Q intervals}
For any $\ha,\hb \in \calQ$, we have $I_{\calQ}(\ha,\hb) = \{\hz \in \calQ| \hz_U \in [\ha_U,\hb_U] \hspace{.05in}  \text{for all} \hspace{.05in} U \in \calU\}$.
\end{lemma}

\begin{proof}
Let $\ha,\hb \in \calQ$ and let $\hz \in I_{\calQ}(\ha,\hb)$, which we observe contains any combinatorial geodesic between $\ha,\hb$.  Note that the $\ell^1$-distance sums corresponding to $d_{\calQ}(\ha, \hb), d_{\calQ}(\ha, \hz),$ and $d_{\calQ}(\hz, \hb)$ all have finitely-many positive terms.  Hence we have the following by definition of $\hz \in I_{\calQ}(\ha,\hb)$,

\begin{eqnarray*}
\sum_{V \in \calU} d_{\hT_V}(\ha,\hb) &=& d_{\calQ}(\ha, \hb)\\
&=& d_{\calQ}(\ha,\hz) + d_{\calQ}(\hz,\hb)\\
&=& \sum_{V \in \calU} d_{\hT_V}(\ha_V, \hz_V) + \sum_{V \in \calU} d_{\hT_V}(\hz_V, \hb_V).
\end{eqnarray*}

But if $\hz_U \notin [\ha_U,\hb_U]$ for some $U \in \calU$, then $d_{\hT_U}(\ha_U, \hz_U) + d_{\hT_U}(\hz_U,\hb_U) \gneq d_{\hT_U}(\ha_U,\hb_U)$.  Since the triangle inequality forces $d_{\hT_W}(\ha_W, \hz_W) + d_{\hT_W}(\hz_W,\hb_W) \geq d_{\hT_W}(\ha_W,\hb_W)$ for all $W \in \calU$, this says that the last line of the above equation is bigger than the first, giving us a contradiction and completing the proof.
 
\end{proof}

Observe that each $\hT_U$ is a tree, and hence has a median metric structure, where $\med_{\hT_U}(\ha_U,\hb_U,\hc_U)$ is the center point of the unique tripod with $\ha_U,\hb_U,\hc_U \in \hT_U$ as its leaves.  Equivalently, 
$$\med_{\hT_U}(\ha_U,\hb_U,\hc_U) = [\ha_U, \hb_U] \cap [\ha_U, \hc_U] \cap [\hb_U, \hc_U] = I_{\hT_U}(\ha_U,\hb_U) \cap I_{\hT_U}(\ha_U, \hc_U) \cap I_{\hT_U}(\hb_U, \hc_U),$$
where $[\ha_U, \hb_U]$ is the unique geodesic in $\hT_U$ between $\ha_U, \hb_U$, etc.

The natural candidate for a median for $\calQ$ is to take the tuple of medians in each coordinate, namely an exact version of the HHS setup:

\begin{lemma} \label{lem:Q median}
For any $\ha,\hb,\hc \in \calQ$, we have $\left(\med_{\hT_U}(\ha_U,\hb_U,\hc_U)\right) \in \calQ$.
\end{lemma}

\begin{proof}
For each $U \in \calU$, set $\med_U = \med_{\hT_U}(\ha_U,\hb_U,\hc_U)$.  We claim that the tuple $(\med_U)$ is $0$-consistent and canonical.

For consistency, let $U \pitchfork V$.  By $0$-consistency of $\ha,\hb,\hc$, up to switching the roles of $U,V$ and substituting $\hc$ for one of $\ha,\hb$, we have $\ha_V = \hb_V = \hd^U_V$, with Lemma \ref{lem:Q intervals} forcing $\med_V = \hd^U_V$, since it lies on the unique geodesic between $\ha_V$ and $\hb_V$ which is a point.

Now let $U \nest V$, and we may assume that $\med_V \neq \hd^U_V$, otherwise we are done.  Up to exchanging the roles of $\ha,\hb,\hc$, we have that $[\ha_V, \hb_V] \cap \hd^U_V = \emptyset$.  Then $0$-consistency and the BGI property of Lemma \ref{lem:collapsed tree control} forces that $\ha_U = \hb_U = \hd^V_U([\ha_V, \hb_V]) = \hd^V_U(\med_V)$, which shows that $\med_U = \hd^V_U(\med_V)$, as required.

Finally, for canonicality, observe that since $\ha,\hb,\hc$ are canonical, we have that $\ha_U =\hb_U =\hc_U$ for all but finitely-many $U \in \calU$, and hence $(\med_U)$ is canonical by Definition \ref{defn:Q consistent}.  This completes the proof.
\end{proof}

\begin{remark}[Median structure on $\oQ$]
It is possible to show that $\oQ$ admits the structure of a median algebra, where the median is defined as above.  Notably, medians of non-canonical tuples need not be non-canonical, just as medians of triples of $0$-cubes in the Roller boundary of a cube complex can be contained in the interior of the space.  The median structure on $\oQ$ seems useful and we leave any investigation of it for future work.
\end{remark}

With these lemmas in hand, we can prove our structural statement:

\begin{proposition}\label{prop:Q median structure}
For any $\ha,\hb,\hc \in \calQ$, set $\med_{\calQ}(\ha,\hb,\hc)  = \left(\med_{\hT_U}(\ha_U,\hb_U,\hc_U)\right)$.  Then

$$I_{\calQ}(\ha,\hb) \cap I_{\calQ}(\ha, \hc) \cap I_{\calQ}(\hb,\hc) = \{\med_{\calQ}(\ha,\hb,\hc)\}.$$
\begin{itemize}
\item In particular, the map $\med_{\calQ}:\calQ^3 \to \calQ$ defined a metric median structure on $\calQ$.
\end{itemize}
\end{proposition}

\begin{proof}
By Lemma \ref{lem:Q intervals}, any point in $I_{\calQ}(\ha,\hb) \cap I_{\calQ}(\ha, \hc) \cap I_{\calQ}(\hb,\hc)$ must have $U$-coordinate $\med_{\hT_U}(\ha_U,\hb_U,\hc_U)$ for each $U \in \calU$, so we are done by Lemma \ref{lem:Q median}, which says that the tuple $\med_{\calQ}(\ha,\hb,\hc)  = \left(\med_{\hT_U}(\ha_U,\hb_U,\hc_U)\right)$ is contained in $\calQ$.  This completes the proof.
\end{proof}

\begin{remark}[$\calQ$ is a median sublagebra of $\calY$]
In fact, one can show that taking tuples of medians (as we have done above) constitutes a median algebra structure on $\calY = \prod_{U \in \calU} \hT_U$, and hence $\calQ$ is a median subalgebra because it is closed under taking medians.  This fact is useful, as discussed at the beginning of this section, to prove that $\calQ$ admits a CAT(0) metric, even when the $\hT_U$ are not simplicial trees.  See \cite[Subsection 3.1]{Bowditch_medianbook} for the appropriate definitions.
\end{remark}

We can also use this median structure to prove that any tree-trimming map produced by Theorem \ref{thm:tree trimming} is $0$-median:

\begin{corollary}\label{cor:tt median}
Any tree-trimming map $\Delta:\calQ \to \calQ'$ satisfying the assumptions of Theorem \ref{thm:tree trimming} is $0$-median.
\end{corollary}

\begin{proof}
The map $\Delta:\calQ \to \calQ'$ is a coordinate-wise combination of collapsing maps $\Delta_U:\hT_U \to \hT'_U$, which only collapses subtrees to points.  Hence each $\Delta_U$ preserves the coordinate medians, meaning $\Delta$ preserves global medians.
\end{proof}

\subsection{Cubical media}

Our next goal is to understand the median structure on a CAT(0) cube complex, so that we can show that the above median structure on $\calQ$ and its cubical dual $\Dual(\calQ)$ coincide.

Intervals in cube complexes have particular nice descriptions, which we state at the following standard lemma \cite{Chepoi_median}:

\begin{lemma} \label{lem:cubical interval}
Let $a, b \in X$ be $0$-cells of a CAT(0) cube complex.  Then 
\begin{enumerate}
\item Any $\ell^1$-geodesic between $a,b$ is contained in $I_X(a,b)$.
\item Any $0$-cell $z \in I_X(a,b)$ is contained in an $\ell^1$-geodesic between $a,b$.
\item For any $0$-cell $z \in I_X(a,b)$ corresponds to a consistent collection of half-spaces  with the following orientation constraints:
\begin{itemize}
\item If $\hh$ is a hyperplane which separates $a$ from $b$, then $z$ is in the half-space of $\hh$ containing $a,b$.
\item Otherwise, $\hh$ can choose either orientation for $z$.
\end{itemize}
\end{enumerate}
\end{lemma}

This allows us to give the following equivalent characterization of the median of a triple of points in a cube complex:

\begin{lemma}\label{lem:cube median}
For any three $0$-cells in a CAT(0) cube complex $a,b,c \in X$, the median $\med_X(a,b,c)$ is the $0$-cell corresponding to tuple of half-space orientations where each hyperplane $\hh$ chooses the half-space containing the majority of $a,b,c$.
\end{lemma}

\subsection{Cubical medians in $\calQ$}

In this subsection, we now assume that $\calQ$ has the simplicial structure as produced by Corollary \ref{cor:simplicial structure}, namely where each $\hT_U$ is a simplicial tree.

\begin{proposition}\label{prop:Q cube median}
The dualization map $\Dual:\calQ \to \Dual(\calQ)$ is $0$-median.
\end{proposition}

\begin{proof}
Since $\Dual:\calQ \to \Dual(\calQ)$ is an isometry with inverse $\Dual^{-1}$, it suffices to prove that $\Dual^{-1}: \Dual(\calQ) \to \calQ$ is $0$-median.  

Let $a,b,c \in \Dual(\calQ)$ be consistent, canonical collections of orientations and $\ha,\hb,\hc$ their images under $\Dual^{-1}$.  Lemma \ref{lem:cube median} says that $\med_{\Dual(\calQ)}(a,b,c)$ is the unique set of orientations, where each hyperplane chooses the half-space containing the majority of $a,b,c$.  Recall that Lemma \ref{lem:wall separation} says that a hyperplane $\hh_U$ separates $\hx, \hy \in \calQ$ if and only if its component hyperplane $h_U$ separates $\hx_U$ from $\hy_U$.

Hence for each $\calQ$-hyperplane $\hh_U$, we have that $\med_{\Dual(\calQ)}(a,b,c)$ chooses the half-space of $\calQ$ corresponding to the half-tree of $\hT_U - \{h_U\}$ containing the majority of $\ha_U,\hb_U,\hc_U$.  But this means that $\med_{\Dual(\calQ)}(a,b,c)$ forces all tree-hyperplanes in $\hT_U$ to choose the half-tree containing $\med_{\hT_U}(\ha_U,\hb_U,\hc_U)$, meaning that $\left(\Dual^{-1}(\med_{\Dual(\calQ})(a,b,c))\right)_U = \med_{\hT_U}(\ha_U,\hb_U,\hc_U)$ for all $U \in \calU$, completing the proof.
\end{proof}

\subsection{$\hPsi:H \to \calQ$ is quasi-median}

Finally, we prove that our map $\hPsi:H \to \calQ$ (Subsection \ref{subsec:hPsi defined}) is quasi-median, which completes the proof of Theorem \ref{thmi:main model}:

\begin{theorem}\label{thm:quasi-median}
There exists $QM = QM(\calX, |F\cup \Lambda|)>0$ so that the map $\hPsi:H \to \calQ$ is $QM$-quasi-median.
\end{theorem}

\begin{proof}
It suffices to prove that $d_{\calQ}(\hPsi(\med_H(a,b,c)), \med_{\calQ}(\ha,\hb,\hc))$ is bounded in terms of $\calX, |F \cup \Lambda|$, where $a,b,c \in H$ and $\hPsi(a)= \ha$, $\hPsi(b)= \hb$, and $\hPsi(c)= \hc$.  Set $\hPsi(\med_H(a,b,c)) = \hmed'$ and $\med_{\calQ}(\ha,\hb,\hc)) = \hmed$.  Hence it suffices to bound $d_{\hT_U}(\hmed'_U, \hmed_U)$ for each $U \in \calU$.

Recall that the map $\hPsi:H \to \calQ$ is defined coordinate-wise via $\hpsi_U = q_U \circ \psi_U = q_U \circ \phi_U^{-1} \circ r_U \circ \pi_U$, where $r_U:\calC(U) \to \phi(T_U)$ is the retraction onto $\phi_U(T_U)$, $\phi_U:T_U \to \calC(U)$ is the tree-model map (which is a uniform quasi-median quasi-isometry by Lemmas \ref{lem:ray trees exist} and \ref{lem:median hyp}), and the quotient map $q_U:T_U \to \hT_U$ which collapses cluster subtrees to points.

Set $a_U = \psi_U(a) = \phi_U^{-1} \circ r_U \circ \pi_U(a)$, and similarly for $b,c$.

\begin{claim}\label{claim:median tree level}
For each $U \in \calU$, we have $d_{T_U}(\med_{T_U}(a_U,b_U,c_U), \psi_U(\med_{H}(a,b,c)))$ is bounded in terms of $\calX, |F \cup \Lambda|$.

\end{claim}

\begin{proof}
By combining Lemmas \ref{lem:median hyp} and \ref{lem:med convex}, we have that $r_U \circ \pi_U(\med_{H}(a,b,c))$ is uniformly close (in $\calX, |F \cup\Lambda|$) to $\phi_U(\med_{T_U}(a_U,b_U,c_U))$ in $\calC(U)$, and hence the claim follows from Lemma \ref{lem:median hyp}.
\end{proof}

\begin{claim}\label{claim:collapse median}
For any $x,y,z \in T_U$, we have $q_U(\med_{T_U}(x,y,z)) = \med_{\hT_U}(\hx,\hy,\hz)$.  That is, $q_U:T_U \to \hT_U$ is $0$-median.
\end{claim}

\begin{proof}
This follows immediately from the fact that $q_U$ just collapses subtrees to points.
\end{proof}

With these in hand, we can now prove the theorem.  Let $U \in \calU$.  Claim \ref{claim:median tree level} implies that $d_{T_U}(\med_{T_U}(a_U,b_U,c_U), \psi_U(\med_{H}(a,b,c)))$ is bounded, while Claim \ref{claim:collapse median} says that $q_U(\med_{T_U}(x,y,z)) = \med_{\hT_U}(\hx,\hy,\hz$ coincide.  On the other hand, $q_U:T_U \to \hT_U$ is $1$-Lipschitz, so we are done.
\end{proof}

\subsection{Restating main theorem in full generality}

We finish this section by restating our cubical model theorem in full generality.

\begin{theorem}[Cubical model theorem, general version]\label{thm:model general}
Let $\calX$ be a (proper) hierarchically hyperbolic space, $F$ a nonempty finite set of interior points, and $\Lambda$ a finite set of hierarchy rays, with $H$ their hierarchical hull.  Let $\calU = \Rel_K(F \cup \Lambda)$ be the set of relevant domains where $K$ is as in Subsection \ref{subsec:K}.  Let $\{T_U\}_{U \in \calU}$ be the trees in a reduced tree system for $F \cup \Lambda$ as in Definition \ref{defn:tree axioms}, and let $\{\hT_U\}_{U \in \calU}$ be the family of collapsed trees with $\calQ$ the $0$-consistent canonical subset of $\prod_{U \in \calU} \hT_U$.  Then there exists $C= C(\calX,|F\cup \Lambda|)>0$ so that:

\begin{enumerate}
\item The map $\hPsi:H \to \calQ$ (Subsection \ref{subsec:hPsi defined}) is a $C$-quasi-median $(C,C)$-quasi-isometry.
\item We can replace the trees $\{\hT_U\}_{U \in \calU}$ with their simplicial versions as in Corollary \ref{cor:simplicial structure}, and this induces a $0$-median $(C,C)$-quasi-isometry between canonical $0$-consistent sets.
\item The simplicial version of $\calQ$ is 0-median isometric to its cubical dual (Subsection \ref{subsec:dual to Q}).
\end{enumerate}
\end{theorem}

\begin{proof}
That any family of Gromov modeling trees for $H_U = \hull_{\calC(U)}(F \cup \Lambda)$ (which exist by Lemma \ref{lem:ray trees exist}) for each $U \in \calU$ can be converted into a reduced tree system (Definition \ref{defn:tree axioms}) is explained in Subsection \ref{subsec:tree axioms}.  The rest of the construction happens in that generality, with the first conclusion being Theorem \ref{thm:dual}, the second conclusion being Corollary \ref{cor:simplicial structure}, and the third conclusion being a combination of Corollary \ref{cor:Q qi to H} and Theorem \ref{thm:quasi-median}.  This completes the proof.
\end{proof}

\section{Simplicial and hierarchical boundaries} \label{sec:boundary compare}

In this section, we analyze the simplicial boundary $\partial_{\Delta} \calQ$ of the cubical models $\calQ$ for the hierarchical hull $H$ of some finite set of points and hierarchy rays $F \cup \Lambda$.  As above, let $K$ be sufficiently large (Subsection \ref{subsec:K}) and $\calU = \Rel_K(F \cup \Lambda)$. 

We begin with a brief discussion of the simplicial boundary of a cube complex (introduced by Hagen \cite{Hagen_simplicial}) and the (untopologized) simplicial structure of the hierarchical boundary $\partial H \subset \partial \calX$ for $H$ (introduced by Hagen, Sisto, and the author \cite{DHS_boundary}).  With these preliminaries in hand, we then begin our investigation in their connection in earnest.  Our ultimate goal is the following theorem:

\begin{theorem}\label{thm:boundary talk}
There exists a simplicial isomorphism $\partial \hO: \partial_{\Delta} \calQ \to \partial H$.  Moreover, $\partial \hO$ extends the map $\hO:\calQ \to H$ in the following sense:

\begin{itemize}
\item If $\calW$ is a UBS representing a simplex $\nu \subset \partial_{\Delta}(\calQ)$ with $U \in \supp(\calW)$ and $\Dust_U(\calW)$ converging to $\pi_U(\lambda) \in \partial \calC(U)$ for some $\lambda \in \Lambda$, then $\partial \hO(\nu)$ contains a boundary $0$-simplex $\sigma$ so that $\supp(\sigma) = U$ and $\bp(\sigma) = \pi_U(\lambda)$.  In fact, all $0$-simplices of $\partial \hO(\nu)$ arise in this way.
\end{itemize}

\end{theorem}

In the above, the \emph{dust} of a UBS refers to (relative) projection data encoded into its hyperplanes, see Definition \ref{defn:dust}.

\subsection{UBSes}

In this subsection, we introduce the notion of a UBS in and the simplicial boundary of a CAT(0) cube complex.  Our treatment follows Hagen \cite{Hagen_simplicial}, but see also \cite{HFF23}.

\medskip 

A UBS in a CAT(0) cube complex is a family of hyperplanes that ``point'' in a given ``direction''.

\begin{definition}[Unidirectional, inseparable, facing triples, and UBSes]\label{defn:UBS}
Let $X$ be a CAT(0) cube complex and $\calW$ a set of hyperplanes in $X$.
\begin{itemize}
\item We say that $\calW$ is \emph{unidirectional} if for every hyperplane $\hh \in \calW$, one half-space of $\hh$ contains all but finitely-many hyperplanes in $\calW$.
\item We say that $\calW$ is \emph{inseparable} if whenever $\hh$ is a hyperplane separating  $\hh_1, \hh_2 \in \calW$, then we have $\hh \in \calW$
\item Three hyperplanes $\hh_1,\hh_2,\hh_3$ in $X$ form a \emph{facing triple} if for each $i \in \{1,2,3\}$, some half-space of $\hh_i$ contains both of $\hh_j, \hh_k$ for $j,k \in \{1,2,3\} - \{i\}$.
\item We say that $\calW$ is a \emph{unidirectional boundary system} (UBS) if it is infinite, unidirectional, inseparable, and contains no facing triples.
\end{itemize}
\end{definition}

The defining properties of a UBS generalize the properties of the family of hyperplanes crossed by a combinatorial geodesic ray \cite[Lemma 3.3]{Hagen_simplicial}:

\begin{lemma}\label{lem:ray UBS}
Let $\gamma$ be a combinatorial geodesic ray in a CAT(0) cube complex $X$ and $\calW(\gamma)$ the set of hyperplanes crossed by $\gamma$.  Then $\calW(\gamma)$ is a UBS.
\end{lemma}

Note that the definition of a UBS does not prevent one from containing infinite subfamilies of hyperplanes, so that every hyperplane in one family crosses (almost) every hyperplane in the other.  In fact, every UBS essentially admits a decomposition into such a collection of families.

\begin{definition}[Almost equivalence, minimality] \label{defn:minimal UBS}
We say that two UBSes $\calV, \calW$ in a CAT(0) cube complex are \emph{almost equivalent}, written $\calV \sim \calW$, if $\#(\calV \triangle \calW) < \infty$.
\begin{itemize}
\item A UBS $\calW$ is \emph{minimal} if any UBS $\calV \subset \calW$ satisfies $\calV \sim \calW$.
\end{itemize}
\end{definition}

By Proposition \ref{prop:pocset}, there is bound on the cardinality of a set of pairwise intersecting hyperplanes in $\calQ$.  Hence the following lemma implies that any UBS in $\calQ$ contains a minimal UBS \cite[Lemma 3.7]{Hagen_simplicial}:

\begin{lemma}\label{lem:min UBS exists}
If a CAT(0) cube complex $X$ contains no infinite family of pairwise-crossing hyperplanes, and $\calW$ contains a UBS, then $\calW$ contains a minimal UBS.
\end{lemma}

The following useful theorem due to Hagen \cite[Theorem 3.10]{Hagen_simplicial} (see also \cite[Theorem A]{Hagen_simp_cor}) implies that every UBS in $\calQ$ admits a \emph{minimal decomposition} in the following sense:

\begin{theorem}\label{thm:minimal decomp}
Suppose that $X$ is a finite dimensional CAT(0) cube complex.  Let $\nu$ be an almost-equivalence class of UBSes.  Then $\nu$ has a representative of the form $\calW = \bigsqcup_{i=1}^n \calV_i$, where $n< \infty$, each $\calV_i$ is minimal, and for all $1 \leq i < j \leq k$, if $\hh \in \calV_j$, then $\hh$ crosses all but finitely-many hyperplanes in $\calV_i$.
\begin{itemize}
\item Moreover, this decomposition is unique in the sense that if $\calW' = \bigsqcup_{i=1}^k\calU_i$ is almost-equivalent to $\calW$, and each $\calU_i$ is minimal, then $n=k$ and, up to re-ordering, we have $\calV_i \sim \calU_i$ for all $i$.
\end{itemize}
\end{theorem}

The above theorem will allow us to focus on understanding the structure of minimal UBSes in $\calQ$.  Recall that $\calQ$ is finite dimensional by Corollary \ref{cor:dimension}, so we can apply this theorem directly.  But before we begin our analysis, we need one more tool.

\begin{definition}
Given a set of hyperplanes $\calV$ in a CAT(0) cube complex $X$, the \emph{inseparable closure} $\oV$ of $\calV$ is the intersection of all inseparable sets of hyperplanes containing $\calV$.
\begin{itemize}
\item Equivalent, $\oV$ consists of $\calV$ along with every hyperplane $\hh$ for which there exists $\hh_1, \hh_2 \in \calV$ so that $\hh$ separates $\hh_1$ from $\hh_2$.
\end{itemize}
\end{definition}

We make a basic but useful observation about inseparable closures of subsets of UBSes in any CAT(0) cube complex:

\begin{lemma}\label{lem:closed UBS}
Suppose that $\calA$ is a UBS in a CAT(0) cube complex and $\calV \subset \calA$ is an infinite subset.  Then $\oV \subset \calA$ and $\oV$ is a UBS.  Moreover, if $\calA$ is minimal, then $\oV$ is minimal and almost equivalent to $\calA$.
\end{lemma}

\begin{proof}
The fact that $\oV \subset \calA$ is by definition of inseparable closure, since $\calA$ is inseparable.  Hence $\oV$ is unidirectional, infinite, inseparable, and has no facing triples because $\calA$ does not, making $\oV$ a UBS.  The fact that $\oV$ is minimal and almost equivalent to $\calA$ is immediate from the definition of minimality.
\end{proof}

\subsection{Simplicial boundary}

Hagen's Theorem \ref{thm:minimal decomp} allows for the definition of the simplicial boundary.

\begin{definition}[Simplices at infinity] \label{defn:UBS simplices}
Let $X$ be a finite dimensional CAT(0) cube complex.
\begin{itemize}
\item A UBS $\calW$ is \emph{$n$-dimensional} if the minimal decomposition of $\calW$ has $n$ factors.
\item A \emph{$0$-simplex at infinity} is an almost-equivalence class of minimal UBSes.
\item A \emph{$n$-simplex at infinity} is an almost-equivalence class of $(n+1)$-dimensional UBSes.
\item If $\nu, \omega$ are simplices at infinity, we say $\nu \leq \mu$ if and only if there exist representatives  $\calV$ of $\nu$ and $\calW$ of $\omega$ so that $\calV \subset \calW$.
\end{itemize}
\end{definition}

We note that the relation $\leq$ defined above is a partial-order by Theorem \ref{thm:minimal decomp}.

\begin{definition}[Simplicial boundary]\label{defn:simplicial boundary}
Let $X$ be a CAT(0) cube complex in which every collection of pairwise crossing hyperplanes is finite.  The \emph{simplicial boundary} $\partial_{\Delta} X$ of $X$ is the geometric realization of the abstract simplicial complex whose set of simplices is the set of simplices of $X$ at infinity, in which $\nu$ is a face of $\omega$ if and only if $\nu \leq \omega$.
\end{definition}

\subsection{Simplices in the HHS boundary}\label{subsec:HHS simplex}

We now turn to discussing the (untopologized) HHS boundary of our hierarchical hull $H = \hull_{\calX}(F \cup \Lambda)$.  We begin with a general definition of the simplicial structure of the HHS boundary $\partial \calX$ of any HHS $\calX$.  This construction was introduced in \cite{DHS_boundary}.

The difference between the following definition and Definition \ref{defn:boundary point} is that we are only describing the simplices in $\partial \calX$, and Definition \ref{defn:boundary point} describes a point in a simplex.  The underlying definitions are otherwise the same.

\begin{definition}[Boundary simplex, support set] \label{defn:HHS simplex}
Given an HHS $\calX$, we can define \emph{boundary simplicies} and their support sets as follows:

\begin{itemize}
\item A \emph{boundary $0$-simplex} $\sigma$ consists of 
\begin{itemize}
\item A \emph{support domain} $\supp(\sigma) \in \mathfrak S$,
\item A single point $\bp(\sigma) \in \partial \calC(\supp(\sigma))$. 
\end{itemize}
\item A \emph{boundary $n$-simplex} $\Delta$ is a collection of $0$-simplices $\{\sigma_1, \dots, \sigma_n\}$ whose supports are pairwise orthogonal.
\begin{itemize}
\item The \emph{support set} $\supp(\Delta)$ of an $n$-simplex $\Delta$ is $\supp(\Delta) = \{\supp(\sigma_1), \dots, \supp(\sigma_n)\}$.
\end{itemize}
\item If $\Delta, \Delta'$ are two boundary simplices, let $\Delta \leq \Delta'$ mean that $\Delta \subset \Delta'$.
\end{itemize}
\end{definition}

Clearly, the relation $\leq$ is partial order.  This will allow us to define the simplicial structure on $\partial H$:

\begin{definition}
For any HHS $\calX$, the \emph{HHS boundary} $\partial \calX$ of $\calX$ is the geometric realization of the abstract simplicial complex whose set of simplices is the set of boundary simplices as in Definition \ref{defn:HHS simplex}, in which a simplex $\Delta$ is a face of $\Delta'$ if and only if $\Delta \leq \Delta'$.  
\end{definition}

\begin{remark}
We note that every boundary point $\lambda \in \partial \calX$ as in Definition \ref{defn:boundary point} determines a unique boundary simplex in the sense of Definition \ref{defn:HHS simplex}:  Indeed, if $\lambda$ is as in Definition \ref{defn:boundary point}, then $\lambda$ lies in the unique simplex $\Delta$ whose $0$-simplices correspond to the $\lambda_U \in \partial \calC(U)$ for each $U \in \supp(\lambda)$, where $\supp(\lambda)$ is pairwise orthogonal by definition.  As noted in Remark \ref{rem:simplex constants}, the associated constants $\{a_U| U \in \supp(\lambda)\}$ with $\sum_{U \in \supp(\lambda)} a_U = 1$ just determine a point in the interior of the simplex $\Delta$.
\end{remark}

We note the similarity between Definition \ref{defn:simplicial boundary} and Definition \ref{defn:HHS simplex}.  In Theorem \ref{thm:boundary iso} below, we will see that this similarity is actually the result of an isomorphism.

\subsection{Simplicial structure the HHS boundary of the hierarchical hull $H$} \label{subsec:H simplex}

Our next goal is to describe the simplicial structure on the HHS boundary of the hierarchical hull $H = \hull_{\calX}(F \cup \Lambda)$.

For this part, we use the induced HHS structure on $H$, as discussed in Lemma \ref{lem:hqc induce} and Subsection \ref{subsec:hier hull}.  Using this structure, we get the following description of the boundary $0$-simplices in the HHS boundary $\partial H$ of $H$, which of course also describes all higher simplices:

\begin{lemma}\label{lem:H simplex}
For every boundary $0$-simplex $\sigma \in \partial H$, there exists $\lambda \in \Lambda$ so that $\supp(\sigma) \in \supp(\lambda)$ and $\bp(\sigma) = \pi_{\supp(\sigma)}(\lambda)$.
\end{lemma}

\begin{proof}
This is immediate from the fact that the boundaries of the hyperbolic spaces in the induced HHS structure on $H$ are completely determined by the projections of the elements in $\Lambda$.
\end{proof}

\subsection{Ray representatives of minimal UBSes}

The main work in the rest of this section is an analysis of the UBSes in the cubical structure of $\calQ$, which is provided by Theorem \ref{thm:dual}.

In this subsection, we analyze minimal UBSes in $\calQ$.  The idea will be to associate to each minimal UBS $\calW$ a domain in $\calU$ on which it is supported, a corresponding point $\lambda \in \Lambda$, and a particular representative of $\calW$ in its almost equivalence class whose ``dust'' in $\calC(U)$ is effectively a ray representing $\pi_U(\lambda)$.

Recall that every hyperplane $\hh \in \calW$ in a UBS $\calW$ is labeled by its \emph{support domain} $\supp(\hh) \in \calU$.  We can consider $\calW_U = \{\hh \in \calW | \supp(\hh) \sqsubseteq U\}$, the set of hyperplanes whose domain labels $V$ satisfy $V \nest U$ or $V = U$.  To any such hyperplane $\hh$ we can associate a point $\Dust_U(\hh) \in \calC(U)$ as follows:

\begin{definition}[Dust]\label{defn:dust}
Let $\calW$ be a UBS and $U \in \calU$.  If $\hh \in \calW_U$ with $\supp(\hh) = V$, the \emph{dust} of $\hh$ in $\calC(U)$, denoted by $\Dust_U(\hh)$, is 
\begin{itemize}
\item $\rho^V_U$ when $V \nest U$, or
\item $\phi_U(q_U^{-1}(h_V))$ when $V = U$.
\end{itemize}
where $q_U:T_U \to \hT_U$ is the quotient map from Definition \ref{defn:collapsed tree} and $\phi_U:T_U \to \calC(U)$ is the map associated to $T_U$.
\end{definition}

Of course, we can apply the dust map $\Dust_U$ to a large collections of hyperplanes, e.g. all of $\calW_U$, and we will frequently do so.

The following is immediate from the construction:

\begin{lemma}\label{lem:dust close}
If $\calW$ is a UBS in $\calQ$ and $U \in \calU$, then
$$\Dust_U(\calW_U) \subset \calN_{2E}(\phi_U(T_U)).$$
\end{lemma}

Our next goal is to show that the dust of a UBS spreads out in at least one domain.

\begin{definition}\label{defn:dust support}
Let $\calW$ be a UBS in $\calQ$.  The \emph{support set} of $\calW$ is
$$\supp(\calW) = \{U \in \calU| \diam_U(\Dust_U(\calW_U)) = \infty\}.$$
\end{definition}

The first observation is that every UBS has non-empty support:

\begin{lemma}\label{lem:dust spread}
For any UBS $\calW$ in $\calQ$, we have $\supp(\calW) \neq \emptyset$.
\end{lemma}

\begin{proof}
We claim that since $\calW$ is infinite, there exists $U \in \calU$ so that $\calW_U$ is infinite.  Now, if $\#\{\hh \in \calW| \supp(\hh) = U\} = \infty$ for some $U \in \calU$, then $U \in \supp(\calW)$ and we are done as the associated tree-hyperplanes must have infinite diameter in $\hT_U$ and hence $\diam_U(\Dust_U(\calW)) = \infty$.  On the other hand, if $\#\{\hh \in \calW| \supp(\hh) = U\}<\infty$ for all $U \in \calU$, then there exists an infinite subset $\calV \subset \calW$ so that if $\hh \neq \hh' \in \calV$, then $\supp(\hh) \neq \supp(\hh')$.  Passing to an infinite subset if necessary, we may assume that $\supp(\calV) \subset \Rel_K(a,b)$ for some fixed $a,b \in F \cup \Lambda$.  In this case, Strong Passing-up \ref{prop:SPU} implies that there exists $W \in \calU$ and an infinite subset $\calV' \subset \calV$ so that $Z \nest W$ for all $Z \in \supp(\calV')$, and $\diam_W(\bigcup_{Z \in \supp(\calV')} \rho^Z_W) = \infty$.  This completes the proof.
\end{proof}

The next goal is to prove that the support set of a minimal UBS contains a unique domain, and moreover that any such minimal UBS admits a ``ray'' representative in its equivalence class.

\begin{lemma}\label{lem:minimal spread}
Let $\calW$ be a minimal UBS in $\calQ$.  Then there exists $U \in \calU$ such that $\supp(\calW) = \{U\}$.  Moreover, for any $A>0$, there exists $\lambda \in \Lambda$ and a minimal UBS $\calA_A$ with $\calA_A \sim \calW$ so that
\begin{enumerate}
\item For all $\hh \in \calA_A$, we have $\supp(\hh) \nest U$ or $\supp(\hh) = U$.
\item $d_U(\Dust_U(\calA_A), \hull_U(F \cup \Lambda - \{\lambda\})) > A.$
\end{enumerate}
\begin{itemize}
\item In particular, we have that $\Dust_U(\calA_A)) \subset \calC(U)$ converges to $\pi_U(\lambda) \in \calC(U)$.
\end{itemize}
\end{lemma}

\begin{proof}

Let $U \in \supp(\calW)$.  By Lemma \ref{lem:dust close} and the definition of $U \in \supp(\calW)$, there exists $\lambda \in \Lambda$ and $f \in F$ so that $\Dust_U(\calW_U)$ spreads out along a ray from $f$ to $\pi_U(\lambda)$ in $\calC(U)$, i.e.,
$$\diam_U(\Dust_U(\calW_U) \cap \calN_{2E}(\hull_U(f, \pi_U(\lambda)))) = \infty.$$

In particular, we can choose an infinite sequence $\calV \subset \calW_U$ so that $\Dust_U(\calV)$ represents $[\pi_U(\lambda)]\in \partial \calC(U)$.

As such, for any $A>0$, we can choose an infinite subsequence $\calV_A \subset \calV$ so that $\Dust_U(\calV_A)$ still represents $[\pi_U(\lambda)]\in \partial \calC(U)$, while also
$$d_U(\Dust_U(\calV_A), \hull_U(F \cup \Lambda - \{\lambda\}))> A.$$

Observe that since $\calW$ is a UBS, we have that $\calV_A$ is unidirectional, infinite, and without any facing triples.  Hence by Lemma \ref{lem:closed UBS}, we have that its inseparable closure $\oV_A$ is a minimal UBS which is almost equivalent to $\calW$.  It remains to prove that $\oV_A$ satisfies the rest of the conclusion in the lemma.

\begin{claim}\label{claim:dust nest}
For all $\hh \in \oV_A$, we have $\supp(\hh) \sqsubseteq U$.
\end{claim}

By definition of $\hh \in \oV_A$, there exist $\oQ$-hyperplanes $\hh_B, \hh_C \in \calV_A$ which are separated by $\hh$.   Let $V = \supp(\hh), B = \supp(\hh_B),$ and $C = \supp(\hh_C)$.

Observe that since $\hh$ separates $\hh_B$ from $\hh_C$, Proposition \ref{prop:pocset} implies that $B,C$ are not orthogonal to $V$.  Since both $B,C \sqsubseteq U$, it follows that $U$ and $V$ are not orthogonal.  Since we are also done if either $V \sqsubseteq B$ or $V \sqsubseteq C$ holds, we have that one of $B \nest V$ or $B \pitchfork V$ holds, and similarly either $C \nest V$ or $C \pitchfork V$ holds.  In particular, $B,C$ have well-defined relative projections to $V$.

Finally, note that if $U \nest V$ or $U \pitchfork V$, then $\hd^B_V = \hd^C_V = \hd^U_V$ because $B, C \nest U$, whereas Lemma \ref{lem:wall separation} requires that $\hpi_V(\hh)$ separate $\hpi_V(\hh_B) = \hd^B_V$ from $\hpi_V(\hh_C) = \hd^C_V$, which is impossible.  Hence $V \sqsubseteq U$, proving the claim. \qed

The following claim now completes the proof of the lemma:

\begin{claim}\label{claim:dust contain}
There exists $\zeta = \zeta(\mathfrak S, |F\cup \Lambda|)>0$ so that if $\hh \in \oV_A$ separates $\hh_B, \hh_C \in \calV_B$, then $\Dust_U(\hh) \subset \calN_{\zeta}(\hull_U(\Dust_U(\hh_B) \cup \Dust_U(\hh_C)))$.
\end{claim}
 
As above, let $V = \supp(\hh), B = \supp(\hh_B),$ and $C = \supp(\hh_C)$.  Let $\hx \in \hh_B$ and $\hy \in \hh_C$, so that $\hx_V \in \hpi_V(\hh_B)$ and $\hy_V \in \hpi_V(\hh_C)$.  Set $h_V = \hpi_V(\hh)$ and $h_U = \hd^V_U$ if $V \nest U$ and the tree-hyperplane $h \in \hT_U$ if $V = U$.

First suppose that $h_U, \hx_U, \hy_U$ are all distinct in $\hT_U$.  Then $h_U$ separates $\hx_U$ from $\hy_U$.  This is because $h_V$ separates $\hx_V, \hy_V$ in $\hT_V$, and is immediate when $V = U$, and otherwise when $V \nest U$, this follows from the BGI property (7) of Lemma \ref{lem:collapsed tree control} and $0$-consistency of $\hx,\hy$.  But this implies that $\Dust_U(\hh)$ coarsely separates $\Dust_U(\hh_A)$ from $\Dust_U(\hh_B)$ in $\calC(U)$, with the coarseness depending only on $\mathfrak S, |F \cup \Lambda|$, so we are done in this case.

Now suppose that $h_A, \hx_U,$ and $\hy_U$ are not distinct.  This forces the non-distinct $\hT_U$-coordinates to be at a cluster point, which forces $V \sqsubsetneq U$ and at least one of $A \sqsubsetneq U$ or $B \sqsubsetneq U$.

There are several subcases, involving similar arguments.  The first main subcase is when $A \pitchfork V$ and $B \pitchfork V$.  In this case, Lemma \ref{lem:partition} implies that $\hx_V = \hd^A_V$ and $\hy_V = \hd^B_V$.  But since $h_V$ separates $\hx_V$ from $\hy_V$, this says that $\hd^A_V \neq \hd^B_V$, which is impossible if $\rho^V_U \notin \mathcal \calN_{2E}(\hull_U(\rho^A_U \cup \rho^B_U))$.

For the remaining subcases, we may assume without loss of generality that $A \nest U$ and $\hd^A_U = \hd^V_U \neq \hy_U$.  The two cases are when $B = U$ and $B \nest U$, and we explain the former case, leaving the latter for the reader.

We may assume that $A \pitchfork V$, otherwise $d_U(\rho^A_U, \rho^V_U) < E$ and we are done.  Let $T_U$ denote the Gromov tree for $F \cup \Lambda$, and let $y_U = q^{-1}_U(\hy_U)$.  Let $C$ be the cluster in $T_U$ containing both $\delta^V_U$ and $\delta^A_U$.  Assuming that $\rho^V_U \notin \mathcal \calN_{\zeta}(\hull_U(\rho^A_U \cup \phi^{-1}_U(y_U)))$ for some $\zeta = \zeta(\mathfrak S, |F \cup \Lambda|)>0$, there exits a $\nest_{\calU}$-minimal bipartite domain $W \in \calU$ so that $W \nest U$, $W \pitchfork U$, $W \pitchfork A$, and $\delta^W_U$ separates $\delta^V_U$ from $\delta^A_U$ and $y_U$ in $T_U$.

This, along with the BGI property (7) of Lemma \ref{lem:collapsed tree control} and Lemma \ref{lem:partition}, implies that $\hpi_W(\hh) = \hd^V_W$ and $\hd^A_W = \hy_W = \hx_W$ are at the two distinct endpoints of $\hT_W$.  But then $0$-consistency implies that $\hd^W_V = \hx_V = \hy_V$, which is a contradiction of the assumption that $h_V$ separates $\hx_V$ from $\hy_V$.  This completes the proof.

\end{proof}

\subsection{General UBSes}

Now let $\calW$ be any UBS in $\calQ$.  The goal is to use Lemma \ref{lem:minimal spread} above to study the components of a minimal decomposition of $\calW$ provided by Hagen's Theorem \ref{thm:minimal decomp}:

We begin with a simple hierarchical fact:

\begin{lemma}\label{lem:force orth}
Suppose that $V,W,U,Z \in \mathfrak S$ so that $V, W \nest U$ with $d_U(\rho^V_U, \rho^W_U)>3E$, while also $Z \perp V$ and $Z \perp W$.  Then $Z \perp U$.
\end{lemma}

\begin{proof}
We cannot have $U \nest Z$ or $U = Z$.  And if $Z \nest U$ or $Z \pitchfork U$, then $d_U(\rho^Z_U, \rho^V_U)<E$ and $d_U(\rho^Z_U, \rho^W_U)<E$, showing $d_U(\rho^V_U, \rho^W_U)<2E$, which is a contradiction.  Hence $Z \perp U$ and we are done.
\end{proof}

\begin{lemma}\label{lem:orth decomp}
Let $\calW$ be a UBS in $\calQ$ and $\calW = \bigcup_{i=1}^n \calV_i$ its minimal decomposition.  Then $\supp(\calV_i) \perp \supp(\calV_j)$ for all $i \neq j$.
\end{lemma}

\begin{proof}

By Lemma \ref{lem:minimal spread}, we have that $\supp(\calV_i) = \{U_i\}$ for some $U_i \in \calU$ and each $i$.  In particular, $\supp(\calW) = \{U_1, \dots, U_n\}$ and any UBS almost equivalent to $\calW$ has the same support set by Hagen's Theorem \ref{thm:minimal decomp}.  Lemma \ref{lem:minimal spread} also provides a representative of each $\calV_i$ so that $\supp(\hh) \sqsubseteq U_i$ for each $\hh \in \calV_i$ and each $i$.

Let $\hh \in \calV_i$.  Then by Theorem \ref{thm:minimal decomp}, for any $j \neq i$, there exists a cofinite subset $\calV'_j \subset \calV_j$ so that $\supp(\hh) \perp \supp(\hh')$ for any $\hh' \in \calV'_j$.  Since $\supp(\calV'_j) = \{U_j\}$ and $\supp(\hh') \sqsubseteq U_j$ for all $\hh' \in \calV'_j$, either some $\hh' \in \calV'_j$ has $\supp(\hh') = U_j$, or there exists $\hh'_1, \hh'_2 \in \calV'_j$ whose $\rho$-sets are far apart in $\calC(U_j)$, and hence Lemma \ref{lem:force orth} implies that $\supp(\hh) \perp U_j$.

Now either there exists infinitely-many $\hh \in \calV_i$ with $\supp(\hh) = U_i$, or otherwise there exist $\hh_1, \hh_2 \in \calV_i$ whose $\rho$-sets are far apart in $\calC(U_i)$ and hence Lemma \ref{lem:force orth} implies that $U_i \perp U_j$.  Either way, we are done.
\end{proof}

The following proposition simply combines Lemmas \ref{lem:minimal spread} and \ref{lem:orth decomp} plus the uniqueness of the minimal decomposition of $\calW$ up to almost-equivalence from Theorem \ref{thm:minimal decomp}:

\begin{proposition}\label{prop:UBS decomp}
For any UBS $\calW$ in $\calQ$, any minimal decomposition $\calW = \bigcup_{i=1}^n \calV_i$ has the following properties:

\begin{enumerate}
\item For each $i$, there exists $U_i \in \calU$ so that $\supp(\calV_i) = \{U_i\}$.
\item For $i \neq j$, we have $U_i \perp U_j$.
\item For each $i$, the trail of dust $\Dust_{U_i}(\calV_i)$ converges to $\pi_{U_i}(\lambda_i) \in \partial \calC(U_i)$ for some $\lambda_i \in \Lambda$.
\end{enumerate}

\begin{itemize}
\item In particular, the $U_i$ and $\lambda_i$ are independent of the choice of representative of $\calW$ from its almost equivalence class. 
\end{itemize}
\end{proposition}

\begin{remark}
In (3) above, the dust set $\Dust_{U_i}(\calV_i)$ is countable.  The claim that its \emph{trail converges} to $\lambda_i$ is that, given any enumeration, the corresponding sequence of points in $\calC(U_i)$ converges to $\lambda_i$.
\end{remark}

\subsection{Visibility of $0$-simplices and surjectivity}

As we will see below in Theorem \ref{thm:boundary iso}, the above structural results are enough to provide a simplicial embedding of $\partial_{\Delta} \calQ$ into $\partial H$.  To see surjectivity, we need to combinatorial rays, such as those constructed in Proposition \ref{prop:NR path construct}.

Recall from Lemma \ref{lem:ray UBS} above that if $\gamma$  is a combinatorial geodesic in $\calQ$, then the set $\calW(\gamma)$ of hyperplanes crossed by $\gamma$ determines a UBS $\calW(\gamma)$.

\begin{definition}[Visibility]\label{defn:visible}
A simplex $\nu$ of $\partial_{\Delta} \calQ$ is \emph{visible} if there exists a combinatorial geodesic ray $\gamma$ such that $\calW(\gamma)$ represents $\nu$.  We say $\partial_{\Delta} \calQ$ is \emph{fully visible} if every simplex is visible.
\end{definition}

The following lemma proves that every simplex in $\partial_{\Delta} \calQ$ is visible, and is moreover required for proving that our boundary map is surjective in Theorem \ref{thm:boundary iso} below. 

\begin{lemma}\label{lem:boundary surj}
Let $\Delta$ be a simplex in $\partial H$, with the $0$-simplices $\sigma_1, \dots, \sigma_n$ of $\Delta$ satisfying $\supp(\sigma_i) = U_i \in \calU$ and $\bp(\sigma_i) = \lambda_i \in \Lambda$ for each $i$.  Then there exists a combinatorial geodesic ray $\gamma$ in $\calQ$ so that the corresponding UBS $\calW(\gamma)$ has $\supp(\calW(\gamma)) = \{U_1, \dots, U_n\}$, and for each $i$, the trail of $\Dust_{U_i}(\calW(\gamma))$ in $\calC(U_i)$ converges to $\pi_U(\lambda_i)$ in $\partial \calC(U)$.
\end{lemma}

\begin{proof}
Define a $0$-consistent tuple $\hx \in \calQ^{\infty}$ as follows:
\begin{itemize}
\item For each $i$, set $\hx_{U_i} = \hlam_{i,U_i}$, i.e., the coordinate of $\hlam_i$ in $\hT_{U_i}$.
\item For any $V \in \calU$ with $V \nest U_i$ or $V \pitchfork U_i$ for some $i$, set $\hx_V = \hd^{U_i}_V = \hlam_{i,V}$.
\item For any $V \in \calU$ with $U_i \nest V$ for some $i$, set $\hx_V = \hd^{U_i}_V = \hlam_V$.
\item Fix $f \in F$, and for any $V \in \calU$ with $V \perp U_i$ for all $i$, set $\hx_V = \hf_V$.
\end{itemize}

It is a straight-forward exercise to confirm that  $\hx$ is $0$-consistent, after observing that the fact that the $U_i$ are pairwise orthogonal means that, for instance, $\hd^{U_i}_V = \hd^{U_j}_V$ if $U_i, U_j \pitchfork V$.  Observe also that $\hx$ is not canonical, because for each $i$ either
\medskip

\begin{enumerate}
\item $\hlam_{i,U} \in \partial \hT_{U_i}$, or
\item  There exists an infinite-diameter cluster $C$ in $T_{U_i}$ which contains a ray limiting to $\lambda_{i,U} \in \partial T_{U_i}$.
\end{enumerate}
\medskip

In item (1), $\hx$ fails property (1) of Definition \ref{defn:Q consistent} of canonicality.  In item (2), there exists an infinite-diameter cluster $C$ in $T_U$ which contains a ray limiting to $\lambda_U \in \partial T_U$.  In this latter case, there is an infinite sequence of bipartite $\nest$-minimal domains $(W_k)$ so that, for all $k$, we have $W_k \nest U_i$ with one endpoint of $\hT_{W_k}$ labeled by $F \cup \Lambda - \{\lambda_i\}$ and the other endpoint labeled by $\lambda_i$, and hence coinciding with $\hlam_{W_k}$, which violates property (2) of the definition of canonicality.

Moreover, we can use the above cases to argue that $\calW(\hf|\hx)$ is infinite as follows:  In case (1), there exists an infinite sequence of hyperplanes labeled by $U_i$ which separate $\hf_{U_i}$ from $\hlam_{U_i}$, and in case (2), each domain in the sequence $(W_k)$ supports a hyperplane $\hh_k$ whose tree-hyperplane $h_k$ separates $\hf_{W_k}$ from $\hlam_{W_k}$, while also $\hh_k$ separates $\hh_{k-1}$ from $\hh_{k+1}$ in $\calQ$ for each $k\geq 2$ (this again uses the fact that each $W_k$ is bipartite and $\nest$-minimal).

Now apply Proposition \ref{prop:NR path construct} to produce a combinatorial geodesic ray $\gamma$ in $\calQ$, so that if $\calW(\gamma) = \calW(\hf|\hx)$.

We claim that $\supp(\calW(\gamma)) = \{U_1, \dots, U_n\}$.  To see this, observe that the above argument shows that $\diam_{U_i}(\Dust_{U_i}(\calW(\gamma))) = \infty$ for each $i$, and hence $U_i \in \supp(\calW(\gamma))$ for each $i$.  On the other hand, $\supp(\calW(\gamma))$ is pairwise orthogonal by Lemma \ref{lem:orth decomp}, so it suffices to eliminate candidate domains $V$ which satisfy $V \perp U_i$ for all $i$.  However, by construction $\gamma$ projects to $\hf_V$ for all $V \perp U_i$ for all $i$, and since $\hf$ is canonical (Definition \ref{defn:Q consistent}) by Proposition \ref{prop:hPsi defined}, we have that $\#\{V \in \calU| \hf_V \neq \ha_V\}< \infty$ for all $a \in F$.  Hence $\Dust_V(\calW(\gamma))$ is contained in a bounded neighborhood of $\pi_V(F)$ for all $V \perp U_i$ for all $i$, and thus no domain orthogonal to all of $\{U_1, \dots, U_n\}$ can be in the support of $\calW(\gamma)$.

Finally, we observe that for each $i$,  Proposition \ref{prop:UBS decomp} implies that $\Dust_{U_i}(\calW(\gamma))$ converges to some boundary point, and hence it must converge to $\pi_{U_i}(\lambda_i) \in \partial \calC(U_i)$ by construction and by the existence of sequences of hyperplanes of type (1) or (2) above.  This completes the proof.

\end{proof} 

As a corollary, we get:

\begin{corollary}\label{cor:visibility}
Every simplex in $\partial_{\Delta} \calQ$ is visible.  In particular, $\partial_{\Delta} \calQ$ is fully visible.
\end{corollary}

\subsection{The map $\partial_{\Delta} \calQ \to \partial H$}

With the above in hand, we can prove our isomorphism theorem:

\begin{theorem}\label{thm:boundary iso}
There exists a simplicial isomorphism $\partial \hO: \partial_{\Delta} \calQ \to \partial H$.  Moreover, $\partial \hO$ extends the map $\hO:\calQ \to H$ in the following sense:

\begin{itemize}
\item If $\calW$ is a UBS representing a simplex $\nu \subset \partial_{\Delta}(\calQ)$ with $U \in \supp(\calW)$ and $\Dust_U(\calW)$ converging to $\pi_U(\lambda) \in \partial \calC(U)$ for some $\lambda \in \Lambda$, then $\partial \hO(\nu)$ contains a boundary $0$-simplex $\sigma$ so that $\supp(\sigma) = U$ and $\bp(\sigma) = \pi_U(\lambda)$.  In fact, all $0$-simplices of $\partial \hO(\nu)$ arise in this way.
\end{itemize}
\end{theorem}

\begin{proof}
Let $\nu$ be a simplex in $\partial_{\Delta} \calQ$.  We define a simplex $\partial \hO(\nu)$ in $\partial H$ as follows:

Let $\calW$ be a UBS representing $\nu$ with minimal decomposition $\calW = \bigsqcup_{i=1}^n \calV_i$.  We now use Proposition \ref{prop:UBS decomp}.

For each $i$, let $\supp(\calV_i) = \{U_i\}$ and let $\lambda_i \in \Lambda$ so that the trail of $\Dust_{U_i}(\calV_i)$ converges to $\pi_{U_i}(\lambda_i) \in \partial \calC(U_i)$.

Define a boundary $n$-simplex $\Sigma$ in $\partial H$ as follows:
\begin{itemize}
\item For each $i$, let $\sigma_i$ be the boundary $0$-simplex with $\supp(\sigma_i) = U_i$ and $\bp(\sigma_i) = \pi_{U_i}(\lambda_i) \in \partial \calC(U_i)$.
\item Set $\Sigma$ to be the unique $n$-simplex with $0$-simplices $\sigma_1, \dots, \sigma_n$.
\end{itemize}

Observe that $\Sigma$ is indeed a boundary $n$-simplex in $\partial H$, and moreover all of its $0$-simplices arise in the way described in the ``moreover'' part of the statement of the theorem.  Also observe that $\partial \hO$ clearly respects the simplicial structures of $\partial_{\Delta} \calQ$ and $\partial H$, and hence defines a simplicial embedding.

Finally, Lemma \ref{lem:boundary surj} immediately gives that $\partial \hO$ is surjective and hence an isomorphism. This completes the proof.

\end{proof}

\subsection{Equivariant boundary isomorphisms} \label{subsec:auto boundary}

In this subsection, we briefly observe that HHS automorphisms induce isomorphisms between simplicial boundaries of a cubical model and its isometric image.  This basically boils down to an application of Corollary \ref{cor:auto model}.

Recall that a map $h:X \to Y$ between cube complexes is a \emph{cubical isomorphism} if it induces a bijection between half-spaces that preserves inclusion and complementation.  The reader can confirm that isomorphic cube complexes have isomorphic simplicial boundaries.

\begin{corollary}\label{cor:auto boundary}
If $g \in \mathrm{Aut}(\calX, \mathfrak S)$ is an HHS automorphism, the isometry $\calQ \to \calQ_g$ provided by Corollary \ref{cor:auto model} is a cubical isomorphism and hence it extends to a simplicial isomorphism $\partial_{\Delta}\calQ \to \partial_{\Delta} \calQ_g$.
\end{corollary} 

\begin{proof}
Since the half-spaces in the cubical structure of $\calQ$ and $\calQ_g$ are determined by the half-trees of the $\hT_U$ and $\hT_{gU}$, respectively, which are isometrically identified via $g$, its clear that the set of half-spaces are in bijective correspondence under $g$.  Moreover, this correspondence clearly preserves complementation and that is preserves nesting is an application of Lemma \ref{lem:hs nest}.  This completes the proof.
\end{proof}

With this corollary, we complete the proofs of the equivariance parts of Theorems \ref{thmi:main model} and \ref{thmi:main boundary} from the introduction.

\section{Capturing curve graph distance from cubical models}\label{sec:sep hyp}

In this section, we give an application of our refined cubical machinery.  The goal is to show that one can encode distance in the top-level hyperbolic space associated to many HHSes $\calX$ by the combinatorics of hyperplanes in the cubical models.  For this purpose, we will work in a slightly restricted class of HHSes, which notably includes all of the main examples, including all HHGs.  We explain this context next.

\subsection{Unbounded products and the ABD structure}\label{subsec:ABD}

The following definitions come from joint work of the author with Abbott and Behrstock \cite{ABD}.

\begin{definition}\label{defn:BDD}
An HHS $\calX$ has 
\begin{itemize}
\item the \emph{bounded domain dichotomy} if there exists $B>0$ so that if $U \in \mathfrak S$ has $\diam(\calC(U))>B$, then $\diam(\calC(U)) = \infty$;
\item \emph{unbounded products} if for any $U \in \mathfrak S$ not $\nest$-maximal with $\diam(\mathbb{E}_U) = \infty$, we have $\diam(\mathbb{F}_U) = \infty$;
\item \emph{clean containers} if for each $U \in \mathfrak S$ and each $V \nest U$ with $\{Z \in \mathfrak S_V| Z \perp V\} \neq \emptyset$, then the associated container provided by the orthogonality axiom \eqref{item:dfs_orthogonal} is orthogonal to $V$.
\end{itemize}
\end{definition}

Note that all HHGs have the bounded domain dichotomy because of the requirement of a cofinite action of the set of domain labels, as do all of the main examples.  On the other hand, some natural HHS structures, like those on right-angled Coxeter groups, do not have unbounded products.  However, all of the main examples have the clean containers property, making the following meaningful \cite[Theorem 3.11]{ABD}: 

\begin{theorem}\label{thm:ABD}
Any hierarchically hyperbolic space $\calX$ with clean containers satisfying the bounded domain dichotomy admits an HHS structure with unbounded products, clean containers, and the bounded domain dichotomy.
\end{theorem}

For the rest of this section, we will refer to the HHS structure in Theorem \ref{thm:ABD} on an HHS with clean containers and the bounded domain dichotomy as the \emph{ABD structure}.

The following is the only consequence of this fact that we will need:

\begin{lemma}\label{lem:UP}
Suppose that $(\calX, \mathfrak S)$ is an HHS with the bounded domain dichotomy and unbounded products.  If $U \in \mathfrak S$ has $\diam \calC(U) = \infty$, then there exists $V \in \mathfrak S$ with $V \perp U$ so that $\diam(\calC(V)) = \infty$.
\end{lemma}

\begin{proof}
If $\diam (\calC(U)) = \infty$, then the fact that $\calX$ is normalized (which we have made a blanket assumption, see Definition \ref{defn:normalized}) means that there exist $x, y \in \calX$ so that $d_U(x,y)$ is as large as we need.  In particular, the Distance Formula \ref{thm:DF} on $\mathbb{F}_U$ implies that $\diam(\mathbb{F}_U) = \infty$.  Since $\calX$ has unbounded products, we have $\diam(\mathbb{F}_U) = \infty$, and hence for any $n$ there exist $a_n,b_n \in \mathbb{E}_U$ with $d_{\mathbb{E}_U}(a_n,b_n)> n$.  Since $\mathbb{E}_U$ also has the bounded domain dichotomy with constant $B$, the distance formula on $\mathbb{E}_U$ then implies that for some sufficiently large $N_B$, we have
$$\sum_{V \perp U} [d_V(a_{N_B},b_{N_B})]_B > B,$$
which implies that for at least one $V \perp U$, we have $d_V(a_{N_B},b_{N_B})>B$, and hence $\diam( \calC(V)) = \infty$.  Since $U \perp V$, we are done.
\end{proof}

\subsection{Separated hyperplanes and the statement of the theorem}

Our goal is to encode curve graph distance via hyperplane combinatorics in cubical models.  In order to complete the setup, we need the following definition due to Genevois \cite{Gen_hyp}:

\begin{definition}[Separated hyperplanes]\label{defn:sep hyp}
Two disjoint hyperplanes $H_1,H_2$ in a CAT(0) cube complex $X$ are $L$-separated ($L \geq 0$) if the number of hyperplanes meeting them both and containing no facing triple has cardinality at most $L$.  A set $\{H_i\}$ hyperplanes is an \emph{$L$-separated chain} if each $H_i$ separates $H_j$ from $H_k$ for all $j<i<k$, and the set is pairwise $L$-separated.
\begin{itemize}
\item For any $x, y\in X^{(0)}$ and $L\geq 0$, we let $d^X_L(x,y)$ denote the length of any maximal $L$-separated chain of hyperplanes separating $x,y$.
\end{itemize}
\end{definition}

Notably, Genevois proves that this function $d^X_L$ defines a distance on the vertices of any CAT(0) cube complex $X$.  While we will not be able to define a new distance on an HHS using our current cubulation techniques, we can capture distance in the top-level hyperbolic space of a broad through restricted class of HHSes as follows:

\begin{theorem}\label{thm:sep hyp}
Let $(\calX,\mathfrak S)$ be an HHS with the bounded domain dichotomy equipped with its ABD structure, and $\nest$-maximal domain $S \in \mathfrak S$.  For any $x,y \in \calX$, there exists a cubical model $\calQ$ for $\hull_{\calX}(x,y)$, so that
$$d_{\calC(S)}(x,y) \asymp d^{\calQ}_0(x,y),$$
where the coarse constants associated to $\asymp$ depend only on the structure $(\calX, \mathfrak S)$.
\end{theorem}

The lower bound is \cite[Theorem C]{DZ22}, though there we need to take some separation constant $L = L(\calX)\geq0$, as the limiting models from that paper do not encode domain relations in hyperplane labels, as is the case with our cubical models $\calQ$.  We give a short proof of the above better lower bound next in Lemma \ref{lem:sep LB} which does not require a special cubical model.  The upper bound requires adapting an observation from the world of surfaces.

\subsection{Lower bound}

The following is an improved version of \cite[Theorem C]{DZ22}:

\begin{lemma}\label{lem:sep LB}
Let $(\calX,\mathfrak S)$ be an HHS with the bounded domain dichotomy equipped with its ABD structure, and $\nest$-maximal domain $S \in \mathfrak S$.  For any $x \in \calX$ and $y \in \calX \cup \calX^{\infty}$ and any model $\calQ$ for $\hull_{\calX}(x,y)$, we have
$$d_{\calC(S)}(x,y) \prec d^{\calQ}_0(x,y),$$
with constants depending only on $\calX$ and the constants in the construction of $\calQ$.
\end{lemma}

\begin{proof}
Let $\gamma'$ be a combinatorial geodesic (ray) in $\calQ$ between $\hx = \hPsi(x)$ and $\hy = \hPsi(y)$, where $\hPsi:\hull_{\calX}(x,y) \to \calQ$ is a quasi-isometry from Corollary \ref{cor:Q qi to H} with constants depending only on $\calX$.  Let $\gamma = \hO(\gamma')$ be the corresponding cubical path (ray) between $x,y$ in $\calX$.  Let $\calW(\gamma')$ denote the set of hyperplanes $\gamma'$ crosses in $\calQ$.

It follows from Claim \ref{claim:dense shadow} of Proposition \ref{prop:cube paths hp} and Lemma \ref{lem:dust close} that $d^{Haus}_S(\Dust_S(\calW(\gamma')), [x,y]_S) < \alpha = \alpha(\calX, \calQ)$ is uniformly bounded in terms of $\calX$ and the construction of $\calQ$, where $[x,y]_S$ is any geodesic in $\calC(S)$ between $x,y$.

Choose a $(10\alpha, \alpha)$-net $z_1, \dots, z_n$ along $[x,y]_S$ (Definition \ref{defn:net}), and for each $i$ let $w_i \in \pi_S(\gamma)$ so that $d_S(z_i, w_i) < \alpha.$  For each $i$, let $\hh_i \in \calW(\gamma')$ be some hyperplane crossing a cube adjacent to $\gamma'(t_i)$, where $w_i \in \pi_S(\hO(\gamma'(t_i)))$.  Finally, for each $i$, let $U_i \in \Rel_K(x,y)$ denote the domain label for $\hh_i$.

We claim that the $(\hh_i)$ form a $0$-separated chain between $\hx,\hy$.  To see this, note that as long as $\alpha > 5E$ and $i \neq j$, we have that $d_S(\Dust_S(\hh_i,, \hh_j) > 5E$, and hence there does not exist a domain $V \in \mathfrak S$ so that $V \perp U_i$ and $V \perp U_j$.  In particular, not only do $\hh_i$ and $\hh_j$ not cross, no hyperplane in $\calQ$ crosses both of them, for any $i \neq j$ by item (2) of Lemma \ref{lem:hyperplane char}.

It follows then that $d^{\calQ}_0(\hx,\hy) \geq n \asymp d_S(x,y)$, where the constants in $\asymp$ depend only on $\calX$ and the choices made in the construction of $\calQ$, and in particular not on $x,y$.  This completes the proof.
\end{proof}

\subsection{Standing assumption}

For the rest of this section, we fix $(\calX, \mathfrak S)$ to be an HHS with the bounded domain dichotomy and unbounded products endowed with its ABD structure, for which $S \in \mathfrak S$ is the $\nest$-maximal domain.  The goal is to prove the reverse bound of Lemma \ref{lem:sep LB} for some variation of our cubical model construction, whose definition involves an idea from the world of surfaces.

\subsection{Motivation: Filling curve pairs}

Let $\Sigma$ be a finite type surface and $\alpha,\beta$ simple closed curves.  Then $\alpha, \beta$ are \emph{filling} if, up to isotopy, the complement $\Sigma - (\alpha \cup \beta)$ is a union of disks and once-punctured disks.  Let $\calC(\Sigma)$ denote the curve graph of $\Sigma$, whose vertices are isotopy classes of simple closed curves on $\Sigma$, with two classes connected by an edge if and only if they have disjoint representatives.  Then $\alpha, \beta$ are filling if and only if $d_{\calC(\Sigma)}(\alpha, \beta) \geq 3$, because if $\gamma$ is a simple closed curve disjoint (up to isotopy) from $\alpha \cup \beta$, then $d_{\calC(\Sigma)}(\alpha, \beta) \leq 2$ by the triangle inequality.

Suppose now that $\mu,\mu'$ are a pair of (complete) markings on $\Sigma$ (in the sense of \cite{MM00}) which have sufficiently large projections to the curve graphs associated to a family of subsurfaces $\calV$ of $\Sigma$.  If $\alpha$ is any geodesic between the projections of $\mu,\mu'$ to $\calC(\Sigma)$, then the Bounded Geodesic Image Theorem \cite[Theorem 3.1]{MM00} says that $d_{\calC{\Sigma}}(\alpha, \partial V) \leq 1$ for any $V \in \calV$.  It is not hard to show that if in addition the subsurfaces in $\calV$ have pairwise filling boundaries, then $\diam_{\calC(\Sigma)}(\bigcup_{V \in \calV} \partial V) \asymp \#\calV$, since each has to be pairwise at least distance $3$ away, while also being distance $1$ from a geodesic.

In the rest of this section, we will first generalize this filling philosophy to the setting of HHSes with unbounded products, and then tweak our cubical models to encode this filling property into the structure of their hyperplanes.

\subsection{Filling pairs of domains}

We begin with a definition that is motivated by the notion of a pair of filling simple closed curves on a finite-type surface.

\begin{definition}[Filling domains]\label{defn:filling}
We say a pair of domains $U, V \in \mathfrak S$ is \emph{filling} if there does not exist a domain $W \in \mathfrak S$ with $\diam (\calC(W))= \infty$ so that $W \perp U$ and $W \perp V$.
\end{definition}

Observe that if $U,V$ are not filling, then $d_S(\rho^U_S, \rho^V_S) \leq 2E$.  On the other hand:

\begin{lemma} \label{lem:filling pairs}
Suppose $U, V \in \mathfrak S$ are filling and at least one of $\diam(\calC(U)) = \infty$ or $\diam(\calC(V)) = \infty$.  Then if $W \in \mathfrak S$ with $U \nest W$ or $V \nest W$, we have $W = S$.
\end{lemma}

\begin{proof}
Suppose instead that $W \sqsubsetneq S$.  Since at least one of $\diam(\calC(U)) = \infty$ or $\diam(\calC(V)) = \infty$, the Distance Formula \ref{thm:DF} provides that $\diam_{\calX}(\mathbb{F}_W) = \infty$, and hence $\diam_{\calX}(\mathbb{E}_W) = \infty$ by our unbounded products assumption.  But then Lemma \ref{lem:UP} implies that some $Z \perp W$ has $\diam(\calC(Z)) = \infty$, which contradicts the assumption that $U,V$ fill, giving us a contradiction.   
\end{proof}

\begin{lemma}\label{lem:filling UB}
For any $K_0 > \max\{B, E\}$, if $x,y \in \calX$ and $\calV \subset \Rel_{K_0}(x,y)$ is a pairwise filling collection of domains, then $\# \calV \prec \diam_S(\bigcup_{V \in \calV} \rho^V_S)$. 
\end{lemma}

\begin{proof}
Note that by choosing $K_0 > B$ (the bounded domain dichotomy constant, Definition \ref{defn:BDD}), we have that $\diam(\calC(V)) = \infty$ for all $V \in \calV$.  Then the Large Links Axiom \ref{ax:LL} provides domains $T_1, \dots, T_N \in \mathfrak S$ with $N \asymp d_S(x,y)$ so that each of the domains in $\calV$ nest into one of the $T_i$.  On the other hand, Lemma \ref{lem:filling pairs} says the only domain in which any pair of the domains in $\calV$ nests is $S$.  Hence each $V \in \calV$ nests into at most one $T_i$, and thus $\#\calV \leq N \asymp d_S(x,y)$, as required.
\end{proof}

\subsection{Well-separated cubical models} \label{subsec:well-sep cube}

In this subsection, we explain how to modify the standard cubical model for the hull of a pair of points $x,y\in \calX$, so that any collection of $0$-separated hyperplanes is labeled by a pairwise filling collection of domains.  The construction uses the setting of hierarchical family of trees as in Definition \ref{defn:HFT}, not in the HHS itself, as well as Tree Trimming \ref{thm:tree trimming}.

Fix points $x,y \in \calX$, their hull $H = \hull_{\calX}(x,y)$, and set of relevant domains $\calU = \Rel_K(x,y)$, where $K$ is the largeness constant as in Section \ref{sec:constants}, nd also larger than the bounded domain dichotomy constant $B$ so taht we can eventually apply Lemma \ref{lem:filling pairs}.  Let $(\hT_U)_{U \in \calU}$ be the associated collection of simplicial trees with their relative projections, as produced by Tree Trimming \ref{thm:tree trimming} as in Corollary \ref{cor:simplicial structure}.

We begin with the definition of extra domains which act as a firewall against fake $0$-separation in the cubical model for $x,y$:

\begin{definition}[Firewall domains]
For each $U \in \calU$ with $U \neq S$, the $\nest$-maximal domain, the \emph{firewall} for $U$ is a domain label $V_U$ and a copy of the unit interval $\hT_{V_U} = [0,1]$, along with the following relation and (collapsed) relative projection data:
\begin{itemize}
\item $V_U \nest S$ and $\hd^S_{V_U} \equiv 0$ is a constant map.
\item For each $W \in \calU$ with $W, U$ not filling, we set $V_U \perp W$.
\item If $W \in \calU - \{S\}$ forms a filling pair with $U$, then we set $W \pitchfork V_U$, $\hd^{V_U}_W = \hd^U_W$ and $\hd^W_{U_V} = 0$.
\item We set $\hx_{V_U} = \hy_{V_U} = 0$.
\end{itemize} 
\end{definition}

\begin{lemma}\label{lem:firewall HFT}
The set of domain labels $\calU' = \calU \cup \{V_U| U\in \calU - \{S\}\}$ along with its associated set of trees, projections, and relative projections, is a hierarchical family of trees.  Hence we can define its $0$-consistent set $\calQ'$.
\end{lemma}

\begin{proof}
We discuss the non-immediate items of Definition \ref{defn:HFT}.  Let $U \in \calU - \{S\}$.  The BGI property follows from defining $\hd^S_{V_U} \equiv 0$.     Finally, the tuples $(\hx_U), (\hy_U)$ are in $\calQ'$ because in each of the addition $\hT_{V_U}$ domains, the coordinate coincides with all possible relative projections (which all come from transverse relations).  This completes the proof.

\end{proof}

As proven in Theorem \ref{thm:dual}, the $0$-consistent $\calQ'$ is isometric to a cube complex, and we refer to $\calQ'$ as the \emph{firewall cubical model} for $\hull_{\calX}(x,y)$.  The following explains the relationship between $\calQ$ and $\calQ'$:

\begin{lemma}\label{lem:Q and Q'}
The map $\Delta:\calQ' \to \calQ$ which collapses each of the new unit intervals $\hT_{V_U}$ for $U \in \calU - \{S\}$ is a uniform quasi-isometry, with constants depending only on $\calX$.  Moreover, the map $\Xi:\prod_{U \in \calU} \hT_U \to \prod_{U \in \calU'} \hT_U$ which fixes coordinates for $U \in \calU$ and assigns $0$ for the domains in $\calU'-\calU$ is an isometric embedding.
\end{lemma}

\begin{proof}
The first statement is an immediate consequence of Tree Trimming \ref{thm:tree trimming}.  The second statement is an immediate consequence of the observation that the only new hyperplanes in $\calQ'$ that are not in $\calQ$ are labeled by domains in $\calU'-\calU$, and these new hyperplanes cannot separate vertices whose new coordinates are all the same by Lemma \ref{lem:wall separation}.
\end{proof}

With this in hand, we can prove the following, which is the main point of the construction:

\begin{proposition}\label{prop:firewall cubes}
A pair of $\calQ'$-hyperplanes $\hh, \hh'$ are $0$-separated if and only if their domain labels are filling.
\end{proposition}

\begin{proof}

We first observe that 

\begin{enumerate}
\item Any hyperplane in $\calQ'$ labeled by $S$ crosses no other hyperplane, and
\item For every $U \in \calU - \{S\}$, there is a hyperplane $\hh_{V_U}$ in $\calQ'$ labeled by $V_U$ which crosses every hyperplane $\hh$ which is labeled by a domain $W \in \calU'$ for which $U,W$ is not a filling pair.
\end{enumerate}

(1) is an immediate consequence of Proposition \ref{prop:pocset}.  (2) is an immediate consequence of the construction, plus Lemma \ref{lem:partition}, which says that the midpoint $\frac{1}{2} = h_{V_U} \in \hT_{V_U} = [0,1]$ determines a nonempty hyperplane $\hh_{V_U}$, and Proposition \ref{prop:pocset}, which says that $\hh_{V_U}$ crosses any hyperplane labeled by a domain orthogonal to $V_U$, which is, by definition, any domain not making a filling pair with $U$.

With these observations in hand, first suppose that $\hh, \hh'$ are $0$-separated, with domain labels $U, W \in \calU'$ respectively.  Then (1) and (2) above says that $U,W$ are a filling a pair (with possibly one or both being $S$).  On the other hand, if $U,W$ are a filling pair, then we are done by Lemma \ref{lem:filling pairs} plus our assumption that $(\calX, \mathfrak S)$ satisfies the bounded domain dichotomy.  In particular, by choosing the largeness constant $K$ to be larger than the constant $B$ in Definition \ref{defn:BDD}, we guarantee that the set of domains in $\calU$ orthogonal to both $U,W$ is empty.   Hence no hyperplanes can cross both $\hh,\hh'$ by Lemma \ref{lem:hyperplane char}.  This completes the proof.
\end{proof}

We also need the following observation which is true of any cubical model in any HHS $\calX$:

\begin{lemma}\label{lem:S hyp distance}
Let $x,y \in \calX$ be two points in any HHS and $\calQ$ any cubical model for the hull of $x,y$.  For any $U \in \calU$, if $m$ is the number of $\calQ$-hyperplanes labeled by $U$ separating $x$ from $y$, then $m \prec d_U(x,y)$, where the constant of $\prec$ depends only on $\calX$.
\end{lemma}

\begin{proof}

By item (4) of Lemma \ref{lem:collapsed tree control}, we have $m \leq d_{\hT_U}(\hx_U,\hy_U) \leq d_{T_U}(x_U, y_U) \prec d_U(x,y)$, so we are done.
\end{proof}

We are now ready to prove the upper bound:

\begin{proposition}\label{prop:sep UB}
For any $x,y \in \calX$ and the firewall cubical model $\calQ'$ for $\hull_{\calX}(x,y)$ satisfies
$$d_{\calC(S)}(x,y) \succ d^{\calQ'}_0(x,y),$$
with constants depending only on $\calX$.
\end{proposition}

\begin{proof}
 Let $\hh_1, \dots, \hh_n$ be a $0$-separated chain separating $\hx,\hy \in \calQ'$.  By Proposition \ref{prop:firewall cubes} and Lemma \ref{lem:filling UB}, the number $m$ of domains $U \sqsubsetneq S$ labeling hyperplanes in the chain satisfies
 $$m \prec \diam_S(\bigcup_{V \in \calV} \rho^V_S) \prec d_S(x,y),$$
  with the latter coarse inequality provided by the BGIA \ref{ax:BGIA}.  On the other hand, the number of domains in the chain labeled by $S$ is $n-m$.  Hence we are done by Lemma \ref{lem:filling UB} and Lemma \ref{lem:S hyp distance}.
\end{proof}

\bibliography{Recubulation}{}
\bibliographystyle{alpha}

 \end{document}